\documentclass[reqno,11pt]{amsart}
\usepackage{amsthm,amsfonts,amssymb,euscript,mathrsfs,graphics,color,amsmath,amssymb,latexsym,marginnote}
\usepackage{soul}

\usepackage{amsthm,amsfonts,amssymb,euscript,mathrsfs,graphics,color,amsmath,amssymb,latexsym,marginnote}

\usepackage{hyperref}

\usepackage[margin=1.05in]{geometry}
\numberwithin{equation}{section}

\def\C{{\mathbb{C}}}

\def\be{{\beta}}

\def\eps{\epsilon}

\def\si{\sigma}

\def\varep{\varepsilon}

\def\al{\alpha}

\def\sgn{\mathrm{sgn}}

\def\curl{{\,\mbox{curl}\,}}

\def\wt{\widetilde}
\def\bar{\overline}
\def\R{{\mathbb R}}
\def\e{{\varepsilon}}

\def\H{{\mathbb H}}

\def\R{{\bf R}}

\def\Z{{\mathbb Z}}

\def\bar{\overline}
\def\what{\widehat}
\def\R{\mathbb{R}}

\def\T{{\mathbb T}}

\definecolor{bluegreen}{rgb}{0.0, 0.3, 0.9}

\newtheorem{theorem}{Theorem}[section]
\newtheorem{lemma}[theorem]{Lemma}
\newtheorem{proposition}[theorem]{Proposition}

\newtheorem{definition}[theorem]{Definition}
\newtheorem{remark}[theorem]{Remark}

\begin{document}


\author{Yu Deng}
\address{University of Southern California}
\email{yudeng@usc.edu}

\author{Alexandru D. Ionescu}
\address{Princeton University}
\email{aionescu@math.princeton.edu}

\author{Fabio Pusateri}
\address{University of Toronto}
\email{fabiop@math.toronto.edu}

\thanks{
\noindent
Y. D. was supported in part by NSF DMS-1900251 and a Sloan Fellowship.
A. D. I. was supported in part by NSF grant DMS-2007008 and a Simons Collaboration Grant on Weak Turbulence.
F. P. was supported in part by a start-up grant from the University of Toronto, 
and NSERC grant RGPIN-2018-06487. 
}

\title[Wave turbulence theory of gravity water waves in 2D, I]{On the wave turbulence theory of 2D gravity waves, I:\\
deterministic energy estimates}


\begin{abstract}
Our goal in this paper is to initiate the rigorous investigation of wave turbulence and derivation of wave kinetic equations (WKE) 
for water waves models. This problem has received intense attention 
in recent years in the context of semilinear models, such as 
Schr\"odinger equations \cite{BGHS1, BGHS2, ColGer1, ColGer2, DeHa1, DeHa2, DeHa3} 
or multi-dimensional KdV-type equations \cite{StafTran, XMa}. 
However, our situation here is different since the water waves equations are quasilinear and 
solutions cannot be constructed by iteration of the Duhamel formula due to unavoidable derivative loss. 

This is the first of two papers in which we design 
a new strategy to address this issue. 
We investigate solutions of the gravity water waves system in two dimensions.
In the irrotational case this system can be reduced to an evolution equation on the one-dimensional
interface, which is a large torus $\T_R$ of size $R\geq 1$. 

Our first main result is a deterministic  
energy inequality which provides control of (possibly large) Sobolev norms of solutions
for long times, under the condition that a certain $L^\infty$-type norm is small. 
This energy inequality is of ``quintic'' type: if the $L^\infty$ norm is $O(\varepsilon)$, 
then the increment of the high order energies is controlled for times of the order $\varepsilon^{-3}$.
This is consistent with the approximate quartic integrability of the system, 
as proved recently in \cite{BeFePu,WuNew}. 
In the second paper in this sequence we will show how to use this
energy estimate and a propagation of randomness argument to prove a 
probabilistic regularity result up to times of the order $\e^{-3} R^{3/2-}$, 
in a suitable scaling regime relating $\e \rightarrow 0$ and $R \rightarrow \infty$. 

For our second main result we combine the quintic energy inequality with a bootstrap argument 
using a suitable $Z$-norm of Strichartz-type
to prove that deterministic solutions with Sobolev data of size $\e \ll 1$ 
are regular for times of the order $\e^{-3} \min \big( \e^{-3}, R^{3/4} \big)$. 
In particular, on the real line
solutions exist for times of order $O(\varepsilon^{-6})$. 
This improves substantially on the earlier extended lifespan results 
with times of order $O(\varepsilon^{-2})$ (proved in \cite{WuAG, IoPu2, ADb, HIT}) 
or of order $O(\varepsilon^{-3})$ (proved recently in \cite{WuNew}) or of order $O(\varepsilon^{-4})$ (proved recently in \cite{Zheng22}).
\end{abstract}

\maketitle

\setcounter{tocdepth}{1}

\begin{quote}
\tableofcontents
\end{quote}

\medskip
\section{Introduction}

\subsection{Free boundary Euler equations and water waves} 
The evolution of an inviscid perfect fluid that occupies a domain $\Omega_t \subset \R^n$, for $n \geq 2$,
at time $t$,
is described by the free boundary incompressible Euler equations.
If $v$ and $p$ denote the velocity and the pressure of the fluid (with constant density equal to $1$)
at time $t$ and position $x \in \Omega_t$, these equations are
\begin{equation}\label{E}
(\partial_t + v \cdot \nabla) v = - \nabla p - g e_n,\qquad \nabla \cdot v = 0,\qquad x \in \Omega_t,
\end{equation}
where $g$ is the gravitational constant. The first equation in \eqref{E} is the conservation of momentum equation, while the second is the incompressibility condition.
The free surface $S_t := \partial \Omega_t$ moves with the normal component of the velocity according to the kinematic boun\-dary condition
 \begin{equation}\label{BC1}
\partial_t + v \cdot \nabla  \,\, \mbox{is tangent to} \,\, {\bigcup}_t S_t \subset \R^{n+1}_{x,t}.
\end{equation}
The pressure on the interface is given by
\begin{equation}
\label{BC2}
p (x,t) = \sigma \kappa(x,t),  \qquad x \in S_t,
\end{equation}
where $\kappa$ is the mean-curvature of $S_t$ and $\sigma \geq 0$ is the surface tension coefficient. 
In this paper we consider the {\it{gravity water waves system}}, 
which corresponds to the case when the surface tension coefficient $\sigma$ 
is equal to $0$ and, without loss of generality, the gravity coefficient $g$ is equal to $1$.

The free boundary Euler equations \eqref{E}-\eqref{BC2} can be analyzed in various 
types of domains $\Omega_t$ (bounded, periodic, unbounded), for fluids with different 
characteristics (with or without vorticity, with gravity and/or surface tension),
or even more complicated scenarios where the moving interface separates two fluids. 
In this paper we consider fluid domains described
as regions in the plane below the graph of function over $\T_R^d :=[\R/(2\pi R\Z)]^d$, $R\geq 1$.
We restrict our attention to irrotational flows, i.e. $\rm{\curl} v = 0$, 
in which case one can reduce \eqref{E}-\eqref{BC2} to a system on the boundary.
More precisely let $h=h(x,t):\T_R^d \times I \rightarrow \R$, and let
\begin{align*}
\Omega_t = \{ (x,y) \in \T_R^d\times \R \, : y \leq h(x,t) \} \quad \mbox{and} \quad S_t = \{ (x,y) : y = h(x,t) \}.
\end{align*}
Let $\Phi$ denote the velocity potential, $\nabla_{x,y} \Phi(x,y,t) = v (x,y,t)$ for $(x,y) \in \Omega_t$.
If $\phi(x,t) := \Phi (x, h(x,t),t)$ is the restriction of $\Phi$ to the boundary $S_t$,
the equations of motion reduce to a system of equations for the functions 
$h, \phi : \T_R^d\times I \rightarrow \R$:
\begin{equation}\label{gWW}
\begin{cases}
&\partial_t h = G(h) \phi,
\\
&\partial_t \phi = -h - \dfrac{1}{2} |\nabla_x\phi|^2 + \dfrac{{\left( G(h)\phi
  + \nabla_xh\cdot\nabla_x\phi\right)}^2}{2(1+|\nabla_xh|^2)}.
\end{cases}
\end{equation}
Here
\begin{equation}
\label{defG0}
G(h) := \sqrt{1+|\nabla_xh|^2} \mathcal{N}(h),
\end{equation}
and $\mathcal{N}(h)$ is the Dirichlet-Neumann map associated to the domain $\Omega_t$.

The system \eqref{gWW} admits the conserved Hamiltonian
\begin{equation}\label{gWWHam}
\mathcal{H}(h,\phi) := \frac{1}{2} \int_{\T_R^d} G(h)\phi \cdot \phi \, dx + \frac{1}{2} \int_{\T_R^d} h^2 \, dx,
\end{equation}
which is the sum of the kinetic energy corresponding to the $L^2$ norm of the velocity 
field and the potential energy due to gravity. We refer to  \cite[chap. 11]{SulemBook} 
or the book of Lannes \cite{LannesBook} for the derivation of the system \eqref{gWW}. 
It was first observed by Zakharov \cite{Zak0} that \eqref{gWW} is the Hamiltonian flow associated to \eqref{gWWHam}.

The local well-posedness theory of water waves systems is well-understood in a variety of different scenarios, due to the contributions of many authors.
Without being exhaustive we mention \cite{Wu1,Wu2,CL,Lannes,CS,ShZ3,ABZ1} 
and refer the reader to the survey \cite[Section $2$]{IoPuRev} or to the book of Lannes \cite[Chapter 4]{LannesBook}
for extensive discussions concerning local well-posedness results for general water waves systems, and a more complete list of references. 
In short, for sufficiently regular Sobolev initial data classical smooth solutions exist
on a (small) time interval that depends on the size of the initial data 
and the arc-chord constant of the initial interface. 

\subsection{Motivation and main ideas}

Our work is motivated by the outstanding problem of understanding 
rigorously the wave turbulence theory of water waves.
This is a classical problem in Mathematical Physics, going back to work of Peierls 
\cite {Pei} and Hasselmann \cite{Has1, Has2}; 
we refer the reader to the recent book \cite{Naz} 
for a more complete description of the problem and many examples.

The wave turbulence problem has received intense attention in recent years, mainly in the context of 
dispersive models such as semilinear Schr\"{o}dinger equations;
see for example \cite{BGHS1, BGHS2, ColGer1, ColGer2, JinMa, DeHa1, DeHa2, DeHa3, StafTran}. 
The focus in these papers has been the rigorous derivation of the {\it{wave kinetic equation}} (WKE).
At the analytical level, this amounts 
to understanding precisely the dynamics of solutions corresponding to randomized initial data, 
in the large box and weakly nonlinear limit, and on long time intervals.

Our goal here is to initiate a similar investigation for irrotational water waves models. 
There is a fundamental difference between water waves, which are quasilinear evolutions, 
and semilinear models: the solution cannot be constructed using the Duhamel formula,
due to unavoidable loss of derivatives. In particular, the solution cannot be described using 
iterations of the Duhamel formula (Feynman trees), which is the central idea of the analysis of semilinear models. 

To address this fundamental difficulty we propose here a new mechanism, based on a combination of two main ingredients:

\smallskip
(1) {\it{Energy estimates:}} We prove first estimates on the energy increment 
of the highest order derivatives of the solution. 
These are {\it{deterministic}} estimates, in the sense that they hold for all solutions. 
The key point is to show that the increment is controlled by an $L^\infty$-based 
norm of the solution (not an $L^2$-based norm), which can then be exploited effectively
in the analysis of solutions with random initial data.

\smallskip
(2) {\it{Propagation of randomness and pointwise bounds:}} 
The second step consists in proving pro\-pa\-gation of randomness 
(in a suitable form), and using Khintchine-type arguments to
derive optimal pointwise bounds on randomized solutions in a large box. 
For this one can use the iterated Duhamel formula; the point is that losses 
of derivatives are allowed at this stage, thanks to the accompanying estimate on the high order energy.

\smallskip
In this paper we complete the first part of this program, for the problem of pure
gravity water waves in two dimensions (1D interface). More precisely we prove a refined energy inequality for solutions of the system \eqref{gWW} in a large box $\T_R$, $R\geq 1$; see Theorem \ref{IncremBound}.
The a priori energy inequality \eqref{BootstrapFull}-\eqref{MainGrowth} 
is of ``acceptable quintic'' type, consistent with the ``almost integrability'' of the system \eqref{gWW}
(on $\T$) up to order 4; see the works of Dyachenko-Zakharov \cite{Zak0},
Craig-Worfolk \cite{CW} and Craig-Sulem \cite{CS} on the ``almost integrability'' of the system,
and the rigorous proofs by Berti-Feola-Pusateri \cite{BeFePu} and Wu \cite{WuNew}. 
The second part of the program will be completed in our second paper \cite{DeIoPu2}. 

As a byproduct of our analysis, in this paper we also show that we can use this ``acceptable quintic energy inequality'' 
to prove a deterministic improved lifespan 
result on solutions of Sobolev size $\varepsilon\ll 1$. 
In particular we show that in the Euclidean case $R=\infty$ solutions of 
the gravity water waves system of Sobolev size $O(\varepsilon)$
remain regular on times of order $O(\varepsilon^{-6})$. 
This improves significantly on all the earlier enhanced lifespan results, 
with times of order $O(\varepsilon^{-2})$ (proved in \cite{WuAG, IoPu2, ADb, HIT})
or of order $O(\varepsilon^{-3})$ (proved in \cite{WuNew}) or of order $O(\varepsilon^{-4})$ (proved in \cite{Zheng22}).



\smallskip
\subsection{An acceptable quintic energy inequality}\label{IntroQuin} 
Our first main theorem provides a quintic a priori estimate on the increment of energy of solutions 
of the gravity water waves system in 2D.

\begin{theorem}\label{IncremBound}
Assume that $N_0\geq 6$, $N_1\in[4,N_0-2]$, $T_1\leq T_2\in\mathbb{R}$, and $(h,\phi)\in C([T_1,T_2]:H^{N_0}\times \dot{H}^{N_0,1/2})$
is a solution of the system \eqref{gWW}. Moreover, assume that\footnote{See 
\eqref{HW} for the definition of the spaces $\dot{W}^{N,b}$ and 
Lemma \ref{DNmainpro} for the definition of the good unknown $\omega=\phi-T_Bh$.}  
\begin{equation}\label{BootstrapFull}
\begin{split}
\|h(t)\|_{H^{N_0}}+\||\partial_x|^{1/2}\omega(t)\|_{H^{N_0}} +\||\partial_x|^{1/2}\phi(t)\|_{H^{N_0-1/2}}&\leq A,
\\
\|h(t)\|_{\dot{W}^{N_1,0}}+\||\partial_x|^{1/2}\omega(t)\|_{\dot{W}^{N_1,0}}
  + \||\partial_x|^{1/2}\phi(t)\|_{\dot{W}^{N_1-1/2,0}}&\leq \varep(t),
\end{split}
\end{equation}
for any $t\in[T_1,T_2]$, where $A\in (0,\infty)$ and $\varep:[T_1,T_2]\to [0,\infty)$ 
is a continuous function satisfying $\varep(t)\ll 1$.
Then, for any $t_1\leq t_2\in[T_1,T_2]$,
\begin{equation}\label{MainGrowth}
\big\|(h+i|\partial_x|^{1/2}\omega)(t_2)\big\|_{H^{N_0}}^2
  \lesssim \big\|(h+i|\partial_x|^{1/2}\omega)(t_1)\big\|_{H^{N_0}}^2+A^2\int_{t_1}^{t_2}[\varep(s)]^3\,ds,
\end{equation}
\end{theorem}

\subsubsection{Remarks}\label{ssecdisc1}

We make some remarks on the statement of Theorem \ref{IncremBound} and discuss some of the main ideas of its proof.

\smallskip

(1) {\it The acceptable quintic increment}.
The main point of the theorem is to show that the energy increment is quintic and acceptable, 
that is, the integral term in \eqref{MainGrowth} involves $A^2$ and the third power
$[\varep(s)]^3$ of the $L^\infty$-based norm.
For future applications when $R$ is large, in particular in the random data setting and to the small-data result in Theorem \ref{LongTermDet} below,
we emphasize that it is crucial to distinguish these two quantities and allow $A$ (the high Sobolev energy) to be large,
while only $\e(t)$ (the controlling $L^\infty$ norm) is kept small. Acceptable energy increments can contain only 2 copies of the large Sobolev norm of the solution.

Several ``quartic'' energy inequalities, which improve on the standard cubic 
estimates used to establish local well-posedness, 
have been proved in the last ten years or so, in both one and two dimensions; 
see for example \cite{WuAG,Wu3DWW,IoPu2,ADb,HIT,ZhengWW} for gravity waves.
The key common ingredient in all of these results is the absence of quadratic resonances. 

Berti-Feola-Pusateri \cite{BFPnote,BeFePu} proved a quintic energy inequality in dimension $d=1$ for $R=1$, 
as a consequence of the conjugation of the flow of \eqref{gWWHam} to its integrable Birkhoff normal form
at order $4$. More recently, Wu \cite{WuNew} also proved a quintic energy inequality for the problem on $\R$. 

Our quintic energy inequality \eqref{MainGrowth} in the case $R=\infty$ is not directly comparable 
with the result of Wu. Our main concern, motivated by applications to random data
problems and wave turbulence, is to show that the energy 
increment can be bounded using 3 copies of the $L^\infty$-type norm in \eqref{BootstrapFull}, 
even for solutions of large (high-order) energy; this is what we call {\it acceptable}.
The focus of \cite{WuNew} instead is to prove a quintic energy inequality 
for energy-small initial data at improved levels of regularity and allow for the presence of steep waves.

\smallskip
(2) {\it Relevance to the random data problem in the large box limit}.
As explained earlier, our main goal is to apply Theorem \ref{IncremBound} to investigate the 
energy distribution for large times of 
solutions of \eqref{gWW} with randomized initial data, in the limit $R \rightarrow \infty$.
To do this, we first need to construct smooth solutions on long time intervals. While long-time
or even global regularity is not an issue in many semilinear problems, even at the deterministic level, 
this becomes a significant difficulty in quasilinear problems, where conservation laws are 
scaling-supercritical and cannot be used.
High-order energy inequalities such as \eqref{MainGrowth} are essentially the only known tool available
in the quasilinear setting to propagate regularity.

Roughly speaking, the randomization of the data can give enhanced control on $L^\infty$-type norms. More precisely, for a solution with Sobolev norm of size $A$ one can hope to prove that, with high probability, the $L^\infty$-norm is of size $\varepsilon\approx AR^{-1/2}$, consistent with the conservation of mass 
and the solution being uniformly distributed over $\T_R$. 
In principle, one can hope to prove such an enhanced $L^\infty$ estimate by propagating randomness (using Feynman's diagram expansions in the spirit of the recent works on NLS cited earlier), and combine this with the deterministic energy estimate \eqref{MainGrowth} to construct regular solutions up to times of the order $\e^{-3}$. 
We will prove a rigorous result of this type in the second paper of this sequence \cite{DeIoPu2}. 

\smallskip
(3) {\it Quadratic resonances.}
The main obstruction to obtaining an energy inequality of quintic type is the presence
of quadratic and cubic resonances in the system.
It is well-known that (in any dimension) there are no non-trivial quadratic resonances 
(or resonant three-waves interactions) for \eqref{gWW}; 
that is, there are no numbers $\xi_1,\xi_2,\xi_3\in\mathbb{R}\setminus \{0\}$ solving
\begin{equation}\label{quadres}
\sigma_1 |\xi_1|^{1/2} + \sigma_2 |\xi_2|^{1/2} + \sigma_3 |\xi_3|^{1/2} = 0,
  \qquad \xi_1 + \xi_2 + \xi_3 = 0,
\end{equation}
where $\sigma_1,\sigma_2,\sigma_3\in\{-1,1\}$.
This algebraic fact implies that the quadratic nonlinearities are inessential to the dynamics
and can be (at least formally) eliminated by normal forms or time-averaging.
Using appropriate variables, such as the ``good unknown'' $\omega$, or
Riemann mapping or holomorphic coordinates,
one can avoid losses of derivatives in normal form procedures 
and obtain an ``acceptable quartic'' energy inequality of the form 
\begin{equation}\label{MainGrowth2}
\big\|(h+i|\partial_x|^{1/2}\omega)(t_2)\big\|_{H^{N_0}}^2
  \lesssim \big\|(h+i|\partial_x|^{1/2}\omega)(t_1)\big\|_{H^{N_0}}^2+A^2\int_{t_1}^{t_2}[\varep(s)]^2\,ds.
\end{equation}
Basic inequalities of this type have been proven in various ways in several works, 
including the already cited \cite{WuAG,Wu3DWW,IoPu2,IoPu4,ADa,IT,IT2,ZhengWW}; 
for a longer discussion and more references also about the Klein-Gordon equation
and parameter dependent PDEs, see \S1.4 in \cite{IoPu5}.

\smallskip
(4) {\it Cubic resonances.} To improve \eqref{MainGrowth2} and obtain \eqref{MainGrowth} 
one needs to look at ``cubic resonances'' (or resonant 4-waves interactions) 
that is, solutions of the system
\begin{equation}\label{cubicres}
\sigma_1 |\xi_1|^{1/2} + \sigma_2 |\xi_2|^{1/2} + \sigma_3 |\xi_3|^{1/2} + \sigma_4 |\xi_4|^{1/2} = 0,
  \qquad \xi_1 + \xi_2 + \xi_3 +\xi_4 = 0,
\end{equation}
for $\xi_j\in\R\setminus\{0\}$ and $\sigma_j\in\{-1,1\}$, $j=1,\dots 4$.
Equation \eqref{cubicres} has a set of trivial solutions of the form $(\xi,\eta,-\xi,-\eta)$ with $\sigma_1=\sigma_2=1$,
$\sigma_3=\sigma_4=-1$,
together with their natural permutations; these are the so-called ``trivial resonances''
and cannot be avoided.
However, due to the Hamiltonian/reversible nature of the equations, these resonances give
an integrable contribution and do not appear in the energy increment.
We refer the reader to the introduction of \cite{BeFePu} for a more detailed discussion about this.

Moreover, a simple argument shows that \eqref{cubicres} has exactly one other two-parameter family of solutions (up to natural symmetries),
the so called {\it{Benjamin-Feir resonances}} \cite{DZ,CW,DLZ}, which are given by $(\sigma_1,\sigma_2,\sigma_3,\sigma_4) = (1,1,-1,-1)$ and
\begin{equation}\label{BFres}
\mathcal{R}_{BF} := \bigcup_{\lambda\in\R \setminus \{0\},\,b\in[0,\infty)} \Big\{ 
 \xi_1 = - \lambda (b^2+b+1)^2,\,\xi_2 = \lambda b^2,\,\xi_3 = \lambda (b+1)^2,\,\xi_4= \lambda (b+1)^2 b^2 \Big\}.
\end{equation}
The resonances \eqref{BFres} are a potential obstruction to proving quintic energy inequalities. 
The key point is that one can show that the contribution of these resonances to the energy increment also vanishes.
This is due to the fact that the symbol of the cubic interactions (after the first normal form step)
vanishes on the set \eqref{BFres}, after suitable symmetrizations.
This algebraic fact was already known since the works \cite{DZ,CW,DLZ,CS}.

However, substantial work is needed to turn this into a rigorous quintic energy inequality.
The first step is to calculate precisely the quartic Fourier multipliers appearing in our problem 
after performing the first normal form; see Lemmas \ref{quartic1}, \ref{quartic2}, and \ref{quartic3}.
In order to perform a second normal form step and obtain acceptable quintic expressions,
we need to identify a proper class of ``admissible symbols'' (see Definition \ref{defA}), 
which allows certain types of singularities. These singularities are present in our problem and cannot be avoided; 
such singularities could in principle lead to quintic bulk terms 
containing Hilbert transforms acting on the product of two unknowns, 
and preclude the basic inequality \eqref{MainGrowth}. 
The point is that this does not happen in our case due to the special structure of the equation.

Finally we need to verify that the symbols vanish at Benjamin-Feir resonant sites defined in \eqref{BFres}, and justify the division.
For this we rely instead on an explicit algebraic computation, which is done in Section \ref{uyt22sec} and requires full symmetrization of all the symbols. The symbols of the quartic bulks obtained in this way are long algebraic expressions with no clear simplifications, and the calculations necessary to verify the vanishing 
at the non-trivial resonances are lengthy, but manageable.{\footnote{This vanishing is expected, however, for conceptual reasons, as it is already known to hold at the formal level (the key observation of Zakharov and Dyatchenko \cite{Zak0}), which corresponds to setting $\chi=0$. We need a nontrivial function $\chi$ in our paper, in order to define correctly paradifferential operators, which increases substantially the complexity of the calculation, but this should be thought of as a technical tool needed to avoid derivative loss which does not affect the intrinsic structure of the system.} Once vanishing is proved, the division can be justified easily due to the smoothness properties of the symbols.

\smallskip
(5) {\it{Eulerian coordinates.}} In this paper we work in Eulerian coordinates. While other choices are possible
(such as Lagrangian coordinates as in \cite{WuAG,WuNew} or holomorphic coordinates as in \cite{HIT,IT2}), 
at the conceptual level Eulerian coordinates are the natural and robust choice to analyze small-amplitude waves, 
as they apply equally well to both two and three-dimensional problems. 
This choice leads to slightly longer calculations compared to other choices,
but these calculations are explicit and can be performed without any computer assistance.
\smallskip

(6) {\it Classes of symbols.} At a basic level, many of our energy estimates rely on controlling quartic expressions of the form
\begin{equation}\label{intro7}
I_M:=\Big|\frac{1}{(2\pi R)^3}\sum_{(\xi_1,\xi_2,\xi_3,\xi_4)\in\mathbb{H}^3_R}\widehat{h_1}(\xi_1)\widehat{h_2}(\xi_2)\widehat{h_3}(\xi_3)\widehat{h_4}(\xi_4) M(\xi_1,\xi_2,\xi_3,\xi_4)\Big|,
 \end{equation}
where $h_1,h_2,h_3,h_4$ are functions related to water waves variables and $\mathbb{H}^3_R:=\{(\xi_1,\xi_2,\xi_3,\xi_4)\in(\Z/R)^4:\,\xi_1+\xi_2+\xi_3+\xi_4=0\}$. Such expressions appear many times when estimating bulk contributions (in the Fourier space), for various multipliers $M$. In our applications, we have control of the functions $h_l$ in $L^2$-based Sobolev norms, $\|h_l\|_{H^{N_0}}\leq A_2$, as well as in $L^\infty$-based Sobolev norms, $\|h_l\|_{\widetilde{W}^2}\leq A_\infty$, and we would like to prove sharp bounds of the form
\begin{equation}\label{intro7.1}
I_M\lesssim A_2^2A_\infty^2.
\end{equation}

In order to prove such bounds (uniformly in $R$) we need good information on the symbols $M$. The ideal situation is when the symbols $M$ are {\it{regular}}, as described precisely in Definition \ref{defMNab}, satisfying suitable differential bounds of the H\"{o}rmander-Michlin type. But this is insufficient\footnote{These considerations are important only as $R\to\infty$, in order to prove bounds that are uniform in $R$. When $R$ is bounded, differentiation of symbols is not needed and one can use just pointwise bounds to measure symbols. At a more basic level, $L^\infty$-based norms in the physical space are not needed at all when $R$ is bounded, since $L^\infty$ is in fact stronger than $L^2$ in this case. This effectively eliminates one of the key difficulties of the problem, namely the requirement that acceptable energy estimates can contain only two copies of the (large) Sobolev norms.}  in our case: while many of our symbols are indeed regular, the application of normal forms leads to symbols that contain factors of the form $|\xi_i+\xi_j|$, which do not obey suitable derivative bounds. To deal with these contributions we need to define a special class of symbols, which we call {\it{admissible}}, see Definition \ref{defA}, which are well adapted to our problem. These symbols are more general than the regular symbols, but still lead to optimal quartic bounds like \eqref{intro7.1} and retain sufficient smoothness (after factoring out the singular factors) to allow for symmetrization and control of small divisors.
\smallskip

(7) {\it Small divisors.} One of the main issues in implementing normal form procedures is the loss of derivatives
that could occur when dividing by resonant phases.
On a fixed torus \cite{BeFePu} these losses can be avoided thanks to 
paradifferential reductions to smoothing operators.
In this paper we do not reduce matters to smoothing operators
(and this may not be actually possible uniformly as $R\rightarrow \infty$).
Our strategy here is different, based on using explicit formulas and identifying suitable properties of the symbols. 
The general idea is that if one knows that a symbol (1) vanishes at a resonant site
and (2) satisfies suitable differential bounds (such as those satisfied by regular symbols and by admissible symbols), 
then one can solve the division problem with no loss of derivatives.
For a relatively simple implementation of this idea
we refer the reader to Step 2 in the proof of Lemma \ref{hol2},
which deals with the frequency configuration where Benjamin-Feir resonances are present
and the symbols are regular.
The implementation in the case of the trivial resonances in the subsequent steps is more involved and uses
the full structure of the admissible class of symbols.
\smallskip

(8) {\it Symmetrization of symbols.} Effective estimates, for example on quartic expressions like \eqref{intro7}, can only be proved after accounting and exploiting the symmetries of the problem. For example, if $h_1=h_2$ then one can replace the symbol $M$ with its symmetrized version $[M(\xi_1,\xi_2,\xi_3,\xi_4)+M(\xi_2,\xi_1,\xi_3,\xi_4)]/2$ in the first two variables. 

In our proof, symmetrization serves three basic purposes: (i) in certain cases it eliminates the derivative loss coming from the quasi-linear nature of the problem and normal forms, (ii) it allows us to verify the key vanishing condition of the quartic symbols at the Benjamin-Feir resonant sites, and (iii) at a more subtle level, it prevents the emergence of worse singularities at low frequencies (like $\sgn(\xi_i+\xi_j)$ which appear naturally before symmetrization due to normal forms and Hilbert transforms) which would ultimately prevent us from proving \eqref{MainGrowth}.
\smallskip

(9) {\it Optimality.}
Finally, we note that a quintic energy inequality like \eqref{MainGrowth}
is expected to be optimal in terms of homogeneity in $\e(t)$, 
in view of the presence of quintic resonances as exhibited by Craig-Worfolk \cite{CW} 
and Dyachenko-Lvov-Zakharov \cite{DZ}. 
\smallskip

\subsection{Small data long-term regularity}\label{Application} 
As discussed above, Theorem \ref{IncremBound} is the first of two main steps
towards understanding the energy distribution for gravity water waves with random initial data.
In addition, we can also utilize it in the case of deterministic 
initial data with small Sobolev norm, to prove new extended lifespan results. 
More precisely we have:

\begin{theorem}\label{LongTermDet}
Assume that $N_0\geq 7$ and $(h_0,\phi_0)\in H^{N_0+1/2}(\T_R)\times H^{N_0+1/2}(\T_R)$ 
is initial data satisfying the assumptions
\begin{equation}
\label{mai1}
{\|\,\langle\partial_x\rangle^{1/2}h_0\|}_{H^{N_0}} + {\| \,|\partial_x|^{1/2} \phi_0\|}_{H^{N_0}} \leq \e \leq 1.
\end{equation}
Then there is a unique solution $(h,\phi)\in C([0,T_{R,\varep}]:H^{N_0} \times \dot{H}^{N_0,1/2})$ 
of the system \eqref{gWW} on $\T_R\times [0,T_{R,\varep}]$, with initial data $(h(0),\phi(0)) = (h_0,\phi_0)$, 
where
\begin{equation}\label{mai2}
T_{R,\varep} \gtrsim \e^{-3} \min \big( \e^{-3}, R^{3/4} \big).
\end{equation}
Moreover, letting $\omega = \phi - T_B h$ as in \eqref{DefBV}, the solution satisfies the bounds 
\begin{align}\label{mai3}
\sup_{t\in[0,T_{R,\varep}]} & {\big\|(h+i|\partial_x|^{1/2}\omega)(t) \big\|}_{H^{N_0}} \lesssim \e.
\end{align} 

\end{theorem}

\subsubsection{Remarks and main ideas}\label{ssecdisc2}

The main point of Theorem \ref{LongTermDet} is the enhanced lifespan of solutions corresponding 
to small initial data in Sobolev spaces, without any additional localization. This problem has been considered by many authors and for many dispersive models. We discuss below some related results in this direction, and some of the main ideas of our proof.

\smallskip
(1) {\it The case of a fixed-size torus}.
When $R\approx 1$, the local well-posedness theory
gives a time of existence of $O(\varep^{-1})$.
Existence times of the order $O(\varep^{-2})$ for gravity waves
follow the from the quartic energy inequalities in the already cited \cite{WuAG,IoPu2,ADb,HIT}.
Similar results have also been proved for capillary waves ($g=0,\sigma>0$) in \cite{IoPu4,IT2} and gravity-capillary waves ($g,\sigma>0$) in \cite{BeFeFr}.  

The more recent papers \cite{BeFePu, WuNew} provided the first regularity results up to times $O(\e^{-3})$. 
Theorem \ref{LongTermDet} matches these results in the case of a fixed torus $R\approx 1$, but improves significantly for large tori $R\gg 1$. 
This result is probably optimal for general Sobolev data due to the presence of resonant 
$5$-wave interactions, see \cite{CW,DLZ}.

Extended lifespan results have also been proved for many other dispersive and hyperbolic models. 
Without attempting to be exhaustive, we mention the monograph of Berti-Delort \cite{BertiDelort},
who obtained regularity up to times $O(\e^{-N})$, for any fixed $N$, for standing gravity-capillary waves over $\T$, 
for almost all values 
of the surface tension and gravity parameters.
In general, longer existence times can be achieved by iterating normal forms procedures, 
when non-trivial resonances can be avoided for almost all values of external parameters.
See also the work of Delort on Klein-Gordon equations \cite{Del2} and references therein. 
In the case $d=2$, an extended $O(\varep^{-2})$ lifespan for pure gravity waves
is obtained implicitly in the work of 
Wu \cite{Wu3DWW}, and is included 
in the more general results of Zheng \cite{ZhengWW} on $\mathbb{T}^2_R$.
For the gravity-capillary case on $\mathbb{T}^2$ an extended lifespan
for almost all values of the parameters $(g,\sigma)$ was obtained in \cite{IoPu5}.

\smallskip
(2) {\it Large $R$ and the Euclidean case}. If $R\gtrsim\e^{-4}$ or in the Euclidean case $R=\infty$, 
we obtain a time of existence $T_{R,\varep}\approx \varep^{-6}$. 
This improves substantially on all the previous results on 2D water waves with Sobolev initial data,
including the recent $O(\e^{-3})$ lifespan result of \cite{WuNew} for $R=\infty$ based on a quartic-type energy inequality,
and the very recent $O(\e^{-2} \min(\e^{-2}, \sqrt{R}))$ 
result of Zheng \cite{Zheng22} which uses the simpler quartic energy inequality \eqref{MainGrowth2} and a version of a Strichartz norm of the form \eqref{Zmai}. Similar extended lifespan results,
depending on the size $\varepsilon$ of the initial data and the size $R$ of the periodic box, 
have been proved recently by Zheng \cite{ZhengWW} in the case of pure gravity waves in 3D.

We mention that in the Euclidean case $R=\infty$, 
for localized initial data in a weighted Sobolev space, one can actually prove full global
regularity and scattering results,
relying on dispersion and pointwise decay.
We refer the reader to \cite{WuAG,IoPu2,ADb,IT,Wa1} for the case of $1$d interfaces, 
and to \cite{Wu3DWW,GMS2,DIPP,Wa2} for $2$d interfaces (see \cite[Section 3]{IoPuRev} 
for an overview of these results). However, localization is not compatible with randomized initial data,
so these results are not applicable to the problem of understanding wave turbulence.

\smallskip
(3) {\it Ideas of the proof}.
Our proof of Theorem \ref{LongTermDet} is based on what we call {\it{the $Z$-norm method}}. 
This general method is essentially a bootstrap argument 
that combines high-order energy estimates with dispersive and decay 
estimates at a lower order of regularity. 
In our problem the key $Z=Z_{\varep,R}$-norm is defined by means of the following Strichartz-type norm
\begin{align}\label{Zmai}
{\|F\|}_Z := \sup_{t_1\leq t_2\in[0,T],\,|t_2-t_1|\leq\,\min(\e^{-4},R)} &
  {\big\|F\big\|}_{L^4_{[t_1,t_2]}\dot{W}_x^{N_1-1/4,0}},
\end{align} 
for any function $F\in C([0,T]:H^{N_0}(\T_R))$, where $N_1 = N_0-3 \geq 4$.
See the second norm in \eqref{mai3'}, and \eqref{HW}, for the exact definition of the norm.

The main new ingredient is the Strichartz estimates in Lemma \ref{LemStrichartz},
which provide sharp $L^4_tL^\infty_x$ decay of linear solutions in Sobolev spaces.
We can then combine these Strichartz estimates (which are, essentially, 
efficient substitutes for Sobolev's embedding) and the control on high Sobolev norms from \eqref{MainGrowth}, 
to propagate regularity of solutions on the time interval $[0,T_{R,\varepsilon}]$,
with $T_{R,\varepsilon}$ as in \eqref{mai2}.

\subsection{Organization} In Section \ref{SecPrelim} we collect several basic definitions and results,
including some classes of multipliers that we will encounter, multilinear product estimates, 
the definitions of our function spaces, and elements of pseudo-differential calculus.

In Section \ref{paralinearization} we paralinearize of the water waves system, using 
the good unknown of Alinhac and following the general approach of Alazard-Burq-Zuily \cite{ABZ1,ABZ2} 
and our paralinearization of the Dirichlet-Neumann (DN) operator in Appendix B of \cite{IoPu4}.
The proofs of the main technical results, Lemmas \ref{DNmainpro} and \ref{BVexpands},
are postponed to Section \ref{TechnicalProofs}.
The equations that are eventually derived in Lemma \ref{Usor1} 
for the good variables $(h^\ast,\omega^\ast)$ are the starting point for our analysis.

In Section \ref{EnergyBounds} we define our energy functional \eqref{eni3}
and provide a bound for its increment in Lemma \ref{quartic1}, in terms of quartic ``bulk terms'';
this corresponds to a first normal form. We then rewrite these quartic bulk integrals in terms of our main variables $(h^\ast,\omega^\ast)$. 
These are mostly straightforward algebraic calculations, 
but need to be performed very accurately in order to be able to verify the 
necessary vanishing conditions at the non-trivial resonances later on.

In Section \ref{QuartGen} we prove acceptable quintic bounds for quartic bulk integrals, 
under suitable assumptions on the defining symbols. The key issue is to identify a suitable class 
of ``admissible multipliers'' (see Definition \ref{defA}), 
sufficiently large to capture the multipliers encountered in our problem, 
but restrictive enough to allow an additional normal form. The main result of this section, Lemma \ref{hol2},
gives the desired quintic-type bounds in the resonant case, 
provided that certain vanishing and regularity conditions are met.

The main acceptable quintic energy inequality \eqref{MainGrowth} is obtained in Subsection \ref{ssecquintic},
as a consequence of the work done in Sections \ref{uyt22sec} and \ref{secadmver}
to verify the necessary admissibility and vanishing conditions for the symbols close to both sets of resonances. 
To deal with the Benjamin-Feir resonances we need to pass to complex variables and fully symmetrize the expressions.
 
In Section \ref{Deterministic} we prove Theorem \ref{LongTermDet}. 
We first prove Strichartz estimates for the linear flow $e^{it|\partial_x|^{1/2}}$ on $\T_R$ 
in Lemma \ref{LemStrichartz}
and then bootstrap the $Z$-norm \eqref{Zmai} in Subsection \ref{LongTermDetpr}.

\medskip
\section{Preliminaries}\label{SecPrelim}

In this section we summarize some of our main notation and prove a few lemmas that are used later in the paper.

\subsection{The Fourier transform and Littlewood-Paley projections} 
For functions $f\in L^2(\mathbb{T}_R)$ we define the Fourier transform
\begin{equation}\label{FT1}
\widehat{f}(\xi)=\mathcal{F}(f)(\xi):=\int_{\mathbb{T}_R}f(x)e^{-i\xi x}\,dx,\qquad \xi\in\mathbb{Z}/R:=\{n/R:\,n\in\mathbb{Z}\}.
\end{equation}
The definition shows easily that
\begin{equation}\label{FT2}
\begin{split}
f(x)=\frac{1}{2\pi R}\sum_{\xi\in\mathbb{Z}/R}\widehat{f}(\xi)e^{i\xi x},\\
\int_{\mathbb{T}_R}f(x)\overline{g(x)}\,dx=\frac{1}{2\pi R}\sum_{\xi\in\mathbb{Z}/R}\widehat{f}(\xi)\overline{\widehat{g}(\xi)},\\
\mathcal{F}(fg)(\xi)=\frac{1}{2\pi R}\sum_{\eta\in\mathbb{Z}/R}\widehat{f}(\xi-\eta)\widehat{g}(\eta).
\end{split}
\end{equation}

For any $k\in\mathbb{Z}$ let $k^+:=\max(k,0)$ and $k^-:=\min(k,0)$. We fix an even smooth function $\varphi: \R\to[0,1]$ supported in $[-8/5,8/5]$ and equal to $1$ in $[-5/4,5/4]$,
and define
\begin{equation*}
\varphi_k(x) := \varphi(x/2^k) - \varphi(x/2^{k-1}) , \qquad \varphi_{\leq k}(x):=\varphi(x/2^k),
  \qquad\varphi_{\geq k}(x) := 1-\varphi(x/2^{k-1}),
\end{equation*}
for any $k\in\mathbb{Z}$. Let $P_k$, $P_{\leq k}$, and $P_{\geq k}$ denote the operators on $\mathbb{T}_R$ defined by the Fourier multipliers $\varphi_k$,
$\varphi_{\leq k}$, and $\varphi_{\geq k}$ respectively.
Moreover, let
\begin{equation}\label{P'_k}
P'_k := \sum_{a\in\{-2,-1,0,1,2\}}P_{k+a}\qquad \mbox{and} \qquad \varphi'_k :=  \sum_{a\in\{-2,-1,0,1,2\}}\varphi_{k+a}
\end{equation}
Notice that $P_k=0$ if $2^kR\leq 1/2$.
We also let $\overline{\Z}:=\Z\cup\{-\infty\}$ and define the operator $P_{-\infty}$ by
\begin{equation*}
P_{-\infty}f(x):=\frac{1}{2\pi R}\int_{\T_R}f(y)\,dy.
\end{equation*}

The operators $P_k$ are given by convolution in the physical space. 
More precisely, for $f\in L^1(\T_R)$ and $k\in\mathbb{Z}$ we have
\begin{equation}\label{Pkconv1}
\begin{split}
P_kf(x)&=\int_{\T_R}f(y) K_k(x-y),
\\
K_k(x)&:=\frac{1}{2\pi R}\sum_{\xi\in\Z/R}\varphi_k(\xi)e^{i x\xi}
  =\frac{1}{2\pi R}\sum_{n\in\Z}\varphi_0(n/(R2^k))e^{i (x/R)n},
\end{split}
\end{equation}
as a consequence of the Fourier inversion formula in \eqref{FT2}. The kernels $K_k$ can be estimated using the Poisson summation formula,
\begin{equation}\label{Pkconv2}
K_k(x)=\sum_{n\in\mathbb{Z}}\frac{1}{2\pi R}\int_{\R}\varphi_0(\eta/(R2^k))e^{i (x/R)\eta}e^{-i2\pi n\eta}\,d\eta=\sum_{n\in\Z}2^k\widetilde{\varphi_0}\big(2^k(x-2\pi nR)\big),
\end{equation}
where $\widetilde{\varphi_0}$ denotes the Euclidean inverse Fourier transform of $\varphi_0$. In particular $\|K_k\|_{L^1(\T_R)}\lesssim 1$ for any $k\in\Z$, and therefore the operators $P_k$ are uniformly bounded on $L^p$, $p\in[1,\infty]$.

\subsection{Multipliers and associated operators} We will often work with multipliers $m:\mathbb{R}^2\to\mathbb{C}$ or $m:\mathbb{R}^3\to\mathbb{C}$, and operators defined by such multipliers. We define the class of symbols
\begin{equation}
\label{Sinfty}
S^\infty := \{m: (\Z/R)^d \to \mathbb{C} :\, {\| m \|}_{S^\infty} := {\|\mathcal{F}^{-1}(m)\|}_{L^1(\T_R^d)} < \infty \}.
\end{equation}
Given a suitable symbol $m$ we define the associated bilinear operator $M$ by
\begin{equation}\label{Moutm}
\mathcal{F}\big[M(f,g)\big](\xi)=\frac{1}{2\pi R}\sum_{\eta\in\mathbb{Z}/R}m(\xi-\eta,\eta)\widehat{f}(\xi-\eta)\widehat{g}(\eta).
\end{equation}

For $d\geq 2$ we define also the hyperplanes
\begin{equation}\label{hyperplane}
\begin{split}
&\mathbb{H}^d_R:=\{(\xi_1,\ldots,\xi_{d+1})\in(\Z/R)^{d+1}:\,\xi_1+\ldots+\xi_{d+1}=0\},\\
&\mathbb{H}^d_\infty:=\{(\xi_1,\ldots,\xi_{d+1})\in\R^{d+1}:\,\xi_1+\ldots+\xi_{d+1}=0\},
\end{split}
\end{equation}
and the associated class of symbols
\begin{equation}\label{hyperplane2}
\begin{split}
&\widetilde{S}^\infty =\widetilde{S}^\infty(\mathbb{H}^d_R):= \big\{m: \mathbb{H}_R^d \to \mathbb{C} :\\
&{\| m \|}_{\widetilde{S}^\infty} := \Big\|\frac{1}{(2\pi R)^d}\sum_{\xi_1,\ldots,\xi_d\in\Z/R}m(\xi_1,\ldots,\xi_d,-\xi_1-\ldots-\xi_d)e^{i(x_1\xi_1+\ldots+x_d\xi_d)}\Big\|_{L^1_x(\T_R^d)} < \infty \big\}.
\end{split}
\end{equation}

\smallskip
Lemma \ref{touse} below summarizes some properties of symbols and associated operators.

\begin{lemma}\label{touse} 
With the notation \eqref{Sinfty} and \eqref{hyperplane2} we have the following:

(i) Assume $p,q,r\in[1,\infty]$ satisfy $1/p+1/q=1/r$, and $m\in S^\infty$. Then
\begin{equation}\label{mk6}
\|M(f,g)\|_{L^r(\mathbb{T}_R)} \lesssim \|m\|_{S^\infty} \|f\|_{L^p(\mathbb{T}_R)} \|g\|_{L^q(\mathbb{T}_R)}.
\end{equation}
for any $f,g\in L^2(\mathbb{T}_R)$. 

(ii) More generally, if $p_1,\ldots,p_n\in[1,\infty]$, $n\geq 3$, 
are exponents that satisfy $1/p_1+\ldots+1/p_n=1$ and $m\in\widetilde{S}^{\infty}(\mathbb{H}^{n-1}_R)$ then
\begin{equation}\label{mk6.5}
\begin{split}
\Big|\frac{1}{(2\pi R)^{n-1}}\sum_{(\xi_1,\ldots,\xi_{n})\in\mathbb{H}^{n-1}_R}
  \widehat{f_1}(\xi_1)\cdot\ldots\cdot\widehat{f_n}(\xi_{n})m(\xi_1,\ldots,\xi_n)\Big|
  \\
\lesssim \|f_1\|_{L^{p_1}(\mathbb{T}_R)}\cdot\ldots\cdot\|f_n\|_{L^{p_n}(\mathbb{T}_R)}\|m\|_{\widetilde{S}^\infty}.
\end{split}
\end{equation}
\end{lemma}

\begin{proof}The bounds \eqref{mk6} follow from \eqref{mk6.5} with $n=3$. To prove \eqref{mk6.5} we rewrite
\begin{equation*}
\begin{split}
\frac{1}{(2\pi R)^{n-1}}&\sum_{(\xi_1,\ldots,\xi_{n})\in\mathbb{H}^{n-1}_R}\widehat{f_1}(\xi_1)\cdot\ldots\cdot\widehat{f_n}(\xi_{n})m(\xi_1,\ldots,\xi_n)\\
&=\int_{\T_R^n}f_1(x_1)\cdot\ldots\cdot f_n(x_n)K(x_n-x_1,\ldots,x_n-x_{n-1})\,dx_1\ldots dx_n\\
&=\int_{\T_R^n}f_1(z-y_1)\cdot\ldots\cdot f_{n-1}(z-y_{n-1})f_n(z)K(y_1,\ldots,y_{n-1})\,dy_1\ldots dy_{n-1}dz,
\end{split}
\end{equation*}
where 
\begin{equation*}
K(y_1,\ldots,y_{n-1}):=\frac{1}{(2\pi R)^{n-1}}\sum_{\xi_1,\ldots,\xi_{n-1}\in\Z/R}m(\xi_1,\ldots,\xi_{n-1},-\xi_1-\ldots-\xi_{n-1})e^{i(y_1\xi_1+\ldots+y_{n-1}\xi_{n-1})}.
\end{equation*}
Since $\|K\|_{L^1}={\| m \|}_{\widetilde{S}^\infty}$ the desired bounds \eqref{mk6.5} follow using 
H\"{o}lder's inequality in $z$.
\end{proof}

To apply Lemma \ref{touse} we need to bound the $S^\infty$ and $\widetilde{S}^\infty$ norms of various multipliers. 
We will often use the following simple properties:

\begin{lemma}\label{multiBound}
The following hold true: 

(i) If $m,m'\in S^\infty$ then $m\cdot m'\in S^\infty$ and
\begin{equation}\label{al8}
\|m\cdot m'\|_{S^\infty}\lesssim \|m\|_{S^\infty}\|m'\|_{S^\infty}.
\end{equation}
Similarly, if $m,m'\in \widetilde{S}^\infty$ then $m\cdot m'\in \widetilde{S}^\infty$ and
\begin{equation}\label{al8.4}
\|m\cdot m'\|_{\widetilde{S}^\infty}\lesssim \|m\|_{\widetilde{S}^\infty}\|m'\|_{\widetilde{S}^\infty}.
\end{equation}
Moreover, if $d\geq 2$, $\sigma:\{1,\ldots,d+1\}\to\{1,\ldots,d+1\}$ is a permutation, $m\in\widetilde{S}^\infty(\mathbb{H}_R^d)$, and $m_\sigma(\xi_1,\ldots,\xi_{d+1}):=m(\xi_{\sigma(1)},\ldots,\xi_{\sigma(d+1)})$ then 
\begin{equation}\label{al8.44}
\|m_\sigma\|_{\widetilde{S}^\infty}=\|m\|_{\widetilde{S}^\infty}.
\end{equation}

(ii) Assume that $d\geq 1$ and $m:\R^d\to\mathbb{C}$ is a continuous compactly supported function and $\check{m}\in L^1(\R^d)$ is the Euclidean inverse Fourier transform of $m$. Then
\begin{equation}\label{al8.6}
\Big\|\frac{1}{(2\pi R)^d}\sum_{\xi_1,\ldots,\xi_d\in\Z/R}m(\xi_1,\ldots,\xi_d)e^{i(x_1\xi_1+\ldots+x_d\xi_d)}\Big\|_{L^1_x(\T_R^d)}\lesssim \|\check{m}\|_{L^1(\R^d)}.
\end{equation}

(iii) In particular, assume that $\underline{k}=(k_1,\ldots,k_d)\in\mathbb{Z}^d$ and let $\varphi_{\underline{k}}(\xi):=\varphi_{k_1}(\xi_1)\ldots\varphi_{k_d}(\xi_d)$. Assume that the multiplier $m:(\Z/R)^d\to\mathbb{C}$ admits an extension $m':\R^d\to\mathbb{C}$ (such that $m(\xi)\varphi_{\underline{k}}(\xi)=m'(\xi)\varphi_{\underline{k}}(\xi)$ for all $\xi\in(\Z/R)^d$) satisfying the differential inequalities
\begin{equation}\label{al8.7}
\big|2^{\alpha\cdot \underline{k}}\partial^\al_{\xi}\big[m'(\xi_1,\ldots,\xi_d)\varphi_{k_1}(\xi_1)\ldots\varphi_{k_d}(\xi_d)\big]\big|\lesssim \Lambda
\end{equation}
and for any multi-index $\alpha=(\alpha_1,\ldots,\alpha_d)$ with $|\alpha|:=\alpha_1+\ldots+\alpha_d\leq (d+2)/2$ and any $\xi_1,\ldots,\xi_d\in \R$, where $2^{\alpha\cdot \underline{k}}=2^{\alpha_1k_1+\ldots+\alpha_dk_d}$, $\partial^\al_{\xi}=\partial^{\al_1}_{\xi_1}\ldots\partial^{\al_d}_{\xi_d}$, and $\Lambda\in(0,\infty)$. Then
\begin{equation}\label{al8.77}
\big\|m(\xi)\varphi_{\underline{k}}(\xi)\big\|_{S^\infty}\lesssim \Lambda.
\end{equation}
The conclusion also holds if one replaces some of the functions $\varphi_{k_j}(\xi_j)$ with $\varphi_{\leq k_j}(\xi_j)$ in both the assumption \eqref{al8.7} and the conclusion \eqref{al8.77}.

(iv) Similarly, if $\underline{k}=(k_1,\ldots,k_{d+1})\in\mathbb{Z}^{d+1}$, $\Lambda\in(0,\infty)$, and the multiplier $m:\mathbb{H}^d_R\to\mathbb{C}$ admits an extension $m'':\R^{d+1}\to\mathbb{C}$ (such that $m(\xi)\varphi_{\underline{k}}(\xi)=m''(\xi)\varphi_{\underline{k}}(\xi)$ for all $\xi\in\mathbb{H}_R^d$) that satisfies the differential inequalities
\begin{equation}\label{al8.8}
\big|2^{\alpha\cdot \underline{k}}\partial^\al_{\xi}\big[m''(\xi_1,\ldots,\xi_{d+1})\varphi_{k_1}(\xi_1)\ldots\varphi_{k_{d+1}}(\xi_{d+1})\big]\big|\lesssim \Lambda
\end{equation}
and for any multi-index $\alpha=(\alpha_1,\ldots,\alpha_{d+1})$ with $|\alpha|\leq (d+2)/2$ and any $\xi\in\mathbb{H}^d_\infty$, then
\begin{equation}\label{al8.87}
\big\|m(\xi)\varphi_{\underline{k}}(\xi)\big\|_{\widetilde{S}^\infty}\lesssim \Lambda.
\end{equation}
The conclusion also holds if one replaces some of the functions $\varphi_{k_j}(\xi_j)$
with $\varphi_{\leq k_j}(\xi_j)$ in both the assumption \eqref{al8.8} and the conclusion \eqref{al8.87}.
\end{lemma}

\begin{proof} 
(i) The claims \eqref{al8}--\eqref{al8.44} follow easily from definitions and \eqref{FT2}. 

(ii) To prove \eqref{al8.6} we use the Poisson summation formula to write
\begin{equation*}
\begin{split}
\frac{1}{(2\pi R)^d}\sum_{\xi\in(\Z/R)^d}m(\xi)e^{ix\cdot\xi}&=\frac{1}{(2\pi R)^d}\sum_{\xi\in\Z^d}m(\xi/R)e^{i(x/R)\cdot\xi}\\
&=\frac{1}{(2\pi R)^d}\sum_{a\in\Z^d}\int_{\R^d}m(\xi/R)e^{i(x/R)\cdot\xi}e^{-2\pi ia\cdot\xi}\,d\xi\\
&=\sum_{a\in\Z^d}\check{m}(x-2\pi Ra).
\end{split}
\end{equation*}
The desired bound \eqref{al8.6} follows.

(iii) In view of \eqref{al8.6} we have
\begin{equation}\label{alho2.6}
\big\|m(\xi)\varphi_{\underline{k}}(\xi)\big\|_{S^\infty}\lesssim\Big\|\int_{\R^d}m'(\xi)\varphi_{\underline{k}}(\xi)e^{ix\cdot\xi}\,d\xi\Big\|_{L^1_x}.
\end{equation}
We integrate by parts in $\xi$, and use the assumptions \eqref{al8.7} and Plancherel theorem to see that
\begin{equation}\label{alho2.7}
\Big|\int_{\R^d}m'(\xi)\varphi_{\underline{k}}(\xi)e^{ix\cdot\xi}\,d\xi\Big|\lesssim \frac{2^{k_1/2}\cdot\ldots\cdot 2^{k_d/2}}{(1+2^{k_1}|x_1|+\ldots+2^{k_d}|x_d|)^{(d+1)/2}}K(x_1,\ldots,x_d)
\end{equation}
for any $(x_1,\ldots,x_d)\in\R^d$, where the kernel $K$ satisfies $\|K\|_{L^2}\lesssim \Lambda$. The desired bounds \eqref{alge2.2} follow using the Cauchy-Schwarz inequality.

(iv) Without loss of gene\-ra\-lity we may assume that $k_1=\max(k_1,\ldots,k_{d})$. Recalling the definition \eqref{hyperplane2} and using the bounds \eqref{al8.6} it suffices to show that
\begin{equation}\label{algh6}
\begin{split}
\Big\|\int_{\R^{d}}&m''(-\xi_2-\ldots-\xi_{d+1},\xi_2,\ldots,\xi_{d+1})\varphi_{k_1}(-\xi_2-\ldots-\xi_{d+1})\\
&\times\varphi_{k_2}(\xi_2)\ldots\varphi_{k_{d+1}}(\xi_{d+1})e^{i(x_2\xi_2+\ldots+x_{d+1}\xi_{d+1})}\,d\xi_2\ldots d\xi_{d+1}\Big\|_{L^1(\R^{d})}\lesssim \Lambda.
\end{split}
\end{equation}
Using the assumptions \eqref{al8.8}, integration by parts, and the Plancherel theorem we have
\begin{equation*}
\begin{split}
\Big|\int_{\R^{d}}m''(-\xi_2&-\ldots-\xi_{d+1},\xi_2,\ldots,\xi_{d+1})\varphi_{k_1}(-\xi_2-\ldots-\xi_{d+1})\varphi_{k_2}(\xi_2)\ldots\varphi_{k_{d+1}}(\xi_{d+1})\\
&\times e^{i(x_2\xi_2+\ldots+x_{d+1}\xi_{d+1})}\,d\xi_2\ldots d\xi_{d+1}\Big|\lesssim \frac{2^{k_2/2}\cdot\ldots\cdot 2^{k_{d+1}/2}\cdot K(x_2,\ldots,x_{d+1})}{(1+2^{k_2}|x_2|+\ldots+2^{k_{d+1}}|x_{d+1}|)^{(d+1)/2}},
\end{split}
\end{equation*}
where $\|K\|_{L^2}\lesssim\Lambda$. The bounds \eqref{algh6} follow.
\end{proof}

The lemma above motivates the following definition:

\begin{definition}[H\"ormander-Mikhlin norm]\label{defHM}
Given a multiplier $m:\mathbb{H}^d_R\to\mathbb{C}$, 
a lattice point $\underline{k}=(k_1,\ldots,k_{d+1})\in\mathbb{Z}^{d+1}$, an integer $a\geq 1$,
and a number $\Lambda\in(0,\infty)$, we let 
$\varphi_{\underline{k}}(\xi):=\varphi_{k_1}(\xi_1)\ldots\varphi_{k_{d+1}}(\xi_{d+1})$ and write
\begin{equation}\label{HMD1}
\big\|m\cdot\varphi_{\underline{k}}\big\|_{\mathcal{HM}_a}\lesssim \Lambda
\end{equation}
if $m$ admits an extension $m'':\R^{d+1}\to\mathbb{C}$ such that 
$m(\xi)\varphi_{\underline{k}}(\xi)=m''(\xi)\varphi_{\underline{k}}(\xi)$ for all $\xi\in\mathbb{H}_R^d$ and 
\begin{equation}\label{HMD2}
\big|2^{\alpha\cdot \underline{k}}\partial^\al_{\xi}
\big[m''(\xi_1,\ldots,\xi_{d+1})\varphi_{k_1}(\xi_1)\ldots\varphi_{k_{d+1}}(\xi_{d+1})\big]\big|\lesssim \Lambda
\end{equation}
for any multi-index $\alpha=(\alpha_1,\ldots,\alpha_{d+1})$ with $|\alpha|\leq a$ 
and any $\xi\in\mathbb{H}^d_\infty$. 
Here $2^{\alpha\cdot \underline{k}}=2^{\alpha_1k_1+\ldots+\alpha_dk_d}$ 
and $\partial^\al_{\xi}=\partial^{\al_1}_{\xi_1}\ldots\partial^{\al_d}_{\xi_d}$.
\end{definition}

\subsection{Function spaces} We control our solutions using two types of spaces: the $L^2$-based spaces $\dot{H}^{N,b}(\T_R)$ and the $L^\infty$-based spaces $\dot{W}^{N,b}(\T_R)$, defined by the norms
\begin{equation}\label{HW}
\begin{split}
\|f\|_{\dot{H}^{N,b}}&:=\Big\{\|P_{-\infty}f\|_{L^2}^2R^{-2b}+\sum_{k\in\mathbb{Z}}\|P_kf\|_{L^2}^2(2^{2Nk}+2^{2bk})\Big\}^{1/2},\\
\|f\|_{\dot{W}^{N,b}}&:=\|P_{-\infty}f\|_{L^\infty}R^{-b}+\sum_{k\in\mathbb{Z}}\|P_kf\|_{L^\infty}(2^{Nk}+2^{bk}),
\end{split}
\end{equation}
where $b\in[-1,1]$ and $N\geq \max(b,0)$. Notice that $\dot{H}^{N,0}=H^{N}$ (the standard Sobolev spaces on $\T_R$). The exponent $b$ is important as it allows us to measure accurately low frequency contributions, which is a key issue in the long-term dynamics of solutions. We notice also that the spaces $\dot{W}^{N,b}$ are invariant under the action of the Hilbert transform.

For technical reasons, connected to the algebra properties in Lemma \ref{algeProp} (i) below, we use also an additional family of $L^\infty$-based Sobolev spaces $\widetilde{W}^N$, defined by the norms
\begin{equation}\label{Wtilde}
\|f\|_{\widetilde{W}^{N}}:=\|f\|_{L^\infty}+\sum_{k\geq 0}2^{Nk}\|P_kf\|_{L^\infty}.
\end{equation}

Recall the Sobolev embedding estimates
\begin{equation}\label{SobEmb}
\|P_kf\|_{L^\infty}\lesssim (2^{k/2}+R^{-1/2})\|P_kf\|_{L^2}\qquad\text{ for any }k\in\mathbb{Z}\cup\{-\infty\}.
\end{equation}
In particular, for any $N\geq 1$ and $f\in H^{N+1}(\T_R)$,
\begin{equation}\label{SobEmb2}
\|f\|_{\widetilde{W}^N}\lesssim \|f\|_{\dot{W}^{N,0}}\lesssim \|f\|_{\widetilde{W}^N}\ln\{2+\|f\|_{H^{N+1}}/\|f\|_{\widetilde{W}^N}\}.
\end{equation}
For any $l\in\R$ let $l^+:=\max(l,0)$ and $l^-:=\min(l,0)$. We will often use the following two lemmas to estimate products of functions and multilinear integrals.

\begin{lemma}\label{algeProp}
(i) Assume $s_0\geq 4$, $s_1\in[1,s_0-1]$, and the functions $f,g:\mathbb{T}_R\to\mathbb{C}$ satisfy the bounds
\begin{equation}\label{alge1}
\|f\|_{H^{s_0}}+\|g\|_{H^{s_0}}\lesssim A_2,\qquad \|f\|_{\widetilde{W}^{s_1}}+\|g\|_{\widetilde{W}^{s_1}}\lesssim A_\infty,
\end{equation}
for some numbers $A_\infty\lesssim A_2\in(0,\infty)$. If $M$ is a bilinear operator defined as in \eqref{Moutm} then
\begin{equation}\label{alge2}
\|M(f,g)\|_{H^{s_0}}\lesssim A_2A_\infty,\qquad \|M(f,g)\|_{\widetilde{W}^{s_1}}\lesssim A^2_\infty,
\end{equation}
provided that the multiplier $m:(\Z/R)^2\to\C$ of the operator $M$ satisfies the bounds
\begin{equation}\label{alge2.2}
\big\|\mathcal{F}^{-1}[m(\rho,\eta)\varphi^\ast_k(\rho)\varphi_j^\ast(\eta)\varphi_l^\ast(\rho+\eta)]\big\|_{L^1(\T_R^2)}\lesssim 2^{s_1\min(k,j)}
\end{equation}
for any $k,j,l\in\Z_+$, where $\varphi^\ast_a:=\varphi_a$ if $a\geq 1$ and $\varphi^\ast_0:=\varphi_{\leq 0}$. In particular
\begin{equation}\label{alge2.3}
\|f\cdot g\|_{H^{s_0}}\lesssim A_2A_\infty,\qquad \|f\cdot g\|_{\widetilde{W}^{s_1}}\lesssim A^2_\infty.
\end{equation}

(ii) The bounds \eqref{alge2.2} hold if $m$ extends to a $C^2$ function $m':\R^2\to\mathbb{C}$ satisfying
\begin{equation}\label{alge2.4}
2^{\alpha k}\big|\partial_\rho^\al[m'(\rho,\eta)\varphi^\ast_k(\rho)\varphi_j^\ast(\eta)]\big|+2^{\alpha j}\big|\partial_\eta^\al[m'(\rho,\eta)\varphi^\ast_k(\rho)\varphi_j^\ast(\eta)]\big|\lesssim 2^{s_1\min(k,j)},
\end{equation}
for any $k,j\in\Z_+$, $\al\in\{0,1,2\}$, and $\rho,\eta\in\R$.

(iii) The bounds \eqref{alge2.2} also hold if we assume that 

(1) $m(\rho,0)=m(0,\rho)=m(-\rho,\rho)=0$ for all $\rho\in\Z/R$;

(2) for any $(k_1,k_2,k_3)\in\Z^3$ there is a $C^2$ function $m'':\R^3\to\mathbb{C}$ such that
\begin{equation}\label{alge3.1}
m(\rho,\eta)\varphi_{k_1}(\rho)\varphi_{k_2}(\eta)\varphi_{k_3}(\rho+\eta)=m''(\rho,\eta,-\rho-\eta)\varphi_{k_1}(\rho)\varphi_{k_2}(\eta)\varphi_{k_3}(\rho+\eta)
\end{equation}
for all $\rho,\eta\in\Z/R$, and
\begin{equation}\label{alge3.2}
\big|\varphi_{k_1}(\xi_1)\varphi_{k_2}(\xi_2)\varphi_{k_3}(\xi_3)\cdot 2^{\alpha\cdot k}(\partial^\al_{\xi}m'')(\xi_1,\xi_2,\xi_3)\big|\lesssim 2^{s_1\min(k_1^+,k_2^+)}2^{\min(k_1^-,k_2^-,k_3^-)/8}
\end{equation}
for any multi-index $\alpha=(\alpha_1,\alpha_2,\alpha_3)$ with $|\alpha|\leq 2$ and any $\xi_1,\xi_2,\xi_3\in \R$ satisfying $\xi_1+\xi_2+\xi_3=0$. Here $2^{\alpha\cdot k}=2^{\alpha_1k_1+\alpha_2k_2+\alpha_3k_3}$ and  $\partial^\al_{\xi}=\partial^{\al_1}_{\xi_1}\partial^{\al_2}_{\xi_2}\partial^{\al_3}_{\xi_3}$. 
\end{lemma}

\begin{proof} (i) For $k\geq 0$ and $p\in\{2,\infty\}$ we estimate
\begin{equation*}
\begin{split}
\|P_kM(f,g)\|_{L^p}&\lesssim \|P_kM(P'_kf,P_{\leq k-4}g)\|_{L^p}+\|P_kM(P_{\leq k-4}f,P'_kg)\|_{L^p}\\
&+\sum_{a,b\geq k-4,\,|a-b|\leq 8}\|P_kM(P_af,P_bg)\|_{L^p}\\
&\lesssim A_\infty\sum_{a\geq k-4}\big(\|P_af\|_{L^p}+\|P_ag\|_{L^p}\big),
\end{split}
\end{equation*}
using the assumption \eqref{alge2.2} and the bounds \eqref{mk6} (with the low frequency placed in $L^\infty$ and the high frequency in $L^p$). Moreover, letting $P_a^\ast:=P_a$ if $a\geq 1$ and $P_0^\ast:=P_{\leq 0}$, we also have
\begin{equation*}
\begin{split}
\|P_{\leq 0}M(f,g)\|_{L^p}&\lesssim \sum_{a,b\geq 0,\,|a-b|\leq 8}\|P_0^\ast M(P^\ast_af,P^\ast_bg)\|_{L^p}\lesssim A_\infty\sum_{a\geq 0}\big(\|P_a^\ast f\|_{L^p}+\|P_a^\ast g\|_{L^p}\big).
\end{split}
\end{equation*}
The desired bounds \eqref{alge2} follow from these two estimates.

(ii) We show now that the differential assumptions \eqref{alge2.4} imply the $S^\infty$ bounds \eqref{alge2.2}. This is similar to the proof of Lemma \ref{multiBound} (iii). Indeed, in view of \eqref{al8} and \eqref{al8.6} we have
\begin{equation}\label{alge2.6}
\begin{split}
\big\|\mathcal{F}^{-1}[m(\rho,\eta)\varphi^\ast_k(\rho)\varphi_j^\ast(\eta)&\varphi_l^\ast(\rho+\eta)]\big\|_{L^1(\T_R^2)}\lesssim \big\|\mathcal{F}^{-1}[m(\rho,\eta)\varphi^\ast_k(\rho)\varphi_j^\ast(\eta)]\big\|_{L^1(\T_R^2)}\\
&\lesssim\Big\|\int_{\R^2}m'(\rho,\eta)\varphi^\ast_k(\rho)\varphi_j^\ast(\eta)e^{i(x\rho+y\eta)}\,d\rho d\eta\Big\|_{L^1_{x,y}}.
\end{split}
\end{equation}
Let $\Lambda:=2^{s_1\min(k,j)}$. We integrate by parts in $\rho$ and $\eta$, and use the assumptions \eqref{alge2.4} and the Plancherel theorem to see that, for any $(x,y)\in\R^2$,
\begin{equation}\label{alge2.7}
\Big|\int_{\R^2}m'(\rho,\eta)\varphi^\ast_k(\rho)\varphi_j^\ast(\eta)e^{i(x\rho+y\eta)}\,d\rho d\eta\Big|\lesssim \frac{2^{k/2}2^{j/2}}{(1+2^k|x|)^2+(1+2^j|y|)^2}K(x,y)
\end{equation}
where $K$ satisfies $\|K\|_{L^2}\lesssim \Lambda$. The bounds \eqref{alge2.2} follow using the Cauchy-Schwarz inequality.

(iii) In view of Lemma \ref{multiBound} (iv), for any $k_1,k_2,k_3\in\Z$ we have
\begin{equation}\label{alge3.4}
\big\|\mathcal{F}^{-1}[m(\rho,\eta)\varphi_{k_1}(\rho)\varphi_{k_2}(\eta)\varphi_{k_3}(\rho+\eta)]\big\|_{L^1(\T_R^2)}\lesssim 2^{s_1\min(k_1^+,k_2^+)}2^{\min(k_1^-,k_2^-,k_3^-)/8}.
\end{equation}
The bounds \eqref{alge2.2} follow by summation if any of the parameters $k$, $j$, or $l$ is equal to $0$.
\end{proof}

\begin{lemma}\label{algeProp7}
Assume $d\geq 3$, $s_0\geq 4$, $s_1\in[1,s_0-1]$, and $f_1,\ldots,f_d:\mathbb{T}_R\to\mathbb{C}$ are functions that satisfy the bounds
\begin{equation}\label{alge17}
\|f_j\|_{H^{s_0}}\lesssim A_2,\qquad \|f_j\|_{\widetilde{W}^{s_1}}\lesssim A_\infty
\end{equation}
for some numbers $A_\infty\lesssim A_2\in(0,\infty)$, and any $j\in\{1,\ldots,d\}$. Then
\begin{equation}\label{algh1}
\Big|\frac{1}{(2\pi R)^{d-1}}\sum_{(\xi_1,\ldots,\xi_{d})\in\mathbb{H}^{d-1}_R}\widehat{f_1}(\xi_1)\cdot\ldots\cdot\widehat{f_d}(\xi_d)m(\xi_1,\ldots,\xi_d)\Big|\lesssim A_2^2A_{\infty}^{d-2},
\end{equation}
provided that the multiplier $m$ vanishes when one of the arguments is equal to $0$, and 
\begin{equation}\label{algh2}
\begin{split}
&\big\|m(\xi_1,\ldots,\xi_d)\varphi_{k_1}(\xi_1)\ldots\varphi_{k_d}(\xi_d)\big\|_{\widetilde{S}^\infty}\lesssim \Lambda(k_1,\ldots,k_d),\\
&\Lambda(k_1,\ldots,k_d):=2^{s_0\max(\widetilde{k}_1,0)}2^{s_0\max(\widetilde{k}_2,0)}2^{\min(\widetilde{k}_d,0)/8}\prod_{j=3}^d2^{s_1\max(\widetilde{k}_j,0)},
\end{split}
\end{equation}
for any $k_1,\ldots,k_d\in\Z$, where $\widetilde{k}_1\geq\widetilde{k}_2\geq\ldots\geq\widetilde{k}_d$ is the decreasing rearrangement of the set $k_1,\ldots,k_d$. 
\end{lemma}

\begin{proof}
For any $k_1,\ldots,k_d\in\Z$ let
\begin{equation*}
I_{k_1,\ldots,k_d}:=\frac{1}{(2\pi R)^{d-1}}\sum_{(\xi_1,\ldots,\xi_{d})\in\mathbb{H}^{d-1}_R}\widehat{P_{k_1}f_1}(\xi_1)\cdot\ldots\cdot\widehat{P_{k_d}f_d}(\xi_d)m(\xi_1,\ldots,\xi_d).
\end{equation*}
For $p\in\{2,\infty\}$ and $k\in\mathbb{Z}$ let
\begin{equation*}
B_{k,p}:=\sum_{j\in\{1,\ldots,d\}}\|P'_kf_j\|_{L^p}.
\end{equation*}
Using the assumption \eqref{algh2} and the bounds \eqref{mk6.5} we have
\begin{equation*}
|I_{k_1,\ldots,k_d}|\lesssim (2^{s_0\max(\widetilde{k}_1,0)}B_{\widetilde{k}_1,2})(2^{s_0\max(\widetilde{k}_2,0)}B_{\widetilde{k}_2,2})\prod_{j=3}^d(2^{s_1\max(\widetilde{k}_j,0)}2^{\min(\widetilde{k}_j,0)/(8d)}B_{\widetilde{k_j},\infty}).
\end{equation*}
Recall that $m$ vanishes when one of the arguments is equal to $0$. The bounds \eqref{algh1} follow,
\begin{equation*}
\Big|\frac{1}{(2\pi R)^{d-1}}\sum_{(\xi_1,\ldots,\xi_{d})\in\mathbb{H}^{d-1}_R}\widehat{f_1}(\xi_1)\cdot\ldots\cdot\widehat{f_d}(\xi_d)m(\xi_1,\ldots,\xi_d)\Big|\lesssim \sum_{k_1,\ldots,k_d\in\mathbb{Z}}|I_{k_1,\ldots,k_d}|\lesssim A_2^2A_{\infty}^{d-2},
\end{equation*}
as desired.
\end{proof}

\subsection{Elements of paradifferential calculus} Our main theorems are based on establishing suitable energy estimates. For this we use a paradifferential approach. 

The paradifferential operators on $\T_R$ are defined for functions $f,g\in H^1(\T_R)$ by the formula
\begin{equation}\label{Taf}
\begin{split}
\mathcal{F} \big( T_fg\big)(\xi) &:= \frac{1}{2\pi R} 
  \sum_{\eta\in\Z/R} \chi (\xi-\eta,\eta)\widehat{f}(\xi-\eta) \widehat{g}(\eta),\\
\chi(x,y)&:=\sum_{k\in\mathbb{Z}}\varphi_{\leq k-20}(x)\varphi_k(y+x/2).
\end{split}
\end{equation}
Notice that 
\begin{equation}\label{chieq}
\chi(x,y)=\chi(x,z)\qquad\text{ if }x,y,z\in\R\text{ satisfy }x+y+z=0.
\end{equation}
Using the formulas \eqref{FT2} we have
\begin{equation}\label{Taf2}
\int_{\T_R}T_af\cdot g\,dx=\int_{\T_R}f\cdot T_ag\,dx=\frac{1}{(2\pi R)^2}\sum_{\xi,\eta\in\Z/R}\chi(-\xi-\eta,\eta)\widehat{a}(-\xi-\eta)\widehat{f}(\eta)\widehat{g}(\xi),
\end{equation}
for any functions $f,g,a\in H^1(\T_R)$. In addition $\overline{T_fg}=T_{\overline{f}}\overline{g}$, so the function $T_fg$ is real-valued if both $f$ and $g$ are real-valued.

For any $f,g\in H^1(\T_R)$ we define the bilinear operator
\begin{equation}\label{Hfg}
\mathcal{H}(f,g):=fg-T_fg-T_gf,
\end{equation}
which can be rewritten in the Fourier space as 
\begin{equation}\label{HfgFo}
\begin{split}
\mathcal{F}[\mathcal{H}(f,g)](\xi)&= \frac{1}{2\pi R} 
  \sum_{\eta\in\Z/R} \widetilde{\chi} (\xi-\eta,\eta)\widehat{f}(\xi-\eta) \widehat{g}(\eta),\\
\widetilde{\chi}(x,y)&:=1-\chi(x,y)-\chi(y,x).
\end{split}
\end{equation}
The remainder $\mathcal{H}$ is smoothing in the sense that for any $a_1,a_2>0$
\begin{equation}\label{aux1.1}
\begin{split}
&{\|P_{\geq 0}\mathcal{H}(f,g) \|}_{H^{a_1+a_2}} \lesssim_{a_1,a_2} \min\big\{\|f\|_{H^{a_1}} \|g\|_{\widetilde{W}^{a_2}},\,\|f\|_{\widetilde{W}^{a_2}} \|g\|_{H^{a_1}}\big\},\\
&{\|P_{\geq 0}\mathcal{H}(f,g) \|}_{\widetilde{W}^{a_1+a_2}} \lesssim_{a_1,a_2}\|f\|_{\widetilde{W}^{a_1}} \|g\|_{\widetilde{W}^{a_2}}.
\end{split}
\end{equation}

\medskip
\section{Paralinearization of the water waves system}\label{paralinearization}

In this section we consider smooth solutions of the system \eqref{gWW} on some time interval.
Our goal is to rewrite this system in a way that is amenable to energy estimates.

\subsection{The Dirichlet-Neumann operator and the good variable} 
Assume $h:\T_R\to\mathbb{R}$ is a smooth function and let $\Omega:=\{(x,z)\in\T_R\times \R:\,z< h(x)\}$. Given a sufficiently smooth function $\phi:\T_R\to\mathbb{R}$ let $\Phi$ denote the (unique in a suitable space) harmonic function in $\Omega$, continuous on $\overline{\Omega}$, and satisfying $\Phi(x,h(x))=\phi(x)$. We define the Dirichlet-Neumann map as
\begin{equation}\label{DefDN}
G(h)\phi=\sqrt{1+h_x^2}\,(\nu\cdot\nabla\Phi)
\end{equation}
where $\nu$ denotes the outward pointing unit normal to the domain $\Omega$. The first result of this subsection is the following paralinearization of the Dirichlet-Neumann map.

\begin{lemma}
\label{DNmainpro}
(i) Assume that $\varep, A\in (0,\infty)$ are numbers such that $\varep\ll 1$ and $\varep\lesssim A$. 
Assume that $h,\phi:\T_R\to\mathbb{R}$ satisfy the bounds 
\begin{equation}\label{Asshu}
\|h\|_{H^{N_0}}+\||\partial_x|^{1/2}\phi\|_{H^{N_0-1/2}}\leq A,\qquad \|h\|_{\dot{W}^{N_1,0}}+\||\partial_x|^{1/2}\phi\|_{\dot{W}^{N_1-1/2,0}}\leq \varep.
\end{equation}
Then there is a unique harmonic map $\Phi\in C^4_{\mathrm{loc}}(\Omega)$ such that $\Phi(x,h(x))=\phi(x)$ and $\nabla\Phi\in L^2(\Omega)$. 

(ii) We define $G(h)\phi$ as in \eqref{DefDN} and let
\begin{align}\label{DefBV}
B := \frac{G(h)\phi+\partial_xh\partial_x\phi}{1+(\partial_xh)^2}, 
  \qquad V := \partial_x \phi - B\partial_xh, \qquad \omega := \phi - T_B h.
\end{align}
Then for any $F\in\big\{\partial_x\phi,\,G(h)\phi,\,B,\,V\big\}$ we have
\begin{equation}\label{DefBV2}
\|F\|_{H^{N_0-1}}\lesssim A,\qquad\|F\|_{\widetilde{W}^{N_1-1}}\lesssim\varep.
\end{equation}

(iii) We can paralinearize the Dirichlet-Neumann operator as
\begin{equation}
\label{DNmainformula}
G(h)\phi = |\partial_x|\omega- \partial_x(T_V h) + G_2(\phi,h) + G_3(\phi,h,h) +G_{\geq 4},
\end{equation}
where
\begin{equation}\label{DNmain2}
\begin{split}
\mathcal{F}[G_2(\phi,h)](\xi)&=\frac{1}{2\pi R}\sum_{\eta\in\Z/R}q_2(\xi-\eta,\eta)\widehat{\phi}(\xi-\eta)\widehat{h}(\eta),
\\
q_2(\rho,\eta) &:=\widetilde{\chi}(\rho,\eta)[(\eta+\rho)\rho-|\eta+\rho||\rho|],
\end{split}
\end{equation}
\begin{equation}\label{DNmain3}
\begin{split}
\mathcal{F}[G_3(\phi,h,h)](\xi)&=\frac{1}{(2\pi R)^2}\sum_{\rho,\eta\in\Z/R}q_3(\xi-\eta-\rho,\eta,\rho)\widehat{\phi}(\xi-\eta-\rho)\widehat{h}(\eta)\widehat{h}(\rho),\\
q_3(\sigma,\eta,\rho)&:=\frac{|\sigma+\eta+\rho||\sigma|}{2}\Big\{[1-\chi(\sigma+\rho,\eta)][|\sigma+\rho|-|\sigma|-\rho\,\sgn(\sigma+\eta+\rho)]\\
&\qquad\qquad+[1-\chi(\sigma+\eta,\rho)][|\sigma+\eta|-|\sigma|-\eta\,\sgn(\sigma+\eta+\rho)]\\
&\qquad\qquad+\big[|\sigma|-\sigma\,\sgn(\sigma+\eta+\rho)\big]\Big\},
\end{split}
\end{equation}
\begin{equation}\label{DNmain4}
\|G_{\geq 4}\|_{\dot{H}^{N_0,-3/4}}\lesssim A\varep^3,\qquad \|G_{\geq 4}\|_{\dot{W}^{N_1,-3/4}}\lesssim \varep^4.
\end{equation}
The operators $G_2$ and $G_3$ are smoothing in the sense that their symbols satisfy the identities
\begin{equation}\label{smoothing1}
q_2(\rho,\eta)=0\qquad\text{ if }\,\,|\rho|\leq 2^{-40}|\eta|\,\,\text{ or }\,\,|\eta|\leq 2^{-40}|\rho|,
\end{equation}
\begin{equation}\label{smoothing2}
q_3(\sigma,\eta,\rho)=0\qquad\text{ if }\,\,|\sigma|+|\eta|\leq 2^{-40}|\rho|\,\,\text{ or }\,\,|\eta|+|\rho|\leq 2^{-40}|\sigma|\,\,\text{ or }\,\,|\rho|+|\sigma|\leq 2^{-40}|\eta|.
\end{equation}
\end{lemma}

We will also need an extension of this lemma to the case of time evolutions. Assume that $T_1\leq T_2\in\mathbb{R}$ and $(h,\phi)\in C([T_1,T_2]:H^{N_0}\times \dot{H}^{N_0,1/2})$ is a solution of the system \eqref{gWW}. Moreover, assume that for any $t\in[T_1,T_2]$ we have
\begin{equation}\label{maineqas}
\begin{split}
\|h(t)\|_{H^{N_0}}+\||\partial_x|^{1/2}\phi(t)\|_{H^{N_0-1/2}}+\||\partial_x|^{1/2}\omega(t)\|_{H^{N_0}} &\leq A,\\
\|h(t)\|_{\dot{W}^{N_1,0}} + \||\partial_x|^{1/2}\phi(t)\|_{\dot{W}^{N_1-1/2,0}} + \||\partial_x|^{1/2}\omega(t)\|_{\dot{W}^{N_1,0}}&\leq \varep(t),
\end{split}
\end{equation}
where $A\in (0,\infty)$ and $\varep:[T_1,T_2]\to [0,\infty)$ is a continuous function satisfying $\varep(t)\ll 1$.

\begin{lemma}\label{BVexpands} 
Assume that \eqref{maineqas} hold and that $B,V$ are as in \eqref{DefBV}. 
Then we can expand
\begin{equation}\label{expand0}
\begin{split}
& \partial_t h = |\partial_x|\phi + \dot{h}_2 + \dot{h}_{\geq 3},
\qquad \dot{h}_2 := -\partial_x(h\partial_x\phi) - |\partial_x|(h|\partial_x|\phi),
\\
& \partial_t \phi = -h + \dot{\phi}_2 + \dot{\phi}_{\geq 3},
\qquad \dot{\phi}_2 := -\frac{1}{2} (\partial_x\phi)^2 + \frac{1}{2}(|\partial_x|\phi)^2,
\end{split}
\end{equation}
and 
\begin{equation}\label{expand1}
\begin{split}
& B=|\partial_x|\phi+B_2+B_{\geq 3},\qquad B_2:=-|\partial_x|(h|\partial_x|\phi)-h\partial_x^2\phi,\\
& V=\partial_x\phi+V_2+V_{\geq 3},\qquad V_2:=-\partial_xh|\partial_x|\phi,
\end{split}
\end{equation}
where the cubic remainders satisfy the bounds
\begin{equation}\label{expand2}
\begin{split}
\|\dot{h}_{\geq 3}(t)\|_{H^{N_0-1}} + \|\dot{\phi}_{\geq 3}(t)\|_{H^{N_0-1}}+\|B_{\geq 3}(t)\|_{H^{N_0-1}} + \|V_{\geq 3}(t)\|_{H^{N_0-1}}\lesssim A[\varep(t)]^2,\\
\|\dot{h}_{\geq 3}(t)\|_{\widetilde{W}^{N_1-1}} + \|\dot{\phi}_{\geq 3}(t)\|_{\widetilde{W}^{N_1-1}}+\|B_{\geq 3}(t)\|_{\widetilde{W}^{N_1-1}} + \|V_{\geq 3}(t)\|_{\widetilde{W}^{N_1-1}}\lesssim [\varep(t)]^3,
\end{split}
\end{equation}
for any $t\in[T_1,T_2]$. Moreover,
\begin{equation}\label{expand3}
\begin{split}
\partial_t B & = -|\partial_x|h+\dot{B}_2+\dot{B}_{\geq 3},\\
\dot{B}_2 & :=-\frac{1}{2}|\partial_x|(\partial_x\phi)^2
 -\frac{1}{2}|\partial_x|(|\partial_x|\phi)^2 - |\partial_x|\phi\,\partial_x^2\phi + h\partial_x^2h 
 + |\partial_x|(h|\partial_x|h),
\end{split}
\end{equation}
\begin{equation}\label{expand4}
\partial_tV=-\partial_xh+\dot{V}_2+\dot{V}_{\geq 3},
  \qquad \dot{V}_2 :=-\frac{1}{2}\partial_x(\partial_x\phi)^2+\partial_xh|\partial_x|h,
\end{equation}
and
\begin{equation}\label{expand5}
\begin{split}
\partial_t^2B&=-|\partial_x|^2\phi+\ddot{B}_2+\ddot{B}_{\geq 3},
\\
\ddot{B}_2&:=2|\partial_x|(\partial_xh\,\partial_x\phi)+2|\partial_x|(|\partial_x|h\,|\partial_x|\phi)+\partial_x^2h\,|\partial_x|\phi+|\partial_x|h\,\partial_x^2\phi-2\partial_xh\,\partial_x|\partial_x|\phi,
\end{split}
\end{equation}
where the cubic remainders $\dot{B}_{\geq 3},\dot{V}_{\geq 3},\ddot{B}_{\geq 3}$ satisfy the estimates
\begin{equation}\label{expand6.5}
\begin{split}
\|\dot{B}_{\geq 3}(t)\|_{H^{N_0-2}} + \|\dot{V}_{\geq 3}(t)\|_{H^{N_0-2}}+\|\ddot{B}_{\geq 3}(t)\|_{H^{N_0-3}}\lesssim A[\varep(t)]^2,\\
\|\dot{B}_{\geq 3}(t)\|_{\widetilde{W}^{N_1-2}} + \|\dot{V}_{\geq 3}(t)\|_{\widetilde{W}^{N_1-2}}+\|\ddot{B}_{\geq 3}(t)\|_{\widetilde{W}^{N_1-3}}\lesssim [\varep(t)]^3.
\end{split}
\end{equation}
\end{lemma}

These lemmas are similar to Proposition B.1 in \cite{IoPu4}.
For convenience, we provide details of the proofs in Section \ref{TechnicalProofs}. 

\subsection{Evolution equations} We derive now an evolution equation for the good variable $\omega$. 

\begin{lemma}\label{propmaineq}
Assume that $T_1\leq T_2\in\mathbb{R}$ and $(h,\phi)\in C([T_1,T_2]:H^{N_0}\times \dot{H}^{N_0,1/2})$ is a solution of the system \eqref{gWW} satisfying \eqref{maineqas}. Then the variable $\omega$ satisfies the evolution equation
\begin{equation}\label{hphieq1}
\partial_t \omega = -h- T_{\partial_tB + V\partial_xB} h - T_V \partial_x\omega + L_2(\phi,\phi)
  + L_{3}(h,\phi,\phi) + L_{\geq 4},
\end{equation} 
where
\begin{align}\label{hoeq2}
\begin{split}
\mathcal{F}[L_2(\phi,\phi)](\xi)&= \frac{1}{2\pi R}\sum_{\eta\in\Z/R} \ell_2(\xi-\eta,\eta)\widehat{\phi}(\xi-\eta) \widehat{\phi}(\eta) ,
\\
\ell_2(\rho,\eta) &:= \frac{1}{2}\wt{\chi}(\rho,\eta) [|\eta||\rho| + \eta\rho],
\end{split}
\end{align}
\begin{equation}\label{hoeq31}
\begin{split}
\mathcal{F}\big[L_3(h,\phi,\phi)\big](\xi) & = \frac{1}{(2\pi R)^2}\sum_{\rho,\eta\in\Z/R} \ell_{3}(\xi-\eta-\rho,\eta,\rho)\widehat{h}(\xi-\eta-\rho)\widehat{\phi}(\eta) \widehat{\phi}(\rho),\\
\ell_{3}(\sigma,\eta,\rho)&:= \eta  |\rho|\big\{(\sigma+\rho)\chi(\eta,\sigma+\rho) \chi(\rho,\sigma)+\sigma\chi(\sigma+\rho,\eta) - \sigma\chi(\rho,\sigma+\eta)\\
&\qquad\quad -\rho\chi(\eta+\rho,\sigma)+ (|\rho| - |\sigma+\rho|)\sgn(\eta)\wt{\chi}(\sigma+\rho,\eta)\big\},
\end{split}
\end{equation}
and, for any $t\in[T_1,T_2]$,
\begin{align}\label{hoeqrem}
\begin{split}
{\| |\partial_x|^{1/2}L_{\geq 4} (t)\|}_{\dot{H}^{N_0,-1/4}} \lesssim A\varep(t)^3,\qquad {\| |\partial_x|^{1/2}L_{\geq 4} (t)\|}_{\dot{W}^{N_1,-1/4}} \lesssim \varep(t)^4.
\end{split}
\end{align}
The operators $L_2$ and $L_3$ are smoothing in the sense that their symbols satisfy the identities
\begin{equation}\label{smoothing3}
\ell_2(\rho,\eta)=0\qquad\text{ if }\,\,|\rho|\leq 2^{-40}|\eta|\,\,\text{ or }\,\,|\eta|\leq 2^{-40}|\rho|,
\end{equation}
\begin{equation}\label{smoothing4}
\ell_3(\sigma,\eta,\rho)=0\qquad\text{ if }\,\,|\sigma|+|\eta|\leq 2^{-40}|\rho|\,\,\text{ or }\,\,|\eta|+|\rho|\leq 2^{-40}|\sigma|\,\,\text{ or }\,\,|\rho|+|\sigma|\leq 2^{-40}|\eta|.
\end{equation}
\end{lemma}

\begin{proof} We start from the second equation in \eqref{gWW} and use \eqref{DefBV} to calculate
\begin{equation}\label{equdtphi}
\begin{split}
- \frac{1}{2}(\partial_x\phi)^2 + \frac{(G(h)\phi+\partial_xh\partial_x\phi)^2}{2(1+(\partial_xh)^2)}&=
  \frac{1}{2}\big[B^2(1+(\partial_xh)^2)-(V+B\partial_xh)^2\big]\\
&= \frac{1}{2}\big[B^2-2VB\partial_xh-V^2\big].
\end{split}
\end{equation}
Therefore, using the definition $\omega=\phi-T_Bh$, we have
\begin{align}
\label{dtomega1}
\begin{split}
\partial_t\omega = \partial_t \phi - T_{\partial_tB}h - T_B\partial_t h = -h- T_{\partial_t B} h + I + II,
\end{split}
\end{align}
where
\begin{align*}
\begin{split}
I & := (1/2)(B^2-V^2) - V B \partial_xh,
\qquad
II  := - T_B(G(h)\phi).
\end{split}
\end{align*}

We use the identities $V + B\partial_xh= \partial_x\phi = \partial_x\omega+ \partial_x T_B h$ and recall \eqref{Hfg} to calculate
\begin{align}\label{dtomegaI}
\begin{split}
I & = (B^2-V^2)/2 - V B \partial_xh
\\
& = T_B B + (1/2)\mathcal{H}(B,B) - T_V V -(1/2) \mathcal{H}(V,V) - T_V (B \partial_xh) - T_{B\partial_xh} V - \mathcal{H}(V, B\partial_xh)
\\
& = T_B B - T_V \partial_x\omega + (1/2) \mathcal{H}(|\partial_x|\phi,|\partial_x|\phi) 
	- (1/2) \mathcal{H}(\partial_x\phi,\partial_x\phi)
\\ 
& - T_V \big( \partial_x T_B h \big) - T_{B\partial_xh} V + \mathcal{H}(B, B -|\partial_x|\phi) 
  + I_{\geq 4},
\end{split}
\end{align}
where $\mathcal{H}$ is defined as in \eqref{Hfg} and
\begin{align*}
\begin{split}
I_{\geq 4} & := (1/2)\mathcal{H}(B,B) 
  - (1/2) \mathcal{H}(|\partial_x|\phi,|\partial_x|\phi) - \mathcal{H}(B,B-|\partial_x|\phi)
  \\
  & \, -(1/2) \mathcal{H}(V,V) + (1/2) \mathcal{H}(\partial_x\phi,\partial_x\phi) - \mathcal{H}(V,B\partial_xh).
\end{split}
\end{align*}
Using again $V = \partial_x\phi-B\partial_xh$ we calculate
\begin{align*}
I_{\geq 4} = -(1/2)\mathcal{H}(B-|\partial_x|\phi,B-|\partial_x|\phi) + (1/2) \mathcal{H}(B\partial_xh, B\partial_xh),
\end{align*}
which is an acceptable quartic remainder, due to \eqref{aux1.1}.

From \eqref{DefBV} we have $B -V\partial_xh=G(h)\phi$ so that $II = -T_B B + T_B (V\partial_xh)$.
Then,
\begin{align}\label{I+II}
\begin{split}
I+II & = - T_V \partial_x\omega + (1/2)\mathcal{H}(|\partial_x|\phi,|\partial_x|\phi) 
  - (1/2)\mathcal{H}(\partial_x\phi,\partial_x\phi)
\\
& + \big[ - T_V( \partial_x T_B h ) + T_B (V\partial_xh) -  T_{B\partial_xh} V \big]+ \mathcal{H}(B,B-|\partial_x|\phi) + I_{\geq 4}.
\end{split}
\end{align}
The term $-T_V\partial_x\omega$ is the same as in \eqref{hphieq1}, while the $\mathcal{H}$ terms in the first line of \eqref{I+II} give $L_2(\phi,\phi)$ as in \eqref{hoeq2}. For the terms in the second line we can write
\begin{align*}
- T_V( \partial_x T_B h ) + T_B (V\partial_xh) -  T_{B\partial_xh} V  = - T_{V  \partial_x B} h + R_1(V,B,h)
\end{align*}
where
\begin{align}\label{I+II1}
\begin{split}
\mathcal{F}\big[R_1(V,B,h)\big](\xi) & = \frac{1}{(2\pi R)^2} 
	\sum_{\rho,\eta\in\Z/R} \widehat{V}(\eta) \widehat{B}(\rho) \widehat{h}(\xi-\eta-\rho) r_1(\xi,\eta,\rho),
\\
r_1(\xi,\eta,\rho) & := -\chi(\eta,\xi-\eta) i (\xi-\eta) \chi(\rho,\xi-\eta-\rho) + \chi(\rho,\xi-\rho)i(\xi-\eta-\rho)
	\\ &\quad\,\,\,  - \chi(\xi-\eta,\eta) i(\xi-\eta-\rho)+ \chi(\eta+\rho,\xi-\eta-\rho) i\rho.
\end{split}
\end{align}
Notice that $R_1(h,V,B)$ is a cubic term which is smoothing,
in the sense that $r_1(\xi,\eta,\rho)=0$ if $|\eta|+|\rho|\leq 2^{-40}|\xi|$ 
or if $|\eta|+|\xi-\eta-\rho|\leq 2^{-40}|\xi|$ or if $|\rho|+|\xi-\eta-\rho|\leq 2^{-40}|\xi|$.
Therefore we can replace $V$ with $\partial_x\phi$ and $B$ with $|\partial_x|\phi$ 
at the expense of acceptable quartic errors. Moreover, in view of \eqref{expand1} 
we may replace the remaining term $\mathcal{H}(B,B-|\partial_x|\phi)$ in \eqref{I+II}
with $\mathcal{H}(|\partial_x|\phi,-|\partial_x|(h|\partial_x|\phi)-h\partial_x^2\phi)$
at the expense of acceptable quartic remainders. 
A simple calculation shows that the resulting smoothing cubic terms are 
of the form $L_3(h,\phi,\phi)$ as claimed in \eqref{hoeq31}.
\end{proof}

We define now our main variables $h^\ast$ and $\omega^\ast$. These variables are optimal for our analysis, in particular for combining normal forms and energy estimates (without loss of derivatives) for two main reasons: (1) they satisfy the good paradifferential evolution equations \eqref{Usor2}, with well-defined quasi-linear and semi-linear terms, and (2) they are both at the level of derivatives of the Hamiltonian .

\begin{lemma}\label{Usor1}
Assume that $T_1\leq T_2\in\mathbb{R}$ and $(h,\phi)\in C([T_1,T_2]:H^{N_0}\times \dot{H}^{N_0,1/2})$ 
is a solution of the system \eqref{gWW} satisfying \eqref{maineqas}. 
We define
\begin{equation}\label{Usor2}
\begin{split}
&\alpha:=\partial_tB + V\partial_xB,\qquad\beta:=\sqrt{1+\alpha}-1,\\
&h^\ast:=T_{\sqrt{1+\alpha}}h=h+T_\beta h,\qquad \omega^\ast:=|\partial_x|^{1/2}\omega.
\end{split}
\end{equation}
Then the variables $h^\ast,\omega^\ast$ satisfy the evolution equations
\begin{equation}\label{Usor3}
\begin{split}
\partial_t h^\ast-|\partial_x|^{1/2}\omega^\ast
  & = T_{\beta}|\partial_x|^{1/2}\omega^\ast-\partial_x T_V h^\ast+H_2^\ast 
  + H^\ast_{\geq 3}=:\mathcal{N}_{\geq 2}^{h^\ast},
  \\
\partial_t \omega^\ast+|\partial_x|^{1/2} h^\ast 
  & = -T_{\beta}|\partial_x|^{1/2} h^\ast-\partial_x T_V \omega^\ast+\Omega_2^\ast 
  + \Omega^\ast_{\geq 3}=:\mathcal{N}_{\geq 2}^{\omega^\ast},
\end{split}
\end{equation} 
where
\begin{equation}\label{Usor5.1}
\begin{split}
&H_2^\ast=H^{ri}_2(h^\ast,\omega^\ast),\qquad\mathcal{F}\{H^{ri}_2(h^\ast,\omega^\ast)\}(\xi)=\frac{1}{2\pi R}\sum_{\eta\in\Z/R}\widehat{h^\ast}(\xi-\eta)\widehat{\omega^\ast}(\eta)n^{ri}(\xi-\eta,\eta),
\\
&n^{ri}(\rho,\eta):=\widetilde{\chi}(\eta,\rho)\frac{(\rho+\eta)\eta-|\rho+\eta||\eta|}{|\eta|^{1/2}}-\chi(\eta,\rho)\frac{|\eta|^{3/2}}{2},
\end{split}
\end{equation}
\begin{equation}\label{Usor4}
\Omega_2^\ast=\Omega_2^{rr}(h^\ast,h^\ast)+\Omega^{ii}_2(\omega^\ast,\omega^\ast),
\end{equation}
\begin{equation}\label{Usor5}
\begin{split}
&\mathcal{F}\{\Omega_2^{rr}(h^\ast,h^\ast)\}(\xi)=\frac{1}{2\pi R}\sum_{\eta\in\Z/R}\widehat{h^\ast}(\xi-\eta)\widehat{h^\ast}(\eta)n^{rr}(\xi-\eta,\eta),\\
&n^{rr}(\rho,\eta):=\chi(\rho,\eta)\frac{|\rho|}{2}(|\rho+\eta|^{1/2}-|\eta|^{1/2}),
\end{split}
\end{equation}
\begin{equation}\label{Usor5.2}
\begin{split}
& \mathcal{F}\{\Omega^{ii}_2(\omega^\ast,\omega^\ast)\}(\xi) =
	\frac{1}{2\pi R}\sum_{\eta\in\Z/R}\widehat{\omega^\ast}(\xi-\eta)\widehat{\omega^\ast}(\eta)n^{ii}(\xi-\eta,\eta),
\\
& n^{ii}(\rho,\eta) := \widetilde{\chi}(\rho,\eta)|\rho+\eta|^{1/2}\frac{|\eta||\rho|+\eta\rho}{2|\eta|^{1/2}|\rho|^{1/2}} 
	+ \chi(\rho,\eta)\rho\frac{|\rho+\eta|^{1/2}\eta-|\eta|^{1/2}(\rho+\eta)}{|\eta|^{1/2}|\rho|^{1/2}},
\end{split}
\end{equation}
and the cubic remainders $H^\ast_{\geq 3},\Omega^\ast_{\geq 3}$ satisfy the bounds
\begin{align}\label{Usor6}
\begin{split}
{\|H^\ast_{\geq 3} (t)\|}_{H^{N_0}}+{\|\Omega^\ast_{\geq 3} (t)\|}_{H^{N_0}} &\lesssim A\varep(t)^2,\\
{\|H^\ast_{\geq 3} (t)\|}_{\widetilde{W}^{N_1}} +{\|\Omega^\ast_{\geq 3} (t)\|}_{\widetilde{W}^{N_1}}&\lesssim\varep(t)^3,
\end{split}
\end{align}
for any $t\in[T_1,T_2]$.
\end{lemma}

\begin{proof} We use the equations \eqref{hphieq1} and $\partial_th=G(h)\phi$, and the identities \eqref{DNmainformula}. Thus
\begin{equation*}
\begin{split}
\partial_t h^\ast& = T_{\partial_t\sqrt{1+\alpha}}h+T_{\sqrt{1+\alpha}}|\partial_x|\omega-T_{\sqrt{1+\alpha}}\partial_xT_Vh+T_{\sqrt{1+\alpha}}G_2(\phi,h)+H_{\geq 3}^1,\\
\partial_t\omega^\ast&=-|\partial_x|^{1/2}T_{1+\alpha} h - |\partial_x|^{1/2}T_V \omega_x + |\partial_x|^{1/2}L_2(\phi,\phi)+\Omega_{\geq 3}^1,
\end{split}
\end{equation*}
where $H_{\geq 3}^1,\Omega_{\geq 3}^1$ are acceptable cubic remainders satisfying \eqref{Usor6}. Therefore
\begin{equation}\label{Utor7}
\begin{split}
\partial_t h^\ast-T_{\sqrt{1+\alpha}}|\partial_x|^{1/2}\omega^\ast+\partial_xT_Vh^\ast& = \big\{\partial_xT_VT_{\sqrt{1+\alpha}}h-T_{\sqrt{1+\alpha}}\partial_xT_Vh\}\\
&+T_{\sqrt{1+\alpha}}G_2(\phi,h)+T_{\partial_t\sqrt{1+\alpha}}h+H_{\geq 3}^1,
\end{split}
\end{equation}
and
\begin{equation}\label{Utor7.1}
\begin{split}
\partial_t\omega^\ast+T_{\sqrt{1+\alpha}}|\partial_x|^{1/2}h^\ast&+\partial_xT_V\omega^\ast=\big\{T_{\sqrt{1+\alpha}}|\partial_x|^{1/2}T_{\sqrt{1+\alpha}}h-|\partial_x|^{1/2}T_{1+\alpha} h\big\}\\
&+\{\partial_xT_V|\partial_x|^{1/2}\omega- |\partial_x|^{1/2}T_V \omega_x\} + |\partial_x|^{1/2}L_2(\phi,\phi)+\Omega_{\geq 3}^1.
\end{split}
\end{equation}
Using the formula $\alpha:=\partial_tB + V\partial_xB$ and Lemma \ref{BVexpands} we can expand
\begin{equation}\label{Utor10}
\begin{split}
&\alpha=-|\partial_x|h+\alpha_2+\alpha_{\geq 3},
\\
& \alpha_2 := -\frac{1}{2}|\partial_x|(\partial_x\phi)^2 - \frac{1}{2}|\partial_x|(|\partial_x|\phi)^2 
	 - |\partial_x|\phi \partial_x^2\phi + \partial_x\phi\partial_x|\partial_x|\phi + h\partial_x^2h + |\partial_x|(h|\partial_x|h),
\\
& \|\alpha_{\geq 3}(t)\|_{\widetilde{W}^{N_1-2}}\lesssim \varepsilon(t)^3
  \qquad \text{ for any }t\in[T_1,T_2],
\end{split}
\end{equation}
and
\begin{equation}\label{Utor11}
\begin{split}
& \partial_t\alpha = -|\partial_x|^2\phi+\dot{\alpha}_2+\dot{\alpha}_{\geq 3},
\\
& \dot{\alpha}_2 := 2|\partial_x|[\partial_xh\,\partial_x\phi+|\partial_x|h\,|\partial_x|\phi]+\partial_x^2h\,|\partial_x|\phi+|\partial_x|h\,\partial_x^2\phi-3\partial_xh\,\partial_x|\partial_x|\phi-\partial_x|\partial_x|h\,\partial_x\phi,
  \\
& \|\dot{\alpha}_{\geq 3}(t)\|_{\widetilde{W}^{N_1-3}}\lesssim \varepsilon(t)^3
  \qquad \text{ for any }t\in[T_1,T_2].
\end{split}
\end{equation}

Since $\beta=\sqrt{1+\alpha}-1$ we have $(1+\beta)^2=1+\alpha$, therefore
\begin{equation*}
\begin{split}
T_{\sqrt{1+\alpha}}|\partial_x|^{1/2}T_{\sqrt{1+\alpha}}h-|\partial_x|^{1/2}T_{1+\alpha} h&=\big\{T_\beta|\partial_x|^{1/2}h-|\partial_x|^{1/2}T_\beta h\big\}\\
&+\big\{T_\beta|\partial_x|^{1/2}T_\beta h-|\partial_x|^{1/2}T_{\beta^2} h\big\}.
\end{split}
\end{equation*}
Using the definitions and \eqref{Utor10} we therefore have 
\begin{equation*}
\begin{split}
&T_{\sqrt{1+\alpha}}|\partial_x|^{1/2}T_{\sqrt{1+\alpha}}h-|\partial_x|^{1/2}T_{1+\alpha} h=\Omega_2^1(h^\ast,h^\ast)+\Omega_{\geq 3}^2,\\
&\mathcal{F}\{\Omega_2^1(h^\ast,h^\ast)\}(\xi)=\frac{1}{2\pi R}\sum_{\eta\in\Z/R}\widehat{h^\ast}(\xi-\eta)\widehat{h^\ast}(\eta)\chi(\xi-\eta,\eta)\frac{|\xi-\eta|}{2}(|\xi|^{1/2}-|\eta|^{1/2}),
\end{split}
\end{equation*}
where $\Omega_{\geq 3}^2$ is an acceptable cubic remainder. Similarly, using also the approximate identities in Lemma \ref{BVexpands}, we can rewrite the other terms in the right-hand side of \eqref{Utor7.1} as
\begin{equation*}
\begin{split}
&\partial_xT_V |\partial_x|^{1/2}\omega-|\partial_x|^{1/2}T_V \partial_x\omega=\Omega_2^{2,1}(\omega^\ast,\omega^\ast)+\Omega_{\geq 3}^3,\\
&\mathcal{F}\{\Omega_2^{2,1}(\omega^\ast,\omega^\ast)\}(\xi)=\frac{1}{2\pi R}\sum_{\eta\in\Z/R}\widehat{\omega^\ast}(\xi-\eta)\widehat{\omega^\ast}(\eta)\chi(\xi-\eta,\eta)(\xi-\eta)\frac{|\xi|^{1/2}\eta-|\eta|^{1/2}\xi}{|\eta|^{1/2}|\xi-\eta|^{1/2}},
\end{split}
\end{equation*}
\begin{equation*}
\begin{split}
&|\partial_x|^{1/2}L_2(\phi,\phi)=\Omega_2^{2,2}(\omega^\ast,\omega^\ast)+\Omega_{\geq 3}^4,\\
&\mathcal{F}\{\Omega_2^{2,2}(\omega^\ast,\omega^\ast)\}(\xi)=\frac{1}{2\pi R}\sum_{\eta\in\Z/R}\widehat{\omega^\ast}(\xi-\eta)\widehat{\omega^\ast}(\eta)\widetilde{\chi}(\xi-\eta,\eta)|\xi|^{1/2}\frac{|\eta||\xi-\eta|+\eta(\xi-\eta)}{2|\eta|^{1/2}|\xi-\eta|^{1/2}},
\end{split}
\end{equation*}
for some acceptable cubic remainders $\Omega_{\geq 3}^3,\Omega_{\geq 3}^4$. Similarly, the terms in the right-hand side of \eqref{Utor7} can be rewritten as
\begin{equation*}
\partial_xT_VT_{\sqrt{1+\alpha}}h-T_{\sqrt{1+\alpha}}\partial_xT_Vh=H_{\geq 3}^2,
\end{equation*}
\begin{equation*}
\begin{split}
&T_{\sqrt{1+\alpha}}G_2(\phi,h)=H_2^1(h^\ast,\omega^\ast)+H_{\geq 3}^3,\\
&\mathcal{F}\{H_2^1(h^\ast,\omega^\ast)\}(\xi)=\frac{1}{2\pi R}\sum_{\eta\in\Z/R}\widehat{h^\ast}(\xi-\eta)\widehat{\omega^\ast}(\eta)\widetilde{\chi}(\eta,\xi-\eta)\frac{\xi\eta-|\xi||\eta|}{|\eta|^{1/2}},
\end{split}
\end{equation*}
\begin{equation*}
\begin{split}
&T_{\partial_t\sqrt{1+\alpha}}h=H_2^2(h^\ast,\omega^\ast)+H_{\geq 3}^4,\\
&\mathcal{F}\{H_2^2(h^\ast,\omega^\ast)\}(\xi)=\frac{1}{2\pi R}\sum_{\eta\in\Z/R}\widehat{h^\ast}(\xi-\eta)\widehat{\omega^\ast}(\eta)\chi(\eta,\xi-\eta)(-|\eta|^{3/2}/2),
\end{split}
\end{equation*}
for some acceptable cubic remainders $H_{\geq 3}^2,H_{\geq 3}^3,H_{\geq 3}^4$. 
The desired quadratic formulas \eqref{Usor5.1}--\eqref{Usor5.2} follow by collecting the terms in the formulas above.
\end{proof}

\medskip
\section{Energy estimates and normal forms}\label{EnergyBounds}

We are now ready to start the proof of Theorem \ref{IncremBound}. Assume that $T_1\leq T_2\in\mathbb{R}$ and $(h,\phi)\in C([T_1,T_2]:H^{N_0}\times \dot{H}^{N_0,1/2})$ is a solution of the system \eqref{gWW}. As in section \ref{paralinearization} we assume that
\begin{equation}\label{eni1}
\begin{split}
\|h(t)\|_{H^{N_0}}+\||\partial_x|^{1/2}\phi(t)\|_{H^{N_0-1/2}}+\||\partial_x|^{1/2}\omega(t)\|_{H^{N_0}} &\leq A,\\
\|h(t)\|_{\dot{W}^{N_1,0}} + \||\partial_x|^{1/2}\phi(t)\|_{\dot{W}^{N_1-1/2,0}} + \||\partial_x|^{1/2}\omega(t)\|_{\dot{W}^{N_1,0}}&\leq \varep(t),
\end{split}
\end{equation}
for any $t\in[T_1,T_2]$, where $A\in (0,\infty)$ is a number and $\varep:[T_1,T_2]\to [0,\infty)$ is a continuous function satisfying $\varep(t)\ll 1$. 

For $\delta\in(0,1/100)$ let
\begin{equation}\label{m_choice}
m_{N_0}(\xi)=m_{N_0,\delta}(\xi):=\varphi_{\geq 0}(\xi)\langle\xi\rangle^{N_0}/(1+\delta\langle\xi\rangle^{N_0})
\end{equation}
and let $J^{N_0}$ denote the operator defined by the multiplier $m_{N_0}$. Let 
\begin{equation}\label{eni1.1}
A(t):=\big\|(h+i|\partial_x|^{1/2}\omega)(t)\big\|_{H^{N_0}}\lesssim A.
\end{equation}

We define the energy functional
\begin{equation}\label{eni3}
\mathcal{E}_2:=\int_{\T_R}J^{N_0}h\cdot J^{N_0}h\,dx+\int_{\T_R}J^{N_0}|\partial_x|\omega\cdot J^{N_0}\omega\,dx.
\end{equation}

\subsection{The first normal form} Our goal in this section is to eliminate the cubic bulk terms (using essentially a Shatah normal form), and show that the increment of the energy functional $\mathcal{E}_2$ is controlled by a space-time integral that can be written as sum of quartic expressions in the main variables $h^\ast$ and $\omega^\ast$ and an acceptable quintic contribution. At this stage it is very important to get precise formulas for the quartic increment, as described in Lemma \ref{quartic3}, and it is convenient to work with the real variables $h^\ast$ and $\omega^\ast$.

\begin{lemma}\label{quartic1}
For any $t_1\leq t_2\in[T_1,T_2]$ we have
\begin{equation}\label{eni3.5}
|\mathcal{E}_2(t_2)-\mathcal{E}_2(t_1)|\lesssim \sum_{j\in\{1,2\}}\varep(t_j)A(t_j)^2+\Big|\int_{t_1}^{t_2}\sum_{a\in\{1,\ldots,11\}}\mathcal{B}_{\geq 4}^a(s)\,ds\Big|+A^2\int_{t_1}^{t_2}[\varep(s)]^3\,ds,
\end{equation}
where the quartic expressions $\mathcal{B}_{\geq 4}^a$ are given by
\begin{equation}\label{eni5.1}
\mathcal{B}_{\geq 4}^1 := 2\int_{\T_R}J^{N_0}G_3(\phi,h,h)\cdot J^{N_0}h\,dx,
\end{equation}
\begin{equation}\label{eni5.2}
\mathcal{B}_{\geq 4}^2 := 2\int_{\T_R}J^{N_0}|\partial_x|\omega\cdot J^{N_0}L_3(h,\phi,\phi)\,dx,
\end{equation}
\begin{equation}\label{eni5.3}
\mathcal{B}_{\geq 4}^3 := -2\int_{\T_R}J^{N_0}\partial_x(T_Vh)\cdot T_\al J^{N_0}h\,dx
  + 2\int_{\T_R}J^{N_0}G_2(\phi,h)\cdot T_\al J^{N_0}h\,dx,
\end{equation}
\begin{equation}\label{eni5.5}
\mathcal{B}_{\geq 4}^4 := \frac{2}{(2\pi R)^2}\sum_{(\xi,\eta,\rho)\in\mathbb{H}^2_R}
  \big\{\widehat{h}(\xi)\widehat{h}(\eta)\widehat{\phi}(\rho)|\rho|^{1/2}
   -\widehat{h^\ast}(\xi)\widehat{h^\ast}(\eta)\widehat{\omega^\ast}(\rho)\big\}
   \frac{m_{N_0}^2(\xi)q_2(\rho,\eta)}{|\rho|^{1/2}},
\end{equation}
\begin{equation}\label{eni5.6}
\mathcal{B}_{\geq 4}^5 := \frac{2}{(2\pi R)^2}\sum_{(\xi,\eta,\rho)\in\mathbb{H}^2_R}
  \widehat{\omega}(\xi)\big\{\widehat{\phi}(\eta)|\eta|^{1/2}\widehat{\phi}(\rho)|\rho|^{1/2}
  -\widehat{\omega^\ast}(\eta)\widehat{\omega^\ast}(\rho)\big\}
  \frac{m_{N_0}^2(\xi)|\xi|\ell_2(\rho,\eta)}{|\eta|^{1/2}|\rho|^{1/2}},
\end{equation}
\begin{equation}\label{eni5.7}
\mathcal{B}_{\geq 4}^6 := \frac{2}{(2\pi R)^2}\sum_{(\xi,\eta,\rho)\in\mathbb{H}^2_R}
  \big\{\widehat{h}(\xi)\widehat{h}(\eta)\widehat{V}(\rho)
  -\widehat{h^\ast}(\xi)\widehat{h^\ast}(\eta)\widehat{\omega^\ast}(\rho)
  i\rho|\rho|^{-1/2}\big\}m_{N_0}^2(\xi)\chi(\rho,\eta)i\xi,
\end{equation}
\begin{equation}\label{eni5.8}
\mathcal{B}_{\geq 4}^7 := \frac{-2}{(2\pi R)^2}\sum_{(\xi,\eta,\rho)\in\mathbb{H}^2_R}
  \widehat{\omega}(\xi)\widehat{\omega}(\eta)
  \big\{\widehat{V}(\rho)-\widehat{\omega^\ast}(\rho)i\rho|\rho|^{-1/2}\big\}m_{N_0}^2(\xi)\chi(\rho,\eta)i\eta|\xi|,
\end{equation}
\begin{equation}\label{eni5.9}
\begin{split}
\mathcal{B}_{\geq 4}^8 := \frac{1}{(2\pi R)^2}\sum_{(\xi,\eta,\rho)\in\mathbb{H}^2_R}
  &\big\{\widehat{h}(\xi)\widehat{h}(\eta)\widehat{\partial_t\alpha}(\rho)\\
&+ \widehat{h^\ast}(\xi)\widehat{h^\ast}(\eta)\widehat{\omega^\ast}(\rho)
  |\rho|^{3/2}\big\}m_{N_0}(\xi)m_{N_0}(\eta)\chi(\rho,\eta),
\end{split}
\end{equation}
\begin{equation}\label{eni5.10}
\begin{split}
\mathcal{B}_{\geq 4}^9 := \frac{2}{(2\pi R)^2} \sum_{(\xi,\eta,\rho)\in\mathbb{H}^2_R}
  & \widehat{\omega^\ast}(\xi)\big\{ \widehat{h}(\eta)\widehat{\alpha}(\rho)
  \\
& + \widehat{h^\ast}(\eta)\widehat{h^\ast}(\rho)
  |\rho|\big\}m_{N_0}(\xi)|\xi|^{1/2}[m_{N_0}(\eta)-m_{N_0}(\xi)]\chi(\rho,\eta),
\end{split}
\end{equation}
and, with the symbols $A,B,D:\mathbb{H}^2_R\to\mathbb{R}$ defined as in \eqref{eni21}--\eqref{eni22} and \eqref{eni42},
\begin{equation}\label{eni5.11}
\begin{split}
\mathcal{B}^{10}_{\geq 4}
  &:=\frac{1}{(2\pi R)^2}\sum_{(\xi,\eta,\rho)\in\mathbb{H}^2_R}\frac{A(\xi,\eta,\rho)}{D(\xi,\eta,\rho)}\Big\{6|\rho|^{1/2}(|\xi|+|\eta|-|\rho|)\widehat{\mathcal{N}^{\omega^\ast}_{\geq 2}}(\xi)\widehat{\omega^\ast}(\eta)\widehat{h^\ast}(\rho)
  \\
&+3|\rho|^{1/2}(|\xi|+|\eta|-|\rho|)\widehat{\omega^\ast}(\xi)\widehat{\omega^\ast}(\eta)\widehat{\mathcal{N}^{h^\ast}_{\geq 2}}(\rho)+ 6|\xi|^{1/2}|\eta|^{1/2}|\rho|^{1/2}\widehat{h^\ast}(\xi)\widehat{h^\ast}(\eta)\widehat{\mathcal{N}^{h^\ast}_{\geq 2}}(\rho)\Big\},
\end{split}
\end{equation}
\begin{equation}\label{eni5.12}
\begin{split}
&\mathcal{B}^{11}_{\geq 4} 
  := \frac{1}{(2\pi R)^2}\sum_{(\xi,\eta,\rho)\in\mathbb{H}^2_R}\frac{B(\xi,\eta,\rho)}{D(\xi,\eta,\rho)}\\
&\times\Big\{2|\xi|^{1/2}|\eta|^{1/2}|\rho|^{1/2}
  \big[2\widehat{\mathcal{N}^{\omega^\ast}_{\geq 2}}(\xi)\widehat{\omega^\ast}(\eta)\widehat{h^\ast}(\rho) 
  + \widehat{\omega^\ast}(\xi)\widehat{\omega^\ast}(\eta)\widehat{\mathcal{N}^{h^\ast}_{\geq 2}}(\rho)\big]\\
& + |\rho|^{1/2}(|\xi|+|\eta|-|\rho|)
  \big[2\widehat{\mathcal{N}^{h^\ast}_{\geq 2}}(\xi)\widehat{h^\ast}(\eta)\widehat{h^\ast}(\rho)
  + \widehat{h^\ast}(\xi)\widehat{h^\ast}(\eta)\widehat{\mathcal{N}^{h^\ast}_{\geq 2}}(\rho)\big]\\
&-2|\eta|^{1/2}(|\xi|-|\eta|+|\rho|)
  \big[\widehat{\mathcal{N}^{h^\ast}_{\geq 2}}(\xi)\widehat{\omega^\ast}(\eta)\widehat{\omega^\ast}(\rho)
  + \widehat{h^\ast}(\xi)\widehat{\mathcal{N}^{\omega^\ast}_{\geq 2}}(\eta)\widehat{\omega^\ast}(\rho)
  + \widehat{h^\ast}(\xi)\widehat{\omega^\ast}(\eta)\widehat{\mathcal{N}^{\omega^\ast}_{\geq 2}}(\rho)\big]\Big\}.
\end{split}
\end{equation}
\end{lemma}

\begin{remark}  
The $\delta$-mollification of the weights in \eqref{m_choice} 
is only needed to gua\-ran\-tee that  our algebraic manipulations 
are justified (the integrals and sums converge absolutely), 
but all the bounds are independent on $\delta$. We will later let $\delta\to 0$ to prove Theorem \ref{IncremBound}.

We notice also that our multipliers $m_{N_0}$ are supported at frequencies $|\xi|\gtrsim 1$.
This is useful to avoid some low-frequency issues, related to our functional framework. The low-frequency part of the energy can be controlled separately, using the conservation 
of the Hamiltonian \eqref{gWWHam}, see section \ref{secadmver} for details.
\end{remark}

\begin{proof} 
The variables $h$ and $\omega$ satisfy the evolution equations
\begin{equation}\label{eni10}
\left\{
\begin{array}{l}
\partial_t h = |\partial_x|\omega- \partial_x(T_V h) + G_2(\phi,h) + G_3(\phi,h,h) +G_{\geq 4},
\\
\partial_t \omega = -h- T_{\alpha} h - T_V \partial_x\omega + L_2(\phi,\phi)
  + L_{3}(h,\phi,\phi) + L_{\geq 4},
\end{array}
\right.
\end{equation}
 as a consequence of \eqref{gWW}, \eqref{DNmainformula}, and \eqref{hphieq1}. 

{\bf{Step 1.}} With $\mathcal{E}_2$ defined as in \eqref{eni3} we start by calculating
\begin{equation}\label{eni11}
\begin{split}
\partial_t\mathcal{E}_2&=2\int_{\T_R}J^{N_0}\partial_th\cdot J^{N_0}h\,dx
+2\int_{\T_R}J^{N_0}|\partial_x|\omega\cdot J^{N_0}\partial_t\omega\,dx
\\
&=\mathcal{B}_{\geq 3}^1+\mathcal{B}_{\geq 3}^2+\mathcal{B}_{\geq 3}^3
+\mathcal{B}_{\geq 3}^4+\mathcal{B}_{\geq 3}^\#+\mathcal{B}_{\geq 4}^1+\mathcal{B}_{\geq 4}^2+\mathcal{B}_{\geq 5}^1,
\end{split}
\end{equation}
where the quartic terms $\mathcal{B}_{\geq 4}^1$ and $\mathcal{B}_{\geq 4}^2$
are defined as in \eqref{eni5.1}--\eqref{eni5.2}, 
$\mathcal{B}^1_{\geq 5}$ is an acceptable remainder satisfying 
$|\mathcal{B}^1_{\geq 5}(t)|\lesssim A^2[\varep(t)]^3$, and
\begin{equation}\label{eni12}
\begin{split}
\mathcal{B}_{\geq 3}^1&:=2\int_{\T_R}J^{N_0}G_2(\phi,h)\cdot J^{N_0}h\,dx,\qquad \mathcal{B}_{\geq 3}^2:=2\int_{\T_R}J^{N_0}|\partial_x|\omega\cdot J^{N_0}L_2(\phi,\phi)\,dx,\\
\mathcal{B}_{\geq 3}^3&:=-2\int_{\T_R}J^{N_0}\partial_xT_Vh\cdot J^{N_0}h\,dx,\qquad \mathcal{B}_{\geq 3}^4:=-2\int_{\T_R}J^{N_0}|\partial_x|\omega\cdot J^{N_0}T_V\partial_x\omega\,dx,\\
\mathcal{B}_{\geq 3}^\#&:=-2\int_{\T_R}J^{N_0}|\partial_x|\omega\cdot J^{N_0}T_\al h\,dx.
\end{split}
\end{equation}

Let
\begin{equation}\label{eni4}
\mathcal{E}^1_{\geq 3}:=\int_{\T_R}J^{N_0}h\cdot T_\al J^{N_0}h\,dx,
\end{equation}
where $\alpha:=\partial_tB + V\partial_xB$ as before. Using also \eqref{Taf2}, we calculate
\begin{equation}\label{eni13}
\begin{split}
\partial_t\mathcal{E}^1_{\geq 3}&=2\int_{\T_R}J^{N_0}\partial_th\cdot T_\al J^{N_0}h\,dx+\int_{\T_R}J^{N_0}h\cdot T_{\partial_t\al} J^{N_0}h\,dx\\
&=-\mathcal{B}_{\geq 3}^\#+\mathcal{B}_{\geq 3}^5+\mathcal{B}_{\geq 3}^6+\mathcal{B}_{\geq 4}^3+\mathcal{B}_{\geq 5}^2,
\end{split}
\end{equation}
where the cubic term $\mathcal{B}_{\geq 3}^\#$ is defined as in \eqref{eni12}, the quartic term $\mathcal{B}_{\geq 4}^3$
is defined as in \eqref{eni5.3},
$\mathcal{B}^2_{\geq 5}$ is an acceptable remainder satisfying $|\mathcal{B}^2_{\geq 5}(t)|\lesssim A^2[\varep(t)]^3$, 
and
\begin{equation}\label{eni14}
\begin{split}
\mathcal{B}_{\geq 3}^5&:=\int_{\T_R}J^{N_0}h\cdot T_{\partial_t\al} J^{N_0}h\,dx,\\
\mathcal{B}_{\geq 3}^6&:=2\int_{\T_R}J^{N_0}|\partial_x|\omega\cdot \big\{T_\al J^{N_0}h-J^{N_0}T_\al h\big\}\,dx.
\end{split}
\end{equation}
The point of the correction cubic energy functional $\mathcal{E}^1_{\geq 3}$ is to cancel the term $\mathcal{B}_{\geq 3}^\#$ which loses derivatives. To summarize,
\begin{equation}\label{eni15}
\begin{split}
\partial_t(\mathcal{E}_2+\mathcal{E}^1_{\geq 3})&=(\mathcal{B}_{\geq 3}^1+\mathcal{B}_{\geq 3}^2+\mathcal{B}_{\geq 3}^3+\mathcal{B}_{\geq 3}^4+\mathcal{B}_{\geq 3}^5+\mathcal{B}_{\geq 3}^6)\\
&+(\mathcal{B}_{\geq 4}^1+\mathcal{B}_{\geq 4}^2+\mathcal{B}_{\geq 4}^3)+(\mathcal{B}_{\geq 5}^1+\mathcal{B}_{\geq 5}^2),
\end{split}
\end{equation}
where $\mathcal{B}_{\geq 5}^1$ and $\mathcal{B}_{\geq 5}^2$ are acceptable quintic errors, 
and the cubic terms $\mathcal{B}_{\geq 3}^1,\ldots,\mathcal{B}_{\geq 3}^6$ and the quartic 
terms $\mathcal{B}_{\geq 4}^1, \mathcal{B}_{\geq 4}^2, \mathcal{B}_{\geq 4}^3$ are defined as above.

{\bf{Step 2.}} We would like now to express the remaining cubic bulk terms in terms of the normalized variables $h^\ast=T_{\sqrt{1+\alpha}}h$ and $\omega^\ast=|\partial_x|^{1/2}\omega$ defined in Lemma \ref{Usor1}. For this we pass to the Fourier space and rewrite, 
\begin{equation}\label{eni16}
\begin{split}
&\mathcal{B}_{\geq 3}^1=\frac{2}{(2\pi R)^2}\sum_{(\xi,\eta,\rho)\in\mathbb{H}^2_R}\widehat{h}(\xi)\widehat{h}(\eta)\widehat{\phi}(\rho)m_{N_0}^2(\xi)q_2(\rho,\eta)=\mathcal{B}_{\geq 3}^{1,0}+\mathcal{B}_{\geq 4}^4,\\
&\mathcal{B}_{\geq 3}^{1,0}:=\frac{2}{(2\pi R)^2}\sum_{(\xi,\eta,\rho)\in\mathbb{H}^2_R}\widehat{h^\ast}(\xi)\widehat{h^\ast}(\eta)\widehat{\omega^\ast}(\rho)m_{N_0}^2(\xi)\frac{q_2(\rho,\eta)}{|\rho|^{1/2}},
\end{split}
\end{equation}
where $q_2$ is as in \eqref{DNmain2} and $\mathcal{B}_{\geq 4}^4$ is defined as in \eqref{eni5.5}. Similarly
\begin{equation}\label{eni17}
\begin{split}
&\mathcal{B}_{\geq 3}^2=\frac{2}{(2\pi R)^2}\sum_{(\xi,\eta,\rho)\in\mathbb{H}^2_R}\widehat{\omega}(\xi)\widehat{\phi}(\eta)\widehat{\phi}(\rho)m_{N_0}^2(\xi)|\xi|\ell_2(\rho,\eta)=\mathcal{B}_{\geq 3}^{2,0}+\mathcal{B}_{\geq 4}^5,\\
&\mathcal{B}_{\geq 3}^{2,0}:=\frac{2}{(2\pi R)^2}\sum_{(\xi,\eta,\rho)\in\mathbb{H}^2_R}\widehat{\omega^\ast}(\xi)\widehat{\omega^\ast}(\eta)\widehat{\omega^\ast}(\rho)m_{N_0}^2(\xi)|\xi|^{1/2}\frac{\ell_2(\rho,\eta)}{|\eta|^{1/2}|\rho|^{1/2}},
\end{split}
\end{equation}
where $\ell_2$ is as in \eqref{hoeq2} and $\mathcal{B}_{\geq 4}^5$ is defined as in \eqref{eni5.6}. 
Also,
\begin{equation}\label{eni18}
\begin{split}
&\mathcal{B}_{\geq 3}^3=\frac{2}{(2\pi R)^2}\sum_{(\xi,\eta,\rho)\in\mathbb{H}^2_R}
  \widehat{h}(\xi)\widehat{h}(\eta)\widehat{V}(\rho)m_{N_0}^2(\xi)\chi(\rho,\eta)i\xi=
  \mathcal{B}_{\geq 3}^{3,0}+\mathcal{B}_{\geq 4}^6,
  \\
&\mathcal{B}_{\geq 3}^{3,0}:=\frac{-2}{(2\pi R)^2}\sum_{(\xi,\eta,\rho)\in\mathbb{H}^2_R}
  \widehat{h^\ast}(\xi)\widehat{h^\ast}(\eta)\widehat{\omega^\ast}(\rho)m_{N_0}^2(\xi)\chi(\rho,\eta)\xi\rho|\rho|^{-1/2},
\end{split}
\end{equation}
\begin{equation}\label{eni19}
\begin{split}
&\mathcal{B}_{\geq 3}^4=\frac{-2}{(2\pi R)^2}\sum_{(\xi,\eta,\rho)\in\mathbb{H}^2_R}\widehat{\omega}(\xi)\widehat{\omega}(\eta)\widehat{V}(\rho)m_{N_0}^2(\xi)\chi(\rho,\eta)i\eta|\xi|=\mathcal{B}_{\geq 3}^{4,0}+\mathcal{B}_{\geq 4}^7,\\
&\mathcal{B}_{\geq 3}^{4,0}:=\frac{2}{(2\pi R)^2}\sum_{(\xi,\eta,\rho)\in\mathbb{H}^2_R}\widehat{\omega^\ast}(\xi)\widehat{\omega^\ast}(\eta)\widehat{\omega^\ast}(\rho)m_{N_0}^2(\xi)\chi(\rho,\eta)\eta|\eta|^{-1/2}\rho|\rho|^{-1/2}|\xi|^{1/2},
\end{split}
\end{equation}
where $\mathcal{B}_{\geq 4}^6$ and $\mathcal{B}_{\geq 4}^7$ are as in \eqref{eni5.7}--\eqref{eni5.8}. Finally, using also \eqref{Utor10}--\eqref{Utor11},
\begin{equation}\label{eni20}
\begin{split}
&\mathcal{B}_{\geq 3}^5=\frac{1}{(2\pi R)^2}\sum_{(\xi,\eta,\rho)\in\mathbb{H}^2_R}\widehat{h}(\xi)\widehat{h}(\eta)\widehat{\partial_t\alpha}(\rho)m_{N_0}(\xi)m_{N_0}(\eta)\chi(\rho,\eta)=\mathcal{B}_{\geq 3}^{5,0}+\mathcal{B}_{\geq 4}^8,\\
&\mathcal{B}_{\geq 3}^{5,0}:=\frac{-1}{(2\pi R)^2}\sum_{(\xi,\eta,\rho)\in\mathbb{H}^2_R}\widehat{h^\ast}(\xi)\widehat{h^\ast}(\eta)\widehat{\omega^\ast}(\rho)m_{N_0}(\xi)m_{N_0}(\eta)\chi(\rho,\eta)|\rho|^{3/2},
\end{split}
\end{equation}
\begin{equation}\label{eni20.5}
\begin{split}
&\mathcal{B}_{\geq 3}^6=\frac{2}{(2\pi R)^2}\sum_{(\xi,\eta,\rho)\in\mathbb{H}^2_R}\widehat{\omega}(\xi)\widehat{h}(\eta)\widehat{\alpha}(\rho)m_{N_0}(\xi)|\xi|[m_{N_0}(\eta)-m_{N_0}(\xi)]\chi(\rho,\eta)=\mathcal{B}_{\geq 3}^{6,0}+\mathcal{B}_{\geq 4}^9,\\
&\mathcal{B}_{\geq 3}^{6,0}:=\frac{2}{(2\pi R)^2}\sum_{(\xi,\eta,\rho)\in\mathbb{H}^2_R}\widehat{\omega^\ast}(\xi)\widehat{h^\ast}(\eta)\widehat{h^\ast}(\rho)m_{N_0}(\xi)|\xi|^{1/2}|\rho|[m_{N_0}(\xi)-m_{N_0}(\eta)]\chi(\rho,\eta)
\end{split}
\end{equation}
where $\mathcal{B}_{\geq 4}^8,\mathcal{B}_{\geq 4}^9$ are defined as in \eqref{eni5.9}--\eqref{eni5.10}.

We examine the cubic expressions $\mathcal{B}^{2,0}_{\geq 3}$ and $\mathcal{B}^{4,0}_{\geq 3}$. We symmetrize in $\xi,\eta,\rho$ to avoid derivative loss and rewrite
\begin{equation}\label{eni21}
\begin{split}
&\mathcal{B}^{2,0}_{\geq 3}+\mathcal{B}^{4,0}_{\geq 3}
=\mathcal{B}^{7}_{\geq 3}:=\frac{1}{(2\pi R)^2}\sum_{(\xi,\eta,\rho)\in\mathbb{H}^2_R}
  \widehat{\omega^\ast}(\xi)\widehat{\omega^\ast}(\eta)\widehat{\omega^\ast}(\rho)A(\xi,\eta,\rho),
\\
& A(\xi,\eta,\rho):=\frac{\widetilde{A}(\xi,\eta,\rho)+\widetilde{A}(\eta,\rho,\xi)+\widetilde{A}(\rho,\xi,\eta)}{3},
\\
& \widetilde{A}(\xi,\eta,\rho):=m_{N_0}^2(\xi)|\xi|^{1/2}\frac{\eta\rho+\widetilde{\chi}(\rho,\eta)|\eta||\rho|}{|\eta|^{1/2}|\rho|^{1/2}}.
\end{split}
\end{equation}
Similarly, we combine the expressions $\mathcal{B}^{1,0}_{\geq 3}$, $\mathcal{B}^{3,0}_{\geq 3}$, 
$\mathcal{B}^{5,0}_{\geq 3}$, $\mathcal{B}^{6,0}_{\geq 3}$, and symmetrize in $\xi$ and $\eta$ to rewrite
\begin{equation}\label{eni22}
\begin{split}
&\mathcal{B}^{1,0}_{\geq 3}+\mathcal{B}^{3,0}_{\geq 3}+\mathcal{B}^{5,0}_{\geq 3}
+\mathcal{B}^{6,0}_{\geq 3}=\mathcal{B}^{8}_{\geq 3}:=\frac{1}{(2\pi R)^2}
\sum_{(\xi,\eta,\rho)\in\mathbb{H}^2_R}\widehat{h^\ast}(\xi)\widehat{h^\ast}(\eta)\widehat{\omega^\ast}(\rho)B(\xi,\eta,\rho),
\\
& B(\xi,\eta,\rho):=-|\rho|^{-1/2}\big\{m_{N_0}^2(\xi)\widetilde{\chi}(\rho,\eta)(\xi\rho+|\xi||\rho|)
+m_{N_0}^2(\eta)\widetilde{\chi}(\rho,\xi)(\eta\rho+|\eta||\rho|)\big\}
\\
&\quad-\rho|\rho|^{-1/2}\big\{m_{N_0}^2(\xi)\xi\chi(\rho,\eta)+m_{N_0}^2(\eta)\eta\chi(\rho,\xi)\big\}
-m_{N_0}(\xi)m_{N_0}(\eta)|\rho|^{3/2}\frac{\chi(\rho,\eta)+\chi(\rho,\xi)}{2}
\\
&\quad+ m_{N_0}(\rho)|\rho|^{1/2}\big\{|\xi|[m_{N_0}(\rho)-m_{N_0}(\eta)]\chi(\xi,\eta)+|\eta|[m_{N_0}(\rho)-m_{N_0}(\xi)]\chi(\eta,\xi)\big\}.
\end{split}
\end{equation}
Notice that the symbol $A$ is symmetric in all three variables, while the symbol $B$ is symmetric in the first two variables. 

It is easy to see that the multipliers $A$ and $B$ satisfy the 
bounds
\begin{equation}\label{ABmultHM}
\begin{split}
&\big\|A(\xi,\eta,\rho)\cdot\varphi_{k_1}(\xi)\varphi_{k_2}(\eta)\varphi_{k_3}(\rho)\big\|_{\mathcal{HM}_9}
\lesssim 2^{3\min(k_1,k_2,k_3)/2}2^{2N_0\max(k_1,k_2,k_3)},\\
&\big\|B(\xi,\eta,\rho)\cdot\varphi_{k_1}(\xi)\varphi_{k_2}(\eta)\varphi_{k_3}(\rho)\big\|_{\mathcal{HM}_9}\lesssim 2^{3\min(k_1,k_2,k_3)/2}2^{2N_0\max(k_1,k_2,k_3)}2^{-|k_1-k_2|/2},
\end{split}
\end{equation}
for any $k_1,k_2,k_3\in\Z$ (recall the definition \eqref{HMD1}--\eqref{HMD2}). In particular
\begin{equation}\label{ABmultBounds}
\begin{split}
&\big\|A(\xi,\eta,\rho)\cdot\varphi_{k_1}(\xi)\varphi_{k_2}(\eta)\varphi_{k_3}(\rho)\big\|_{\widetilde{S}^\infty}
\lesssim 2^{3\min(k_1,k_2,k_3)/2}2^{2N_0\max(k_1,k_2,k_3)},\\
&\big\|B(\xi,\eta,\rho)\cdot\varphi_{k_1}(\xi)\varphi_{k_2}(\eta)\varphi_{k_3}(\rho)\big\|_{\widetilde{S}^\infty}\lesssim 2^{3\min(k_1,k_2,k_3)/2}2^{2N_0\max(k_1,k_2,k_3)}2^{-|k_1-k_2|/2},
\end{split}
\end{equation}
for any $k_1,k_2,k_3\in\Z$, due to Lemma \ref{multiBound} (iv) (recall the definition of the $\widetilde{S}^\infty$ norm in \eqref{hyperplane2}).

To summarize, we showed that
\begin{equation}\label{eni25}
\partial_t(\mathcal{E}_2+\mathcal{E}^1_{\geq 3})
=\mathcal{B}_{\geq 3}^7+\mathcal{B}_{\geq 3}^8+\Big\{\sum_{a\in\{1,\ldots,9\}}\mathcal{B}_{\geq 4}^a\Big\}
+(\mathcal{B}_{\geq 5}^1+\mathcal{B}_{\geq 5}^2),
\end{equation}
where $\mathcal{B}_{\geq 5}^1$ and $\mathcal{B}_{\geq 5}^2$ are acceptable quintic errors. 
The cubic terms $\mathcal{B}_{\geq 3}^7, \mathcal{B}_{\geq 3}^8$ 
and the quartic terms $\mathcal{B}_{\geq 4}^a$ are defined as in the identities 
\eqref{eni21}--\eqref{eni22} and \eqref{eni5.1}--\eqref{eni5.10} above.

\medskip
{\bf{Step 3.}} 
We use now a normal form transformation to remove the cubic bulk terms $\mathcal{B}_{\geq 3}^7$ and $\mathcal{B}_{\geq 3}^8$. For this we define 
\begin{align}\label{eni26}
\begin{split}
D(\xi,\eta,\rho) &:=|\xi|(|\xi|-|\eta|-|\rho|)+|\eta|(|\eta|-|\rho|-|\xi|)+|\rho|(|\rho|-|\xi|-|\eta|)\\
&=|\xi|^2 + |\eta|^2 + |\rho|^2 - 2|\xi||\eta| - 2|\xi||\rho| - 2|\eta||\rho|,
\end{split}
\end{align}
and two cubic energy functionals 
\begin{equation}\label{eni40}
\begin{split}
\mathcal{E}^2_{\geq 3}:=\frac{1}{(2\pi R)^2}\sum_{(\xi,\eta,\rho)\in\mathbb{H}^2_R}\frac{A(\xi,\eta,\rho)}{D(\xi,\eta,\rho)}\Big\{&3\widehat{\omega^\ast}(\xi)\widehat{\omega^\ast}(\eta)\widehat{h^\ast}(\rho)|\rho|^{1/2}(-|\xi|-|\eta|+|\rho|)\\
&-2\widehat{h^\ast}(\xi)\widehat{h^\ast}(\eta)\widehat{h^\ast}(\rho)|\xi|^{1/2}|\eta|^{1/2}|\rho|^{1/2}\Big\},
\end{split}
\end{equation}
and
\begin{equation}\label{eni41}
\begin{split}
\mathcal{E}^3_{\geq 3}:=&\frac{1}{(2\pi R)^2}\sum_{(\xi,\eta,\rho)\in\mathbb{H}^2_R}\frac{B(\xi,\eta,\rho)}{D(\xi,\eta,\rho)}\Big\{2\widehat{h^\ast}(\xi)\widehat{\omega^\ast}(\eta)\widehat{\omega^\ast}(\rho)|\eta|^{1/2}(|\xi|-|\eta|+|\rho|)\\
&-2\widehat{\omega^\ast}(\xi)\widehat{\omega^\ast}(\eta)\widehat{h^\ast}(\rho)|\xi|^{1/2}|\eta|^{1/2}|\rho|^{1/2}+\widehat{h^\ast}(\xi)\widehat{h^\ast}(\eta)\widehat{h^\ast}(\rho)|\rho|^{1/2}(-|\xi|-|\eta|+|\rho|)\Big\}.
\end{split}
\end{equation}

Notice that if $\xi,\eta,\rho\in\mathbb{R}$ and $\xi+\eta+\rho=0$ then
\begin{equation}\label{eni42}
D(\xi,\eta,\rho)=\frac{-4|\xi||\eta||\rho|}{\max(|\xi|,|\eta|,|\rho|)}.
\end{equation}
Using \eqref{ABmultBounds}, Lemma \ref{multiBound} (iv), and \eqref{al8.4}, and letting $A':=A/D$ and $B'=B/D$, we have
\begin{equation}\label{ABmultBounds2}
\begin{split}
\begin{split}
&\big\|A'(\xi,\eta,\rho)\cdot\varphi_{k_1}(\xi)\varphi_{k_2}(\eta)\varphi_{k_3}(\rho)\big\|_{\widetilde{S}^\infty}
  \lesssim 2^{\min(k_1,k_2,k_3)/2}2^{(2N_0-1)\max(k_1,k_2,k_3)},
\\
&\big\|B'(\xi,\eta,\rho)\cdot\varphi_{k_1}(\xi)\varphi_{k_2}(\eta)\varphi_{k_3}(\rho)\big\|_{\widetilde{S}^\infty}
  \lesssim 2^{\min(k_1,k_2,k_3)/2}2^{(2N_0-1)\max(k_1,k_2,k_3)}2^{-|k_1-k_2|/2},
\end{split}
\end{split}
\end{equation}
for any $k_1,k_2,k_3\in\Z$. Recalling also the definition \eqref{eni4} and using Lemma \ref{touse} (ii) (with the two high frequencies estimated in $L^2$ and the low frequency estimated in $L^\infty$), we have
\begin{equation}\label{eni43}
\mathcal{E}^1_{\geq 3}(t)+\mathcal{E}^2_{\geq 3}(t)+\mathcal{E}^3_{\geq 3}(t)\lesssim\varep(t)\big\|(h+i|\partial_x|^{1/2}\omega)(t)\big\|_{H^{N_0}}
\end{equation} 
for any $t\in[T_1,T_2]$. Moreover, the formulas $\partial_t h^\ast=|\partial_x|^{1/2}\omega^\ast+\mathcal{N}_{\geq 2}^{h^\ast}$ and 
$\partial_t \omega^\ast=-|\partial_x|^{1/2} h^\ast+\mathcal{N}_{\geq 2}^{\omega^\ast}$ in \eqref{Usor3} show that
\begin{equation*}
\partial_t\mathcal{E}^2_{\geq 3} = \mathcal{B}_{\geq 3}^7 - \mathcal{B}_{\geq 4}^{10}, 
\qquad \partial_t\mathcal{E}^3_{\geq 3} = \mathcal{B}_{\geq 3}^8 - \mathcal{B}_{\geq 4}^{11},
\end{equation*}
where the quartic terms $\mathcal{B}_{\geq 4}^{10}$ and $\mathcal{B}_{\geq 4}^{11}$ are defined in \eqref{eni5.11}--\eqref{eni5.12}. Recalling also \eqref{eni25}, the desired conclusion of the lemma follows.
\end{proof}

\subsection{Analysis of the quartic bulk integrals}\label{EBssec1}
We examine now more closely the space-time integrals generated in Lemma \ref{quartic1}. 
We define first two general classes of symbols, ``weakly admissible'' and ``regular'',
that appear in many of our operators.
In the next section we will need to introduce the more specific class of ``admissible'' symbols
(see Definition \ref{defA}). 




\smallskip
\begin{definition}[Weakly Admissible Symbols]\label{AccSymb}
A symbol $M:\mathbb{H}_R^3\to\mathbb{C}$ will be called ``weakly admissible'' if:

\smallskip

(i) $M(\xi_1,\xi_2,\xi_3,\xi_4)=0$ if one of the arguments $\xi_j$ is equal to $0$; 

\smallskip
(ii) For any  lattice point $\underline{k}=(k_1,k_2,k_3,k_4)\in\Z^4$
\begin{equation}\label{hoe3}
{\| M \cdot \varphi_{\underline{k}} \|}_{\wt{S}^\infty} 
  \lesssim 2^{\widetilde{k}_4/2}2^{2\widetilde{k}_3}2^{N_0\max(\widetilde{k}_1,0)}2^{N_0\max(\widetilde{k}_2,0)},
\end{equation}
see \eqref{hyperplane2}.
Here $\varphi_{\underline{k}}(\xi)=\varphi_{k_1}(\xi_1)\varphi_{k_2}(\xi_2)\varphi_{k_3}(\xi_3)\varphi_{k_4}(\xi_4)$ 
and $\widetilde{k}_1\geq\widetilde{k}_2\geq\widetilde{k}_3\geq\widetilde{k}_4$ 
is the decreasing rearrangement of the set $k_1,k_2,k_3,k_4$.

\end{definition}

\smallskip
\begin{definition}[Regular Symbols]\label{defMNab}
A symbol $m:\mathbb{H}^3_R \to \C$, will be called ``regular of class $\mathcal{M}^b$''
if for every $\underline{k}=(k_1,\ldots,k_4)\in \Z^4$
we have
\begin{equation}\label{HMD1'}
{\big\| m \cdot \varphi_{\underline{k}} \big\|}_{\mathcal{HM}_8}
  \lesssim 2^{\widetilde{k}_4/2} 2^{(b-1/2)\widetilde{k}_3} 
  \cdot 2^{N_0\max(\wt{k}_1,0)} 2^{N_0\max(\wt{k}_2,0)},
\end{equation}
(recall Definition \ref{defHM}) for some $b \geq 0$.
\end{definition}

\begin{remark}\label{defRrem} (i) A regular symbol of class $\mathcal{M}^{5/2}$ is weakly admissible, due to Lemma \ref{multiBound} (iv). Regular symbols of class $\mathcal{M}^{5/2}$ will often be called simply ``regular" in the rest of the paper.
\smallskip

(ii) Using Lemma \ref{algeProp7} and Lemma \ref{multiBound} (iv), we have
\begin{equation}\label{shortcut}
\Big|\frac{1}{(2\pi R)^3}\sum_{(\xi_1,\xi_2,\xi_3,\xi_4)\in\mathbb{H}^3_R}
  \widehat{h_1}(\xi_1)\widehat{h_2}(\xi_2)\widehat{h_3}(\xi_3)\widehat{h_4}(\xi_4)
  M(\xi_1,\xi_2,\xi_3,\xi_4)\Big|\lesssim A_2^2A_\infty^2,
\end{equation}
provided that $M$ is a weakly admissible quartic symbol and the
functions $h_j$ satisfy $\|h_j\|_{H^{N_0}}\lesssim A_2$ and $\|h_j\|_{\widetilde{W}^2}
  \lesssim A_\infty$, $j\in\{1,2,3,4\}$; recall \eqref{Wtilde}.

\end{remark}

\smallskip
We now calculate explicitly the quartic bulk terms in Lemma \ref{quartic1}.
We will further manipulate and re-arrange these symbols in Lemma \ref{quartic3}
to obtain more convenient expressions for the purpose of verifying
suitable non-resonance conditions later on.

\begin{lemma}\label{quartic2}
With the hypothesis and notation in Lemma \ref{quartic1}, for any $t\in[T_1,T_2]$ we have
\begin{equation}\label{reh1}
\sum_{a\in\{1,\ldots,11\}}\mathcal{B}_{\geq 4}^a(t)
  = \mathcal{Q}_1(h^\ast,h^\ast,h^\ast,\omega^\ast)(t)+\mathcal{Q}_2(\omega^\ast,\omega^\ast,\omega^\ast,h^\ast)(t)
  +\mathcal{B}_{\geq 5}(t),
\end{equation}
where $\mathcal{B}_{\geq 5}(t)$ is an acceptable error satisfying 
\begin{align}\label{quartic2B5}
|\mathcal{B}_{\geq 5}(t)|\lesssim A^2[\varep(t)]^3,
\end{align}
and
\begin{equation}\label{reh2}
\begin{split}
\mathcal{Q}_1(h^\ast,h^\ast,h^\ast,\omega^\ast)&=\frac{1}{(2\pi R)^3}\sum_{(\xi,\eta,\rho,\sigma)\in\mathbb{H}^3_R}
  \widehat{h^\ast}(\xi)\widehat{h^\ast}(\eta)\widehat{h^\ast}(\rho)\widehat{\omega^\ast}(\sigma)
  M_1(\xi,\eta,\rho,\sigma),
  \\
\mathcal{Q}_2(\omega^\ast,\omega^\ast,\omega^\ast,h^\ast)&=\frac{1}{(2\pi R)^3}
  \sum_{(\xi,\eta,\rho,\sigma)\in\mathbb{H}^3_R}
  \widehat{\omega^\ast}(\xi)\widehat{\omega^\ast}(\eta)\widehat{\omega^\ast}(\rho)\widehat{h^\ast}(\sigma)
  M_2(\xi,\eta,\rho,\sigma).
\end{split}
\end{equation}
The symbols $M_1$ and $M_2$ are given by
\begin{equation}\label{reh3}
\begin{split}
& M_1(\xi,\eta,\rho,\sigma) = \sum_{a\in\{1, \ldots, 14 \}} M_1^a(\xi,\eta,\rho,\sigma),
\\
& M_2(\xi,\eta,\rho,\sigma) = \sum_{a\in\{1, \ldots, 6\}}M_2^a(\xi,\eta,\rho,\sigma),
\end{split}
\end{equation}
where $M_1^a$ and $M_2^b$ are weakly admissible 
(see Definition \ref{AccSymb})
given explicitly by
\begin{equation}\label{mu11}
\begin{split}
M^1_1(\xi,\eta,\rho,\sigma)&:=
  m_{N_0}^2(\xi) |\sigma|^{1/2} \Big\{[1-\chi(\sigma+\rho,\eta)][|\sigma+\rho||\xi|-|\sigma||\xi|+\rho\xi]
  \\
  & + [1-\chi(\sigma+\eta,\rho)][|\sigma+\eta||\xi|-|\sigma||\xi|+\eta\xi]+\big[|\sigma||\xi|+\sigma\xi\big]\Big\},
\end{split}
\end{equation}
\begin{align}\label{mu12}
\begin{split}
M_1^2(\xi,\eta,\rho,\sigma) :=&- m_{N_0}(\xi+\si) m_{N_0}(\eta) (\xi+\si) \si|\si|^{-1/2} 
  |\rho| \chi(\si,\xi) \chi(\rho,\eta)
  \\ & - m_{N_0}(\eta+\si) m_{N_0}(\xi) (\eta+\si) \si|\si|^{-1/2} |\rho| \chi(\si,\eta) \chi(\rho,\xi),
\end{split}
\end{align}
\begin{align}\label{mu13}
M_1^3(\xi,\eta,\rho,\sigma) &:=  - 2 m_{N_0}(\xi+\si) m_{N_0}(\eta) 
[(\xi+\si)\si-|\xi+\si||\si|] |\si|^{-1/2} |\rho|\widetilde{\chi}(\si,\xi)\chi(\rho,\eta),
\end{align}
\begin{align}\label{mu14}
M_1^4(\xi,\eta,\rho,\sigma) &:= \frac{m_{N_0}^2(\xi+\eta)\widetilde{\chi}(\si,\rho)\chi(\eta,\xi)|\eta|[(\rho+\sigma)\sigma-|\rho+\sigma||\sigma|]}{|\si|^{1/2}},
\end{align}
\begin{align}\label{mu15}
M_1^5(\xi,\eta,\rho,\sigma) & := m_{N_0}^2(\xi)\frac{|\rho|q_2(\si,\eta+\rho)\chi(\rho,\eta) +2
  |\si|q_2(\rho + \si,\eta)\chi(\si,\rho)}{|\si|^{1/2}},
\end{align}
\begin{align}\label{mu16}
\begin{split}
M_1^6(\xi,\eta,\rho,\sigma):=&-\frac{1}{2}m_{N_0}^2(\xi+\rho)\chi(\si,\eta)\chi(\rho,\xi) (\xi+\rho)|\rho|\si|\si|^{-1/2}\\
&-\frac{1}{2}m_{N_0}^2(\eta+\rho)\chi(\si,\xi)\chi(\rho,\eta) (\eta+\rho)|\rho|\si|\si|^{-1/2},
\end{split}
\end{align}
\begin{align}\label{mu17}
\begin{split}
M_1^7(\xi,\eta,\rho,\sigma) :=&-\frac{1}{2}\xi m_{N_0}^2(\xi)\chi(\si,\eta+\rho)\chi(\rho,\eta)|\rho|\si|\si|^{-1/2}\\
&-\frac{1}{2}\eta m_{N_0}^2(\eta)\chi(\si,\xi+\rho)\chi(\rho,\xi)|\rho|\si|\si|^{-1/2},
\end{split}
\end{align}
\begin{align}\label{mu18}
\begin{split}
M_1^8(\xi,\eta,\rho,\sigma):=& -|\si|^{1/2} \big[ (\si+\rho) \chi(\si,\rho) - \rho\big]\\
&\times\big\{\xi m_{N_0}^2(\xi)\chi(\rho+\si,\eta)+\eta m_{N_0}^2(\eta)\chi(\rho+\si,\xi)\big\},
\end{split}
\end{align}
\begin{align}\label{mu19}
M_1^9(\xi,\eta,\rho,\si) := m_{N_0}(\xi)m_{N_0}(\eta)\chi(\rho+\si,\eta) b(\si,\rho),
\end{align}
\begin{align}\label{mu110}
M_1^{10}(\xi,\eta,\rho,\si) & :=-\frac{1}{2}  m_{N_0}(\xi+\rho) m_{N_0}(\eta) [\chi(\si,\eta)+\chi(\si,\xi+\rho)]\chi(\rho,\xi)|\rho| |\si|^{3/2},
\end{align}
\begin{align}\label{mu111}
M_1^{11}(\xi,\eta,\rho,\si):= - m_{N_0}(\si) [m_{N_0}(\xi+\eta) - m_{N_0}(\si)] \chi(\rho,\xi+\eta)\chi(\xi,\eta) |\rho| |\si|^{1/2} \, |\xi| ,
\end{align}
\begin{align}\label{mu112}
M_1^{12}(\xi,\eta,\rho,\si):= 2m_{N_0}(\si) [m_{N_0}(\eta) - m_{N_0}(\si)] \chi(\rho+\xi,\eta) |\si| ^{1/2} a_2^{rr}(\rho,\xi),
\end{align}
\begin{equation}\label{mu113}
\begin{split}
M_1^{13}(\xi,\eta,\rho,\si)&:=6A'(\eta+\rho,\sigma,\xi)|\xi|^{1/2}(|\eta+\rho|+|\sigma|-|\xi|)n_2^{rr}(\eta,\rho)\\
&+6A'(\xi,\eta,\rho+\sigma)|\xi|^{1/2}|\eta|^{1/2}|\rho+\sigma|^{1/2}n_2^{ri}(\rho,\sigma),
\end{split}
\end{equation}
\begin{equation}\label{mu114}
\begin{split}
M_1^{14}(\xi,\eta,\rho,\si)&:=4B'(\eta+\rho,\sigma,\xi)|\eta+\rho|^{1/2}|\xi|^{1/2}|\sigma|^{1/2}n_2^{rr}(\eta,\rho)\\
&+2B'(\rho+\sigma,\eta,\xi)|\xi|^{1/2}(|\rho+\sigma|+|\eta|-|\xi|)n_2^{ri}(\rho,\sigma)\\
&+B'(\xi,\eta,\rho+\sigma)|\rho+\sigma|^{1/2}(|\xi|+|\eta|-|\rho+\sigma|)n_2^{ri}(\rho,\sigma)\\
&-2B'(\xi,\eta+\rho,\sigma)|\eta+\rho|^{1/2}(|\xi|+|\sigma|-|\eta+\rho|)n_2^{rr}(\eta,\rho)\\
&-2B'(\xi,\sigma,\eta+\rho)|\sigma|^{1/2}(|\xi|+|\eta+\rho|-|\sigma|)n_2^{rr}(\eta,\rho),
\end{split}
\end{equation}
and
\begin{equation}\label{mu21}
\begin{split}
M^1_2(\xi,\eta,\rho,\sigma) &:=\frac{2m_{N_0}^2(\xi)|\xi|^{1/2}|\rho|^{1/2}}{|\eta|^{1/2}}
  \big\{\eta(\sigma+\rho)\chi(\eta,\sigma+\rho) \chi(\rho,\sigma)+\eta\sigma\chi(\sigma+\rho,\eta)\\
&-\eta\sigma\chi(\rho,\sigma+\eta)-\rho\eta\chi(\eta+\rho,\sigma) 
  + (|\rho| - |\sigma+\rho|)|\eta|\wt{\chi}(\sigma+\rho,\eta)\big\},
\end{split}
\end{equation}
\begin{align}\label{mu22}
M_2^2(\xi,\eta,\rho,\sigma) & := \frac{2m_{N_0}^2(\xi) \chi(\eta,\si)\widetilde{\chi}(\rho,\eta+\si)|\xi|^{1/2}|\eta|^{1/2}[|\eta+\sigma||\rho|+(\eta+\sigma)\rho]}{|\rho|^{1/2}},
\end{align}
\begin{equation}\label{mu23}
\begin{split}
M_2^3(\xi,\eta,\rho,\sigma)&:=|\rho|^{1/2}\big[(\rho+\si) \chi(\rho,\si) -  \si\big]\\
&\times\big\{|\xi|^{1/2}m_{N_0}^2(\xi)\chi(\rho+\si,\eta) \eta|\eta|^{-1/2}+|\eta|^{1/2}m_{N_0}^2(\eta)\chi(\rho+\si,\xi) \xi|\xi|^{-1/2}\big\},
\end{split}
\end{equation}
\begin{align}\label{mu24}
M_2^4(\xi,\eta,\rho,\si) & := 2m_{N_0}(\xi)[m_{N_0}(\si)-m_{N_0}(\xi)] \chi(\rho+\eta,\si) |\xi|^{1/2} a_2^{ii}(\rho,\eta),
\end{align}
\begin{equation}\label{mu25}
\begin{split}
M_2^5(\xi,\eta,\rho,\si)&:=6A'(\xi+\rho,\eta,\sigma)|\sigma|^{1/2}(|\xi+\rho|+|\eta|-|\sigma|)n_2^{ii}(\xi,\rho)\\
&+3A'(\xi,\eta,\rho+\sigma)|\rho+\sigma|^{1/2}(|\xi|+|\eta|-|\rho+\sigma|)n_2^{ri}(\sigma,\rho),
\end{split}
\end{equation}
\begin{equation}\label{mu26}
\begin{split}
M_2^6(\xi,\eta,\rho,\si)&:=4B'(\xi+\rho,\eta,\sigma)|\xi+\rho|^{1/2}|\eta|^{1/2}|\sigma|^{1/2}n_2^{ii}(\xi,\rho)\\
&+2B'(\xi,\eta,\rho+\sigma)|\xi|^{1/2}|\eta|^{1/2}|\rho+\sigma|^{1/2}n_2^{ri}(\sigma,\rho)\\
&-2B'(\rho+\sigma,\eta,\xi)|\eta|^{1/2}(|\rho+\sigma|+|\xi|-|\eta|)n_2^{ri}(\sigma,\rho)\\
&-2B'(\sigma,\xi+\rho,\eta)|\xi+\rho|^{1/2}(|\sigma|+|\eta|-|\xi+\rho|)n_2^{ii}(\xi,\rho)\\
&-2B'(\sigma,\eta,\xi+\rho)|\eta|^{1/2}(|\sigma|+|\xi+\rho|-|\eta|)n_2^{ii}(\xi,\rho).
\end{split}
\end{equation}
The new symbols $a_2^{rr}$, $a_2^{ii}$, $b$, $n_2^{ri}$, $n_2^{rr}$, and $n_2^{ii}$ are defined in \eqref{ite2}, \eqref{ite4}, and \eqref{Nh}--\eqref{Nii2}, and the symbols $A'$ and $B'$ are defined by $A':=A/D$ and $B':=B/D$.
\end{lemma}
\begin{proof}

The idea is to examine the formulas \eqref{eni5.1}--\eqref{eni5.12}, 
extract the purely quartic terms expressed in terms of the main variables $h^\ast$ and $\omega^\ast$, 
and show that the errors are indeed quintic remainders that do not lose derivatives. 
In this proof $X\sim Y$ means that $X-Y$ is an acceptable quintic error,
$|X-Y|(t)\lesssim A^2[\varep(t)]^3$ for any $t\in[T_1,T_2]$.

\smallskip
{\bf{Step 1: the terms $\mathcal{B}_{\geq 4}^1$, $\mathcal{B}_{\geq 4}^2$, $\mathcal{B}_{\geq 4}^3$.}} Using \eqref{DNmain3} and \eqref{hoeq31} we have
\begin{equation*}
\mathcal{B}_{\geq 4}^1 \sim \frac{1}{(2\pi R)^3}\sum_{(\xi,\eta,\rho,\sigma)\in\mathbb{H}^3_R}
   \widehat{h^\ast}(\xi)\widehat{h^\ast}(\eta)\widehat{h^\ast}(\rho)\widehat{\omega^\ast}(\sigma)
   M_1^1(\xi,\eta,\rho,\sigma),
\end{equation*}
and
\begin{equation*}
\mathcal{B}_{\geq 4}^2 \sim \frac{1}{(2\pi R)^3} \sum_{(\xi,\eta,\rho,\sigma)\in\mathbb{H}^3_R}
  \widehat{\omega^\ast}(\xi)\widehat{\omega^\ast}(\eta)\widehat{\omega^\ast}(\rho)\widehat{h^\ast}(\sigma)
  M_2^1(\xi,\eta,\rho,\sigma),
\end{equation*}
where the symbols $M_1^1$ and $M_2^1$ are defined in \eqref{mu11} and \eqref{mu21}. These identities follow by replacing $h$ with $h^\ast$ and $|\partial_x|^{1/2}\phi$ with $\omega^\ast$ at the expense of quadratic errors, and noticing that there is no derivative loss since the operators $G_3$ and $L_3$ are smoothing (see \eqref{smoothing2} and \eqref{smoothing4}).

For the term $\mathcal{B}_{\geq 4}^3$ in \eqref{eni5.3} we use
the expansions \eqref{expand1} for $V$ and \eqref{Utor10} for $\alpha$, the definitions \eqref{Usor2}, 
and the formula \eqref{DNmain2} for $G_2$, 
to see that
\begin{equation*}
\begin{split}
-2&\int_{\T_R}J^{N_0}\partial_x(T_Vh)\cdot T_\al J^{N_0}h\,dx\\
&=\frac{-2}{(2\pi R)^3}\sum_{y,\xi,\eta\in\Z/R}\big\{m_{N_0}(y)iy\chi(y-\xi,\xi)\widehat{V}(y-\xi)\widehat{h}(\xi)\cdot \widehat{\alpha}(-y-\eta)\chi(-y-\eta,\eta)m_{N_0}(\eta)\widehat{h}(\eta)\big\}\\
&\sim\frac{1}{(2\pi R)^3}\sum_{(\xi,\eta,\rho,\sigma)\in \mathbb{H}^3_R}
   \widehat{h^\ast}(\xi)\widehat{h^\ast}(\eta)\widehat{h^\ast}(\rho)\widehat{\omega^\ast}(\sigma)M_1^2(\xi,\eta,\rho,\sigma),
   \end{split}
\end{equation*}
where the symbol $M_1^2$ is defined in \eqref{mu12} by symmetrization in $\xi$ and $\eta$ (to avoid derivative loss such that $M_1^2$ becomes a regular symbol\footnote{For presentation purposes, we postpone the routine proofs of regularity/admissibility of symbols to section 6. The symbols $M_i^j$ are converted into the slightly simpler symbols $K_i^j$ in Lemma \ref{quartic3}, and the symbols $K_i^j$ are then all analyzed, sometimes fully symmetrized, in section 6.}). Also
\begin{equation*}
\begin{split}
2&\int_{\T_R}J^{N_0}G_2(\phi,h)\cdot T_\al J^{N_0}h\,dx\\
&=\frac{2}{(2\pi R)^3}\sum_{x,\xi,\eta\in\Z/R}\big\{m_{N_0}(y)q_2(y-\xi,\xi)\widehat{\phi}(y-\xi)\widehat{h}(\xi)\cdot \widehat{\alpha}(-y-\eta)\chi(-y-\eta,\eta)m_{N_0}(\eta)\widehat{h}(\eta)\big\}\\
&\sim\frac{1}{(2\pi R)^3}\sum_{(\xi,\eta,\rho,\sigma)\in \mathbb{H}^3_R}
   \widehat{h^\ast}(\xi)\widehat{h^\ast}(\eta)\widehat{h^\ast}(\rho)\widehat{\omega^\ast}(\sigma)M_1^3(\xi,\eta,\rho,\sigma),
   \end{split}
\end{equation*}
using \eqref{DNmain2} and \eqref{smoothing1}. Therefore
\begin{align*}
\mathcal{B}_{\geq 4}^3 \sim 
   \frac{1}{(2\pi R)^3}\sum_{(\xi,\eta,\rho,\sigma)\in\mathbb{H}^3_R}
   \widehat{h^\ast}(\xi)\widehat{h^\ast}(\eta)\widehat{h^\ast}(\rho)\widehat{\omega^\ast}(\sigma)
   \big[ M_1^2(\xi,\eta,\rho,\sigma) + M_1^3(\xi,\eta,\rho,\sigma)\big].
\end{align*}

{\bf{Step 2: the terms $\mathcal{B}_{\geq 4}^a$, $a\in\{4, 5, 6, 7\}$.}} To identify properly the quartic terms we expand
\begin{align}
\label{hast-h}
\what{h}(\xi)-\what{h^\ast}(\xi)=\frac{1}{2\pi R}\sum_{\eta\in\Z/R} 
  \frac{1}{2} |\xi-\eta|\chi(\xi-\eta,\eta) \what{h^\ast}(\xi-\eta) \what{h^\ast}(\eta)+\widehat{\mathcal{R}^1_{\geq 3}}(\xi),
\end{align}
\begin{align}
\label{phi-omega}
|\xi|^{1/2}\what{\phi}(\xi) - \what{\omega^\ast}(\xi)= \frac{1}{2\pi R}\sum_{\eta\in\Z/R}
 |\xi|^{1/2}|\xi-\eta|^{1/2}\chi(\xi-\eta,\eta)\what{\omega^\ast}(\xi-\eta) \what{h^\ast}(\eta)+\widehat{\mathcal{R}^2_{\geq 3}}(\xi),
\end{align}
\begin{align}\label{V-omega}
\what{V}(\xi) - i\xi|\xi|^{-1/2}\what{\omega^\ast}(\xi)= \frac{1}{2\pi R}\sum_{\eta\in\Z/R}
  i |\xi-\eta|^{1/2} \big( \xi\chi(\xi-\eta,\eta) - \eta\big) \what{\omega^\ast}(\xi-\eta) \what{h^\ast}(\eta)+\widehat{\mathcal{R}^3_{\geq 3}}(\xi),
\end{align}
where the cubic remainders $\mathcal{R}^j_{\geq 3}$ satisfy
\begin{equation}\label{remainder5}
\begin{split}
\|\mathcal{R}^1_{\geq 3}(t)\|_{H^{N_0}}+\|\mathcal{R}^2_{\geq 3}(t)\|_{H^{N_0-1/2}}+\|\mathcal{R}^3_{\geq 3}(t)\|_{H^{N_0-1}}&\lesssim A\varep(t)^2,\\
\|\mathcal{R}^1_{\geq 3}(t)\|_{\widetilde{W}^{N_1}}+\|\mathcal{R}^2_{\geq 3}(t)\|_{\widetilde{W}^{N_1-1/2}}+\|\mathcal{R}^3_{\geq 3}(t)\|_{\widetilde{W}^{N_1-1}}&\lesssim \varep(t)^3,\\
\end{split}
\end{equation}
for any $t\in [T_1,T_2]$. These identities follow from the definitions \eqref{Usor2}, \eqref{DefBV}, and the expansions \eqref{expand1}--\eqref{expand2} and \eqref{Utor10}--\eqref{Utor11}. 

Using the formulas we can write
\begin{align*}
\mathcal{B}_{\geq 4}^4 \sim 
   \frac{1}{(2\pi R)^3}\sum_{(\xi,\eta,\rho,\sigma)\in\mathbb{H}^3_R}
   \widehat{h^\ast}(\xi)\widehat{h^\ast}(\eta)\widehat{h^\ast}(\rho)\widehat{\omega^\ast}(\sigma)
   \big[ M_1^4(\xi,\eta,\rho,\sigma) + M_1^5(\xi,\eta,\rho,\sigma)\big],
\end{align*}
where $M_1^4$ and $M_1^5$ are as in \eqref{mu14}--\eqref{mu15}. Similarly, the definition \eqref{eni5.6} gives
\begin{align*}
\mathcal{B}_{\geq 4}^5 \sim \frac{1}{(2\pi R)^3} \sum_{(\xi,\eta,\rho,\sigma)\in\mathbb{H}^3_R}
  \widehat{\omega^\ast}(\xi)\widehat{\omega^\ast}(\eta)\widehat{\omega^\ast}(\rho)\widehat{h^\ast}(\sigma)M_2^2(\xi,\eta,\rho,\sigma),
\end{align*}
where $M_2^2$ is defined in \eqref{mu22}. Similarly, using the definition \eqref{eni5.7} we have
\begin{align*}
\mathcal{B}_{\geq 4}^6 \sim \frac{1}{(2\pi R)^3} \sum_{(\xi,\eta,\rho,\sigma)\in\mathbb{H}^3_R}
  \widehat{h^\ast}(\xi)\widehat{h^\ast}(\eta)\widehat{h^\ast}(\rho)\widehat{\omega^\ast}(\sigma)
  \big[ M_1^6(\xi,\eta,\rho,\sigma) + M_1^7(\xi,\eta,\rho,\sigma) + M_1^8(\xi,\eta,\rho,\sigma) \big],
\end{align*}
where the symbols $M_1^6,M_1^7,M_1^8$ are defined in \eqref{mu16}--\eqref{mu18}. Finally, with $M_2^3$ as in \eqref{mu23},
\begin{align*}
\mathcal{B}_{\geq 4}^7 \sim \frac{1}{(2\pi R)^3} \sum_{(\xi,\eta,\rho,\sigma)\in\mathbb{H}^3_R}
  \widehat{\omega^\ast}(\xi)\widehat{\omega^\ast}(\eta)\widehat{\omega^\ast}(\rho)\widehat{h^\ast}(\sigma)
  M_2^3(\xi,\eta,\rho,\sigma).
\end{align*}

These four formulas can be justified in a similar way. We only provide the details for the more complicated identity for $\mathcal{B}_{\geq 4}^6$. We symmetrize first in $\xi$ and $\eta$ to avoid derivative loss, so 
\begin{equation*}
\begin{split}
\mathcal{B}_{\geq 4}^6 = \frac{1}{(2\pi R)^2}\sum_{(\xi,\eta,\rho)\in\mathbb{H}^2_R}
  \big\{\widehat{h}(\xi)\widehat{h}(\eta)\widehat{V}(\rho)
 & -\widehat{h^\ast}(\xi)\widehat{h^\ast}(\eta)\widehat{\omega^\ast}(\rho)
  i\rho|\rho|^{-1/2}\big\}\\
  &\times[i\xi m_{N_0}^2(\xi)\chi(\rho,\eta)+i\eta m_{N_0}^2(\eta)\chi(\rho,\xi)].
  \end{split}
\end{equation*}
Then we use the identities \eqref{V-omega} and \eqref{hast-h} and notice that the symbol $M_1^8$ defined in \eqref{mu18} is regular (as in Definition \ref{defMNab}) to see that
\begin{equation*}
\begin{split}
\mathcal{B}_{\geq 4}^6&\sim\frac{1}{(2\pi R)^2}\sum_{(\xi,\eta,\rho)\in\mathbb{H}^2_R}
  \big\{\widehat{h}(\xi)\widehat{h}(\eta)-\widehat{h^\ast}(\xi)\widehat{h^\ast}(\eta)\big\}\widehat{\omega^\ast}(\rho)i\rho|\rho|^{-1/2}\\
  &\qquad\qquad\qquad\qquad\times[i\xi m_{N_0}^2(\xi)\chi(\rho,\eta)+i\eta m_{N_0}^2(\eta)\chi(\rho,\xi)]\\
  &+\frac{1}{(2\pi R)^3}\sum_{(\xi,\eta,\rho,\sigma)\in\mathbb{H}^3_R}
   \widehat{h^\ast}(\xi)\widehat{h^\ast}(\eta)\widehat{h^\ast}(\rho)\widehat{\omega^\ast}(\sigma)M_1^8(\xi,\eta,\rho,\sigma).
  \end{split}
\end{equation*}
Then we notice that the symbols $M_1^6$ and $M_1^7$ are also regular, so the desired formula follows by expanding $h$ as in \eqref{hast-h} and reorganizing the terms to preserve the symmetry in $\xi$ and $\eta$.

\smallskip
{\bf{Step 3: the terms $\mathcal{B}^{8}_{\geq 4}$, $\mathcal{B}^{9}_{\geq 4}$.}} We examine the formulas \eqref{Utor10}--\eqref{Utor11}, and use \eqref{hast-h}--\eqref{phi-omega} to deduce that
\begin{align}\label{ite1}
\begin{split}
\widehat{\alpha}(\xi) + |\xi|\widehat{h^\ast}(\xi) &=\frac{1}{2\pi R}\sum_{\eta\in\Z/R} a_2^{rr}(\xi-\eta,\eta) \what{h^\ast}(\xi-\eta) \what{h^\ast}(\eta)\\
&+\frac{1}{2\pi R}\sum_{\eta\in\Z/R}a_2^{ii}(\xi-\eta,\eta) \what{\omega^\ast}(\xi-\eta) \what{\omega^\ast}(\eta)+\widehat{\mathcal{R}^4_{\geq 3}}(\xi),
\end{split}
\end{align}
\begin{align}\label{ite2}
\begin{split}
& a_2^{rr}(\rho,\eta) := \frac{1}{4}|\rho+\eta|\big[2|\rho|+2|\eta|-|\rho|\chi(\rho,\eta)-|\eta|\chi(\eta,\rho)\big] - \frac{|\rho|^2+|\eta|^2}{2} ,
\\  
& a_2^{ii}(\rho,\eta) :=  \frac{1}{2|\rho|^{1/2}|\eta|^{1/2}} (|\rho+\eta| - |\rho| -|\eta|) (\rho\eta -|\rho||\eta|),
\end{split}
\end{align}  
and
\begin{align}\label{ite3}
\widehat{\partial_t\alpha}(\xi) + |\xi|^{3/2}\widehat{\omega^\ast}(\xi)=\frac{1}{2\pi R}\sum_{\eta\in\Z/R}
  b(\xi-\eta,\eta) \what{\omega^\ast}(\xi-\eta) \what{h^\ast}(\eta)+\widehat{\mathcal{R}^5_{\geq 3}}(\xi),
\end{align}
\begin{align}\label{ite4}
b(\rho,\eta) &:=\frac{-|\rho+\eta|^2 |\rho|\chi(\rho,\eta)+(2|\rho+\eta|-|\rho|-|\eta|)(|\rho||\eta|-\rho\eta)
 + 2|\rho|\rho\eta}{|\rho|^{1/2}}.
\end{align}
The cubic remainders $\mathcal{R}^4_{\geq 3}$ and $\mathcal{R}^5_{\geq 3}$ satisfy the uniform bounds
\begin{equation}\label{ite5}
\|\mathcal{R}^4_{\geq 3}(t)\|_{\widetilde{W}^{N_1-2}}+\|\mathcal{R}^5_{\geq 3}(t)\|_{\widetilde{W}^{N_1-3}}\lesssim \varep(t)^3.
\end{equation}

We notice that the symbols $M_1^{10}$, $M_1^{11}$ are regular, while the symbols $M_1^9$,  $M_1^{12}$, $M_2^4$ are weakly admissible (due to the singular factors $|\rho+\sigma|$, $|\xi+\rho|$, and $|\eta+\rho|$ respectively). Using the formulas above and the definition \eqref{eni5.9} we have
\begin{align*}
\begin{split}
\mathcal{B}_{\geq 4}^8 \sim \frac{1}{(2\pi R)^3} \sum_{(\xi,\eta,\rho,\sigma)\in\mathbb{H}^3_R}
  &\widehat{h^\ast}(\xi)\widehat{h^\ast}(\eta)\widehat{h^\ast}(\rho)\widehat{\omega^\ast}(\si)\big[ M_1^{9}(\xi,\eta,\rho,\si) + M_1^{10}(\xi,\eta,\rho,\si)],
\end{split}
\end{align*}
where $M_1^9$ and $M_1^{10}$ are defined in \eqref{mu19}--\eqref{mu110}. Moreover, recalling \eqref{eni5.10},
\begin{align}\label{B9}
\begin{split}
\mathcal{B}_{\geq 4}^9  \sim &\frac{1}{(2\pi R)^3} \sum_{(\xi,\eta,\rho,\si) \in \mathbb{H}^3_R}
  \widehat{h^\ast}(\xi)\widehat{h^\ast}(\eta)\widehat{h^\ast}(\rho)\widehat{\omega^\ast}(\si) 
  \big[ M_1^{11}(\xi,\eta,\rho,\si) + M_1^{12}(\xi,\eta,\rho,\si) \big]
\\
& + \frac{1}{(2\pi R)^3} \sum_{(\xi,\eta,\rho,\si) \in \mathbb{H}^3_R}
  \widehat{\omega^\ast}(\xi) \widehat{\omega^\ast}(\eta) \widehat{\omega^\ast}(\rho) \widehat{h^\ast}(\si) M_2^4(\xi,\eta,\rho,\si),
\end{split}
\end{align}
where the multipliers $M_1^{11}, M_1^{12}, M_2^4$ are defined in \eqref{mu111}, \eqref{mu112}, and \eqref{mu24}.

\smallskip
{\bf{Step 4: the terms $\mathcal{B}^{10}_{\geq 4}$, $\mathcal{B}^{11}_{\geq 4}$.}} To analyze the more complicated expressions  in \eqref{eni5.11}--\eqref{eni5.12} 
we examine the nonlinearities $\mathcal{N}_{\geq 2}^{h^\ast}$ 
and $\mathcal{N}_{\geq 2}^{\omega^\ast}$ defined in Lemma \ref{Usor1}. We define the new symbols
\begin{align}\label{Nh}
\begin{split}
n_2^{ri}(\rho,\eta) &:= n^{ri}(\rho,\eta)-\chi(\rho,\eta)|\rho||\eta|^{1/2}/2+(\rho+\eta)\eta|\eta|^{-1/2}\chi(\eta,\rho)
\\
&=\widetilde{\chi}(\eta,\rho)\frac{(\rho+\eta)\eta-|\rho+\eta||\eta|}{|\eta|^{1/2}} 
- \frac{|\rho||\eta|}{2 |\eta|^{1/2}} \chi(\rho,\eta) + \frac{2\rho\eta+\eta^2}{2|\eta|^{1/2}}\chi(\eta,\rho),
\end{split}
\end{align}
\begin{align}\label{Nrr2}
\begin{split}
n_2^{rr}(\rho,\eta) & := \frac{1}{2}\Big\{n^{rr}(\rho,\eta)+n^{rr}(\eta,\rho)+\frac{1}{2}\chi(\eta,\rho)|\eta||\rho|^{1/2}+\frac{1}{2}\chi(\rho,\eta)|\rho||\eta|^{1/2}\Big\}\\
&=\frac{1}{4} |\rho+\eta|^{1/2} \big[ \chi(\rho,\eta) |\rho| + \chi(\eta,\rho) |\eta| \big],
\end{split}
\end{align}
\begin{align}\label{Nii2}
\begin{split}
n_2^{ii}(\rho,\eta) & := \frac{1}{2}\Big\{n^{ii}(\rho,\eta)+n^{ii}(\eta,\rho)+(\rho+\eta)\eta|\eta|^{-1/2}\chi(\eta,\rho)+(\eta+\rho)\rho|\rho|^{-1/2}\chi(\rho,\eta)\Big\}\\
& =  |\rho+\eta|^{1/2} \frac{|\eta||\rho|+\eta\rho}{2|\eta|^{1/2}|\rho|^{1/2}}
	- \frac{1}{2} \big[\chi(\rho,\eta)+\chi(\eta,\rho)\big] |\rho+\eta|^{1/2} |\eta|^{1/2} |\rho|^{1/2}.
\end{split}
\end{align}
Using the identities in Lemma \ref{Usor1} we notice that we have the decompositions
\begin{equation}\label{Poj1}
\begin{split}
&\mathcal{N}^{h^\ast}_{\geq 2}=\mathcal{N}^{h^\ast}_{2}+\mathcal{N}^{h^\ast}_{\geq 3},\\
&\widehat{\mathcal{N}^{h^\ast}_{2}}(\xi):=\frac{1}{2\pi R}\sum_{\eta\in\Z/R}\widehat{h^\ast}(\xi-\eta)\widehat{\omega^\ast}(\eta)n^{ri}_2(\xi-\eta,\eta),\\
&\mathcal{N}^{h^\ast}_{\geq 3}:=T_{\beta+|\partial_x|h^\ast/2}|\partial_x|^{1/2}\omega^\ast-\partial_xT_{V-\partial_x|\partial_x|^{-1/2}\omega^\ast}h^\ast+H^\ast_{\geq 3},
\end{split}
\end{equation}
and
\begin{equation}\label{Poj2}
\begin{split}
&\mathcal{N}^{\omega^\ast}_{\geq 2}=\mathcal{N}^{\omega^\ast}_{2}+\mathcal{N}^{\omega^\ast}_{\geq 3},\\
&\widehat{\mathcal{N}^{\omega^\ast}_{2}}(\xi):=\frac{1}{2\pi R}\sum_{\eta\in\Z/R}\big\{\widehat{h^\ast}(\xi-\eta)\widehat{h^\ast}(\eta)n^{rr}_2(\xi-\eta,\eta)+\widehat{\omega^\ast}(\xi-\eta)\widehat{\omega^\ast}(\eta)n^{ii}_2(\xi-\eta,\eta)\big\},\\
&\mathcal{N}^{\omega^\ast}_{\geq 3}:=-T_{\beta+|\partial_x|h^\ast/2}|\partial_x|^{1/2}h^\ast-\partial_xT_{V-\partial_x|\partial_x|^{-1/2}\omega^\ast}\omega^\ast+\Omega^\ast_{\geq 3}.
\end{split}
\end{equation}

We see easily that
\begin{equation}\label{Poj3}
\begin{split}
\mathcal{B}_{\geq 4}^{10}  &=\frac{1}{(2\pi R)^3} \sum_{(\xi,\eta,\rho,\si) \in \mathbb{H}^3_R}
  \widehat{h^\ast}(\xi)\widehat{h^\ast}(\eta)\widehat{h^\ast}(\rho)\widehat{\omega^\ast}(\si)M_1^{13}(\xi,\eta,\rho,\si)\\
&+ \frac{1}{(2\pi R)^3} \sum_{(\xi,\eta,\rho,\si) \in \mathbb{H}^3_R}
  \widehat{\omega^\ast}(\xi) \widehat{\omega^\ast}(\eta) \widehat{\omega^\ast}(\rho) \widehat{h^\ast}(\si) M_2^5(\xi,\eta,\rho,\si)+\mathcal{B}_{\geq 5}^{10},
\end{split}
\end{equation}
\begin{equation}\label{Poj4}
\begin{split}
\mathcal{B}_{\geq 4}^{11}&=\frac{1}{(2\pi R)^3} \sum_{(\xi,\eta,\rho,\si) \in \mathbb{H}^3_R}
  \widehat{h^\ast}(\xi)\widehat{h^\ast}(\eta)\widehat{h^\ast}(\rho)\widehat{\omega^\ast}(\si) M_1^{14}(\xi,\eta,\rho,\si)\\
&+ \frac{1}{(2\pi R)^3} \sum_{(\xi,\eta,\rho,\si) \in \mathbb{H}^3_R}
  \widehat{\omega^\ast}(\xi) \widehat{\omega^\ast}(\eta) \widehat{\omega^\ast}(\rho) \widehat{h^\ast}(\si) M_2^6(\xi,\eta,\rho,\si)+\mathcal{B}_{\geq 5}^{11},
\end{split}
\end{equation}
where the symbols $M_1^{13}$, $M_1^{14}$, $M_2^5$, $M_2^6$ are defined in \eqref{mu113}, \eqref{mu114},\eqref{mu25}, and \eqref{mu26}, and the quintic remainders $\mathcal{B}_{\geq 5}^{10}$ and $\mathcal{B}_{\geq 5}^{11}$ are defined by
\begin{equation}\label{Suses1}
\begin{split}
\mathcal{B}^{10}_{\geq 5}
  &:=\frac{1}{(2\pi R)^2}\sum_{(\xi,\eta,\rho)\in\mathbb{H}^2_R}A'(\xi,\eta,\rho)\big\{6|\rho|^{1/2}(|\xi|+|\eta|-|\rho|)\widehat{\mathcal{N}^{\omega^\ast}_{\geq 3}}(\xi)\widehat{\omega^\ast}(\eta)\widehat{h^\ast}(\rho)
  \\
&+3|\rho|^{1/2}(|\xi|+|\eta|-|\rho|)\widehat{\omega^\ast}(\xi)\widehat{\omega^\ast}(\eta)\widehat{\mathcal{N}^{h^\ast}_{\geq 3}}(\rho)+ 6|\xi|^{1/2}|\eta|^{1/2}|\rho|^{1/2}\widehat{h^\ast}(\xi)\widehat{h^\ast}(\eta)\widehat{\mathcal{N}^{h^\ast}_{\geq 3}}(\rho)\big\},
\end{split}
\end{equation}
\begin{equation}\label{Suses2}
\begin{split}
\mathcal{B}^{11}_{\geq 5} 
  &:= \frac{1}{(2\pi R)^2}\sum_{(\xi,\eta,\rho)\in\mathbb{H}^2_R}B'(\xi,\eta,\rho)\big\{2|\xi|^{1/2}|\eta|^{1/2}|\rho|^{1/2}
  \big[2\widehat{\mathcal{N}^{\omega^\ast}_{\geq 3}}(\xi)\widehat{\omega^\ast}(\eta)\widehat{h^\ast}(\rho) 
  + \widehat{\omega^\ast}(\xi)\widehat{\omega^\ast}(\eta)\widehat{\mathcal{N}^{h^\ast}_{\geq 3}}(\rho)\big]\\
& + |\rho|^{1/2}(|\xi|+|\eta|-|\rho|)
  \big[2\widehat{\mathcal{N}^{h^\ast}_{\geq 3}}(\xi)\widehat{h^\ast}(\eta)\widehat{h^\ast}(\rho)
  + \widehat{h^\ast}(\xi)\widehat{h^\ast}(\eta)\widehat{\mathcal{N}^{h^\ast}_{\geq 3}}(\rho)\big]\\
&-2|\eta|^{1/2}(|\xi|-|\eta|+|\rho|)
  \big[\widehat{\mathcal{N}^{h^\ast}_{\geq 3}}(\xi)\widehat{\omega^\ast}(\eta)\widehat{\omega^\ast}(\rho)
  + \widehat{h^\ast}(\xi)\widehat{\mathcal{N}^{\omega^\ast}_{\geq 3}}(\eta)\widehat{\omega^\ast}(\rho)
  + \widehat{h^\ast}(\xi)\widehat{\omega^\ast}(\eta)\widehat{\mathcal{N}^{\omega^\ast}_{\geq 3}}(\rho)\big]\big\}.
\end{split}
\end{equation}

We have to prove that these remainders are acceptable quintic error,  i.e. $\mathcal{B}^{10}_{\geq 5}\sim 0$ and $\mathcal{B}^{11}_{\geq 5}\sim 0$. The bounds \eqref{ABmultHM} show that
\begin{equation}\label{Suses3}
\begin{split}
&\big\|A'(\xi,\eta,\rho)\cdot\varphi_{k_1}(\xi)\varphi_{k_2}(\eta)\varphi_{k_3}(\rho)\big\|_{\mathcal{HM}_9}
\lesssim 2^{\min(k_1,k_2,k_3)/2}2^{(2N_0-1)\max(k_1,k_2,k_3)},\\
&\big\|B'(\xi,\eta,\rho)\cdot\varphi_{k_1}(\xi)\varphi_{k_2}(\eta)\varphi_{k_3}(\rho)\big\|_{\mathcal{HM}_9}\lesssim 2^{\min(k_1,k_2,k_3)/2}2^{(2N_0-1)\max(k_1,k_2,k_3)}2^{-|k_1-k_2|/2},
\end{split}
\end{equation}
for any $k_1,k_2,k_3\in\Z$. The functions $H^\ast_{\geq 3}$ and $\Omega^\ast_{\geq 3}$ satisfy suitable cubic bounds without derivative loss (see \eqref{Usor6}), so the corresponding contributions lead to acceptable quintic errors. 

To bound the remaining contributions we need symmetrization to avoid derivative loss. Let 
\begin{equation*}
X:=\beta+|\partial_x|h^\ast/2,\qquad\qquad Y:=V-\partial_x|\partial_x|^{-1/2}\omega^\ast,
\end{equation*}
Since $\beta=\sqrt{1+\alpha}-1$, it follows from Lemma \ref{BVexpands} and \eqref{Utor10} that
\begin{equation}\label{Suses5}
\|X(t)\|_{\widetilde{W}^{N_1-3}}+\|Y(t)\|_{\widetilde{W}^{N_1-2}}\lesssim \varep(t)^2,\qquad t\in[T_1,T_2].
\end{equation}
We need to bound the contributions of the functions $T_{X}|\partial_x|^{1/2}\omega^\ast-\partial_xT_{Y}h^\ast$ and $-T_{X}|\partial_x|^{1/2}h^\ast-\partial_xT_{Y}\omega^\ast$ (compare with the terms $\mathcal{N}^{h^\ast}_{\geq 3}$ and $\mathcal{N}^{\omega^\ast}_{\geq 3}$ in \eqref{Poj1}--\eqref{Poj2}). Then we decompose
\begin{equation*}
\begin{split}
\mathcal{B}^{10}_{\geq 5}&\sim\mathcal{B}^{10}_{\geq 5,1}+\mathcal{B}^{10}_{\geq 5,2}+\mathcal{B}^{10}_{\geq 5,3}+\mathcal{B}^{10}_{\geq 5,4},
\end{split}
\end{equation*}
\begin{equation*}
\begin{split}
\mathcal{B}^{10}_{\geq 5,1}&:=\frac{1}{(2\pi R)^2}\sum_{(\xi,\eta,\rho)\in\mathbb{H}^2_R}A'(\xi,\eta,\rho)(-6)|\xi|^{1/2}|\eta|^{1/2}|\rho|^{1/2}\widehat{h^\ast}(\xi)\widehat{h^\ast}(\eta)\mathcal{F}[\partial_xT_{Y}h^\ast](\rho),
\end{split}
\end{equation*}
\begin{equation*}
\begin{split}
\mathcal{B}^{10}_{\geq 5,2}&:=\frac{1}{(2\pi R)^2}\sum_{(\xi,\eta,\rho)\in\mathbb{H}^2_R}A'(\xi,\eta,\rho)\big\{-6|\rho|^{1/2}(|\xi|+|\eta|-|\rho|)\widehat{\omega^\ast}(\eta)\widehat{h^\ast}(\rho)\mathcal{F}[T_{X}|\partial_x|^{1/2}h^\ast](\xi)\\
&\qquad\qquad\qquad + 6|\xi|^{1/2}|\eta|^{1/2}|\rho|^{1/2}\widehat{h^\ast}(\xi)\widehat{h^\ast}(\eta)\mathcal{F}[T_{X}|\partial_x|^{1/2}\omega^\ast](\rho)\big\},
\end{split}
\end{equation*}
\begin{equation*}
\begin{split}
\mathcal{B}^{10}_{\geq 5,3}&:=\frac{1}{(2\pi R)^2}\sum_{(\xi,\eta,\rho)\in\mathbb{H}^2_R}A'(\xi,\eta,\rho)\big\{-6|\rho|^{1/2}(|\xi|+|\eta|-|\rho|)\widehat{\omega^\ast}(\eta)\widehat{h^\ast}(\rho)\mathcal{F}[\partial_xT_{Y}\omega^\ast](\xi)\\
&\qquad\qquad\qquad -3|\rho|^{1/2}(|\xi|+|\eta|-|\rho|)\widehat{\omega^\ast}(\xi)\widehat{\omega^\ast}(\eta)\mathcal{F}[\partial_xT_Yh^\ast](\rho)\big\},
\end{split}
\end{equation*}
\begin{equation*}
\begin{split}
\mathcal{B}^{10}_{\geq 5,4}&:=\frac{1}{(2\pi R)^2}\sum_{(\xi,\eta,\rho)\in\mathbb{H}^2_R}A'(\xi,\eta,\rho)3|\rho|^{1/2}(|\xi|+|\eta|-|\rho|)\widehat{\omega^\ast}(\xi)\widehat{\omega^\ast}(\eta)\mathcal{F}[T_{X}|\partial_x|^{1/2}\omega^\ast](\rho).
\end{split}
\end{equation*}
By symmetrization we can rewrite these expressions in the form
\begin{equation*}
\begin{split}
\mathcal{B}^{10}_{\geq 5,1}=&\frac{1}{(2\pi R)^3}\sum_{(\xi,\eta,\rho,\sigma)\in\mathbb{H}^3_R}\widehat{h^\ast}(\xi)\widehat{h^\ast}(\eta)\widehat{h^\ast}(\rho)\widehat{Y}(\sigma)L_1^{10}(\xi,\eta,\rho,\sigma),\\
L_1^{10}(\xi,\eta,\rho,\sigma):=&-2A'(\xi,\eta,\rho+\sigma)|\xi|^{1/2}|\eta|^{1/2}|\rho+\sigma|^{1/2}i(\rho+\sigma)\chi(\sigma,\rho)\\
&-2A'(\eta,\rho,\xi+\sigma)|\eta|^{1/2}|\rho|^{1/2}|\xi+\sigma|^{1/2}i(\xi+\sigma)\chi(\sigma,\xi)\\
&-2A'(\rho,\xi,\eta+\sigma)|\rho|^{1/2}|\xi|^{1/2}|\eta+\sigma|^{1/2}i(\eta+\sigma)\chi(\sigma,\eta),
\end{split}
\end{equation*}
\begin{equation*}
\begin{split}
\mathcal{B}^{10}_{\geq 5,2}=&\frac{1}{(2\pi R)^3}\sum_{(\xi,\eta,\rho,\sigma)\in\mathbb{H}^3_R}\widehat{h^\ast}(\xi)\widehat{h^\ast}(\eta)\widehat{\omega^\ast}(\rho)\widehat{X}(\sigma)L_2^{10}(\xi,\eta,\rho,\sigma),\\
L_2^{10}(\xi,\eta,\rho,\sigma):=&-3A'(\xi+\sigma,\rho,\eta)|\xi|^{1/2}|\eta|^{1/2}(|\xi+\sigma|-|\eta|+|\rho|)\chi(\sigma,\xi)\\
&-3A'(\eta+\sigma,\rho,\xi)|\eta|^{1/2}|\xi|^{1/2}(|\eta+\sigma|-|\xi|+|\rho|)\chi(\sigma,\eta)\\
&+6A'(\xi,\eta,\rho+\sigma)|\xi|^{1/2}|\eta|^{1/2}|\rho+\sigma|^{1/2}|\rho|^{1/2}\chi(\sigma,\rho),
\end{split}
\end{equation*}
\begin{equation*}
\begin{split}
\mathcal{B}^{10}_{\geq 5,3}=&\frac{1}{(2\pi R)^3}\sum_{(\xi,\eta,\rho,\sigma)\in\mathbb{H}^3_R}\widehat{\omega^\ast}(\xi)\widehat{\omega^\ast}(\eta)\widehat{h^\ast}(\rho)\widehat{Y}(\sigma)L_3^{10}(\xi,\eta,\rho,\sigma),\\
L_3^{10}(\xi,\eta,\rho,\sigma):=&-3A'(\xi+\sigma,\eta,\rho)|\rho|^{1/2}(|\xi+\sigma|+|\eta|-|\rho|)i(\xi+\sigma)\chi(\sigma,\xi)\\
&-3A'(\eta+\sigma,\xi,\rho)|\rho|^{1/2}(|\eta+\sigma|+|\xi|-|\rho|)i(\eta+\sigma)\chi(\sigma,\eta)\\
&-3A'(\xi,\eta,\rho+\sigma)|\rho+\sigma|^{1/2}(|\xi|+|\eta|-|\rho+\sigma|)i(\rho+\sigma)\chi(\sigma,\rho),
\end{split}
\end{equation*}
\begin{equation*}
\begin{split}
\mathcal{B}^{10}_{\geq 5,4}=&\frac{1}{(2\pi R)^3}\sum_{(\xi,\eta,\rho,\sigma)\in\mathbb{H}^3_R}\widehat{\omega^\ast}(\xi)\widehat{\omega^\ast}(\eta)\widehat{\omega^\ast}(\rho)\widehat{X}(\sigma)L_4^{10}(\xi,\eta,\rho,\sigma),\\
L_4^{10}(\xi,\eta,\rho,\sigma):=&+A'(\xi,\eta,\rho+\sigma)|\rho+\sigma|^{1/2}(|\xi|+|\eta|-|\rho+\sigma|)|\rho|^{1/2}\chi(\sigma,\rho)\\
&+A'(\eta,\rho,\xi+\sigma)|\xi+\sigma|^{1/2}(|\eta|+|\rho|-|\xi+\sigma|)|\xi|^{1/2}\chi(\sigma,\xi)\\
&+A'(\rho,\xi,\eta+\sigma)|\eta+\sigma|^{1/2}(|\rho|+|\xi|-|\eta+\sigma|)|\eta|^{1/2}\chi(\sigma,\eta).
\end{split}
\end{equation*}
In view of \eqref{Suses3}, the multipliers $L_a^{10}$, $a\in\{1,2,3,4\}$, do not lose derivatives after symmetrization, and are all regular. We provide the complete details for the symbol $L_1^{10}$. We look at $$L_1^{10}(\xi,\eta,\rho,\sigma)\varphi_{k_1}(\xi)\varphi_{k_2}(\eta)\varphi_{k_3}(\rho)\varphi_{k_4}(\sigma),$$ for $k_1,k_2,k_3,k_4\in\Z$, and let $\widetilde{k}_1\geq \widetilde{k}_2\geq \widetilde{k}_3\geq\widetilde{k}_4$ denote the decreasing rearrangement of $k_1,k_2,k_3,k_4$. Notice that the localized symbol vanishes if $k_4\geq \widetilde{k}_2-10$, due to the presence of the function $\chi$. So we may assume that $k_4\leq \widetilde{k}_2-10$ and $k_4\leq\widetilde{k}_3$. Since $L_1^{10}(\xi,\eta,\rho,\sigma)=L_1^{10}(\eta,\rho,\xi,\sigma)=L_1^{10}(\rho,\xi,\eta,\sigma)$ we may also assume that $\{k_1,k_2\}=\{\widetilde{k}_1,\widetilde{k}_2\}$, so $\xi,\eta$ are the large variables and $\rho,\sigma$ are the small variables.

We have two cases: in the "smoothing case" $k_4\geq \widetilde{k}_2-40$ there is no derivative loss, and each of the 3 terms in the definition of $L_1^{10}$ satisfies suitable H\"{o}rmander-Michlin bounds which follow from the bounds \eqref{Suses3}. On the other hand, if $k_4\leq \widetilde{k}_2-40$ then we rewrite
\begin{equation*}
\begin{split}
L_1^{10}(\xi,\eta,\rho,\sigma)=&-2i|\eta|^{1/2}|\rho|^{1/2}\big[A'(\eta,\rho,\xi+\sigma)|\xi+\sigma|^{1/2}(\xi+\sigma)-A'(\eta,\rho,\xi)|\xi|^{1/2}\xi\big]\\
&-2i|\xi|^{1/2}|\rho|^{1/2}\big[A'(\rho,\xi,\eta+\sigma)|\eta+\sigma|^{1/2}(\eta+\sigma)-A'(\rho,\xi,\eta)|\eta|^{1/2}\eta\big]\\
&-2i|\xi|^{1/2}|\eta|^{1/2}(\rho+\sigma)\big[A'(\xi,\eta,\rho+\sigma)|\rho+\sigma|^{1/2}\chi(\sigma,\rho)-A'(\xi,\eta,\rho)|\rho|^{1/2}\big],
\end{split}
\end{equation*}
since $\chi(\sigma,\xi)=\chi(\sigma,\eta)=1$, $A'(\xi,\eta,\rho)=A'(\eta,\rho,\xi)=A'(\rho,\xi,\eta)$, and $\xi+\eta+\rho+\sigma=0$. The expression in each line defines a regular symbol, due to the bounds \eqref{Suses3}.

The analysis of the symbols $L_2^{10}, L_3^{10}, L_4^{10}$ is similar. In all cases we may assume that $k_4\leq \widetilde{k}_2-10$. In the "smoothing case" $k_4\geq \widetilde{k}_2-40$ there is no derivative loss, and the expression in each line defines a regular symbol, while in the "unbalanced case" $k_4\leq \widetilde{k}_2-40$ we have to reorganize some of the symbols suitably to avoid derivative loss. 

Using Lemma \ref{algeProp7} it then follows that $|\mathcal{B}^{10}_{\geq 5,a}|\lesssim A^2\varep(t)^3$, $a\in\{1,2,3,4\}$, 
as desired.

The analysis of the quintic term $\mathcal{B}^{11}_{\geq 5}$ is similar. We can decompose as before
\begin{equation*}
\begin{split}
\mathcal{B}^{11}_{\geq 5}&\sim\mathcal{B}^{11}_{\geq 5,1}+\mathcal{B}^{11}_{\geq 5,2}+\mathcal{B}^{11}_{\geq 5,3}+\mathcal{B}^{10}_{\geq 5,4},\\
\mathcal{B}^{11}_{\geq 5,1}=&\frac{1}{(2\pi R)^3}\sum_{(\xi,\eta,\rho,\sigma)\in\mathbb{H}^3_R}\widehat{h^\ast}(\xi)\widehat{h^\ast}(\eta)\widehat{h^\ast}(\rho)\widehat{Y}(\sigma)L_1^{11}(\xi,\eta,\rho,\sigma),\\
\mathcal{B}^{11}_{\geq 5,2}=&\frac{1}{(2\pi R)^3}\sum_{(\xi,\eta,\rho,\sigma)\in\mathbb{H}^3_R}\widehat{h^\ast}(\xi)\widehat{h^\ast}(\eta)\widehat{\omega^\ast}(\rho)\widehat{X}(\sigma)L_2^{11}(\xi,\eta,\rho,\sigma),\\
\mathcal{B}^{11}_{\geq 5,3}=&\frac{1}{(2\pi R)^3}\sum_{(\xi,\eta,\rho,\sigma)\in\mathbb{H}^3_R}\widehat{\omega^\ast}(\xi)\widehat{\omega^\ast}(\eta)\widehat{h^\ast}(\rho)\widehat{Y}(\sigma)L_3^{11}(\xi,\eta,\rho,\sigma),\\
\mathcal{B}^{11}_{\geq 5,4}=&\frac{1}{(2\pi R)^3}\sum_{(\xi,\eta,\rho,\sigma)\in\mathbb{H}^3_R}\widehat{\omega^\ast}(\xi)\widehat{\omega^\ast}(\eta)\widehat{\omega^\ast}(\rho)\widehat{X}(\sigma)L_4^{11}(\xi,\eta,\rho,\sigma).
\end{split}
\end{equation*}
The multipliers $L_a^{11}$, $a\in\{1,2,3,4\}$, can be calculated explicitly using the definition \eqref{Suses2}, and do not lose derivatives after symmetrization. Using again Lemma \ref{algeProp7} it follows that $|\mathcal{B}^{11}_{\geq 5,a}|\lesssim A^2\varep(t)^3$, $a\in\{1,2,3,4\}$. This completes the proof of the lemma.
\end{proof}

\subsection{The quartic terms}\label{SecQuintic}

In this subsection we combine some of the multipliers in Lemma \ref{quartic2} and use (some) of the symmetries 
of the operators $\mathcal{Q}_1$ and $\mathcal{Q}_2$ to obtain more convenient formulas.

\begin{lemma}\label{quartic3}
With the hypothesis and notation in Lemma \ref{quartic2}, for any $t\in[T_1,T_2]$ we have
\begin{equation}\label{simp1}
\begin{split}
\mathcal{Q}_1(h^\ast,h^\ast,h^\ast,\omega^\ast)&=\frac{1}{(2\pi R)^3}\sum_{(\xi,\eta,\rho,\sigma)\in\mathbb{H}^3_R}
  \widehat{h^\ast}(\xi)\widehat{h^\ast}(\eta)\widehat{h^\ast}(\rho)\widehat{\omega^\ast}(\sigma)
  K_1(\xi,\eta,\rho,\sigma),
  \\
\mathcal{Q}_2(\omega^\ast,\omega^\ast,\omega^\ast,h^\ast)&=\frac{1}{(2\pi R)^3}\sum_{(\xi,\eta,\rho,\sigma)\in\mathbb{H}^3_R}
  \widehat{\omega^\ast}(\xi)\widehat{\omega^\ast}(\eta)\widehat{\omega^\ast}(\rho)\widehat{h^\ast}(\sigma)
  K_2(\xi,\eta,\rho,\sigma).
\end{split}
\end{equation}
The symbols $K_1$ and $K_2$ are given by
\begin{equation}\label{simp2}
K_1(\xi,\eta,\rho,\sigma) = \sum_{a\in\{1, \ldots, 9\}} K_1^a(\xi,\eta,\rho,\sigma),\qquad K_2(\xi,\eta,\rho,\sigma) 
  = \sum_{a\in\{1, \ldots,4\}}K_2^a(\xi,\eta,\rho,\sigma),
\end{equation}
where
\begin{equation}\label{ku11}
\begin{split}
&K^1_1(\xi,\eta,\rho,\sigma):=\frac{m_{N_0}^2(\xi)}{|\sigma|^{1/2}}\Big\{|\sigma|\big[(|\xi||\rho+\sigma|+\xi(\rho+\sigma))+(|\xi||\eta+\sigma|+\xi(\eta+\sigma))-(|\xi||\sigma|+\xi\sigma)\big]\\
&-|\sigma|\chi(\rho+\sigma,\eta)\big[|\xi||\rho+\sigma|-|\xi||\sigma|\big]-|\sigma|\chi(\eta+\sigma,\rho)\big[|\xi||\eta+\sigma|-|\xi||\sigma|\big]\\
&-|\sigma|\chi(\sigma,\rho)\widetilde{\chi}(\rho+\sigma,\eta)(|\xi||\rho+\sigma|+\xi(\rho+\sigma))-|\sigma|\chi(\sigma,\eta)\widetilde{\chi}(\eta+\sigma,\rho)(|\xi||\eta+\sigma|+\xi(\eta+\sigma))\\
&-|\sigma|\chi(\sigma,\rho)\chi(\rho+\sigma,\eta)\xi(\sigma+\rho)-|\sigma|\chi(\sigma,\eta)\chi(\eta+\sigma,\rho)\xi(\sigma+\eta)\\
&-\frac{1}{2}[|\rho|\chi(\rho,\eta)+|\eta|\chi(\eta,\rho)]\big[\widetilde{\chi}(\sigma,\eta+\rho)(|\xi||\sigma|+\xi\sigma)+\chi(\sigma,\eta+\rho)\xi\sigma\big]\Big\},
\end{split}
\end{equation}
\begin{equation}\label{ku12}
\begin{split}
K_1^2(\xi,\eta,\rho,&\sigma) :=- \frac{m_{N_0}(\rho+\si) m_{N_0}(\eta)
  |\xi|\chi(\xi,\eta)\big\{\si(\rho+\si)(1-\chi(\rho,\si)) -|\si||\rho+\si|\widetilde{\chi}(\rho,\si)\big\}}{|\sigma|^{1/2}}\\
& - \frac{m_{N_0}(\eta+\si) m_{N_0}(\rho)
  |\xi|\chi(\xi,\rho)\big\{\si(\eta+\si)(1-\chi(\eta,\si)) -|\si||\eta+\si|\widetilde{\chi}(\eta,\si)\big\}}{|\sigma|^{1/2}},
\end{split}
\end{equation}
\begin{equation}\label{ku14}
\begin{split}
K_1^3(\xi,\eta,\rho,\sigma):= &\frac{m_{N_0}^2(\eta+\rho)[|\eta|\chi(\eta,\rho)+|\rho|\chi(\rho,\eta)]}{2|\si|^{1/2}}\\
&\times\big\{((\xi+\si)\si-|\xi+\si||\si|)\widetilde{\chi}(\si,\xi)+(\xi+\si)\si\chi(\si,\xi)\big\},
\end{split}
\end{equation}
\begin{align}\label{ku19}
K_1^4(\xi,\eta,\rho,\si) :=\frac{1}{2}m_{N_0}(\eta) m_{N_0}(\rho)b(\si,\xi)[\chi(\xi+\si,\eta)+\chi(\xi+\si,\rho)],
\end{align}
\begin{align}\label{ku110}
K_1^5(\xi,\eta,\rho,\si) & :=-\frac{m_{N_0}(\eta+\rho) m_{N_0}(\xi) |\si|^{3/2}[\chi(\si,\xi)+\chi(\si,\eta+\rho)][|\rho|\chi(\rho,\eta)+|\eta|\chi(\eta,\rho)]}{4},
\end{align}
\begin{align}\label{ku111}
K_1^6(\xi,\eta,\rho,\si):= -\frac{1}{2}m_{N_0}(\si) [m_{N_0}(\eta+\rho) - m_{N_0}(\si)] \chi(\xi,\eta+\rho)|\si|^{1/2}|\xi|\big[|\rho|\chi(\rho,\eta)+|\eta|\chi(\eta,\rho)],
\end{align}
\begin{align}\label{ku112}
K_1^7(\xi,\eta,\rho,\si):= 2m_{N_0}(\si) [m_{N_0}(\xi) - m_{N_0}(\si)] \chi(\eta+\rho,\xi) |\si| ^{1/2} a_2^{rr}(\eta,\rho),
\end{align}
\begin{equation}\label{ku113}
\begin{split}
K_1^8(\xi,\eta,\rho,\si)&:=6n_2^{rr}(\eta,\rho)A'(\xi,\sigma,\eta+\rho)|\xi|^{1/2}(|\eta+\rho|+|\sigma|-|\xi|)\\
&+6n_2^{ri}(\xi,\sigma)A'(\eta,\rho,\xi+\sigma)|\rho|^{1/2}|\eta|^{1/2}|\xi+\sigma|^{1/2},
\end{split}
\end{equation}
\begin{equation}\label{ku114}
K_1^9(\xi,\eta,\rho,\si):=-2n_2^{rr}(\eta,\rho)B''(\eta+\rho,\sigma,\xi)+n_2^{ri}(\xi,\sigma)B'''(\eta,\rho,\xi+\sigma),
\end{equation}
and
\begin{equation}\label{ku21}
\begin{split}
K^1_2(\xi,\eta,\rho,\sigma) &:= \frac{m_{N_0}^2(\xi)|\xi|^{1/2}|\rho|^{1/2}}{|\eta|^{1/2}}\Big\{\chi(\rho,\sigma)\big[\eta(\sigma+\rho)+|\eta||\sigma+\rho|\widetilde{\chi}(\sigma+\rho,\eta)\big]\\
&\qquad-\eta\sigma\chi(\rho,\sigma+\eta)-\rho\eta\chi(\eta+\rho,\sigma) 
  + (|\rho| - |\sigma+\rho|)|\eta|\wt{\chi}(\sigma+\rho,\eta)\Big\}\\
&+\frac{m_{N_0}^2(\xi)|\xi|^{1/2}|\eta|^{1/2}}{|\rho|^{1/2}}\Big\{\chi(\eta,\sigma)\big[\rho(\sigma+\eta)+|\rho||\sigma+\eta|\widetilde{\chi}(\sigma+\eta,\rho)\big]\\
&\qquad-\rho\sigma\chi(\eta,\sigma+\rho)-\eta\rho\chi(\rho+\eta,\sigma) 
  + (|\eta| - |\sigma+\eta|)|\rho|\wt{\chi}(\sigma+\eta,\rho)\Big\},
\end{split}
\end{equation}
\begin{align}\label{ku24}
K_2^2(\xi,\eta,\rho,\si) & := 2m_{N_0}(\xi)[m_{N_0}(\si)-m_{N_0}(\xi)] \chi(\eta+\rho,\si) |\xi|^{1/2} a_2^{ii}(\eta,\rho),
\end{align}
\begin{equation}\label{ku25}
\begin{split}
K_2^3(\xi,\eta,\rho,\si)&:=6n_2^{ii}(\eta,\rho)A'(\xi,\sigma,\eta+\rho)|\sigma|^{1/2}(|\eta+\rho|+|\xi|-|\sigma|)\\
&+3n_2^{ri}(\sigma,\xi)A'(\eta,\rho,\xi+\sigma)|\xi+\sigma|^{1/2}(|\rho|+|\eta|-|\xi+\sigma|),
\end{split}
\end{equation}
\begin{equation}\label{ku26}
K_2^4(\xi,\eta,\rho,\si):=-2n_2^{ii}(\eta,\rho)B''(\eta+\rho,\xi,\sigma)-n_2^{ri}(\sigma,\xi)B''(\eta,\rho,\xi+\sigma).
\end{equation}
All the symbols $K_i^j$ are real-valued, even, and symmetric in the second and third variables, i.e.
\begin{equation}\label{simp4}
K_i^j(\xi,\eta,\rho,\sigma)=K_i^j(-\xi,-\eta,-\rho,-\sigma), \qquad K_i^j(\xi,\eta,\rho,\sigma)=K_i^j(\xi,\rho,\eta,\sigma)
\end{equation}
for any $(\xi,\eta,\rho,\sigma)\in\mathbb{H}_R^3$. The symbols $B'',B''':\H^2_\infty\to\mathbb{R}$ are given by
\begin{equation}\label{slov20}
\begin{split}
B''(x,y,&z):=-2B'(x,y,z)|x|^{1/2}|y|^{1/2}|z|^{1/2}\\
&+B'(z,x,y)|x|^{1/2}(|z|+|y|-|x|)+B'(z,y,x)|y|^{1/2}(|z|+|x|-|y|),
\end{split}
\end{equation}
\begin{equation}\label{slov21}
\begin{split}
B'''(x,y,&z):=B'(x,y,z)|z|^{1/2}(|x|+|y|-|z|)\\
&+B'(y,z,x)|x|^{1/2}(|y|+|z|-|x|)+B'(z,x,y)|y|^{1/2}(|z|+|x|-|y|).
\end{split}
\end{equation}
\end{lemma}

\begin{proof} 
The symbol $K_1^1$ is obtained by combining the symbols $M_1^1$, $M_1^5$, $M_1^7$, $M_1^8$. 
The symbol $K_1^2$ is obtained by combining the symbols $M_1^2$ and $M_1^3$.
The symbol $K_1^3$ is obtained by combining the symbols $M_1^4$ and $M_1^6$, 
while $K_2^1$ is obtained by combining the symbols $M_2^1$, $M_2^2$, and $M_2^3$. 
In all cases we symmetrize the symbols $K_i^j$ in the variables $\eta$ and $\rho$.
\end{proof}

\medskip
\section{Estimates on quartic space-time integrals}\label{QuartGen}
In view of Lemmas \ref{quartic1} and \ref{quartic2} our main goal is to show that, 
for $T_1\leq T_2\in\mathbb{R}$ and $(h,\phi)\in C([T_1,T_2]:H^{N_0}\times \dot{H}^{N_0,1/2})$ 
a solution of the system \eqref{gWW} satisfying \eqref{BootstrapFull}, we have
\begin{equation}\label{mainquarticeni}
\Big| \int_{t_1}^{t_2} 
  \Big[ \mathcal{Q}_1(h^\ast,h^\ast,h^\ast,\omega^\ast)(s) + \mathcal{Q}_2(\omega^\ast,\omega^\ast,\omega^\ast,h^\ast)(s)
  \Big] \,ds \Big|
  \lesssim A^2\int_{t_1}^{t_2}[\varep(s)]^3\,ds,
\end{equation}
for any $t_1\leq t_2\in[T_1,T_2]$. To prove this we need to integrate by parts in time (the method of normal forms), to convert quartic expressions into quintic expressions. There are two main issues we need to be aware of: the presence of quartic resonances and the potential loss of derivative. 

To address these issues, we first ``complexify'' the quartic terms, that is, we convert them into expressions involving the complex variables $(U^+,U^-)$ defined by
\begin{align}\label{hoe1}
\begin{split}
& U = U^+:=h^\ast+i\omega^\ast,\qquad \overline{U}=U^-:=h^\ast-i\omega^\ast,
\\
& h^\ast = \frac{1}{2}(U^+ + U^-), \qquad \omega^\ast = \frac{1}{2i}(U^+ - U^-).
\end{split}
\end{align}
We also need to symmetrize our symbols. Note that the symbols $K_i^a$ in \eqref{simp2} are symmetric with respect to the second and third variables, as stated in \eqref{simp4}.
However, we may also replace them by their full symmetrization in the first three variables $\widetilde{K}_1^a$, where, by definition, 
\begin{align}\label{Ksymm}
\begin{split}
& \widetilde{K}(\xi,\eta,\rho,\si) 
  := \frac{1}{3} \Big[ K(\xi,\eta,\rho,\sigma) + K(\eta,\rho,\xi,\sigma) + K(\rho,\xi,\eta,\sigma) \Big],
\end{split}
\end{align}
for any symbol $K$ symmetric in the second and third variables. Similarly, we may replace the symbols $K_2^a$ with $\widetilde{K}_2^a$. In view of \eqref{simp1} we have
\begin{align}\label{simp1C}
\begin{split}
(2\pi R)^3\mathcal{Q}_1(h^\ast,h^\ast,h^\ast,\omega^\ast) 
  & =  \sum_{(\xi_1,\xi_2,\xi_3,\xi_4)\in\mathbb{H}^3_R}
  \widehat{U}(\xi_1)\widehat{U}(\xi_2)\widehat{U}(\xi_3)\widehat{U}(\xi_4)
  \frac{i}{16} K_1^{++++} (\xi_1,\xi_2,\xi_3,\xi_4)
  \\
  & + \sum_{(\xi_1,\xi_2,\xi_3,\xi_4)\in\mathbb{H}^3_R} 
  \widehat{U}(\xi_1)\widehat{U}(\xi_2)\widehat{U}(\xi_3)\widehat{\bar{U}}(\xi_4)
  \frac{i}{16} K_1^{+++-} (\xi_1,\xi_2,\xi_3,\xi_4)
\\
& + \sum_{(\xi_1,\xi_2,\xi_3,\xi_4)\in\mathbb{H}^3_R} 
  \widehat{U}(\xi_1)\widehat{U}(\xi_2)\widehat{\bar{U}}(\xi_3)\widehat{\bar{U}}(\xi_4)
  \frac{i}{32} SK_1 (\xi_1,\xi_2,\xi_3,\xi_4)
\\
& + \sum_{(\xi_1,\xi_2,\xi_3,\xi_4)\in\mathbb{H}^3_R} 
  \widehat{U}(\xi_1)\widehat{\bar{U}}(\xi_2)\widehat{\bar{U}}(\xi_3)\widehat{\bar{U}}(\xi_4)
  \frac{i}{16} K_1^{+---} (\xi_1,\xi_2,\xi_3,\xi_4) 
\\
& + \sum_{(\xi_1,\xi_2,\xi_3,\xi_4)\in\mathbb{H}^3_R} 
  \widehat{\bar{U}}(\xi_1)\widehat{\bar{U}}(\xi_2)\widehat{\bar{U}}(\xi_3)\widehat{\bar{U}}(\xi_4)
  \frac{i}{16} K_1^{----} (\xi_1,\xi_2,\xi_3,\xi_4),
\end{split}
\end{align}
where, with $\widetilde{K}_1$ defined as in \eqref{Ksymm},
\begin{align}\label{oths1}
\begin{split}
& K_1^{++++}(\xi_1,\xi_2,\xi_3,\xi_4) := -\widetilde{K}_1(\xi_1,\xi_2,\xi_3,\xi_4) ,
\\
& K_1^{+++-}(\xi_1,\xi_2,\xi_3,\xi_4) :=
  \widetilde{K}_1(\xi_1,\xi_2,\xi_3,\xi_4)
  - \widetilde{K}_1(\xi_4,\xi_1,\xi_2,\xi_3) -\widetilde{K}_1(\xi_4,\xi_2,\xi_3,\xi_1) - \widetilde{K}_1(\xi_4,\xi_3,\xi_1,\xi_2),
\\
& K_1^{+---}(\xi_1,\xi_2,\xi_3,\xi_4) :=\widetilde{K}_1(\xi_1,\xi_2,\xi_3,\xi_4)+\widetilde{K}_1(\xi_1,\xi_3,\xi_4,\xi_2) +\widetilde{K}_1(\xi_1,\xi_4,\xi_2,\xi_3)-  \widetilde{K}_1(\xi_2,\xi_3,\xi_4,\xi_1),
\\
& K_1^{----}(\xi_1,\xi_2,\xi_3,\xi_4) := \widetilde{K}_1(\xi_1,\xi_2,\xi_3,\xi_4),
\end{split}
\end{align}
and the symbol $SK_1$ is defined
by a suitable symmetrization according to the general formula
\begin{equation}\label{verf1}
SL(\xi_1,\xi_2,\xi_3,\xi_4):=3\widetilde{L}(\xi_1,\xi_2,\xi_3,\xi_4)+3\widetilde{L}(\xi_1,\xi_2,\xi_4,\xi_3)-3\widetilde{L}(\xi_3,\xi_4,\xi_1,\xi_2)-3\widetilde{L}(\xi_3,\xi_4,\xi_2,\xi_1)
\end{equation}
for any $(\xi_1,\xi_2,\xi_3,\xi_4)\in\mathbb{H}_R^3$ and any symbol 
$L:\mathbb{H}_R^3\to\mathbb{C}$ symmetric in the second and the third variables (with $\widetilde{L}$ defined as in \eqref{Ksymm}). Notice that $SL$ satisfies the natural symmetries
\begin{equation}\label{verf2}
\begin{split}
&SL(\xi_1,\xi_2,\xi_3,\xi_4)=SL(\xi_2,\xi_1,\xi_3,\xi_4)=SL(\xi_1,\xi_2,\xi_4,\xi_3)=SL(\xi_2,\xi_1,\xi_4,\xi_3),
\\
&SL(\xi_1,\xi_2,\xi_3,\xi_4)=-SL(\xi_3,\xi_4,\xi_1,\xi_2).
\end{split}
\end{equation}

We can proceed similarly for the $\mathcal{Q}_2$ type terms in \eqref{simp1} and write
\begin{align}\label{simp2C}
\begin{split}
(2\pi R)^3\mathcal{Q}_2(\omega^\ast,\omega^\ast,\omega^\ast,h^\ast) 
  & = \sum_{(\xi_1,\xi_2,\xi_3,\xi_4)\in\mathbb{H}^3_R}
  \widehat{U}(\xi_1)\widehat{U}(\xi_2)\widehat{U}(\xi_3)\widehat{U}(\xi_4)
  \frac{-i}{16} K_2^{++++} (\xi_1,\xi_2,\xi_3,\xi_4)
  \\
  & + \sum_{(\xi_1,\xi_2,\xi_3,\xi_4)\in\mathbb{H}^3_R} 
  \widehat{U}(\xi_1)\widehat{U}(\xi_2)\widehat{U}(\xi_3)\widehat{\bar{U}}(\xi_4)
  \frac{i}{16} K_2^{+++-} (\xi_1,\xi_2,\xi_3,\xi_4)
\\
& + \sum_{(\xi_1,\xi_2,\xi_3,\xi_4)\in\mathbb{H}^3_R} 
  \widehat{U}(\xi_1)\widehat{U}(\xi_2)\widehat{\bar{U}}(\xi_3)\widehat{\bar{U}}(\xi_4)
  \frac{-i}{32} SK_2(\xi_1,\xi_2,\xi_3,\xi_4)
\\
& + \sum_{(\xi_1,\xi_2,\xi_3,\xi_4)\in\mathbb{H}^3_R} 
  \widehat{U}(\xi_1)\widehat{\bar{U}}(\xi_2)\widehat{\bar{U}}(\xi_3)\widehat{\bar{U}}(\xi_4)
  \frac{i}{16} K_2^{+---} (\xi_1,\xi_2,\xi_3,\xi_4)
\\
& + \sum_{(\xi_1,\xi_2,\xi_3,\xi_4)\in\mathbb{H}^3_R} 
  \widehat{\bar{U}}(\xi_1)\widehat{\bar{U}}(\xi_2)\widehat{\bar{U}}(\xi_3)\widehat{\bar{U}}(\xi_4)
  \frac{-i}{16} K_2^{----} (\xi_1,\xi_2,\xi_3,\xi_4)
\end{split}
\end{align}
where $SK_2$ is defined via \eqref{verf1}, and, as in \eqref{oths1},
\begin{align}\label{oths2}
\begin{split}
& K_2^{++++}(\xi_1,\xi_2,\xi_3,\xi_4) := - \widetilde{K}_2(\xi_1,\xi_2,\xi_3,\xi_4) ,
\\
& K_2^{+++-}(\xi_1,\xi_2,\xi_3,\xi_4) :=
  \widetilde{K}_2(\xi_1,\xi_2,\xi_3,\xi_4)
  -\widetilde{K}_2(\xi_4,\xi_1,\xi_2,\xi_3) -\widetilde{K}_2(\xi_4,\xi_2,\xi_3,\xi_1) - \widetilde{K}_2(\xi_4,\xi_3,\xi_1,\xi_2),
\\
& K_2^{+---}(\xi_1,\xi_2,\xi_3,\xi_4) :=  \widetilde{K}_2(\xi_1,\xi_2,\xi_3,\xi_4)
  +\widetilde{K}_2(\xi_1,\xi_3,\xi_4,\xi_2)+\widetilde{K}_2(\xi_1,\xi_4,\xi_2,\xi_3) - \widetilde{K}_2(\xi_2,\xi_3,\xi_4,\xi_1),
\\
& K_2^{----}(\xi_1,\xi_2,\xi_3,\xi_4) := -K_2^{++++}(\xi_1,\xi_2,\xi_3,\xi_4).
\end{split}
\end{align}

To prove the bounds \eqref{mainquarticeni} we distinguish between the ``resonant'' contributions, that is, 
those that have two copies of $U$ and two copies of $\bar{U}$, and the other ``non-resonant'' contributions.
The phases associated to the non-resonant contributions are lower bounded, and we estimate them using Lemma \ref{hoe2} below.

The resonant contributions are harder to treat.
Since the associated phases can vanish, see \eqref{cubicres} and \eqref{BFres}, 
we will need to verify
specific non-resonance (and regularity) properties 
of the symbols in order to obtain a bound as in \eqref{mainquarticeni}.
Sufficient non-resonant conditions are those stated in the assumptions of Lemma \ref{hol2},
see \eqref{verf1} and \eqref{uyt22}, which we will then verify in Sections \ref{secadmver} and \ref{uyt22sec}. 

\smallskip
\subsection{Admissible symbols}\label{secA}
To verify these conditions we introduce a suitable class of ``admissible'' symbols. Recall the notation
$\varphi_{\underline{k}}(\xi):=\varphi_{k_1}(\xi_1)\varphi_{k_2}(\xi_2)\varphi_{k_3}(\xi_3)\varphi_{k_4}(\xi_4)$,
and that $\widetilde{k}_1\geq\widetilde{k}_2\geq\widetilde{k}_3\geq\widetilde{k}_4$ 
denotes the decreasing rearrangement of the set $\{k_1,k_2,k_3,k_4\}\in\Z^4$.
Also recall the Definitions \ref{AccSymb} and \ref{defMNab}.

\smallskip
\begin{definition}[Admissible Symbols]
\label{defA}
We define the class $\mathcal{A}$ of 
(quartic) ``admissible symbols'' as the class of symbols $m:\mathbb{H}^3_\infty \to \mathbb{C}$ 
that satisfy the following:

\smallskip
(i) The symbol $m$ is even 
\begin{align}\label{A0}
m(\xi_1,\xi_2,\xi_3,\xi_4) = m(-\xi_1,-\xi_2,-\xi_3,-\xi_4)
\end{align}
for any $(\xi_1,\xi_2,\xi_3,\xi_4)\in \mathbb{H}^3_\infty$ and weakly admissible (when restricted to $\mathbb{H}^3_R$), 
\begin{align}\label{A1}
\begin{split}
{\big\| m \cdot 
  \varphi_{\underline{k}} \big\|}_{\widetilde{S}^\infty}
  \lesssim 2^{\widetilde{k}_4/2} 2^{2\widetilde{k}_3} 
  2^{2N_0\max(\widetilde{k}_1,0)},\qquad\text {for any }\underline{k}\in\Z^4.
\end{split}
\end{align}

\smallskip
 (ii) For any $\underline{k} \in \Z^4$  and $\xi=(\xi_1,\xi_2,\xi_3,\xi_4)\in\mathbb{H}^3_\infty$, one can write
\begin{align}\label{A2.0}
\varphi_{\underline{k}}(\xi)m(\xi_1,\xi_2,\xi_3,\xi_4)=\varphi_{\underline{k}}(\xi) \sum_{\beta\in\{0,1,3/2\},\,i\neq j\in\{1,2,3,4\}}|\xi_i+\xi_j|^{\beta} F_{\beta;i,j}(\xi_1,\xi_2,\xi_3,\xi_4)
\end{align}
where the functions $F_{\beta;i,j}= F_{\beta;i,j;\underline{k}}: \R^4 \to \C$ are even and satisfy the bounds
\begin{align}\label{A2.1}
&\big|2^{\alpha\cdot \underline{k}}\partial^\al_{\xi}
  \big[F_{\beta;i,j}(\xi)\varphi_{\underline{k}}(\xi)\big]\big|
  \lesssim 2^{\widetilde{k}_4/2} 2^{(2-\beta)\widetilde{k}_3} 2^{2N_0\max(\widetilde{k}_1,0)},
\end{align}
for $|\alpha|\leq 6$ and $\xi\in\mathbb{H}^3_\infty$. Moreover, the functions $F_{\beta;i,j}$ can be non-trivial only if $\beta=0$ or $\max(k_i,k_j)\leq\widetilde{k}_3+6$.
\end{definition}

Notice that property (i) follows from the stronger assumptions in property (ii), but we state it here mostly for clarity. Notice also the admisibility condition is invariant under permutations of the coordinates and that regular symbols in the class $\mathcal{M}^{5/2}$ are admissible (with $\beta=0$).

For us the main point is that the symbols $K_1$ and $K_2$ defined in Lemma \ref{quartic3} are admissible after symmetrization in the first three variables, as we verify in section \ref{secadmver}. We can then apply Lemmas \ref{hoe2} and \ref{hol2} below to prove acceptable quintic energy estimates, provided that we can check the non-resonant conditions \eqref{uyt22} corresponding to Benjamin-Feir resonances.

\subsection{Non-resonant interactions}\label{ssecnonres}
\smallskip
Given a multiplier $M:\mathbb{H}_R^3\to\mathbb{C}$ and $f_1,f_2,f_3,f_4:\T_R\to\mathbb{C}$, 
we define the quatrilinear operator
\begin{equation}\label{hoe1.5}
\mathcal{A}_M(f_1,f_2,f_3,f_4):=\frac{1}{(2\pi R)^3}\sum_{(\xi_1,\xi_2,\xi_3,\xi_4)
  \in\mathbb{H}^3_R}M(\xi_1,\xi_2,\xi_3,\xi_4)
  \widehat{f_1}(\xi_1)\widehat{f_2}(\xi_2)\widehat{f_3}(\xi_3)\widehat{f_4}(\xi_4).
\end{equation}

Our first main estimate for quartic bulks in this section is the following:

\begin{lemma}\label{hoe2} 
If $M:\mathbb{H}_R^3\to\mathbb{C}$ is an admissible quartic symbol (see Definition \ref{defA}), 
then
\begin{equation}\label{hoe4}
\Big|\int_{t_1}^{t_2}\mathcal{A}_M(U,U,U,U^{\iota})(s)\,ds\Big|
  \lesssim [A(t_1)\varep(t_1)]^2+[A(t_2)\varep(t_2)]^2+\int_{t_1}^{t_2}A^2[\varep(s)]^3\,ds,
\end{equation}
for any $\iota\in\{+,-\}$ and any $t_1\leq t_2\in[T_1,T_2]$. 
\end{lemma}

\begin{proof} 

{\bf{Step 1.}} The function $U=h^\ast+i\omega^\ast$ satisfies the equation
\begin{equation}\label{hoe7}
\begin{split}
\partial_tU+i|\partial_x|^{1/2}U&=\mathcal{N}_{\geq 2}^{U}=\mathcal{N}_{\geq 2,1}^{U}
  +\mathcal{N}_{\geq 2,2}^{U}+\mathcal{N}_{\geq 2,3}^{U},
  \\
\mathcal{N}_{\geq 2,1}^{U}&:=-iT_{\beta}|\partial_x|^{1/2}U,
\\
\mathcal{N}_{\geq 2,2}^{U}&:=-\partial_x T_VU,\\
\mathcal{N}_{\geq 2,3}^{U}&:=(H_2^\ast+i\Omega_2^\ast)+(H^\ast_{\geq 3}+i\Omega_{\geq 3}^\ast),
\end{split}
\end{equation}
as a consequence of \eqref{Usor3}. For any $t\in[T_1,T_2]$ we also have
\begin{equation}\label{hoe14}
\begin{split}
&\|U(t)\|_{H^{N_0}}\lesssim A,\qquad \|U(t)\|_{\widetilde{W}^{N_1}}\lesssim\varep(t),\\
&\big\|\mathcal{N}_{\geq 2,3}^{U}(t)\big\|_{H^{N_0}}\lesssim A\varep(t),\qquad\big\|\mathcal{N}_{\geq 2,3}^{U}(t)\big\|_{\widetilde{W}^{N_1}}\lesssim [\varep(t)]^2,
\end{split}
\end{equation}
as a consequence of the assumption \eqref{eni1}, Lemma \ref{Usor1}, and Lemma \ref{algeProp}.

For any $s\in[T_1,T_2]$ and $k\in\mathbb{Z}$ we define the frequency envelopes
\begin{equation}\label{hoe8}
\begin{split}
&A_k(s):=\sum_{l\in\Z}2^{-|l-k|/10}\|P_lU(s)\|_{H^{N_0}}\approx\sum_{l\in\Z}2^{-|l-k|/10}\big(2^{N_0l^+}\|P_lU(s)\|_{L^2}\big),\\
&B_k(s):=\sum_{l\in\Z}2^{-|l-k|/10}\big\|P_l\mathcal{N}_{\geq 2,3}^{U}(s)\big\|_{H^{N_0}}\approx\sum_{l\in\Z}2^{-|l-k|/10}\big(2^{N_0l^+}\big\|P_l\mathcal{N}_{\geq 2,3}^{U}(s)\big\|_{L^2}\big),
\end{split}
\end{equation}
and notice that, for any $s\in[T_1,T_2]$,
\begin{equation}\label{hoe9}
\sum_{k\in\Z}[A_k(s)]^2\lesssim \|U(s)\|_{H^{N_0}}^2\lesssim A^2,\qquad \sum_{k\in\Z}[B_k(s)]^2\lesssim \big\|\mathcal{N}_{\geq 2,3}^{U}(s)\big\|_{H^{N_0}}^2\lesssim A^2[\varep(s)]^2.
\end{equation}
For \eqref{hoe4} it suffices to prove that for any $l\in\Z$, $\iota\in\{+,-\}$, and $t_1\leq t_2\in[T_1,T_2]$ we have
\begin{equation}\label{hoe10}
\begin{split}
&2^{[\max(\widetilde{k}_3,0)-\min(\widetilde{k}_4,0)]/10}\Big|\int_{t_1}^{t_2}\mathcal{A}_M(P_{k_1}U,P_{k_2}U,P_{k_3}U,P_{k_4}U^{\iota})(s)\,ds\Big|\\
&\lesssim A_l^2(t_1)[\varep(t_1)]^2+A^2_l(t_2)[\varep(t_2)]^2+\int_{t_1}^{t_2}A_l(s)[\varep(s)]^2[\varepsilon(s)A_l(s)+B_l(s)]\,ds,
\end{split}
\end{equation}
for any $(k_1,k_2,k_3,k_4)\in Z_l^4:=\{(k_1,k_2,k_3,k_4)\in\Z^4:\,l=\max(k_1,k_2,k_3,k_4)\}$.

To bound the left-hand side of \eqref{hoe10} we would like to integrate by parts in time. 
We examine the formula \eqref{hoe7} and let $\mathcal{N}^{U^+}_{\geq 2}=\mathcal{N}^{U}_{\geq 2}$
and $\mathcal{N}^{U^-}_{\geq 2}=\overline{\mathcal{N}^{U}_{\geq 2}}$. 
Let
\begin{equation}\label{hoe11}
\mathcal{T}_M(f_1,f_2,f_3,f_4):=\frac{1}{(2\pi R)^3}\sum_{(\xi_1,\xi_2,\xi_3,\xi_4)\in\mathbb{H}^3_R}\frac{M(\xi_1,\xi_2,\xi_3,\xi_4)}{\Phi_{+++\iota}(\xi_1,\xi_2,\xi_3,\xi_4)}\widehat{f_1}(\xi_1)\widehat{f_2}(\xi_2)\widehat{f_3}(\xi_3)\widehat{f_4}(\xi_4),
\end{equation}
where $\Phi_{+++\iota}(\xi_1,\xi_2,\xi_3,\xi_4):=|\xi_1|^{1/2}+|\xi_2|^{1/2}+|\xi_3|^{1/2}+\iota|\xi_4|^{1/2}$. We start from the identity
\begin{equation*}
\begin{split}
&\mathcal{T}_M(P_{k_1}U,P_{k_2}U,P_{k_3}U,P_{k_4}U^{\iota})(t_2)-\mathcal{T}_M(P_{k_1}U,P_{k_2}U,P_{k_3}U,P_{k_4}U^{\iota})(t_1)\\
&=\int_{t_1}^{t_2}\frac{d}{ds}\mathcal{T}_M(P_{k_1}U,P_{k_2}U,P_{k_3}U,P_{k_4}U^{\iota})(s)\,ds
\end{split}
\end{equation*}
for any $k_1,k_2,k_3,k_4\in\Z$, and use the formula $\partial_tU^{\pm}=\mp i|\partial_x|^{1/2}U^{\pm}+\mathcal{N}^{U^\pm}_{\geq 2}$ to expand the right-hand side. This gives the normal forms identity
\begin{equation}\label{hoe12}
\begin{split}
&\int_{t_1}^{t_2}\mathcal{A}_M(P_{k_1}U,P_{k_2}U,P_{k_3}U,P_{k_4}U^{\iota})(s)\,ds\\
&=-i\int_{t_1}^{t_2}\Big\{\mathcal{T}_M(P_{k_1}\mathcal{N}^{U^+}_{\geq 2},P_{k_2}U,P_{k_3}U,P_{k_4}U^{\iota})(s)+\mathcal{T}_M(P_{k_1}U,P_{k_2}\mathcal{N}^{U^+}_{\geq 2},P_{k_3}U,P_{k_4}U^{\iota})(s)\\
&\qquad\quad\,\,\,\,+\mathcal{T}_M(P_{k_1}U,P_{k_2}U,P_{k_3}\mathcal{N}^{U^+}_{\geq 2},P_{k_4}U^{\iota})(s)+\mathcal{T}_M(P_{k_1}U,P_{k_2}U,P_{k_3}U,P_{k_4}\mathcal{N}^{U^\iota}_{\geq 2})(s)\Big\}\,ds\\
&+i\big\{\mathcal{T}_M(P_{k_1}U,P_{k_2}U,P_{k_3}U,P_{k_4}U^{\iota})(t_2)-\mathcal{T}_M(P_{k_1}U,P_{k_2}U,P_{k_3}U,P_{k_4}U^{\iota})(t_1)\big\}.
\end{split}
\end{equation}
We will use this identity to bound the left-hand side of \eqref{hoe10}.

\smallskip
{\bf{Step 2.}} We would like to describe now the multiplier of the operator $\mathcal{T}_M$
in \eqref{hoe11}, using the assumption that $M$ is admissible.
We notice that for any $\underline{k}=(k_1,k_2,k_3,k_4)\in\mathbb{Z}^4$
\begin{equation}\label{hoe17}
\begin{split}
&\big\|M'\cdot \varphi_{\underline{k}}\big\|_{\widetilde{S}^\infty}\lesssim 2^{\widetilde{k}_4/2}2^{3\widetilde{k}_3/2}2^{2N_0\max(\widetilde{k}_1,0)}\begin{cases}
1&\text{ if }\iota=-\text{ and }k_4\geq \widetilde{k}_1-5,\\
2^{(\widetilde{k}_3-\widetilde{k}_1)/2}&\text{ if }\iota=+\text{ or }k_4\leq \widetilde{k}_1-6.
\end{cases}
\end{split}
\end{equation}
where $M':=M/\Phi_{+++\iota}$.

Indeed, assume first that $\iota=-$ and $k_4\geq \widetilde{k}_1-5$.
The main observation is that
\begin{equation}\label{hoe17.3}
\sqrt{a}+\sqrt{b}+\sqrt{c}-\sqrt{a+b+c}\geq \sqrt{a}+\sqrt{b}-\sqrt{a+b}\geq \sqrt{b}/2
\end{equation}
for any real numbers $a\geq b\geq c\geq 0$, so that
\begin{align}\label{hoe17.5}
{\Big\| \frac{\varphi_{k_1}(\xi_1)\varphi_{k_2}(\xi_2)\varphi_{k_3}(\xi_3)\varphi_{k_4}(-\xi_1-\xi_2-\xi_3)}{\Phi_{+++-}(\xi_1,\xi_2,\xi_3,-\xi_1-\xi_2-\xi_3)} 
  \Big\|}_{S^\infty} \lesssim 2^{-\wt{k}_3/2}.
\end{align}
The bound \eqref{hoe17} then follows using also the product estimates \eqref{al8.4},
and the assumptions \eqref{A1}. 

On the other hand, if $\iota=+$ or $k_4\leq \widetilde{k}_1-6$, 
then the elliptic bounds $\Phi_{+++\iota}(\xi_1,\xi_2,\xi_3,\xi_4)\gtrsim 2^{\widetilde{k}_1/2}$ 
hold in the support of the multiplier, and the bounds \eqref{hoe17} follow again using the product estimates \eqref{al8.4},
and the assumptions \eqref{A1}.

\smallskip
{\bf{Step 3.}} We examine now the terms in the right-hand side of \eqref{hoe12}. 
The boundary terms can be estimated using just \eqref{hoe17}, 
the bounds in the first line of \eqref{hoe14}, and Lemma \ref{touse} (ii),
\begin{equation*}
\begin{split}
\big|\mathcal{T}_M&(P_{k_1}U,P_{k_2}U,P_{k_3}U,P_{k_4}U^{\iota})(s)\big|\\
&\lesssim A_{\widetilde{k}_1}(s)A_{\widetilde{k}_2}(s)(2^{\widetilde{k}_4/2}2^{-N_1\max(\widetilde{k}_4,0)}\varepsilon(s))(2^{3\widetilde{k}_3/2}2^{-N_1\max(\widetilde{k}_3,0)}\varepsilon(s))
\end{split}
\end{equation*}
for $s\in[t_1,t_2]$. This is consistent with the desired estimates \eqref{hoe10}.

To control the space-time integrals we decompose $\mathcal{N}_{\geq 2}^{U^{\pm}}
=\mathcal{N}_{\geq 2,1}^{U^{\pm}}+\mathcal{N}_{\geq 2,2}^{U^{\pm}}+\mathcal{N}_{\geq 2,3}^{U^{\pm}}$ 
as in \eqref{hoe7}. The contribution of the nonlinear terms $\mathcal{N}_{\geq 2,3}^{U^{\pm}}$, 
which do not lose derivatives, can be bounded using just \eqref{hoe14}, \eqref{hoe17}, and Lemma \ref{touse} (ii),
\begin{equation}\label{uyt1}
\begin{split}
&\big|\mathcal{T}_M(P_{k_1}\mathcal{N}^{U^+}_{\geq 2,3},P_{k_2}U,P_{k_3}U,P_{k_4}U^{\iota})(s)\big|+\big|\mathcal{T}_M(P_{k_1}U,P_{k_2}\mathcal{N}^{U^+}_{\geq 2,3},P_{k_3}U,P_{k_4}U^{\iota})(s)\big|\\
&+\big|\mathcal{T}_M(P_{k_1}U,P_{k_2}U,P_{k_3}\mathcal{N}^{U^+}_{\geq 2,3},P_{k_4}U^{\iota})(s)\big|+\big|\mathcal{T}_M(P_{k_1}U,P_{k_2}U,P_{k_3}U,P_{k_4}\mathcal{N}^{U^\iota}_{\geq 2,3})(s)\big|\\
&\lesssim [B_{\widetilde{k}_1}(s)+\varep(s)A_{\widetilde{k}_1}(s)]A_{\widetilde{k}_2}(s)(2^{\widetilde{k}_4/2}2^{-N_1\max(\widetilde{k}_4,0)}\varepsilon(s))(2^{3\widetilde{k}_3/2}2^{-N_1\max(\widetilde{k}_3,0)}\varepsilon(s))
\end{split}
\end{equation}
for $s\in[t_1,t_2]$. This is also consistent with the desired estimates \eqref{hoe10}.

After these reductions, for \eqref{hoe10} it remains to prove that
\begin{equation}\label{uyt2}
\begin{split}
&\big|\mathcal{T}_M(P_{k_1}\mathcal{N}^{U^+}_{\geq 2,j},P_{k_2}U,P_{k_3}U,P_{k_4}U^{\iota})(s)
  +\mathcal{T}_M(P_{k_1}U,P_{k_2}\mathcal{N}^{U^+}_{\geq 2,j},P_{k_3}U,P_{k_4}U^{\iota})(s)
  \\
& +\mathcal{T}_M(P_{k_1}U,P_{k_2}U,P_{k_3}\mathcal{N}^{U^+}_{\geq 2,j},P_{k_4}U^{\iota})(s)
+ \mathcal{T}_M(P_{k_1}U,P_{k_2}U,P_{k_3}U,P_{k_4}\mathcal{N}^{U^\iota}_{\geq 2,j})(s)\big|
\\
&\qquad\qquad\lesssim[\varep(s)]^3A_{l}^2(s)2^{[\min(\widetilde{k}_4,0)-\max(\widetilde{k}_3,0)]/10}
\end{split}
\end{equation}
for any $s\in[t_1,t_2]$, $j\in\{1,2\}$, and $(k_1,k_2,k_3,k_4)\in Z_l^4$. 
For this we need to use symmetrization to avoid losing derivatives; 
we will also need the $L^\infty$-type bounds
\begin{equation}\label{hoe15}
\|V(s)\|_{\widetilde{W}^{N_1-1}}+\|\beta(s)\|_{\widetilde{W}^{N_1-2}}\lesssim\varep(s),
\end{equation}
which follow from Lemma \ref{BVexpands}.

\smallskip
{\bf{Step 4.}} We prove now the bounds \eqref{uyt2} for $j=1$.  
Without loss of generality, we may assume that $k_1\geq k_2\geq k_3$. 
In view of \eqref{hoe15} and \eqref{hoe14}, for any $s\in[t_1,t_2]$ and $n\in\Z$ we have
\begin{equation}\label{uyt3}
\begin{split}
&\big\|P_n\mathcal{N}^{U^{\pm}}_{\geq 2,1}(s)\big\|_{L^2}\lesssim 2^{n/2-N_0\max(n,0)}A_n(s)\varepsilon(s),
\\
&\big\|P_n\mathcal{N}^{U^{\pm}}_{\geq 2,1}(s)\big\|_{L^\infty}\lesssim 2^{n/2-N_1\max(n,0)}[\varepsilon(s)]^2.
\end{split}
\end{equation}
In particular, using also \eqref{hoe17} and Lemma \ref{touse} (ii),
\begin{equation*}
\begin{split}
&\big|\mathcal{T}_M(P_{k_1}\mathcal{N}^{U^+}_{\geq 2,1},P_{k_2}U,P_{k_3}U,P_{k_4}U^{\iota})(s)\big|
+ \big|\mathcal{T}_M(P_{k_1}U,P_{k_2}\mathcal{N}^{U^+}_{\geq 2,1},P_{k_3}U,P_{k_4}U^{\iota})(s)\big|
\\
&+\big|\mathcal{T}_M(P_{k_1}U,P_{k_2}U,P_{k_3}\mathcal{N}^{U^+}_{\geq 2,1},P_{k_4}U^{\iota})(s)\big|
+ \big|\mathcal{T}_M(P_{k_1}U,P_{k_2}U,P_{k_3}U,P_{k_4}\mathcal{N}^{U^\iota}_{\geq 2,1})(s)\big|
\\
&\lesssim A_{\widetilde{k}_1}(s)A_{\widetilde{k}_2}(s)[\varepsilon(s)]^3
  2^{-N_1\max(\widetilde{k}_4,0)}2^{-N_1\max(\widetilde{k}_3,0)}2^{\widetilde{k}_4/2}2^{3\widetilde{k}_3/2}
\begin{cases}
2^{\widetilde{k}_1/2}&\text{ if }\iota=-\text{ and }k_4\geq \widetilde{k}_1-5,
\\
2^{\widetilde{k}_3/2}&\text{ if }\iota=+\text{ or }k_4\leq \widetilde{k}_1-6.
\end{cases}
\end{split}
\end{equation*}
This suffices to prove the desired bounds \eqref{uyt2} if $\iota=+$ or if $\widetilde{k}_3\geq l-20$ or if $l\leq 20$. 

Recalling the assumption $k_1\geq k_2\geq k_3$, it remains to prove \eqref{uyt2}
in the case when $\iota=-$, $k_1,k_4\in[l-6,l]$, $l\geq 20$, and $k_2,k_3\leq l-20$. 
Using \eqref{uyt3}, \eqref{hoe17}, and Lemma \ref{touse} (ii) we have
\begin{equation*}
\begin{split}
\big|\mathcal{T}_M(P_{k_1}U,P_{k_2}\mathcal{N}^{U^+}_{\geq 2,1},P_{k_3}U,P_{k_4}U^{-})(s)\big|
  +\big|\mathcal{T}_M(P_{k_1}U,P_{k_2}U,P_{k_3}\mathcal{N}^{U^+}_{\geq 2,1},P_{k_4}U^{-})(s)\big|
  \\
\lesssim A_{\widetilde{k}_1}(s)A_{\widetilde{k}_2}(s)[\varepsilon(s)]^3
  2^{-N_1\max(\widetilde{k}_4,0)}2^{-N_1\max(\widetilde{k}_3,0)}2^{\widetilde{k}_4/2}2^{2\widetilde{k}_3},
\end{split}
\end{equation*}
which is better than what we need. We combine the remaining two terms 
and recall that $\overline{T_fg}=T_{\overline{f}}\overline{g}$. After changes of variables we can therefore rewrite
\begin{equation}\label{uyt4}
\begin{split}
\mathcal{T}_M&(P_{k_1}\mathcal{N}^{U^+}_{\geq 2,1},P_{k_2}U,P_{k_3}U,P_{k_4}U^-)
+\mathcal{T}_M(P_{k_1}U,P_{k_2}U,P_{k_3}U,P_{k_4}\mathcal{N}^{U^-}_{\geq 2,1})
\\
& = \frac{-i}{(2\pi R)^4}\sum_{(\eta,\xi_1,\xi_2,\xi_3,\xi_4)\in\mathbb{H}^4_R}
  N_1(\eta,\xi_1,\xi_2,\xi_3,\xi_4)\widehat{\beta}(\eta)\widehat{U}(\xi_1)
  \widehat{U}(\xi_2)\widehat{U}(\xi_3)\widehat{\overline{U}}(\xi_4),
\end{split}
\end{equation}
where
\begin{equation}\label{uyt5}
\begin{split}
& N_1(\eta,\xi_1,\xi_2,\xi_3,\xi_4) :=\chi(\eta,\xi_1)|\xi_1|^{1/2}M'(\xi_1+\eta,\xi_2,\xi_3,\xi_4)
    \varphi_{k_1}(\xi_1+\eta)\varphi_{k_2}(\xi_2)\varphi_{k_3}(\xi_3)\varphi_{k_4}(\xi_4)
    \\
& \qquad -\chi(\eta,\xi_4)|\xi_4|^{1/2}M'(\xi_1,\xi_2,\xi_3,\xi_4+\eta)
  \varphi_{k_1}(\xi_1)\varphi_{k_2}(\xi_2)\varphi_{k_3}(\xi_3)\varphi_{k_4}(\xi_4+\eta).
\end{split}
\end{equation}
To complete the proof of \eqref{uyt2} when $j=1$ it suffices to show that
\begin{equation}\label{uyt6}
\|(1+\eta^2)^{-3/8}N_1(\eta,\xi_1,\xi_2,\xi_3,\xi_4)\|_{\widetilde{S}^\infty(\mathbb{H}^4_R)}
  \lesssim 2^{2N_0l}2^{k_3^-/10}2^{2.5k_2^+}.
\end{equation}

To prove \eqref{uyt6} we may replace $\chi(\eta,\xi_1)$ and $\chi(\eta,\xi_4)$ with $\varphi_{\leq l-30}(\eta)$,
at the expense of harmless 
multipliers that do not lose derivatives. Moreover, we may assume that, for some $\beta\in\{0,1,3/2\}$,
\begin{equation}\label{susi1}
\begin{split}
&M(\rho_1,\rho_2,\rho_3,\rho_4)\varphi_{\underline{k}}(\rho)=|\rho_2+\rho_3|^{\beta}F_\beta(\rho_1,\rho_2,\rho_3,\rho_4)\varphi_{\underline{k}}(\rho),\\
&\big|2^{\alpha\cdot \underline{k}}\partial^\al_{\rho}
  \big[F_{\beta}(\rho)\varphi_{\underline{k}}(\rho)\big]\big|
  \lesssim 2^{k_3/2} 2^{(2-\beta)k_2}  
  2^{2N_0l},\qquad\text{ for }|\alpha|\leq 6\text{ and }\rho\in\mathbb{H}^3_\infty.
\end{split}
\end{equation}
Letting $F'_\beta= F_\beta/\Phi_{+++-}$, we write the remaining symbol as
\begin{equation}\label{uyt5'}
\begin{split}
& 
  \varphi_{\leq l-30}(\eta) |\xi_2+\xi_3|^{\beta}
  \big[ |\xi_1|^{1/2}
 F'_\beta(\xi_1+\eta,\xi_2,\xi_3,\xi_4)
    \varphi_{k_1}(\xi_1+\eta)\varphi_{k_2}(\xi_2)\varphi_{k_3}(\xi_3)\varphi_{k_4}(\xi_4)
    \\
& \qquad - 
  |\xi_4|^{1/2}F'_\beta(\xi_1,\xi_2,\xi_3,\xi_4+\eta)
  \varphi_{k_1}(\xi_1)\varphi_{k_2}(\xi_2)\varphi_{k_3}(\xi_3)\varphi_{k_4}(\xi_4+\eta) \big].
\end{split}
\end{equation}
It follows from \eqref{hoe17.5} and \eqref{susi1} that  
\begin{align}\label{uytulF}
\begin{split}
& \big|2^{\alpha\cdot \underline{k}}\partial^\al_{\xi}
  \big[F'_\beta (\xi)\varphi_{\underline{k}}(\xi)\big]\big|
  \lesssim 2^{k_3/2} 2^{(3/2-\beta)k_2} 2^{2 N_0l}\qquad\text{ for }|\alpha|\leq 6\text{ and }\xi\in\mathbb{H}^3_\infty.
\end{split}
\end{align}
We then split \eqref{uyt5'} as the sum of three symbols 
\begin{equation*}
\varphi_{\leq l-30}(\eta) |\xi_2+\xi_3|^\beta \big( N_1^1+N_1^2\big) (\eta,\xi_1,\xi_2,\xi_3,\xi_4)
\end{equation*}
where
\begin{equation*}
\begin{split}
N_1^1(\eta,\xi_1,\xi_2,\xi_3,\xi_4) &:= 
  \big[ |\xi_1|^{1/2}-|\xi_1+\xi_2+\xi_3+\eta|^{1/2} \big]\\
 & \times F'_\beta(\xi_1+\eta,\xi_2,\xi_3,\xi_4)
    \varphi_{k_1}(\xi_1+\eta)\varphi_{k_2}(\xi_2)\varphi_{k_3}(\xi_3)\varphi_{k_4}(\xi_4),
    \end{split}
\end{equation*}
and
\begin{equation*}
\begin{split}
N_1^2&(\eta,\xi_1,\xi_2,\xi_3,\xi_4)  := 
  |\xi_4|^{1/2}\varphi_{k_2}(\xi_2)\varphi_{k_3}(\xi_3)
  \\
& \times \big[ -F'_\beta(\xi_1,\xi_2,\xi_3,\xi_4+\eta)\varphi_{k_1}(\xi_1)\varphi_{k_4}(\xi_4+\eta)
   + F'_\beta(\xi_1+\eta,\xi_2,\xi_3,\xi_4)
    \varphi_{k_1}(\xi_1+\eta)\varphi_{k_4}(\xi_4)\big].
\end{split}
\end{equation*}
Finally, we multiply the two symbols above by $(1+\eta^2)^{-3/8}$ and use \eqref{uytulF}
and Lemma \ref{multiBound} (i) and (iv) to prove the bounds \eqref{uyt6}.

\smallskip
{\bf{Step 5.}} 
For the last step we need to prove the bounds \eqref{uyt2} for $j=2$.  
In view of \eqref{hoe15} and \eqref{hoe14}, for any $s\in[t_1,t_2]$ and $n\in\Z$ we have
\begin{equation}\label{uyt8}
\begin{split}
&\big\|P_n\mathcal{N}^{U^{\pm}}_{\geq 2,2}(s)\big\|_{L^2}\lesssim 2^{n-N_0\max(n,0)}A_n(s)\varepsilon(s),
\\
&\big\|P_n\mathcal{N}^{U^{\pm}}_{\geq 2,2}(s)\big\|_{L^\infty}\lesssim 2^{n-N_1\max(n,0)}[\varepsilon(s)]^2.
\end{split}
\end{equation}
As in the case $j=1$, this suffices to prove the desired bounds \eqref{uyt2} if $\widetilde{k}_3\geq l-20$ or if $l\leq 20$. In the remaining case when $\widetilde{k}_3\leq l-20$ and $l\geq 20$, we may assume that $k_1\geq k_2\geq k_3$. There are two cases, depending on whether $k_4$ is one of the two larger indices, or not. The cases are identical, so assume that $k_4\geq l-6$, so $k_2,k_3\leq l-20$. We can still estimate
\begin{equation*}
\begin{split}
\big|\mathcal{T}_M(P_{k_1}U,P_{k_2}\mathcal{N}^{U^+}_{\geq 2,2},P_{k_3}U,P_{k_4}U^{\iota})(s)\big|+\big|\mathcal{T}_M(P_{k_1}U,P_{k_2}U,P_{k_3}\mathcal{N}^{U^+}_{\geq 2,2},P_{k_4}U^{\iota})(s)\big|\\
\lesssim A_{\widetilde{k}_1}(s)A_{\widetilde{k}_2}(s)[\varepsilon(s)]^32^{-N_1\max(\widetilde{k}_4,0)}2^{-N_1\max(\widetilde{k}_3,0)}2^{\widetilde{k}_4/2}2^{2.5\widetilde{k}_3},
\end{split}
\end{equation*}
using \eqref{uyt8}, \eqref{hoe17}, and Lemma \ref{touse} (ii). We combine the remaining two terms and rewrite
\begin{equation}\label{uyt4'}
\begin{split}
\mathcal{T}_M &(P_{k_1}\mathcal{N}^{U^+}_{\geq 2,2}, P_{k_2}U, P_{k_3}U, P_{k_4}U^\iota)
  +\mathcal{T}_M(P_{k_1}U, P_{k_2}U, P_{k_3}U, P_{k_4}\mathcal{N}^{U^\iota}_{\geq 2,2})
  \\
& =\frac{-i}{(2\pi R)^4}\sum_{(\eta,\xi_1,\xi_2,\xi_3,\xi_4)\in\mathbb{H}^4_R}
  N_2(\eta,\xi_1,\xi_2,\xi_3,\xi_4)\widehat{V}(\eta)
  \widehat{U}(\xi_1)\widehat{U}(\xi_2)\widehat{U}(\xi_3)\widehat{U^\iota}(\xi_4),
\end{split}
\end{equation}
using the fact that $V$ is real-valued, where
\begin{equation*}
\begin{split}
N_2(\eta,\xi_1,\xi_2,&\xi_3,\xi_4):=\chi(\eta,\xi_1)(\xi_1+\eta)
  M'(\xi_1+\eta,\xi_2,\xi_3,\xi_4)\varphi_{k_1}(\xi_1+\eta)\varphi_{k_2}(\xi_2)\varphi_{k_3}(\xi_3)\varphi_{k_4}(\xi_4)
  \\
&+\chi(\eta,\xi_4)(\xi_4+\eta)M'(\xi_1,\xi_2,\xi_3,\xi_4+\eta)
  \varphi_{k_1}(\xi_1)\varphi_{k_2}(\xi_2)\varphi_{k_3}(\xi_3)\varphi_{k_4}(\xi_4+\eta).
\end{split}
\end{equation*}
As in the case $j=1$ above, one can show that
\begin{equation}\label{uyt6'}
\|(1+\eta^2)^{-3/4}N_2(\eta,\xi_1,\xi_2,\xi_3,\xi_4)\|_{\widetilde{S}^\infty(\mathbb{H}^4_R)}
  \lesssim 2^{2N_0l}2^{k_3^-/10}2^{2.5k_2^+},
\end{equation}
which suffices to prove the desired bound \eqref{uyt2} and concludes the proof of the lemma.
\end{proof}

\subsection{Resonant cases}\label{ssecres}
In the next lemma we consider the more subtle case of resonant quartic interactions
which involve two $U$ and two $\bar{U}$ functions and the resonant phase
\begin{equation}\label{uyt28}
\Phi_{++--}(\xi_1,\xi_2,\xi_3,\xi_4):=|\xi_1|^{1/2}+|\xi_2|^{1/2}-|\xi_3|^{1/2}-|\xi_4|^{1/2}.
\end{equation}
Since we will need to distinguish cases depending on the signs of the variables, we define
\begin{align}\label{phi+-}
\begin{split}
& \varphi_k^+(x):=\varphi_k(x)\mathbf{1}_+(x), \qquad \varphi_k^-(x):=\varphi_k(x)\mathbf{1}_-(x),\\
&\varphi_k^{'+}(x):=\varphi'_k(x)\mathbf{1}_+(x), \qquad \varphi_k^{'-}(x):=\varphi'_k(x)\mathbf{1}_-(x),
\\
& \varphi^{\underline{\iota}}_{\underline{k}}(\xi)
:=\varphi^{\iota_1}_{k_1}(\xi_1)\varphi^{\iota_2}_{k_2}(\xi_2)\varphi^{\iota_3}_{k_3}(\xi_3)\varphi^{\iota_4}_{k_4}(\xi_4),\qquad \widetilde{\varphi}^{\underline{\iota}}_{\underline{k}}(\xi)
:=\varphi^{'\iota_1}_{k_1}(\xi_1)\varphi^{'\iota_2}_{k_2}(\xi_2)\varphi^{'\iota_3}_{k_3}(\xi_3)\varphi^{'\iota_4}_{k_4}(\xi_4),
\end{split}
\end{align}
where $\underline{\iota} := (\iota_1,\iota_2,\iota_3,\iota_4)\in\{+,-\}^4$ and $\underline{k} := (k_1,k_2,k_3,k_4)\in\Z^4$, and the functions $\varphi'_k$ are defined as in \eqref{P'_k}.

\begin{lemma}\label{hol2} 
Assume that $M$ is defined according to the symmetrization formula 
\begin{equation}\label{verf1*}
\begin{split}
M(\xi_1,\xi_2,\xi_3,\xi_4):&=m(\xi_1,\xi_2,\xi_3,\xi_4)+m(\xi_1,\xi_2,\xi_4,\xi_3)+m(\xi_2,\xi_1,\xi_3,\xi_4)+m(\xi_2,\xi_1,\xi_4,\xi_3)\\
&-m(\xi_3,\xi_4,\xi_1,\xi_2)-m(\xi_3,\xi_4,\xi_2,\xi_1)-m(\xi_4,\xi_3,\xi_1,\xi_2)-m(\xi_4,\xi_3,\xi_2,\xi_1)
\end{split}
\end{equation}
for any $\xi=(\xi_1,\xi_2,\xi_3,\xi_4)\in\mathbb{H}_\infty^3$, where $m\in\mathcal{A}$ is an admissible symbol as in Definition \ref{defA}. Assume in addition that $M$ vanishes on the set of Benjamin-Feir resonances, i.e.
\begin{align}\label{uyt22}\tag{NR-BF}
\begin{split}
&M(\xi_1,\xi_2,\xi_3,\xi_4)=0\qquad\text{ provided that }\xi=(\xi_1,\xi_2,\xi_3,\xi_4)\in\mathbb{H}_\infty^3\text{ satisfies}\\
&\xi_1\in(-\infty,0),\,\xi_2,\xi_3,\xi_4\in(0,\infty)\text{ and }|\xi_1|^{1/2}+|\xi_2|^{1/2}=|\xi_3|^{1/2}+|\xi_4|^{1/2}.
\end{split}
\end{align}

Then, with $\mathcal{A}_M$ defined as in \eqref{hoe1.5} and for any $t_1\leq t_2\in[T_1,T_2]$ we have
\begin{equation}\label{uyt25}
\Big|\int_{t_1}^{t_2}\mathcal{A}_M(U,U,\overline{U},\overline{U})(s)\,ds\Big|
  \lesssim [A(t_1)\varep(t_1)]^2+[A(t_2)\varep(t_2)]^2+\int_{t_1}^{t_2}A^2[\varep(s)]^3\,ds.
\end{equation}
\end{lemma}

\begin{proof} {\bf{Step 1.}} It follows from \eqref{verf1*} that $M$ satisfies the symmetry assumptions
\begin{equation}\label{sus7}
M(\xi_1,\xi_2,\xi_3,\xi_4)=M(\xi_2,\xi_1,\xi_3,\xi_4)=M(\xi_1,\xi_2,\xi_4,\xi_3)=M(\xi_2,\xi_1,\xi_4,\xi_3).
\end{equation}
Let $M':=M/\Phi_{++--}$. We will show below that for any $\underline{k}\in\mathbb{Z}^4$ and $\underline{\iota}\in\{+,-\}^4$
\begin{equation}\label{sus8}
\big\|M'\cdot \varphi_{\underline{k}}^{\underline{\iota}}\big\|_{\widetilde{S}^\infty}\lesssim 2^{\widetilde{k}_4/2}2^{3\widetilde{k}_3/2}2^{2N_0\max(\widetilde{k}_1,0)}.
\end{equation}
Moreover, if $k_1\geq k_2+12$ and $k_3\geq k_4+12$ then we also show that $M'$ admits an extension $M'':\R^4\to\C$, $M'(\xi)\cdot \varphi_{\underline{k}}^{\underline{\iota}}(\xi)=M''(\xi)\cdot \varphi_{\underline{k}}^{\underline{\iota}}(\xi)$ for any $\xi\in\mathbb{H}^3_\infty$, such that
\begin{equation}\label{sus9}
\big\|2^{k_1}\partial_{\xi_1}[M''\cdot \varphi_{\underline{k}}^{\underline{\iota}}]\big\|_{\widetilde{S}^\infty}+\big\|2^{k_3}\partial_{\xi_3}[M''\cdot \varphi_{\underline{k}}^{\underline{\iota}}]\big\|_{\widetilde{S}^\infty}\lesssim 2^{\widetilde{k}_4/2}2^{3\widetilde{k}_3/2}2^{2N_0\max(\widetilde{k}_1,0)}.
\end{equation}

Assuming \eqref{sus8}--\eqref{sus9}, the desired bounds \eqref{uyt25} follow as in the proof of Lemma \ref{hoe2}. Indeed, as in \eqref{hoe10}, it suffices to show that for any $l\in\Z$ and $t_1\leq t_2\in[T_1,T_2]$ we have
\begin{equation}\label{uyt26}
\begin{split}
&2^{[\max(\widetilde{k}_3,0)-\min(\widetilde{k}_4,0)]/10}\Big|\int_{t_1}^{t_2}
  \mathcal{A}_M(P_{k_1}U,P_{k_2}U,P_{k_3}\overline{U},P_{k_4}\overline{U})(s)\,ds\Big|
  \\
&\lesssim A_l^2(t_1)[\varep(t_1)]^2+A^2_l(t_2)[\varep(t_2)]^2+\int_{t_1}^{t_2}
  A_l(s)[\varep(s)]^2[\varepsilon(s)A_l(s)+B_l(s)]\,ds,
\end{split}
\end{equation}
for any $(k_1,k_2,k_3,k_4)\in Z_l^4=\{(k_1,k_2,k_3,k_4)\in\Z^4:\,l=\max(k_1,k_2,k_3,k_4)\}$, 
where $A_k$ and $B_k$ are defined as in \eqref{hoe8}. To prove this we define the operators
\begin{equation}\label{uyt27}
\mathcal{T}_M(f_1,f_2,f_3,f_4):=\frac{1}{(2\pi R)^3}\sum_{(\xi_1,\xi_2,\xi_3,\xi_4)\in\mathbb{H}^3_R}\frac{M(\xi_1,\xi_2,\xi_3,\xi_4)}{\Phi_{++--}(\xi_1,\xi_2,\xi_3,\xi_4)}\widehat{f_1}(\xi_1)\widehat{f_2}(\xi_2)\widehat{f_3}(\xi_3)\widehat{f_4}(\xi_4),
\end{equation}
use a normal forms identity similar to \eqref{hoe12},
and apply arguments similar to those in the proof of Lemma \ref{hoe2}.
In particular, it suffices to show that
\begin{equation}\label{sus10}
\begin{split}
&\big|\mathcal{T}_M(P_{k_1}\mathcal{N}^{U^+}_{\geq 2,j},P_{k_2}U,P_{k_3}U,P_{k_4}U^{\iota})(s)
  +\mathcal{T}_M(P_{k_1}U,P_{k_2}\mathcal{N}^{U^+}_{\geq 2,j},P_{k_3}U,P_{k_4}U^{\iota})(s)
  \\
& +\mathcal{T}_M(P_{k_1}U,P_{k_2}U,P_{k_3}\mathcal{N}^{U^-}_{\geq 2,j},P_{k_4}U^{\iota})(s)
+ \mathcal{T}_M(P_{k_1}U,P_{k_2}U,P_{k_3}U,P_{k_4}\mathcal{N}^{U^-}_{\geq 2,j})(s)\big|
\\
&\qquad\qquad\lesssim[\varep(s)]^2A_{l}(s)[\varepsilon(s)A_l(s)+B_l(s)]2^{[\min(\widetilde{k}_4,0)-\max(\widetilde{k}_3,0)]/10}
\end{split}
\end{equation}
for any $(k_1,k_2,k_3,k_4)\in Z_l^4$ and $j\in\{1,2,3\}$, where the nonlinearities $\mathcal{N}^{U^+}_{\geq 2,j}=\mathcal{N}^{U^+}_{\geq 2,j}$ and $\mathcal{N}^{U^-}_{\geq 2,j}=\overline{\mathcal{N}^{U^+}_{\geq 2,j}}$ are defined as in \eqref{hoe7}. 

The estimates \eqref{sus10} follows for $j=3$ using the bounds \eqref{hoe14} and \eqref{sus8}. They also follow for $j=1$ or $j=2$ if $2^{\widetilde{k}_3}\approx 2^l$, using \eqref{uyt3} and \eqref{uyt8}, since there is no derivative loss in this case. In the remaining case we may assume that $k_1\geq k_2$ and $k_3\geq k_4$ (due to \eqref{sus7}) and use a symmetrization argument as in the last two steps of the previous lemma. The desired bounds \eqref{sus10} follow in these cases as well, using the assumptions \eqref{sus9}.

\smallskip
{\bf{Step 2.}} We prove now the bounds \eqref{sus8}--\eqref{sus9} in the case when $\underline{\iota}=(\iota_1,\iota_2,\iota_3,\iota_4)$ contains 3 signs $+$ and 1 sign $-$ or 3 signs $-$ and 1 sign $+$. Indeed, in view of \eqref{sus7} we may assume that 
\begin{equation}\label{sus14}
\underline{\iota}=(-,+,+,+).
\end{equation}
By symmetry we may also assume that $k_3\geq k_4$.

Notice that $|\xi_i+\xi_j|\gtrsim 2^{\widetilde{k}_3}$ for any $\xi\in\mathbb{H}^3_\infty$ in the support of the function $\varphi_{\underline{k}}^{\underline{\iota}}$ and any $i\neq j\in\{1,2,3,4\}$. In particular, since $m$ is admissible, we may assume that $m_\sigma\cdot \varphi_{\underline{k}}^{\underline{\iota}}$ does not contain any of the singular factors $|\xi_i+\xi_j|^\beta$ in the definition \eqref{A2.0}, with $\beta\neq 0$, for any permutation $\sigma$ of the set $\{1,2,3,4\}$, where $m_\sigma(\xi_1,\xi_2,\xi_3,\xi_4)=m(\xi_{\sigma(1)},\xi_{\sigma(2)},\xi_{\sigma(3)},\xi_{\sigma(4)})$. This is because such factors could simply be absorbed into the associated smooth symbols. As a consequence, $M$ admits an extension $M^\ast:\R^4\to\C$, $M(\xi)\cdot \varphi_{\underline{k}}^{\underline{\iota}}(\xi)=M^\ast(\xi)\cdot \varphi_{\underline{k}}^{\underline{\iota}}(\xi)$ for any $\xi\in\mathbb{H}^3_\infty$, such that
\begin{equation}\label{sus15}
\big|2^{\alpha\cdot\underline{k}}\partial_{\xi}^\alpha[M^\ast(\xi)\cdot \varphi_{\underline{k}}^{\underline{\iota}}(\xi)]\big|\lesssim 2^{\widetilde{k}_4/2}2^{2\widetilde{k}_3}2^{2N_0\max(\widetilde{k}_1,0)}
\end{equation}
for any $\xi\in\mathbb{H}_\infty^3$ and $|\alpha|\leq 6$.

We notice that if $\xi_1\in(-\infty,0], \xi_2,\xi_3,\xi_4\in[0,\infty)$ and $\xi_1+\xi_2+\xi_3+\xi_4=0$ then we can write
\begin{equation}\label{sus15.1}
\begin{split}
&\Phi_{++--}(\xi_1,\xi_2,\xi_3,\xi_4)=
  |\xi_2+\xi_3+\xi_4|^{1/2}+|\xi_2|^{1/2}-|\xi_3|^{1/2}-|\xi_4|^{1/2}
  \\
& = \frac{2\sqrt{\xi_2}(\sqrt{\xi_3}+\sqrt{\xi_4})-2\sqrt{\xi_3}\sqrt{\xi_4}}{|\xi_2+\xi_3+\xi_4|^{1/2}-|\xi_2|^{1/2}+|\xi_3|^{1/2}+|\xi_4|^{1/2}}
\\
& = \frac{2\xi_2(\sqrt{\xi_3}+\sqrt{\xi_4})^2-2\xi_3\xi_4}{(|\xi_2+\xi_3+\xi_4|^{1/2}+|\xi_3|^{1/2}+|\xi_4|^{1/2}-|\xi_2|^{1/2})(\sqrt{\xi_2}\sqrt{\xi_3}+\sqrt{\xi_2}\sqrt{\xi_4}+\sqrt{\xi_3}\sqrt{\xi_4})}.
\end{split}
\end{equation}
In particular, $\Phi_{++--}(\xi_1,\xi_2,\xi_3,\xi_4)=0$ if and only if
\begin{equation}\label{sus15.2}
\xi_3,\xi_4\in[0,\infty),\qquad\xi_2=\frac{\xi_3\xi_4}{(\sqrt\xi_3+\sqrt\xi_4)^2},\qquad\xi_1=-\frac{(\xi_3+\xi_4+\sqrt\xi_3\sqrt\xi_4)^2}{(\sqrt\xi_3+\sqrt\xi_4)^2}.
\end{equation}

More importantly, we can define the extension $M'':\R^4\to\C$
\begin{equation}\label{sus15.3}
M''(\xi_1,\xi_2,\xi_3,\xi_4):=\frac{M^\ast(-\xi_2-\xi_3-\xi_4,\xi_2,\xi_3,\xi_4)}{\Phi_{++--}(-\xi_2-\xi_3-\xi_4,\xi_2,\xi_3,\xi_4)}
\end{equation}
such that $(M/\Phi_{++--})(\xi)\cdot \varphi_{\underline{k}}^{\underline{\iota}}(\xi)=M''(\xi)\cdot \varphi_{\underline{k}}^{\underline{\iota}}(\xi)$ for any $\xi\in\mathbb{H}^3_\infty$, where $M^\ast$ satisfies \eqref{sus15}. 

We show now that 
\begin{equation}\label{sus15.4}
\big|2^{\alpha\cdot\underline{k}}\partial_{\xi}^\alpha[M''(\xi)\cdot \varphi_{\underline{k}}^{\underline{\iota}}(\xi)]\big|\lesssim 2^{\widetilde{k}_4/2}2^{3\widetilde{k}_3/2}2^{2N_0\max(\widetilde{k}_1,0)}
\end{equation}
for any $\xi\in\mathbb{H}_\infty^3$ and $|\alpha|\leq 5$. Indeed, since $k_3\geq k_4$ we may assume that $2^{k_1}\approx 2^{k_3}\approx 2^{\widetilde{k}_1}\approx 2^{\widetilde{k}_2}$. If $|k_2-k_4|\geq 6$ then the phase 
$\Phi_{++--}$ is elliptic in the support of $\varphi_{\underline{k}}^{\underline{\iota}}$, due to \eqref{sus15.1}, and the bounds \eqref{sus15.4} follow directly from \eqref{sus15}. On the other hand, if $|k_2-k_4|\leq 6$ then $2^{k_2}\approx 2^{k_4}\approx 2^{\widetilde{k}_3}\approx 2^{\widetilde{k}_4}$.  The main assumption \eqref{uyt22} shows that 
\begin{equation*}
M^\ast(-\xi_2-\xi_3-\xi_4,\xi_2,\xi_3,\xi_4)\varphi_{\underline{k}}^{\underline{\iota}}(-\xi_2-\xi_3-\xi_4,\xi_2,\xi_3,\xi_4)=0\qquad\text{ if }\qquad \xi_2=\frac{\xi_3\xi_4}{(\sqrt\xi_3+\sqrt\xi_4)^2}.
\end{equation*}
Therefore we can write, with $\xi^\ast(\xi_3,\xi_4):=\frac{\xi_3\xi_4}{(\sqrt{\xi_3}+\sqrt{\xi_4})^2}$
\begin{equation}\label{Min7.2}
\begin{split}
(M^\ast\varphi_{\underline{k}}^{\underline{\iota}})&(-\xi_2-\xi_3-\xi_4,\xi_2,\xi_3,\xi_4)=\int_{\xi^\ast(\xi_3,\xi_4)}^{\xi_2}\frac{d}{ds}\big[(M^\ast\varphi_{\underline{k}}^{\underline{\iota}})(-s-\xi_3-\xi_4,s,\xi_3,\xi_4)\big]\,ds\\
&=\int_{\xi^\ast(\xi_3,\xi_4)}^{\xi_2}(-\partial_1+\partial_2)(M^\ast\varphi_{\underline{k}}^{\underline{\iota}})(-s-\xi_3-\xi_4,s,\xi_3,\xi_4)\,ds\\
&=(\xi_2-\xi^\ast(\xi_3,\xi_4))M^{\ast\ast}(\xi_2,\xi_3,\xi_4),
\end{split}
\end{equation}
where
\begin{equation}\label{Min7.3}
\begin{split}
M^{\ast\ast}&(\xi_2,\xi_3,\xi_4):=\int_0^1(-\partial_1+\partial_2)(M^\ast\varphi_{\underline{k}}^{\underline{\iota}})
\\
  & (-\xi^\ast(\xi_3,\xi_4) - \alpha(\xi_2-\xi^\ast(\xi_3,\xi_4))-\xi_3-\xi_4,
  \xi^\ast(\xi_3,\xi_4)+\alpha(\xi_2-\xi^\ast(\xi_3,\xi_4)),\xi_3,\xi_4)\,d\alpha.
\end{split}
\end{equation}
One can check that the symbol $M^{**}$ satisfies the differential bounds
\begin{equation}\label{Min7.6}
\sup_{\xi_2,\xi_3,\xi_4\in\R}\sup_{|\alpha|\leq 5}\big|2^{\alpha_2k_2+\alpha_3k_3+\alpha_4k_4}\partial_{\xi}^\alpha[M^{\ast\ast}(\xi)]\big|\lesssim 2^{\widetilde{k}_4/2}2^{\widetilde{k}_3}2^{2N_0\max(\widetilde{k}_1,0)},
\end{equation}
in view of the differential bounds on $M^*$ in \eqref{sus15}, using the condition that 
$2^{k_2}\approx 2^{k_4}\approx 2^{\widetilde{k}_3}\approx 2^{\widetilde{k}_4}$.
Finally, thanks to the identity \eqref{Min7.2} we can explicitly cancel the vanishing factor 
\begin{align*}
& (M''\varphi_{\underline{k}}^{\underline{\iota}})(\xi_1,\xi_2,\xi_3,\xi_4)= \Big(\xi_2-\frac{\xi_3\xi_4}{(\sqrt\xi_3+\sqrt\xi_4)^2}\Big)
  \frac{M^{\ast\ast}(\xi_2,\xi_3,\xi_4)}{\Phi_{++--}(-\xi_2-\xi_3-\xi_4,\xi_2,\xi_3,\xi_4)}
\\
& = \frac{M^{\ast\ast}(\xi_2,\xi_3,\xi_4) }{2(\sqrt\xi_3+\sqrt\xi_4)^2}
  \frac{2\xi_2(\sqrt\xi_3+\sqrt\xi_4)^2 - 2\xi_3\xi_4}{\Phi_{++--}(-\xi_2-\xi_3-\xi_4,\xi_2,\xi_3,\xi_4)}
\\
& = \frac{M^{\ast\ast}(\xi_2,\xi_3,\xi_4)}{2(\sqrt\xi_3+\sqrt\xi_4)^2}(|\xi_2+\xi_3+\xi_4|^{1/2} + |\xi_3|^{1/2} + |\xi_4|^{1/2} - |\xi_2|^{1/2})
  (\sqrt{\xi_2}\sqrt{\xi_3} + \sqrt{\xi_2}\sqrt{\xi_4}+\sqrt{\xi_3}\sqrt{\xi_4}),
\end{align*}
having used (5.51) for the last identity. The claimed bound \eqref{sus15.4} follow from \eqref{Min7.6}. The main bounds \eqref{sus8}--\eqref{sus9} follow from \eqref{sus15.4} and Lemma \ref{multiBound} (iv).

We use this basic strategy several times in the proof this lemma: derivative bounds on symbols and vanishing on suitable sets lead to favorable decompositions like \eqref{Min7.2}, and suffice to perform the division by singular phases with no loss of derivative.

\smallskip
{\bf{Step 3.}} We prove now the bounds \eqref{sus8}--\eqref{sus9} in the case when 
\begin{equation}\label{sus16.1}
\underline{\iota}=(+,+,-,-).
\end{equation}
By symmetry we may assume that $k_1\geq k_2$ and $k_3\geq k_4$; in particular $2^{k_1}\approx 2^{k_3}\approx 2^{\widetilde{k}_1}\approx 2^{\widetilde{k}_2}$.

Notice that if $\xi_1,\xi_2\in[0,\infty), \xi_3,\xi_4\in(-\infty,0]$ and $\xi_1+\xi_2+\xi_3+\xi_4=0$ then we can write
\begin{equation}\label{sus16}
\begin{split}
\Phi_{++--}(\xi_1,\xi_2,\xi_3,\xi_4)&=\frac{2\sqrt{\xi_1\xi_2}-2\sqrt{\xi_3\xi_4}}{
  |\xi_1|^{1/2}+|\xi_2|^{1/2}+|\xi_3|^{1/2}+|\xi_4|^{1/2}}
  \\
& = \frac{2(\xi_1\xi_2-\xi_3\xi_4)}{(\sqrt{\xi_1\xi_2}+\sqrt{\xi_3\xi_4})
  (|\xi_1|^{1/2}+|\xi_2|^{1/2}+|\xi_3|^{1/2}+|\xi_4|^{1/2})}
\\
& = \frac{-2(\xi_2+\xi_3)(\xi_2+\xi_4)}{(\sqrt{\xi_1\xi_2}+\sqrt{\xi_3\xi_4})
  (|\xi_1|^{1/2}+|\xi_2|^{1/2}+|\xi_3|^{1/2}+|\xi_4|^{1/2})}.
\end{split}
\end{equation}
Notice also that if $|k_2-k_4|\geq 4$ then $|\xi_i+\xi_j|$ is elliptically bounded away from $0$ for any $\xi\in\mathbb{H}^3_\infty$ in the support of $\varphi_{\underline{k}}^{\underline{\iota}}$, and the desired bounds \eqref{sus8}--\eqref{sus9} follow easily. It remains to prove these bounds in the main case
\begin{equation}\label{sus16.2}
k_1,k_3\in[l-4,l],\qquad k_2,k_4\in[n-4,n],\qquad \text{ for some }n\leq l\in\Z.
\end{equation}

{\bf{Substep 3.1.}} We show first that we can write
\begin{equation}\label{sus16.4}
(M\varphi_{\underline{k}}^{\underline{\iota}})(\xi)=\varphi_{\underline{k}}^{\underline{\iota}}(\xi)\sum_{\beta\in\{0,1,3/2\}}\big[|\xi_2+\xi_3|^\beta F_\beta^{23}(\xi)+|\xi_2+\xi_4|^\beta F_\beta^{24}(\xi)\big]
\end{equation}
for any $\xi\in\mathbb{H}^3_\infty$, for some functions $F_\beta^{23},F_\beta^{24}:\R^4\to\C$ satisfying
\begin{equation}\label{sus16.5}
\begin{split}
&(F_\beta^j\varphi_{\underline{k}}^{\underline{\iota}})(\rho,\sigma,-\rho,-\sigma)=(F_\beta^j\varphi_{\underline{k}}^{\underline{\iota}})(\rho,\sigma,-\sigma,-\rho)=0,\\
&\big|2^{\alpha\cdot \underline{k}}\partial^\al_{\xi}
  \big[F^j_{\beta}(\xi)\varphi_{\underline{k}}^{\underline{\iota}}(\xi)\big]\big|
  \lesssim 2^{(5/2-\beta)n}
  2^{2N_0\max(l,0)},
\end{split}
\end{equation}
for any $j\in\{23,24\}$, $\beta\in\{0,1,3/2\}$, $\rho,\sigma\in\R$, multi-index $\alpha=(\alpha_1,\alpha_2,\alpha_3,\alpha_4)$ with $|\alpha|\leq 6$, and $\xi\in\mathbb{H}^3_\infty$. Moreover, the functions $F_\beta^{23}$ are identically $0$ if $l\geq n+20$.

For this we use the main formula \eqref{verf1*} and Definition \ref{defA}. We need to be careful however because the vanishing in the first line of \eqref{sus16.5} depends on the symmetrization formula \eqref{verf1*}, while the multipliers in the right-hand side of \eqref{A2.0} are allowed to depend on $\underline{k}$. 

Assume first that $l-n\geq 20$. Since $m$ is admissible, it follows from Definition \ref{defA} and \eqref{sus16.2} that we can write
\begin{equation}\label{sus16.7}
\begin{split}
&m(\rho)=\sum_{\beta\in\{0,1,3/2\}}|\rho_2+\rho_4|^\beta H_\beta^{1}(\rho)\,\,\text{ for any }\rho\in\mathbb{H}^3_\infty\cap B^{+,+,-,-}_{l,n,l,n},\\
&\text { where }B^{\iota_1,\iota_2,\iota_3,\iota_4}_{l_1,l_2,l_3,l_4}:=\big\{\rho\in\R^4:\,\iota_j\rho_j\in[2^{l_j-6},2^{l_j+2}],\,j\in\{1,2,3,4\}\big\},
\end{split}
\end{equation}
for some function $H_\beta^{1}:\R^4\to\C$ satisfying, for $\beta\in\{0,1,3/2\}$,
\begin{equation}\label{sus16.8}
\big|\rho^\alpha\cdot(\partial^\al_{\rho}H^{1}_{\beta})(\rho)\big|\lesssim 2^{(5/2-\beta)n} 
  2^{2N_0\max(l,0)}, \qquad \rho\in\mathbb{H}^3_\infty\cap B^{+,+,-,-}_{l,n,l,n},\,|\alpha|\leq 6.
\end{equation}
This representation follows from the main formula \eqref{A2.0}, by combining non-singular factors $|\rho_i+\rho_j|$ with the corresponding symbols and summing over $k_1,k_3\in [l-8,l+4]$ and $k_2,k_4\in[n-8,n+4]$.

Similarly, we can also write
\begin{equation}\label{sus16.9}
\begin{split}
&m(\rho)=\sum_{\beta\in\{0,1,3/2\}}|\rho_2+\rho_3|^\beta H_\beta^{2}(\rho)\,\,\text{ for any }\rho\in\mathbb{H}^3_\infty\cap B^{+,+,-,-}_{l,n,n,l},\\
&m(\rho)=\sum_{\beta\in\{0,1,3/2\}}|\rho_1+\rho_4|^\beta H_\beta^{3}(\rho)\,\,\text{ for any }\rho\in\mathbb{H}^3_\infty\cap B^{+,+,-,-}_{n,l,l,n},\\
&m(\rho)=\sum_{\beta\in\{0,1,3/2\}}|\rho_1+\rho_3|^\beta H_\beta^{4}(\rho)\,\,\text{ for any }\rho\in\mathbb{H}^3_\infty\cap B^{+,+,-,-}_{n,l,n,l},\\
\end{split}
\end{equation}
for some functions $H_\beta^{j}:\R^4\to\C$, $j\in\{2,3,4\}$, satisfying differential inequalities similar to \eqref{sus16.7}. 

Then we define, for $\xi=(\xi_1,\xi_2,\xi_3,\xi_4)\in B^{+,+,-,-}_{l,n,l,n}$,
\begin{equation}\label{sus16.10}
\begin{split}
F_\beta^{24}(\xi_1,\xi_2,\xi_3,\xi_4):=H_\beta^1(\xi_1,\xi_2,&\xi_3,\xi_4)+H_\beta^2(\xi_1,\xi_2,\xi_4,\xi_3)
+H_\beta^3(\xi_2,\xi_1,\xi_3,\xi_4)+H_\beta^4(\xi_2,\xi_1,\xi_4,\xi_3)\\
&-H_\beta^1(-\xi_3,-\xi_4,-\xi_1,-\xi_2)-H_\beta^2(-\xi_3,-\xi_4,-\xi_2,-\xi_1)\\
&-H_\beta^3(-\xi_4,-\xi_3,-\xi_1,-\xi_2)-H_\beta^4(-\xi_4,-\xi_3,-\xi_2,-\xi_1).
\end{split}
\end{equation}
The identities \eqref{sus16.4} (with $F_\beta^{23}\equiv 0$) follow from the formulas \eqref{verf1*}, \eqref{sus16.7}, and \eqref{sus16.9} (recall also that $m$ is even). The differential inequalities in the second line of \eqref{sus16.5} follow from the differential inequalities \eqref{sus16.8} satisfied by the functions $H_\beta^j$. The identity $(F_\beta^j\varphi_{\underline{k}}^{\underline{\iota}})(\rho,\sigma,-\sigma,-\rho)=0$ holds because $\varphi_{\underline{k}}^{\underline{\iota}}(\rho,\sigma,-\sigma,-\rho)=0$ in our case. The identity $(F_\beta^j\varphi_{\underline{k}}^{\underline{\iota}})(\rho,\sigma,-\rho,-\sigma)=0$ holds because $F_\beta^j(\rho,\sigma,-\rho,-\sigma)=0$ if $\rho\in[2^{l-6},2^{l+2}], \sigma\in[2^{n-6},2^{n+2}]$, due to the definition \eqref{sus16.10}. The point is that even though the functions $H_\beta^1,H_\beta^2,H_{\beta}^3,H_\beta^4$ are not related, this crucial cancellation is already present for each of the four pairs of terms involving each of these functions. This completes the proof of the representation \eqref{sus16.5} in the case $l-n\geq 20$.

Assume now that $l-n\leq 19$. Using again Definition \ref{defA} and \eqref{sus16.2} we can write
\begin{equation}\label{sus16.17}
m(\rho)=\sum_{\beta\in\{0,1,3/2\}}[|\rho_2+\rho_3|^\beta G_\beta^3(\rho)+|\rho_2+\rho_4|^\beta G_\beta^4(\rho)]\end{equation}
for any $\rho=(\rho_1,\rho_2,\rho_3,\rho_4)\in\mathbb{H}^3_\infty$ satisfying $\rho_1,\rho_2,-\rho_3,-\rho_4\in[2^{l-30},2^{l+2}]$. The functions $G_\beta^3,G_\beta^4:\R^4\to\C$ satisfy the differential inequalities 
\begin{equation}\label{sus16.18}
\big|\rho^\alpha\cdot(\partial^\al_{\rho}G^j_{\beta})(\rho)\big|\lesssim 2^{(5/2-\beta)l} 
  2^{2N_0\max(l,0)}, 
\end{equation}
for $\beta\in\{0,1,3/2\}$, $j\in\{3,4\}$, $|\alpha|\leq 6$, and $\rho\in\mathbb{H}^3_\infty$ satisfying $\rho_1,\rho_2,-\rho_3,-\rho_4\in[2^{l-30},2^{l+2}]$.
This representation follows from the main formula \eqref{A2.0}, by summing over $k_1,k_2,k_3,k_4\in [l-32,l+4]$.

Then we define, for $\xi=(\xi_1,\xi_2,\xi_3,\xi_4)$, with $\xi_1,\xi_2,-\xi_3,-\xi_4\in[2^{l-30},2^{l+2}]$,
\begin{equation}\label{sus16.20}
\begin{split}
F_\beta^{23}(\xi_1,\xi_2,\xi_3,\xi_4):=G_\beta^3(\xi_1,\xi_2,&\xi_3,\xi_4)+G_\beta^4(\xi_1,\xi_2,\xi_4,\xi_3)
+G_\beta^4(\xi_2,\xi_1,\xi_3,\xi_4)+G_\beta^3(\xi_2,\xi_1,\xi_4,\xi_3)\\
&-G_\beta^3(-\xi_3,-\xi_4,-\xi_1,-\xi_2)-G_\beta^4(-\xi_3,-\xi_4,-\xi_2,-\xi_1)\\
&-G_\beta^4(-\xi_4,-\xi_3,-\xi_1,-\xi_2)-G_\beta^3(-\xi_4,-\xi_3,-\xi_2,-\xi_1),
\end{split}
\end{equation}
\begin{equation}\label{sus16.21}
\begin{split}
F_\beta^{24}(\xi_1,\xi_2,\xi_3,\xi_4):=G_\beta^4(\xi_1,\xi_2,&\xi_3,\xi_4)+G_\beta^3(\xi_1,\xi_2,\xi_4,\xi_3)
+G_\beta^3(\xi_2,\xi_1,\xi_3,\xi_4)+G_\beta^4(\xi_2,\xi_1,\xi_4,\xi_3)\\
&-G_\beta^4(-\xi_3,-\xi_4,-\xi_1,-\xi_2)-G_\beta^3(-\xi_3,-\xi_4,-\xi_2,-\xi_1)\\
&-G_\beta^3(-\xi_4,-\xi_3,-\xi_1,-\xi_2)-G_\beta^4(-\xi_4,-\xi_3,-\xi_2,-\xi_1).
\end{split}
\end{equation}
The identities \eqref{sus16.4} follow from the formulas \eqref{verf1*} and \eqref{sus16.20}--\eqref{sus16.21} (recall also that $m$ is even). The differential inequalities in the second line of \eqref{sus16.5} follow from the differential inequalities \eqref{sus16.18} satisfied by the functions $G_\beta^j$. The key cancellations in the first line of \eqref{sus16.5} follow from the formulas \eqref{sus16.20}--\eqref{sus16.21}, by canceling suitable pairs of terms.

\smallskip
{\bf{Substep 3.2.}} We can now prove the bounds \eqref{sus8}--\eqref{sus9}. Notice that if $f:\R^3\to\C$ is a $C^2$ function satisfying $f(x,u,0)=f(x,0,v)=0$ for any $x,u,v\in\mathbb{R}$ then 
\begin{equation}\label{uyt35'}
\begin{split}
f(x,u,v) & = uv\cdot g(x,u,v),\\
g(x,u,v) & := \int_0^1\int_0^1(\partial^2_{uv}f)(x,\mu u, \nu v)\,d\mu d\nu,
\end{split}
\end{equation}
for any $x,a,b\in\mathbb{R}$. We apply this identity to the functions
\begin{equation}\label{uyt34.5'}
f(x,u,v):=(F_\beta^j\varphi_{\underline{k}}^{\underline{\iota}})(x+u+v,x,-x-u,-x-v),
\end{equation}
for $\beta\in\{0,1,3/2\}$ and $j\in\{23,24\}$. The formula \eqref{sus16.4} shows that 
\begin{equation}\label{sus16.24}
\begin{split}
(M\varphi_{\underline{k}}^{\underline\iota})&(\xi)=\widetilde{\varphi}_{\underline{k}}^{\underline\iota}(\xi)(\xi_2+\xi_3)(\xi_2+\xi_4)\\
&\times\sum_{\beta\in\{0,1,3/2\}}\big[|\xi_2+\xi_3|^\beta \widetilde{F}_\beta^{23}(\xi_2,\xi_3,\xi_4)+|\xi_2+\xi_4|^\beta \widetilde{F}_\beta^{24}(\xi_2,\xi_3,\xi_4)\big],
\end{split}
\end{equation}
where $\xi\in\H^3_\infty$ and $\widetilde{\varphi}_{\underline{k}}^{\underline\iota}(\xi)$ is defined as in \eqref{phi+-}. In view of \eqref{sus16.5} and \eqref{uyt35'} $\widetilde{F}_\beta^{23}\equiv 0$ if $l\geq n+20$, and the functions $\widetilde{F}_\beta^{23}, \widetilde{F}_\beta^{24}:\R^3\to\C$ satisfy the differential inequalities
\begin{equation}\label{sus16.25}
\big| 2^{\alpha_2 k_2+\alpha_3 k_3+\alpha_4k_4}\partial_{\xi_2}^{\alpha_2}\partial_{\xi_3}^{\alpha_3}\partial_{\xi_4}^{\alpha_4}\widetilde{F}^j_{\beta}(\xi)\varphi_{\underline{k}}^{\underline{\iota}}(\xi)\big]\big|
  \lesssim 2^{(3/2-\beta)n}2^{-l}2^{2N_0\max(l,0)},
  \end{equation}
if $j\in\{23,24\}$, $\alpha_2+\alpha_3+\alpha_4\leq 4$, and $\xi=(\xi_2,\xi_3,\xi_4)\in\R^3$.  

The bounds \eqref{sus8}--\eqref{sus9} follow from \eqref{sus16.24}--\eqref{sus16.25} and \eqref{sus16}. The point is that the singular factors $\xi_2+\xi_3$ and $\xi_2+\xi_4$ cancel, and the desired bounds follow using Lemma \ref{multiBound} (iv).

{\bf{Step 4.}} Finally we prove the bounds \eqref{sus8}--\eqref{sus9} in the case when 
\begin{equation}\label{sus18.1}
\underline{\iota}=(+,-,-,+).
\end{equation}

Notice that if $\xi_1,\xi_4\in[0,\infty), \xi_2,\xi_3\in(-\infty,0]$ and $\xi_1+\xi_2+\xi_3+\xi_4=0$ then we can write
\begin{equation}\label{sus18.3}
\begin{split}
\Phi_{++--}(\xi_1,\xi_2,\xi_3,\xi_4)&=\frac{\xi_1+\xi_3}{|\xi_1|^{1/2}+|\xi_3|^{1/2}}
-\frac{\xi_2+\xi_4}{|\xi_2|^{1/2}+|\xi_4|^{1/2}}
  \\
 & = -(\xi_2+\xi_4)\frac{|\xi_1|^{1/2}+|\xi_4|^{1/2}+|\xi_2|^{1/2}+|\xi_3|^{1/2}}{
 (|\xi_1|^{1/2}+|\xi_3|^{1/2})(|\xi_2|^{1/2}+|\xi_4|^{1/2})}.
\end{split}
\end{equation}
As before, we notice that if $|\widetilde{k}_3-\widetilde{k}_4|\geq 4$ then all the factors of the form $|\xi_i+\xi_j|$ are elliptically bounded away from $0$ for any $\xi\in\mathbb{H}^3_\infty$ in the support of $\varphi_{\underline{k}}^{\underline{\iota}}$, and the bounds \eqref{sus8}--\eqref{sus9} follow easily. In view of the symmetries \eqref{sus7} and recalling that $M$ is even, it remains to prove these bounds in the main case
\begin{equation}\label{sus18.2}
k_1,k_3\in[l-4,l],\qquad k_2,k_4\in[n-4,n],\qquad \text{ for some }n\leq l\in\Z.
\end{equation}

{\bf{Substep 4.1.}} As before, we show first that we can write
\begin{equation}\label{sus18.4}
(M\varphi_{\underline{k}}^{\underline{\iota}})(\xi)=\varphi_{\underline{k}}^{\underline{\iota}}(\xi)\sum_{\beta\in\{0,1,3/2\}}\big[|\xi_2+\xi_4|^\beta G_\beta^{24}(\xi)+|\xi_3+\xi_4|^\beta G_\beta^{34}(\xi)\big]
\end{equation}
for any $\xi\in\mathbb{H}^3_\infty$, for some functions $G_\beta^{24},G_\beta^{34}:\R^4\to\C$ satisfying
\begin{equation}\label{sus18.5}
\begin{split}
&(G_\beta^j\varphi_{\underline{k}}^{\underline{\iota}})(\rho,\sigma,-\rho,-\sigma)=0,\\
&\big|2^{\alpha\cdot \underline{k}}\partial^\al_{\xi}
  \big[G^j_{\beta}(\xi)\varphi_{\underline{k}}^{\underline\iota}(\xi)\big]\big|
  \lesssim 2^{(5/2-\beta)n}
  2^{2N_0\max(l,0)},
\end{split}
\end{equation}
for any $j\in\{24,34\}$, $\beta\in\{0,1,3/2\}$, $\rho,\sigma\in\R$, multi-index $\alpha=(\alpha_1,\alpha_2,\alpha_3,\alpha_4)$ with $|\alpha|\leq 6$, and $\xi\in\mathbb{H}^3_\infty$. Moreover, the functions $G_\beta^{34}$ are identically $0$ if $l\geq n+20$.

This is similar to the proof of \eqref{sus16.4}--\eqref{sus16.5}, slightly simpler because we only prove one cancellation condition in \eqref{sus18.5}. As in \eqref{sus16.7} and \eqref{sus16.9}, since $m$ is admissible, it follows from Definition \ref{defA} and \eqref{sus18.2} that we can write
\begin{equation}\label{sus18.7}
\begin{split}
&m(\rho)=\sum_{\beta\in\{0,1,3/2\}}[|\rho_2+\rho_4|^\beta X_\beta^{1}(\rho)+|\rho_3+\rho_4|^\beta Y_\beta^{1}(\rho)]\,\,\text{ for any }\rho\in\mathbb{H}^3_\infty\cap B^{+,-,-,+}_{l,n,l,n},\\
&m(\rho)=\sum_{\beta\in\{0,1,3/2\}}|\rho_2+\rho_3|^\beta X_\beta^{2}(\rho)+|\rho_3+\rho_4|^\beta Y_\beta^{2}(\rho)]\,\,\text{ for any }\rho\in\mathbb{H}^3_\infty\cap B^{+,-,+,-}_{l,n,n,l},\\
&m(\rho)=\sum_{\beta\in\{0,1,3/2\}}|\rho_1+\rho_4|^\beta X_\beta^{3}(\rho)+|\rho_3+\rho_4|^\beta Y_\beta^{3}(\rho)]\,\,\text{ for any }\rho\in\mathbb{H}^3_\infty\cap B^{-,+,-,+}_{n,l,l,n},\\
&m(\rho)=\sum_{\beta\in\{0,1,3/2\}}|\rho_1+\rho_3|^\beta X_\beta^{4}(\rho)+|\rho_3+\rho_4|^\beta Y_\beta^{4}(\rho)]\,\,\text{ for any }\rho\in\mathbb{H}^3_\infty\cap B^{-,+,+,-}_{n,l,n,l},
\end{split}
\end{equation}
for some function $X^j_\beta,\, Y_\beta^j:\R^4\to\C$ satisfying, for $\beta\in\{0,1,3/2\}$,
\begin{equation}\label{sus18.8}
\big|\rho^\alpha\cdot(\partial^\al_{\rho}X^j_{\beta})(\rho)\big|+\big|\rho^\alpha\cdot(\partial^\al_{\rho}Y^j_{\beta})(\rho)\big|\lesssim 2^{(5/2-\beta)n} 
  2^{2N_0\max(l,0)}, \qquad \rho\in\mathbb{H}^3_\infty\cap B^j,\,|\alpha|\leq 6,
\end{equation}
where $B^j$ are the suitable balls as in \eqref{sus18.7} where the functions $X^j_\beta, Y^j_\beta$ are defined. These representation follow from the main formula \eqref{A2.0}, by summing over indices $k_1,k_3\in [l-8,l+4]$ and $k_2,k_4\in[n-8,n+4]$. Moreover, the functions $Y^j_\beta$ are trivial if $|l-n|\geq 6$.

Then we define, for $\xi=(\xi_1,\xi_2,\xi_3,\xi_4)\in B^{+,-,-,+}_{l,n,l,n}$,
\begin{equation*}
\begin{split}
G_\beta^{24}(\xi_1,\xi_2,\xi_3,\xi_4):=X_\beta^1(\xi_1,\xi_2,&\xi_3,\xi_4)+X_\beta^2(\xi_1,\xi_2,\xi_4,\xi_3)
+X_\beta^3(\xi_2,\xi_1,\xi_3,\xi_4)+X_\beta^4(\xi_2,\xi_1,\xi_4,\xi_3)\\
&-X_\beta^1(-\xi_3,-\xi_4,-\xi_1,-\xi_2)-X_\beta^2(-\xi_3,-\xi_4,-\xi_2,-\xi_1)\\
&-X_\beta^3(-\xi_4,-\xi_3,-\xi_1,-\xi_2)-X_\beta^4(-\xi_4,-\xi_3,-\xi_2,-\xi_1)
\end{split}
\end{equation*}
\begin{equation*}
\begin{split}
G_\beta^{34}(\xi_1,\xi_2,\xi_3,\xi_4):=Y_\beta^1(\xi_1,\xi_2,&\xi_3,\xi_4)+Y_\beta^2(\xi_1,\xi_2,\xi_4,\xi_3)
+Y_\beta^3(\xi_2,\xi_1,\xi_3,\xi_4)+Y_\beta^4(\xi_2,\xi_1,\xi_4,\xi_3)\\
&-Y_\beta^1(-\xi_3,-\xi_4,-\xi_1,-\xi_2)-Y_\beta^2(-\xi_3,-\xi_4,-\xi_2,-\xi_1)\\
&-Y_\beta^3(-\xi_4,-\xi_3,-\xi_1,-\xi_2)-Y_\beta^4(-\xi_4,-\xi_3,-\xi_2,-\xi_1).
\end{split}
\end{equation*}
The identities \eqref{sus18.4} follow from the formulas \eqref{verf1*} and \eqref{sus18.7} (recall also that $m$ is even). The differential inequalities in the second line of \eqref{sus18.5} follow from \eqref{sus18.8} and the identities $(G_\beta^j\varphi_{\underline{k}}^{\underline{\iota}})(\rho,\sigma,-\rho,-\sigma)=0$ follow from the definitions of the functions $G^{24}_\beta$ and $G^{34}_\beta$ above.

\smallskip
{\bf{Substep 4.2.}} We can now prove that bounds \eqref{sus8}--\eqref{sus9}. As in Substep 3.2 above, we notice that if $f:\R^3\to\C$ is a $C^2$ function satisfying $f(x,y,0)=0$ for any $x,y\in\mathbb{R}$ then 
\begin{equation*}
f(x,y,u)= u\cdot\int_0^1(\partial_uf)(x,y,\mu u)\,d\mu.
\end{equation*}
for any $x,y,u\in\mathbb{R}$. We apply this identity to the functions
\begin{equation}\label{sus18.20}
f(x,y,u):=(G_\beta^j\varphi_{\underline{k}}^{\underline{\iota}})(x+u,y,-x,-y-u),
\end{equation}
for $\beta\in\{0,1,3/2\}$ and $j\in\{24,34\}$. The formula \eqref{sus18.4} shows that 
\begin{equation}\label{sus18.24}
\begin{split}
(M\varphi_{\underline{k}}^{\underline\iota})&(\xi)=\widetilde{\varphi}_{\underline{k}}^{\underline\iota}(\xi)(\xi_2+\xi_4)\\
&\times\sum_{\beta\in\{0,1,3/2\}}\big[|\xi_2+\xi_4|^\beta \widetilde{G}_\beta^{24}(\xi_2,\xi_3,\xi_4)+|\xi_3+\xi_4|^\beta \widetilde{G}_\beta^{34}(\xi_2,\xi_3,\xi_4)\big],
\end{split}
\end{equation}
for $\xi=(\xi_1,\xi_2,\xi_3,\xi_4)\in\H^3_\infty$. In view of \eqref{sus18.5} $\widetilde{G}_\beta^{34}\equiv 0$ if $l\geq n+20$, and the functions $\widetilde{G}_\beta^{24}, \widetilde{G}_\beta^{34}:\R^3\to\C$ satisfy the differential inequalities
\begin{equation}\label{sus18.25}
\big| 2^{\alpha_2 k_2+\alpha_3 k_3+\alpha_4k_4}\partial_{\xi_2}^{\alpha_2}\partial_{\xi_3}^{\alpha_3}\partial_{\xi_4}^{\alpha_4}\widetilde{G}^j_{\beta}(\xi)\varphi_{\underline{k}}^{\underline{\iota}}(\xi)\big]\big|
  \lesssim 2^{(3/2-\beta)n}2^{2N_0\max(l,0)},
  \end{equation}
if $j\in\{24,34\}$, $\alpha_2+\alpha_3+\alpha_4\leq 5$, and $\xi=(\xi_2,\xi_3,\xi_4)\in\R^3$.  

As before, the bounds \eqref{sus8}--\eqref{sus9} now follow from \eqref{sus18.24}--\eqref{sus18.25} and \eqref{sus18.3} by cancelling the singular factor $\xi_2+\xi_4$ and using Lemma \ref{multiBound} (iv). This completes the proof of the lemma.
\end{proof}

\medskip
\section{The admissibility condition and proof of Theorem \ref{IncremBound}}\label{secadmver}

In this section we prove that the symbols $\widetilde{K}_1^j$, $j\in\{1,\ldots,9\}$, and $\widetilde{K}_2^j$, $j\in\{1,\ldots,4\}$ obtained by symmetrization in the first three variables of the symbols in Lemma \ref{quartic3} are all admissible. Then we show how to complete the proof of Theorem \ref{IncremBound}, provided that we can verify the vanishing of the fully symmetrized symbol $SK_1-SK_2$ on the set of Benjamin-Feir resonances.

\subsection{Admissibility} We will prove the following lemma:

\begin{lemma}\label{pus1} 
With $K_i^j$ as in \eqref{ku11}--\eqref{ku26}, we define the symmetrized symbols $\widetilde{K}_i^j$ according to the formula \eqref{Ksymm},
\begin{equation}\label{ku11'}
\widetilde{K}_i^j(\xi,\eta,\rho,\sigma):=\frac{1}{3}\big[K_i^j(\xi,\eta,\rho,\sigma)+K_i^j(\eta,\rho,\xi,\sigma)+K_i^j(\rho,\xi,\eta,\sigma)\big].
\end{equation}
Then the symbols $\widetilde{K}_1^j$, $j\in\{1,\ldots,9\}$, and $\widetilde{K}_2^j$, $j\in\{1,\ldots,4\}$, are all admissible.
\end{lemma}

The rest of this subsection is concerned with the proof of this lemma. To prove admissibility, we need to pay attention to three main aspects:
(1) the loss of derivatives, (2) the presence of singular factors of the form $|\xi_l+\xi_k|$, $l\neq k$, and (3) the identification of a lowest frequency factor. In some cases we are able to prove that the symbols $K_i^j$ are already admissible, without symmetrization, but in some cases the symbols become admissible only after symmetrization. 

Recall Definition \ref{defMNab} of the regular symbol classes $\mathcal{M}^b$. For any $x\in\R$ let $\iota(x)\in\{-1,1\}$ denote its sign, $\iota(x):=x/|x|$. For simplicity of notation, for a given $\underline{k}\in\Z^4$ and symbols $m,m':\H^3_\infty\to\C$ we will sometimes write
\begin{equation}\label{slov1.1}
\varphi_{\underline{k}}m\simeq \varphi_{\underline{k}}m'
\end{equation}
if $\varphi_{\underline{k}}\cdot (m-m')$ satisfies the properties \eqref{A2.0}--\eqref{A2.1} in Definition \ref{defA}.

\smallskip
{\bf{The symbol $\widetilde{K}_1^1$.}} We examine the formula \eqref{ku11} and decompose
\begin{equation*}
K_1^1=K_1^{1,1}+K_1^{1,2}+K_1^{1,3}+K_1^{1,4}+K_1^{1,5},
\end{equation*}
\begin{equation}\label{slov1}
\begin{split}
&K^{1,1}_1(\xi,\eta,\rho,\sigma):=m_{N_0}^2(\xi)|\sigma|^{1/2}\big\{(|\xi||\rho+\sigma|+\xi(\rho+\sigma))+(|\xi||\eta+\sigma|+\xi(\eta+\sigma))\\
&\qquad-(|\xi||\sigma|+\xi\sigma)-\chi(\rho+\sigma,\eta)\big[|\xi||\rho+\sigma|-|\xi||\sigma|\big]-\chi(\eta+\sigma,\rho)\big[|\xi||\eta+\sigma|-|\xi||\sigma|\big]\big\},\\
&K^{1,2}_1(\xi,\eta,\rho,\sigma):=-m_{N_0}^2(\xi)|\sigma|^{1/2}\big\{\chi(\sigma,\rho)\widetilde{\chi}(\rho+\sigma,\eta)(|\xi||\rho+\sigma|+\xi(\rho+\sigma))\\
&\qquad\qquad\qquad\qquad\qquad\qquad\qquad+\chi(\sigma,\eta)\widetilde{\chi}(\eta+\sigma,\rho)(|\xi||\eta+\sigma|+\xi(\eta+\sigma))\big\},\\
&K^{1,3}_1(\xi,\eta,\rho,\sigma):=-m_{N_0}^2(\xi)\xi|\sigma|^{1/2}\big\{\chi(\sigma,\rho)\chi(\rho+\sigma,\eta)(\sigma+\rho)+\chi(\sigma,\eta)\chi(\eta+\sigma,\rho)(\sigma+\eta)\big\},\\
&K^{1,4}_1(\xi,\eta,\rho,\sigma):=-\frac{m_{N_0}^2(\xi)}{2|\sigma|^{1/2}}[|\rho|\chi(\rho,\eta)+|\eta|\chi(\eta,\rho)]\widetilde{\chi}(\sigma,\eta+\rho)(|\xi||\sigma|+\xi\sigma),\\
&K^{1,5}_1(\xi,\eta,\rho,\sigma):=-\frac{m_{N_0}^2(\xi)}{2|\sigma|^{1/2}}[|\rho|\chi(\rho,\eta)+|\eta|\chi(\eta,\rho)]\chi(\sigma,\eta+\rho)\xi\sigma.
\end{split}
\end{equation}

It is easy to see that the symbols $K^{1,2}_1$ and $K^{1,4}_1$ are already admissible, more precisely
\begin{equation}\label{slov2}
K^{1,2}_1,\, K^{1,4}_1\in\mathcal{M}^{5/2}.
\end{equation}

The symbols $K_1^{1,1}$, $K_1^{1,3}$, and $K_1^{1,5}$ require symmetrization, to avoid derivative loss. We start with the easier symbol $K^{1,5}_1$. If $\underline{k}=(k_1,k_2,k_3,k_4)\in\Z^4$ is a balanced $4$-tuple (meaning $\widetilde{k}_1-\widetilde{k}_3\leq 9$ as in the proof of Lemma \ref{quartic2}) then there is no derivative loss and it is easy to see that the symbol $\varphi_{\underline{k}}\cdot K^{1,5}_1$ itself satisfies the regular bounds
\begin{equation}\label{slov3}
\|K^{1,5}_1\cdot\varphi_{\underline{k}}\|_{\mathcal{HM}_8}\lesssim 2^{\widetilde{k}_4/2}2^{2\widetilde{k}_3}2^{2N_0\max(\widetilde{k}_1,0)}.
\end{equation}
By permuting the variables we see that similar bounds hold on the symmetrized symbol $\widetilde{K}^{1,5}_1$. 

On the other hand, if $\underline{k}$ is an imbalanced $4$-tuple (meaning $\widetilde{k}_1-\widetilde{k}_3\geq 10$) then we may assume by symmetry that $k_1\geq k_2\geq k_3$. Clearly $\widetilde{K}^{1,5}_1\cdot\varphi_{\underline{k}}\equiv 0$ on $\H^3_\infty$ if $k_4\geq \widetilde{k}_1-5$. Otherwise, if $k_4\leq \widetilde{k}_1-5$ then we may assume that $k_1,k_2$ are large and $k_3,k_4$ are small,
\begin{equation}\label{slov4}
k_1,k_2\in[\widetilde{k}_1-4,\widetilde{k}_1],\qquad k_3,k_4\leq \widetilde{k}_3\leq\widetilde{k}_1-10.
\end{equation}
In this case we replace $\xi$ by $-\eta-\rho-\sigma$ and use the identity $\chi(\sigma,-\sigma-x)=\chi(\sigma,x)$ to write
\begin{equation*}
\begin{split}
\widetilde{K}^{1,5}_1&(\xi,\eta,\rho,\sigma)=\frac{\iota(\sigma)|\sigma|^{1/2}}{6}m_{N_0}^2(\eta+\rho+\sigma)(\rho+\sigma)|\rho|\chi(\rho,\eta)\chi(\sigma,\eta+\rho)\\
&-\frac{\iota(\sigma)|\sigma|^{1/2}}{6}\eta\big[-m_{N_0}^2(\eta+\rho+\sigma)|\rho|\chi(\rho,\eta)\chi(\sigma,\eta+\rho)+m_{N_0}^2(\eta)|\rho|\chi(\rho,\eta+\sigma)\chi(\sigma,\eta)\big],
\end{split}
\end{equation*}
for $(\xi,\eta,\rho,\sigma)\in\H^3_\infty$ in the support of the function $\varphi_{\underline{k}}$. The point is that the symmetrization in $\xi$ and $\eta$ avoids the derivative loss, and the desired bounds
\begin{equation}\label{slov5}
\|\widetilde{K}^{1,5}_1\cdot\varphi_{\underline{k}}\|_{\mathcal{HM}_8}\lesssim 2^{\widetilde{k}_4/2}2^{2\widetilde{k}_3}2^{2N_0\max(\widetilde{k}_1,0)}
\end{equation}
follow in this case as well. 

The analysis of the symbol $\widetilde{K}^{1,3}_1$ is similar. If $\underline{k}=(k_1,k_2,k_3,k_4)\in\Z^4$ is a balanced $4$-tuple then there is no derivative loss and the symbol $\varphi_{\underline{k}}\cdot K^{1,3}_1$ itself satisfies regular bounds similar to \eqref{slov3}. On the other hand, if $\underline{k}$ is imbalanced then we may assume, by symmetry, the frequency configuration \eqref{slov4}, and symmetrize in $\xi$ and $\eta$ as before, by substituting $\xi=-\eta-\rho-\sigma$. 

The analysis of the symbol $\widetilde{K}^{1,1}_1$ is also similar, but slightly longer. We rewrite first
\begin{equation}\label{slov5.5}
\begin{split}
K^{1,1}_1(\xi,\eta,\rho,\sigma)&:=m_{N_0}^2(\xi)|\sigma|^{1/2}|\xi|\big\{(|\rho+\sigma|+|\eta+\sigma|-|\xi|-|\sigma|)\\
&-\chi(\rho+\sigma,\eta)\big[|\rho+\sigma|-|\sigma|\big]-\chi(\eta+\sigma,\rho)\big[|\eta+\sigma|-|\sigma|\big]\big\},
\end{split}
\end{equation}
for $(\xi,\eta,\rho,\sigma)\in\H^3_\infty$. If $\underline{k}=(k_1,k_2,k_3,k_4)\in\Z^4$ is a balanced $4$-tuple then there is no derivative loss and no singular factors, and the symbol $\varphi_{\underline{k}}\cdot K^{1,3}_1$ satisfies regular bounds similar to \eqref{slov3}.

If $\underline{k}$ is an imbalanced $4$-tuple (meaning $\widetilde{k}_1-\widetilde{k}_3\geq 10$) then we may assume that $k_1\geq k_2\geq k_3$, and consider two cases. If $k_4$ is one of the larger frequency parameters, $k_4\geq\widetilde{k}_1-5$ then
\begin{equation}\label{slov6}
k_1,k_4\in[\widetilde{k}_1-4,\widetilde{k}_1],\qquad k_2,k_3\leq \widetilde{k}_3\leq\widetilde{k}_1-10.
\end{equation}
and we use the identity \eqref{slov5.5} to calculate, for $(\xi,\eta,\rho,\sigma)\in \H^3_\infty$ in the support of $\varphi_{\underline{k}}$,
\begin{equation*}
\begin{split}
\widetilde{K}^{1,1}_1(\xi,\eta,&\rho,\sigma)=\frac{1}{3}m_{N_0}^2(\eta)|\sigma|^{1/2}|\eta|\big\{(|\eta+\rho|+|\rho+\sigma|-|\eta|-|\sigma|)-\chi(-\eta-\rho,\rho)\big(|\eta+\rho|-|\sigma|\big)\big\}\\
&+\frac{1}{3}m_{N_0}^2(\rho)|\sigma|^{1/2}|\rho|\big\{(|\eta+\rho|+|\eta+\sigma|-|\rho|-|\sigma|)-\chi(-\eta-\rho,\eta)\big(|\eta+\rho|-|\sigma|\big)\big\}.
\end{split}
\end{equation*}
The desired bounds (similar to \eqref{slov5}) follow in this case, since the multiplier $m_{N_0}$ is applied to the lower frequencies. On the other hand, if $k_4\leq\widetilde{k}_1-5$ then we may assume that the frequency parameters $k_1,k_2,k_3,k_3$ are as in \eqref{slov4}. Then we write, using \eqref{slov5.5},
\begin{equation*}
\begin{split}
\widetilde{K}^{1,1}_1(\xi,\eta,\rho,\sigma)&=\frac{1}{3}\Big\{m_{N_0}^2(\xi)|\sigma|^{1/2}|\xi|\big\{\iota(\xi)\rho+(1-\chi(\rho+\sigma,\eta))\big[|\rho+\sigma|-|\sigma|\big]\big\}\\
&+m_{N_0}^2(\eta)|\sigma|^{1/2}|\eta|\big\{\iota(\eta)\rho+(1-\chi(\rho+\sigma,\xi))\big[|\rho+\sigma|-|\sigma|\big]\big\}\\
&+m_{N_0}^2(\rho)|\sigma|^{1/2}|\rho|(|\xi+\sigma|+|\eta+\sigma|-|\rho|-|\sigma|)\Big\},
\end{split}
\end{equation*}
for $(\xi,\eta,\rho,\sigma)\in \H^3_\infty$ in the support of $\varphi_{\underline{k}}$ (using the identity $|\eta+\sigma|-|\xi|=|\xi+\rho|-|\xi|=\iota(\xi+\rho)(\xi+\rho)-\iota(\xi)\xi=\iota(\xi)\rho$ and a similar identity with $\xi$ and $\eta$ interchanged). The desired bounds (similar to \eqref{slov5}) follow by symmetrizing two terms in $\xi$ and $\eta$ to avoid derivative loss.

\smallskip
{\bf{The symbol $\widetilde{K}_1^2$.}} We will show that the symbol $K_1^2$ is admissible without symmetrization, 
\begin{equation}\label{slov8}
\|K^2_1\cdot\varphi_{\underline{k}}\|_{\mathcal{HM}_8}\lesssim 2^{\widetilde{k}_4/2}2^{2\widetilde{k}_3}2^{2N_0\max(\widetilde{k}_1,0)}
\end{equation}
for any $\underline{k}\in\Z$. By symmetry, we may assume that $k_2\geq k_3$. These bounds follow easily from the formula \eqref{ku12} if $\widetilde{k}_3\geq \widetilde{k}_1-40$, since in this case there is no derivative loss and the factors $|\rho+\sigma|=|\xi+\eta|$ and $|\eta+\sigma|=|\xi+\rho|$ are non-singular in the support of the symbol. 

On the other hand, if $\widetilde{k}_3\leq \widetilde{k}_1-40$ then we have two nontrivial cases. Either 
\begin{equation*}
k_2,k_4\geq\widetilde{k}_1-4,\qquad k_1,k_3\leq\widetilde{k}_3\leq \widetilde{k}_1-40,
\end{equation*}
 in which case the formula \eqref{ku12} gives
 \begin{equation*}
K_1^2(\xi,\eta,\rho,\sigma) =- \frac{m_{N_0}(\eta+\si) m_{N_0}(\rho)
  |\xi|\chi(\xi,\rho)\big\{\si(\xi+\rho)(1-\chi(\eta,\si)) -|\si||\xi+\rho|\widetilde{\chi}(\eta,\si)\big\}}{|\sigma|^{1/2}}
\end{equation*}
for $(\xi,\eta,\rho,\sigma)\in\H^3_\infty$ in the support of the function $\varphi_{\underline{k}}$. There is no derivative loss in this case (since one of the factors $m_{N_0}(\rho)$ applies to one of the low frequencies), and the bounds \eqref{slov8} follow easily. The remaining possibility is
\begin{equation*}
k_2,k_3\geq\widetilde{k}_1-4,\qquad k_1,k_4\leq\widetilde{k}_3\leq \widetilde{k}_1-40,
\end{equation*}
 in which case the formula \eqref{ku12} gives
 \begin{equation*}
K_1^2(\xi,\eta,\rho,\sigma) =- \frac{m_{N_0}(\rho+\si) m_{N_0}(\eta)
  |\xi|\si(\rho+\si)}{|\sigma|^{1/2}}- \frac{m_{N_0}(\eta+\si) m_{N_0}(\rho)
  |\xi|\si(\eta+\si)}{|\sigma|^{1/2}},
\end{equation*}
for $(\xi,\eta,\rho,\sigma)\in\H^3_\infty$ in the support of $\varphi_{\underline{k}}$. The bounds \eqref{slov8} follow in this case as well, by symmetrization in the variables $\eta$ and $\rho$ (replacing $\eta$ with $-\rho-\sigma-\xi$ in the formula above).

\smallskip
{\bf{The symbol $\widetilde{K}_1^3$.}} We examine the formula \eqref{ku14} and decompose $K_1^3=K_1^{3,1}+K_1^{3,2}$,
\begin{equation}\label{slov9}
\begin{split}
K_1^{3,1}(\xi,\eta,\rho,\sigma)&:=\frac{m_{N_0}^2(\eta+\rho)[|\eta|\chi(\eta,\rho)+|\rho|\chi(\rho,\eta)]}{2|\si|^{1/2}}(-(\eta+\rho)\si-|\eta+\rho||\si|)\widetilde{\chi}(\si,\xi),\\
K_1^{3,2}(\xi,\eta,\rho,\sigma)&:= \frac{m_{N_0}^2(\eta+\rho)[|\eta|\chi(\eta,\rho)+|\rho|\chi(\rho,\eta)]}{2|\si|^{1/2}}(\xi+\si)\si\chi(\si,\xi),
\end{split}
\end{equation}
since $\xi+\sigma=-(\eta+\rho)$ along $\H^3_\infty$. It is easy to see that there is no derivative loss in $K_1^{3,1}$ and, in fact, $K_1^{3,1}$ is regular, $K_1^{3,1} \in \mathcal{M}^{5/2}$.

The symbol $K_1^{3,2}$ requires symmetrization, to avoid derivative loss. The analysis is similar to some of the symbols discussed earlier. If $\underline{k}=(k_1,k_2,k_3,k_4)\in\Z^4$ is a balanced $4$-tuple then there is no derivative loss and the symbol $\varphi_{\underline{k}}\cdot K^{3,2}_1$ itself satisfies regular bounds similar to \eqref{slov3}. On the other hand, if $\underline{k}$ is imbalanced then we may assume, by symmetry, the frequency configuration \eqref{slov4}, and symmetrize in $\xi$ and $\eta$ as before, by substituting $\xi=-\eta-\rho-\sigma$. 

\smallskip
{\bf{The symbol $\widetilde{K}_1^4$.}} 
We do not need symmetrization in this case, and we will prove that $K_1^4$ is already admissible. Recall the expression for $b$ from \eqref{ite4}, which contains the possibly singular factor $|\xi+\sigma|=|\eta+\rho|$. There is no derivative loss in this case and $m_{N_0}$ satisfies homogeneous bounds of the form $|\partial_\mu^\alpha m_{N_0}(\mu)|\lesssim |\mu|^{N_0-|\alpha|}$ for $\mu\in\R$ and $|\alpha|\leq 8$ (see definition \eqref{m_choice}). Therefore, $(\varphi_{\underline{k}}K_1^4)\simeq 0$ for any $\underline{k}\in\Z^4$ (see the notation \eqref{slov1.1}), with a possibly singular factor $|\xi+\sigma|$ that could be present if if $|k_1-k_4|\leq 4$ and $|k_2-k_3|\leq 4$.

\smallskip
{\bf{The symbol $\widetilde{K}_1^5$.}}
By inspection of the formula \eqref{ku110} we see directly that $K_1^5 \in \mathcal{M}^{5/2}$.

\smallskip
{\bf{The symbol $\widetilde{K}_1^6$.}} We inspect the formula \eqref{ku111} and notice that there is no derivative loss due to the difference $m_{N_0}(\eta+\rho)-m_{N_0}(\si)$. It follows easily that $K_1^6 \in \mathcal{M}^{5/2}$.

\smallskip
{\bf{The symbol $\widetilde{K}_1^7$.}}
We show that the symbol $K_1^7$ itself is admissible. Recall the formula for $a_2^{rr}(\eta,\rho)$ from \eqref{ite2}. By symmetry, to analyze the localized multiplier $\varphi_{\underline{k}}K_1^7$ we may assume that $k_2\geq k_3$. We have two cases. If $k_1\geq k_2+6$ then $k_1$ and $k_4$ are the large frequency parameters, and we have a suitable decomposition with a possibly singular factor $|\eta+\rho|$. On the other hand, if $k_1\leq k_2+5$ then we replace $\eta+\rho$ with $-\xi-\sigma$ and notice that $\varphi_{\underline{k}}K_1^7$ is nontrivial only if $|k_1-k_4|\leq 2$. We can then see easily that $(\varphi_{\underline{k}}K_1^7)\simeq 0$ for any $\underline{k}\in\Z^4$, with a possibly singular factor $|\xi+\sigma|$.

\smallskip
{\bf{The symbol $\widetilde{K}_1^8$.}}
We examine the formula \eqref{ku113}, decompose $K_1^8=K_1^{8,1}+K_1^{8,2}$,
\begin{align}\label{ku113''}
\begin{split}
& K_1^{8,1}(\xi,\eta,\rho,\si):= 6n_2^{rr}(\eta,\rho)A'(\xi,\sigma,\eta+\rho)|\xi|^{1/2}(|\eta+\rho|+|\sigma|-|\xi|), 
\\
& K_1^{8,2}(\xi,\eta,\rho,\si):= 6n_2^{ri}(\xi,\sigma)A'(\eta,\rho,\xi+\sigma)|\rho|^{1/2}|\eta|^{1/2}|\xi+\sigma|^{1/2},
\end{split}
\end{align}
and recall the formulas \eqref{Nh}--\eqref{Nrr2} for $n_2^{ri}$ and $n_2^{rr}$. The formulas \eqref{eni21} and \eqref{eni42} show that
\begin{equation}\label{sus30}
\begin{split}
A'(x,y,z)&=\frac{\max(|x|,|y|,|z|)}{-12|x|^{1/2}|y|^{1/2}|z|^{1/2}}\big\{m_{N_0}^2(x)[\iota(y)\iota(z)+\widetilde{\chi}(y,z)]\\
&\qquad\qquad+m_{N_0}^2(y)[\iota(z)\iota(x)+\widetilde{\chi}(z,x)]+m_{N_0}^2(z)[\iota(x)\iota(y)+\widetilde{\chi}(x,y)]\big\},
\end{split}
\end{equation}
for any $(x,y,z)\in\H^2_\infty$, since $A'=A/D$. As a consequence, for any $l_1,l_2,l_3\in\Z$, 
\begin{equation}\label{sus31}
\|A'(x,y,z)\cdot\varphi_{l_1}(x)\varphi_{l_2}(y)\varphi_{l_3}(z)\|_{\mathcal{HM}_{12}}\lesssim 2^{\min(l_1,l_2,l_3)/2}2^{(2N_0-1)\max(l_1,l_2,l_3)}.
\end{equation}

If $\underline{k}=(k_1,k_2,k_3,k_4)\in\Z^4$ is a balanced $4$-tuple then there is no derivative loss and the symbol $\varphi_{\underline{k}}\cdot K^{8,1}_1$ satisfies regular bounds similar to \eqref{slov3}. The symbol $\varphi_{\underline{k}}\cdot K^{8,2}_1$ contains the possibly singular factor $|\xi+\sigma|$, but can be expanded as in \eqref{A2.0}, which still suffices to prove admissibility. 

If $\underline{k}=(k_1,k_2,k_3,k_4)\in\Z^4$ is an imbalanced $4$-tuple then we may assume that $k_1\geq k_2\geq k_3$ and consider two cases:

\smallskip
$\bullet$ $k_1,k_2\in[\widetilde{k}_1-4,\widetilde{k}_1]$, $k_3,k_4\leq\widetilde{k}_3\leq \widetilde{k}_1-10$. In this case
\begin{equation*}
\begin{split}
\widetilde{K}^{8,1}_1(\xi,\eta,\rho,\sigma)&=(1/2)|\rho+\eta|^{1/2}|\xi|^{1/2}|\rho|\chi(\rho,\eta)A'(\xi,\sigma,\eta+\rho)(|\xi+\sigma|+|\sigma|-|\xi|)\\
&+(1/2)|\rho+\xi|^{1/2}|\eta|^{1/2}|\rho|\chi(\rho,\xi)A'(\eta,\sigma,\xi+\rho)(|\eta+\sigma|+|\sigma|-|\eta|),
\end{split}
\end{equation*}
for $(\xi,\eta,\rho,\sigma)\in\H^3_\infty$ in the support of $\varphi_{\underline{k}}$, since  $n_2^{rr}(\xi,\eta) = 0$ and therefore $K^{8,1}_1(\rho,\xi,\eta,\si) = 0$. Using \eqref{sus30} it is easy to see that both terms are regular, satisfying bounds similar to \eqref{slov3}.

We consider now the symbol $K_1^{8,2}$ in the second line of \eqref{ku113''}. We write
\begin{equation*}
\begin{split}
\widetilde{K}^{8,2}_1(\xi,\eta,\rho,\sigma)&=2n_2^{ri}(\xi,\sigma)A'(\eta,\rho,\xi+\sigma)|\rho|^{1/2}|\eta|^{1/2}|\xi+\sigma|^{1/2}\\
&+2n_2^{ri}(\eta,\sigma)A'(\rho,\xi,\eta+\sigma)|\xi|^{1/2}|\rho|^{1/2}|\eta+\sigma|^{1/2}\\
&+2n_2^{ri}(\rho,\sigma)A'(\xi,\eta,\rho+\sigma)|\eta|^{1/2}|\xi|^{1/2}|\rho+\sigma|^{1/2}
\end{split}
\end{equation*}
for $(\xi,\eta,\rho,\sigma)\in\H^3_\infty$ in the support of $\varphi_{\underline{k}}$. The formulas \eqref{sus30} and \eqref{Nh} show easily that the multiplier in the last line satisfies suitable admissibility bounds. Moreover, we may replace $n_2^{ri}(\xi,\sigma)$ with $\xi\sigma|\sigma|^{-1/2}\chi(\sigma,\xi)$ and $n_2^{ri}(\eta,\sigma)$ with $\eta\sigma|\sigma|^{-1/2}\chi(\sigma,\eta)$ at the expenses of acceptable errors, therefore, with the notation \eqref{slov1.1},
\begin{equation*}
\begin{split}
(\varphi_{\underline{k}}\widetilde{K}^{8,2}_1)(\xi,\eta,\rho,\sigma)\simeq \varphi_{\underline{k}}(\xi,\eta,\rho,\sigma)\big\{&2\xi\sigma|\sigma|^{-1/2}\chi(\sigma,\xi)A'(\eta,\rho,\xi+\sigma)|\rho|^{1/2}|\eta|^{1/2}|\xi+\sigma|^{1/2}\\
+&2\eta\sigma|\sigma|^{-1/2}\chi(\sigma,\eta)A'(\rho,\xi,\eta+\sigma)|\xi|^{1/2}|\rho|^{1/2}|\eta+\sigma|^{1/2}\big\}.
\end{split}
\end{equation*}
In view of \eqref{sus31} the remaining terms are regular, satisfying bounds similar to \eqref{slov3}.

\smallskip
$\bullet$ $k_1,k_4\in[\widetilde{k}_1-4,\widetilde{k}_1]$, $k_2,k_3\leq\widetilde{k}_3\leq \widetilde{k}_1-10$. In this case we need to recombine the multipliers $K_1^{8,1}$ and $K_1^{8,2}$ to avoid a potential $1/2$ derivative loss. Using \eqref{ku113''} and recalling the definition \eqref{slov1.1} we notice that 
\begin{equation*}
\begin{split}
&\varphi_{\underline{k}}(\xi,\eta,\rho,\sigma)K_1^{8,1}(\xi,\eta,\rho,\si)\simeq 0,\\
&\varphi_{\underline{k}}(\xi,\eta,\rho,\sigma)K_1^{8,1}(\eta,\rho,\xi,\si)\simeq \varphi_{\underline{k}}(\xi,\eta,\rho,\sigma)\cdot 3|\xi|^{1/2}|\rho|\chi(\rho,\xi)A'(\eta,\sigma,\rho+\xi)|\eta|^{1/2}|\sigma|,\\
&\varphi_{\underline{k}}(\xi,\eta,\rho,\sigma)K_1^{8,1}(\rho,\xi,\eta,\si)\simeq \varphi_{\underline{k}}(\xi,\eta,\rho,\sigma)\cdot 3|\xi|^{1/2}|\eta|\chi(\eta,\xi)A'(\rho,\sigma,\eta+\xi)|\rho|^{1/2}|\sigma|,
\end{split}
\end{equation*}
and
\begin{equation*}
\begin{split}
&\varphi_{\underline{k}}(\xi,\eta,\rho,\sigma)K_1^{8,2}(\xi,\eta,\rho,\si)\simeq 0,\\
&\varphi_{\underline{k}}(\xi,\eta,\rho,\sigma)K_1^{8,2}(\eta,\rho,\xi,\si)\simeq \varphi_{\underline{k}}(\xi,\eta,\rho,\sigma)\cdot (-3)|\eta||\sigma|\chi(\eta,\sigma)A'(\rho,\xi,\eta+\sigma)|\xi|^{1/2}|\rho|^{1/2},\\
&\varphi_{\underline{k}}(\xi,\eta,\rho,\sigma)K_1^{8,2}(\rho,\xi,\eta,\si)\simeq \varphi_{\underline{k}}(\xi,\eta,\rho,\sigma)\cdot (-3)|\rho||\sigma|\chi(\rho,\sigma)A'(\eta,\xi,\rho+\sigma)|\xi|^{1/2}|\eta|^{1/2}.
\end{split}
\end{equation*}
Therefore
\begin{equation*}
\begin{split}
&\varphi_{\underline{k}}(\xi,\eta,\rho,\sigma)[K_1^{8,1}(\eta,\rho,\xi,\si)+K_1^{8,2}(\rho,\xi,\eta,\si)]\\
&\simeq \varphi_{\underline{k}}(\xi,\eta,\rho,\sigma)\cdot 3|\rho||\eta|^{1/2}|\xi|^{1/2}|\sigma|\big\{\chi(\rho,\xi)A'(\eta,\rho+\xi,\sigma)-\chi(\rho,\sigma)A'(\eta,\xi,\rho+\sigma)\big\}\simeq 0
\end{split}
\end{equation*}
and
\begin{equation*}
\begin{split}
&\varphi_{\underline{k}}(\xi,\eta,\rho,\sigma)[K_1^{8,1}(\rho,\xi,\eta,\si)+K_1^{8,2}(\eta,\rho,\xi,\si)]\\
&\simeq \varphi_{\underline{k}}(\xi,\eta,\rho,\sigma)\cdot 3|\eta||\rho|^{1/2}|\xi|^{1/2}|\sigma|\big\{\chi(\eta,\xi)A'(\rho,\eta+\xi,\sigma)-\chi(\eta,\sigma)A'(\rho,\xi,\eta+\sigma)\big\}\simeq 0,
\end{split}
\end{equation*}
as desired.

\smallskip
{\bf{The symbol $\widetilde{K}_1^9$.}} We examine the formula \eqref{ku114} and decompose $K_1^9=K_1^{9,1}+K_1^{9,2}$, where
\begin{equation}\label{ku114''}
\begin{split}
& K_1^{9,1}(\xi,\eta,\rho,\si) := -2n_2^{rr}(\eta,\rho)B''(\eta+\rho,\sigma,\xi),\\
& K_1^{9,2}(\xi,\eta,\rho,\si) := n_2^{ri}(\xi,\sigma)B'''(\eta,\rho,\xi+\sigma).
\end{split}
\end{equation}

The functions $B''$ and $B'''$ are defined in \eqref{slov20}--\eqref{slov1}. To calculate them explicitly we recall the formulas \eqref{eni22} and \eqref{eni42}; in particular $B$ and $B'$ are symmetric in the first two variables and we calculate, using \eqref{chieq},
\begin{equation}\label{slov22}
B'(\xi,\eta,\rho)=\frac{m_{N_0}(\rho)|\rho|^{1/2}\big\{|\xi|\chi(\xi,\rho)[m_{N_0}(\rho)-m_{N_0}(\eta)]+|\eta|\chi(\eta,\rho)[m_{N_0}(\rho)-m_{N_0}(\xi)]\big\}}{-4|\xi||\eta|}
\end{equation}
if $(\xi,\eta,\rho)\in\H^2_\infty$ and $|\rho|=\max(|\xi|,|\eta|,|\rho|)$, since $\xi\rho+|\xi||\rho|=\eta\rho+|\eta||\rho|=0$ in this case. Moreover
\begin{equation}\label{slov23}
\begin{split}
B'(\xi,\eta,\rho)&=(1/2)|\rho|^{-1/2}m_{N_0}^2(\eta)+(1/4)\chi(\eta,\xi)m_{N_0}(\rho)|\rho|^{-1/2}[m_{N_0}(\xi)-m_{N_0}(\rho)]\\
&-\frac{\chi(\rho,\xi)[|\xi|m^2_{N_0}(\xi)+|\eta|m^2_{N_0}(\eta)-|\rho|m_{N_0}(\xi)m_{N_0}(\eta)]}{4|\eta||\rho|^{1/2}}
\end{split}
\end{equation}
if $(\xi,\eta,\rho)\in\H^2_\infty$ and $|\xi|=\max(|\xi|,|\eta|,|\rho|)$, using \eqref{chieq} again and the identities $\xi\rho=-|\xi||\rho|$ and $\eta\rho=|\eta|\rho|$ which hold in this case. 

Then we observe that $B''$ is symmetric in the first two variables and calculate
\begin{equation}\label{slov24}
B''(x,y,z)=|x|^{1/2}|y|^{1/2}\big[(1-\chi(y,z))m_{N_0}^2(x)+(1-\chi(x,z))m_{N_0}^2(y)\big]
\end{equation}
if $(x,y,z)\in\H^2_\infty$ and $|z|=\max(|x|,|y|,|z|)$, after careful algebraic simplifications. Similarly,
\begin{equation}\label{slov25}
\begin{split}
B''(x,y,z)=-\frac{1}{2}|x|^{1/2}|y|^{1/2}\big\{&2m_{N_0}^2(y)+\chi(y,x)[m_{N_0}^2(x)-m_{N_0}^2(z)]\\
&-\chi(z,x)[m_{N_0}^2(y)+m_{N_0}^2(x)]\big\}
\end{split}
\end{equation}
if $(x,y,z)\in\H^2_\infty$ and $|x|=\max(|x|,|y|,|z|)$. Finally, $B'''$ is symmetric in all the variables and 
\begin{equation}\label{slov26}
\begin{split}
B'''(x,y,z)=\frac{1}{2}\big\{&2|y|m_{N_0}^2(y)+2|z|m_{N_0}^2(z)\\
&-\chi(y,x)[|x|m^2_{N_0}(x)+|x|m_{N_0}^2(z)-2|y|m_{N_0}(x)m_{N_0}(z)]\\
&-\chi(z,x)[|x|m^2_{N_0}(x)+|x|m_{N_0}^2(y)-2|z|m_{N_0}(x)m_{N_0}(y)]\big\}
\end{split}
\end{equation}
if $(x,y,z)\in\H^2_\infty$ and $|x|=\max(|x|,|y|,|z|)$.

In particular, these formulas show easily that for any $l_1,l_2,l_3\in\Z$, 
\begin{equation}\label{slov27}
\begin{split}
\|B''(x,y,z)\cdot\varphi_{l_1}(x)\varphi_{l_2}(y)\varphi_{l_3}(z)\|_{\mathcal{HM}_{12}}&\lesssim 2^{\min(l_1,l_2,l_3)}2^{2N_0\max(l_1,l_2,l_3)}2^{-|l_1-l_2|/2},\\
\|B'''(x,y,z)\cdot\varphi_{l_1}(x)\varphi_{l_2}(y)\varphi_{l_3}(z)\|_{\mathcal{HM}_{12}}&\lesssim 2^{\min(l_1,l_2,l_3)}2^{2N_0\max(l_1,l_2,l_3)}.
\end{split}
\end{equation}

If $\underline{k}=(k_1,k_2,k_3,k_4)\in\Z^4$ is a balanced $4$-tuple then there is no derivative loss and the formulas \eqref{slov24}--\eqref{slov25} and \eqref{Nrr2} show easily that the symbol $\varphi_{\underline{k}}\cdot K^{9,1}_1$ satisfies regular bounds similar to \eqref{slov3}. The symbol $\varphi_{\underline{k}}\cdot K^{9,2}_1$ contains the possibly singular factor $|\xi+\sigma|$, but can be expanded as in \eqref{A2.0}, which still suffices to prove admissibility. 

If $\underline{k}=(k_1,k_2,k_3,k_4)\in\Z^4$ is an imbalanced $4$-tuple then we may assume that $k_1\geq k_2\geq k_3$ and consider two cases, as before:

\smallskip
$\bullet$ $k_1,k_2\in[\widetilde{k}_1-4,\widetilde{k}_1]$, $k_3,k_4\leq\widetilde{k}_3\leq \widetilde{k}_1-10$. In this case 
\begin{equation*}
\begin{split}
\widetilde{K}^{9,1}_1(\xi,\eta,\rho,\sigma)&=-\frac{1}{6}|\rho+\eta|^{1/2}|\rho|\chi(\rho,\eta)B''(\eta+\rho,\sigma,\xi)-\frac{1}{6}|\rho+\xi|^{1/2}|\rho|\chi(\rho,\xi)B''(\xi+\rho,\sigma,\eta)
\end{split}
\end{equation*}
for $(\xi,\eta,\rho,\sigma)\in\H^3_\infty$ in the support of $\varphi_{\underline{k}}$, using the formulas \eqref{Nrr2} and \eqref{ku114''}. Using \eqref{slov27} it is easy to see that both terms are regular, satisfying bounds similar to \eqref{slov3}. Moreover
\begin{equation*}
\begin{split}
(\varphi_{\underline{k}}\widetilde{K}^{9,2}_1)(\xi,\eta,\rho,\sigma)\simeq \varphi_{\underline{k}}(\xi,\eta,\rho,\sigma)\big\{&\xi\sigma|\sigma|^{-1/2}\chi(\sigma,\xi)B'''(\eta,\rho,\xi+\sigma)\\
&+\eta\sigma|\sigma|^{-1/2}\chi(\sigma,\eta)B'''(\xi,\rho,\eta+\sigma)\big\}/3,
\end{split}
\end{equation*}
using \eqref{Nh}, \eqref{ku114''}, and \eqref{slov26}. In view of \eqref{slov27} the remaining terms are regular, satisfying bounds similar to \eqref{slov3}, due to the symmetrization in $\xi$ and $\eta$.

\smallskip
$\bullet$ $k_1,k_4\in[\widetilde{k}_1-4,\widetilde{k}_1]$, $k_2,k_3\leq\widetilde{k}_3\leq \widetilde{k}_1-10$. In this case 
\begin{equation}\label{slov28}
\begin{split}
&\varphi_{\underline{k}}(\xi,\eta,\rho,\sigma)K_1^{9,1}(\xi,\eta,\rho,\si)\simeq 0,\\
&\varphi_{\underline{k}}(\xi,\eta,\rho,\sigma)K_1^{9,1}(\eta,\rho,\xi,\si)\simeq \varphi_{\underline{k}}(\xi,\eta,\rho,\sigma)\cdot (-1/2)|\xi|^{1/2}|\rho|\chi(\rho,\xi)B''(\rho+\xi,\sigma,\eta),\\
&\varphi_{\underline{k}}(\xi,\eta,\rho,\sigma)K_1^{9,1}(\rho,\xi,\eta,\si)\simeq \varphi_{\underline{k}}(\xi,\eta,\rho,\sigma)\cdot (-1/2)|\xi|^{1/2}|\eta|\chi(\eta,\xi)B''(\eta+\xi,\sigma,\rho),
\end{split}
\end{equation}
and
\begin{equation}\label{slov29}
\begin{split}
&\varphi_{\underline{k}}(\xi,\eta,\rho,\sigma)K_1^{9,2}(\xi,\eta,\rho,\si)\simeq 0,\\
&\varphi_{\underline{k}}(\xi,\eta,\rho,\sigma)K_1^{9,2}(\eta,\rho,\xi,\si)\simeq \varphi_{\underline{k}}(\xi,\eta,\rho,\sigma)\cdot (-1/2)|\eta||\sigma|^{1/2}\chi(\eta,\sigma)B'''(\rho,\xi,\eta+\sigma),\\
&\varphi_{\underline{k}}(\xi,\eta,\rho,\sigma)K_1^{9,2}(\rho,\xi,\eta,\si)\simeq \varphi_{\underline{k}}(\xi,\eta,\rho,\sigma)\cdot (-1/2)|\rho||\sigma|^{1/2}\chi(\rho,\sigma)B'''(\eta,\xi,\rho+\sigma).
\end{split}
\end{equation}
using \eqref{slov27} and \eqref{Nh}--\eqref{Nrr2}. The terms are all regular if $\widetilde{k}_3\geq\widetilde{k}_1-40$. 

On the other hand, if $\widetilde{k}_3\leq\widetilde{k}_1-40$ then we need to recombine these terms to avoid derivative loss. Notice that
\begin{equation}\label{slov30}
\begin{split}
&\varphi_{\underline{k}}(\xi,\eta,\rho,\sigma)\big[K_1^{9,1}(\eta,\rho,\xi,\si)+K_1^{9,2}(\rho,\xi,\eta,\sigma)]\\
&\simeq \varphi_{\underline{k}}(\xi,\eta,\rho,\sigma)\cdot (-1/2)|\xi|^{1/2}|\rho|\chi(\rho,\xi)[B''(\rho+\xi,\sigma,\eta)+B'''(\eta,\xi,\rho+\sigma)].
\end{split}
\end{equation}
We have two cases: if $\iota(\eta)=\iota(\xi)=-\iota(\sigma)$ then we may assume that $|\sigma|= |\rho+\xi|+|\eta|$ and $|\rho+\sigma|= |\xi|+|\eta|$. The formulas \eqref{slov25}--\eqref{slov26} show that
\begin{equation*}
\begin{split}
&2B''(\rho+\xi,\sigma,\eta)+2B'''(\eta,\xi,\rho+\sigma)=|\sigma|^{1/2}|\rho+\xi|^{1/2}[m_{N_0}^2(\sigma)-m_{N_0}^2(\rho+\xi)]\\
&+(2|\xi|-|\rho+\sigma|)m_{N_0}^2(\xi)-|\rho+\sigma|m_{N_0}^2(\rho+\sigma)+2|\eta|m_{N_0}(\xi)m_{N_0}(\rho+\sigma)+2|\eta|m_{N_0}^2(\eta)\\
&=[|\sigma|^{1/2}|\rho+\xi|^{1/2}-|\xi|][m_{N_0}^2(\sigma)-m_{N_0}^2(\rho+\xi)]-|\eta|[m_{N_0}(\xi)-m_{N_0}(\rho+\sigma)]^2\\
&+|\xi|[m_{N_0}^2(|\sigma|)-m_{N_0}^2(|\sigma|-|\eta|)+m_{N_0}^2(|\xi|)-m_{N_0}^2(|\xi|+|\eta|)]+2|\eta|m_{N_0}^2(\eta),
\end{split}
\end{equation*}
for $(\xi,\eta,\rho,\sigma)\in\H^3_\infty$ in the support of $\varphi_{\underline{k}}$ satisfying $\iota(\eta)=\iota(\xi)=-\iota(\sigma)$.
On the other hand, if $\iota(\eta)=-\iota(\xi)=\iota(\sigma)$ then we may assume that $|\sigma|+|\eta|= |\rho+\xi|$ and $|\rho+\sigma|+|\eta|= |\xi|$. The formulas \eqref{slov25}--\eqref{slov26} show that
\begin{equation*}
\begin{split}
&2B''(\rho+\xi,\sigma,\eta)+2B'''(\eta,\xi,\rho+\sigma)=|\sigma|^{1/2}|\rho+\xi|^{1/2}[m_{N_0}^2(\rho+\xi)-m_{N_0}^2(\sigma)]\\
&+(2|\rho+\sigma|-|\xi|)m_{N_0}^2(\rho+\sigma)-|\xi|m_{N_0}^2(\xi)+2|\eta|m_{N_0}(\rho+\sigma)m_{N_0}(\xi)+2|\eta|m_{N_0}^2(\eta)\\
&=[|\sigma|^{1/2}|\rho+\xi|^{1/2}-|\rho+\sigma|][m_{N_0}^2(\rho+\xi)-m_{N_0}^2(\sigma)]-|\eta|[m_{N_0}(\rho+\sigma)-m_{N_0}(\xi)]^2\\
&+|\rho+\sigma|[m_{N_0}^2(|\sigma|+\eta|)-m_{N_0}^2(|\sigma|)+m_{N_0}^2(|\xi|-|\eta|)-m_{N_0}^2(|\xi|)]+2|\eta|m_{N_0}^2(\eta),
\end{split}
\end{equation*}
for $(\xi,\eta,\rho,\sigma)\in\H^3_\infty$ in the support of $\varphi_{\underline{k}}$ satisfying $\iota(\eta)=-\iota(\xi)=\iota(\sigma)$. These two formulas show that in both cases there is a derivative gain, compensating for the $1/2$ derivative loss in the prefactor in \eqref{slov30}. Therefore 
\begin{equation*}
\varphi_{\underline{k}}(\xi,\eta,\rho,\sigma)\big[K_1^{9,1}(\eta,\rho,\xi,\si)+K_1^{9,2}(\rho,\xi,\eta,\sigma)]\simeq 0.
\end{equation*}

A similar calculation with the roles of $\rho$ and $\eta$ interchanged shows that
\begin{equation*}
\varphi_{\underline{k}}(\xi,\eta,\rho,\sigma)\big[K_1^{9,1}(\rho,\xi,\eta,\sigma)+K_1^{9,2}(\eta,\rho,\xi,\si)]\simeq 0
\end{equation*}
as well, as desired. This completes the analysis of the symbol $\widetilde{K}_1^9$. 

\smallskip
{\bf{The symbol $\widetilde{K}_2^1$.}} We examine the formula \eqref{ku21} and notice that $\varphi_{\underline{k}}\cdot K_2^1$ already satisfies regular bounds similar to \eqref{slov3} if $\underline{k}$ is balanced $4$-tuple. If $\underline{k}$ is imbalanced then we may assume that $k_1\geq k_2\geq k_3$ and consider two cases:

\smallskip
$\bullet$ $k_1,k_2\in[\widetilde{k}_1-4,\widetilde{k}_1]$, $k_3,k_4\leq\widetilde{k}_3\leq \widetilde{k}_1-10$. We calculate
\begin{equation*}
\begin{split}
&\varphi_{\underline{k}}(\xi,\eta,\rho,\sigma)K_2^1(\xi,\eta,\rho,\si)\simeq \varphi_{\underline{k}}(\xi,\eta,\rho,\sigma)\cdot \frac{m_{N_0}^2(\xi)|\xi|^{1/2}|\rho|^{1/2}}{|\eta|^{1/2}}\eta\big[(\rho+\sigma)\chi(\rho,\sigma)-\sigma\chi(\rho,\eta+\sigma)\big],\\
&\varphi_{\underline{k}}(\xi,\eta,\rho,\sigma)K_2^1(\eta,\rho,\xi,\si)\simeq \varphi_{\underline{k}}(\xi,\eta,\rho,\sigma)\cdot \frac{m_{N_0}^2(\eta)|\eta|^{1/2}|\rho|^{1/2}}{|\xi|^{1/2}}\xi\big[(\rho+\sigma)\chi(\rho,\sigma)-\sigma\chi(\rho,\xi+\sigma)\big],\\
&\varphi_{\underline{k}}(\xi,\eta,\rho,\sigma)K_2^1(\rho,\xi,\eta,\si)\simeq 0.
\end{split}
\end{equation*}
Since $\eta=|\eta|\iota(\eta)$, $\xi=|\xi|\iota(\xi)$, and $\iota(\xi)=-\iota(\eta)$, it follows easily that $\varphi_{\underline{k}}\widetilde{K}_2^1\simeq 0$, as desired.

\smallskip
$\bullet$ $k_1,k_4\in[\widetilde{k}_1-4,\widetilde{k}_1]$, $k_2,k_3\leq\widetilde{k}_3\leq \widetilde{k}_1-10$. In this case we have
\begin{equation*}
\varphi_{\underline{k}}(\xi,\eta,\rho,\sigma)K_2^1(\xi,\eta,\rho,\si)\simeq\varphi_{\underline{k}}(\xi,\eta,\rho,\sigma)K_2^1(\eta,\rho,\xi,\si)\simeq\varphi_{\underline{k}}(\xi,\eta,\rho,\sigma)K_2^1(\rho,\xi,\eta,\si)\simeq 0,
\end{equation*}
by examining the formula \eqref{ku21}. The desired conclusion follows.

\smallskip
{\bf{The symbol $\widetilde{K}_2^2$.}} We do not need symmetrization in this case. The formulas \eqref{ku24} and \eqref{ite2} show that
\begin{equation*}
\begin{split}
K_2^2(\xi,\eta,\rho,\sigma)&=m_{N_0}(\xi)[m_{N_0}(\sigma)-m_{N_0}(\si)]\chi(-\xi-\sigma,\sigma)\\
&\times|\xi|^{1/2}|\eta|^{1/2}|\rho|^{1/2}(|\eta+\rho|-|\eta|-|\rho|)[\iota(\eta)\iota(\rho)-1]
\end{split}
\end{equation*}
for any $\xi,\eta,\rho,\sigma)\in\H^3_\infty$. In particular $\varphi_{\underline{k}}\cdot K_2^2$ is nontrivial only if $|k_1-k_4|\leq 2$, and it is easy to see that the symbol is admissible with a possibly singular factor of $|\eta+\rho|=|\xi+\sigma|$.

\smallskip
{\bf{The symbol $\widetilde{K}_2^3$.}} As in the analysis of the symbol $\widetilde{K}_1^8$ we decompose $K_2^3=K_2^{3,1}+K_2^{3,2}$,
\begin{align}\label{slov40}
\begin{split}
& K_2^{3,1}(\xi,\eta,\rho,\si):=6n_2^{ii}(\eta,\rho)A'(\xi,\sigma,\eta+\rho)|\sigma|^{1/2}(|\eta+\rho|+|\xi|-|\sigma|),
\\
& K_2^{3,2}(\xi,\eta,\rho,\si):=3n_2^{ri}(\sigma,\xi)A'(\eta,\rho,\xi+\sigma)|\xi+\sigma|^{1/2}(|\rho|+|\eta|-|\xi+\sigma|).
\end{split}
\end{align}
The relevant formulas are \eqref{Nh}, \eqref{Nii2}, and \eqref{sus30}. We notice that the symbol $n_2^{ii}(\rho,\eta)$ is not singular in $|\rho+\eta|$, and in fact $|\rho+\eta|\approx |\rho|+|\eta|$ in the support of this symbol. If $\underline{k}=(k_1,k_2,k_3,k_4)\in\Z^4$ is a balanced $4$-tuple then there is no derivative loss and it is easy to see that the symbols $\varphi_{\underline{k}}\cdot K^{3,1}_2$ and $\varphi_{\underline{k}}\cdot K^{3,2}_2$ satisfy suitable admissible bounds.

If $\underline{k}=(k_1,k_2,k_3,k_4)\in\Z^4$ is imbalanced then we may assume that $k_1\geq k_2\geq k_3$ and, as before, consider two cases:

\smallskip
$\bullet$ $k_1,k_2\in[\widetilde{k}_1-4,\widetilde{k}_1]$, $k_3,k_4\leq\widetilde{k}_3\leq \widetilde{k}_1-10$. In this case
\begin{equation*}
\begin{split}
\varphi_{\underline{k}}(\xi,\eta,\rho,\sigma)K^{3,1}_2(\xi,\eta,\rho,\sigma)&\simeq \varphi_{\underline{k}}(\xi,\eta,\rho,\sigma)\cdot 3\eta|\rho|^{1/2}\iota(\rho)A'(\xi,\sigma,\eta+\rho)|\sigma|^{1/2}(|\eta|+|\xi|),\\
\varphi_{\underline{k}}(\xi,\eta,\rho,\sigma)K^{3,1}_2(\eta,\rho,\xi,\sigma)&\simeq \varphi_{\underline{k}}(\xi,\eta,\rho,\sigma)\cdot 3\xi|\rho|^{1/2}\iota(\rho)A'(\eta,\sigma,\xi+\rho)|\sigma|^{1/2}(|\eta|+|\xi|),\\
\varphi_{\underline{k}}(\xi,\eta,\rho,\sigma)K^{3,1}_2(\rho,\xi,\eta,\sigma)&=0.
\end{split}
\end{equation*}
Using \eqref{sus31} we have $\varphi_{\underline{k}}\widetilde{K}^{3,1}_2\simeq 0$, as desired. Moreover,
\begin{equation*}
\begin{split}
\varphi_{\underline{k}}(\xi,\eta,\rho,\sigma)K^{3,2}_2(\xi,\eta,\rho,\sigma)\simeq \varphi_{\underline{k}}(\xi,\eta,\rho,\sigma)K^{3,2}_2(\eta,\rho,\xi,\sigma)\simeq\varphi_{\underline{k}}(\xi,\eta,\rho,\sigma)K^{3,2}_2(\rho,\xi,\eta,\sigma)\simeq 0,
\end{split}
\end{equation*}
therefore $\varphi_{\underline{k}}\widetilde{K}^{3,2}_2\simeq 0$, as desired.

\smallskip
$\bullet$ $k_1,k_4\in[\widetilde{k}_1-4,\widetilde{k}_1]$, $k_2,k_3\leq\widetilde{k}_3\leq \widetilde{k}_1-10$. Using \eqref{sus30}--\eqref{sus31} we see that
\begin{equation*}
\begin{split}
\varphi_{\underline{k}}(\xi,\eta,\rho,\sigma)K^{3,1}_2(\xi,\eta,\rho,\sigma)&\simeq 0,\\
\varphi_{\underline{k}}(\xi,\eta,\rho,\sigma)K^{3,1}_2(\eta,\rho,\xi,\sigma)&\simeq \varphi_{\underline{k}}(\xi,\eta,\rho,\sigma)\cdot 3\rho|\rho|^{-1/2}\xi|\sigma|^{1/2}(|\rho+\xi|+|\eta|-|\sigma|)A'(\eta,\sigma,\rho+\xi),\\
\varphi_{\underline{k}}(\xi,\eta,\rho,\sigma)K^{3,1}_2(\rho,\xi,\eta,\sigma)&\simeq\varphi_{\underline{k}}(\xi,\eta,\rho,\sigma)\cdot 3\eta|\eta|^{-1/2}\xi|\sigma|^{1/2}(|\eta+\xi|+|\rho|-|\sigma|)A'(\rho,\sigma,\eta+\xi),
\end{split}
\end{equation*}
and
\begin{equation*}
\begin{split}
\varphi_{\underline{k}}(\xi,\eta,\rho,\sigma)K^{3,2}_2(\xi,\eta,\rho,\sigma)&\simeq 0,\\
\varphi_{\underline{k}}(\xi,\eta,\rho,\sigma)K^{3,2}_2(\eta,\rho,\xi,\sigma)&\simeq \varphi_{\underline{k}}(\xi,\eta,\rho,\sigma)\cdot 3\eta|\eta|^{-1/2}\sigma|\sigma|^{-1/2}A'(\rho,\xi,\eta+\sigma)(|\xi|+|\rho|-|\eta+\sigma|),\\
\varphi_{\underline{k}}(\xi,\eta,\rho,\sigma)K^{3,2}_2(\rho,\xi,\eta,\sigma)&\simeq\varphi_{\underline{k}}(\xi,\eta,\rho,\sigma)\cdot 3\rho|\rho|^{-1/2}\sigma|\sigma|^{-1/2}A'(\eta,\xi,\rho+\sigma)(|\xi|+|\eta|-|\rho+\sigma|).
\end{split}
\end{equation*}
As before, we need to combine these terms in order to avoid a $1/2$ derivative loss. Notice that $|\rho+\xi|+|\eta|-|\sigma|=|\xi|+|\eta|-|\rho+\sigma|=|\eta|+\iota(\xi)(\rho+\xi+\sigma)$ for $(\xi,\eta,\rho,\sigma)\in\H^3_\infty$ in the support of $\varphi_{\underline{k}}$. Using again \eqref{sus31} it follows that $\varphi_{\underline{k}}(\xi,\eta,\rho,\sigma)[K_2^{3,1}(\eta,\rho,\xi,\si)+K_2^{3,2}(\rho,\xi,\eta,\si)]\simeq 0$. Similarly, we have $\varphi_{\underline{k}}(\xi,\eta,\rho,\sigma)[K_2^{3,1}(\rho,\xi,\eta,\si)+K_2^{3,2}(\eta,\rho,\xi,\si)]\simeq 0$, so $\varphi_{\underline{k}}\widetilde{K}_2^3\simeq 0$, as desired.

\smallskip
{\bf{The symbol $\widetilde{K}_2^4$.}} We examine the formula \eqref{ku26} and decompose $K_2^4=K_2^{4,1}+K_2^{4,2}$, where
\begin{equation}\label{sam1}
\begin{split}
& K_2^{4,1}(\xi,\eta,\rho,\si) := -2n_2^{ii}(\eta,\rho)B''(\eta+\rho,\xi,\si),\\
& K_2^{4,2}(\xi,\eta,\rho,\si) := -n_2^{ri}(\sigma,\xi)B''(\eta,\rho,\xi+\sigma).
\end{split}
\end{equation}
The function $B''$, $n_2^{ri}$, and $n_2^{ii}$ are defined in \eqref{slov24}--\eqref{slov25}, \eqref{Nh}, and \eqref{Nii2}.

As in the analysis of the multiplier $\widetilde{K}_1^9$, it is easy to see that $\varphi_{\underline{k}}K_2^{4,1}\simeq 0$ and $\varphi_{\underline{k}}K_2^{4,2}\simeq 0$ if $\underline{k}$ is a balanced $4$-tuple. If $\underline{k}=(k_1,k_2,k_3,k_4)\in\Z^4$ is imbalanced then we may assume that $k_1\geq k_2\geq k_3$ and consider two cases, as before:

\smallskip
$\bullet$ $k_1,k_2\in[\widetilde{k}_1-4,\widetilde{k}_1]$, $k_3,k_4\leq\widetilde{k}_3\leq \widetilde{k}_1-10$. In this case, using \eqref{slov27},
\begin{equation*}
\begin{split}
\varphi_{\underline{k}}(\xi,\eta,\rho,\sigma)K^{4,1}_2(\xi,\eta,\rho,\sigma)&\simeq \varphi_{\underline{k}}(\xi,\eta,\rho,\sigma)\cdot (-1)\iota(\rho)|\rho|^{1/2}\eta B''(\eta+\rho,\xi,\sigma),\\
\varphi_{\underline{k}}(\xi,\eta,\rho,\sigma)K^{4,1}_2(\eta,\rho,\xi,\sigma)&\simeq \varphi_{\underline{k}}(\xi,\eta,\rho,\sigma)\cdot (-1)\iota(\rho)|\rho|^{1/2}\xi B''(\xi+\rho,\eta,\sigma),\\
\varphi_{\underline{k}}(\xi,\eta,\rho,\sigma)K^{4,1}_2(\rho,\xi,\eta,\sigma)&=0,
\end{split}
\end{equation*}
and, using also \eqref{slov24}--\eqref{slov25},
\begin{equation*}
\begin{split}
\varphi_{\underline{k}}(\xi,\eta,\rho,\sigma)K^{4,2}_2(\xi,\eta,\rho,\sigma)\simeq \varphi_{\underline{k}}(\xi,\eta,\rho,\sigma)K^{4,2}_2(\eta,\rho,\xi,\sigma)\simeq\varphi_{\underline{k}}(\xi,\eta,\rho,\sigma)K^{4,2}_2(\rho,\xi,\eta,\sigma)\simeq 0.
\end{split}
\end{equation*}
Therefore, using \eqref{slov27} we have $\varphi_{\underline{k}}\widetilde{K}^{4,1}_2\simeq 0$ and $\varphi_{\underline{k}}\widetilde{K}^{4,2}_2\simeq 0$, as desired.

\smallskip
$\bullet$ $k_1,k_4\in[\widetilde{k}_1-4,\widetilde{k}_1]$, $k_2,k_3\leq\widetilde{k}_3\leq \widetilde{k}_1-10$. Using \eqref{slov24}--\eqref{slov25}  and \eqref{slov27} we see that
\begin{equation*}
\begin{split}
\varphi_{\underline{k}}(\xi,\eta,\rho,\sigma)K^{4,1}_2(\xi,\eta,\rho,\sigma)&\simeq 0,\\
\varphi_{\underline{k}}(\xi,\eta,\rho,\sigma)K^{4,1}_2(\eta,\rho,\xi,\sigma)&\simeq \varphi_{\underline{k}}(\xi,\eta,\rho,\sigma)\cdot (-1)\iota(\rho)|\rho|^{1/2}\xi B''(\xi+\rho,\eta,\sigma),\\
\varphi_{\underline{k}}(\xi,\eta,\rho,\sigma)K^{4,1}_2(\rho,\xi,\eta,\sigma)&\simeq\varphi_{\underline{k}}(\xi,\eta,\rho,\sigma)\cdot (-1)\iota(\eta)|\eta|^{1/2}\xi B''(\xi+\eta,\rho,\sigma),
\end{split}
\end{equation*}
and
\begin{equation*}
\begin{split}
\varphi_{\underline{k}}(\xi,\eta,\rho,\sigma)K^{4,2}_2(\xi,\eta,\rho,\sigma)&\simeq 0,\\
\varphi_{\underline{k}}(\xi,\eta,\rho,\sigma)K^{4,2}_2(\eta,\rho,\xi,\sigma)&\simeq \varphi_{\underline{k}}(\xi,\eta,\rho,\sigma)\cdot (-1)\iota(\eta)|\eta|^{1/2}\sigma B''(\rho,\xi,\eta+\sigma),\\
\varphi_{\underline{k}}(\xi,\eta,\rho,\sigma)K^{4,2}_2(\rho,\xi,\eta,\sigma)&\simeq\varphi_{\underline{k}}(\xi,\eta,\rho,\sigma)\cdot (-1)\iota(\rho)|\rho|^{1/2}\sigma B''(\eta,\xi,\rho+\sigma).
\end{split}
\end{equation*}
As before, we combine these terms in order to avoid a $1/2$ derivative loss. Using again \eqref{slov27} it follows that $\varphi_{\underline{k}}(\xi,\eta,\rho,\sigma)[K_2^{4,1}(\eta,\rho,\xi,\si)+K_2^{4,2}(\rho,\xi,\eta,\si)]\simeq 0$ and $\varphi_{\underline{k}}(\xi,\eta,\rho,\sigma)[K_2^{4,1}(\rho,\xi,\eta,\si)+K_2^{4,2}(\eta,\rho,\xi,\si)]\simeq 0$, so $\varphi_{\underline{k}}\widetilde{K}_2^4\simeq 0$, as desired.

\subsection{Proof of Theorem \ref{IncremBound}}\label{ssecquintic}
Let $(h,\phi)\in C([T_1,T_2]:H^{N_0}\times\dot{H}^{N_0,1/2}$ be a solution of \eqref{gWW} on the time-interval $[T_1,T_2]$, satisfying the assumptions of the theorem.
Recall that we defined
\begin{align}\label{pr0}
\omega = \phi - T_B h, 
  \qquad A(t) = {\| (h+i|\partial_x|^{1/2}\omega)(t) \|}_{H^{N_0}}.
\end{align}
In view of \eqref{BootstrapFull} we have $A(t) \lesssim A$ for all $t\in[T_1,T_2]$.
Recall also the definition of the energy functional $\mathcal{E}_2$ in \eqref{eni3} 
with $J^{N_0}$ the operator whose symbol is defined in \eqref{m_choice}, 
and notice that, as $\delta \rightarrow 0^+$ we have
\begin{align}\label{pr1}
{\| P_{\geq 2} h(t) \|}_{H^{N_0}}^2 
  + {\| |\partial_x|^{1/2}P_{\geq 2}\omega(t) \|}_{H^{N_0}}^2\lesssim \mathcal{E}_2(t) \lesssim A^2(t).
  \end{align}

Combining Lemmas \ref{quartic1}, \ref{quartic2} and \ref{quartic3} we have,
for all $t_1\leq t_2 \in [T_1,T_2]$,
\begin{align}\label{pr3}
\begin{split}
& \big| \mathcal{E}_2(t_2)  - \mathcal{E}_2(t_1) \big| 
  \lesssim \e(t_1)A^2(t_1) + \e(t_2)A^2(t_2) 
  \\
  & + \Big| \int_{t_1}^{t_2} \big[ \mathcal{Q}_1(h^\ast,h^\ast,h^\ast,\omega^\ast)(s)
  + \mathcal{Q}_2(\omega^\ast, \omega^\ast, \omega^\ast, h^\ast)(s)\big] \, ds \Big|
  + A^2 \int_{t_1}^{t_2} \e^3(s) \, ds,
\end{split}
\end{align}
where $\mathcal{Q}_1$ and $\mathcal{Q}_2$ are defined by \eqref{simp1}-\eqref{simp2}.
These expressions are defined by the symbols $\widetilde{K}_1^a$ and $\widetilde{K}_2^b$, and all these symbols are admissible in the sense of Definition \ref{defA} (see Lemma \ref{pus1}). We would like to apply the formulas \eqref{simp1C} and \eqref{simp2C}, and then Lemmas \ref{hoe2} and \ref{hol2}. This is possible provided that we verify the key vanishing condition \eqref{uyt22}:

\begin{proposition}\label{BFprop}
The symmetrized symbol $SK_1-SK_2$ (see definition \eqref{verf1}) vanishes on the set of Benjamin-Feir resonances, i.e.
\begin{align}\label{BFprop2}
\begin{split}
&(SK_1-SK_2)(\xi_1,\xi_2,\xi_3,\xi_4)=0\qquad\text{ provided that } \xi=(\xi_1,\xi_2,\xi_3.\xi_4)\in\H_\infty^3\text{ satisfies}\\
&\xi_1\in(-\infty,0),\,\xi_2,\xi_3,\xi_4\in(0,\infty)\text{ and }|\xi_1|^{1/2}+|\xi_2|^{1/2}=|\xi_3|^{1/2}+|\xi_4|^{1/2}.
\end{split}
\end{align}
\end{proposition}

We will prove this proposition in section \ref{uyt22sec}, by explicit calculations. Assuming this, we can apply  Lemmas \ref{hoe2} and \ref{hol2}, and the inequality \eqref{pr3} we conclude that
\begin{align}\label{pr5}
\begin{split}
\mathcal{E}_2(t_2) \lesssim \mathcal{E}_2(t_1) + \e(t_1)A^2(t_1) + \e(t_2)A^2(t_2)
  + A^2 \int_{t_1}^{t_2} \e^3(s) \, ds.
\end{split}
\end{align}

To conclude the proof we need to include low frequencies
using the Hamiltonian \eqref{gWWHam} which is conserved. Under our assumptions
on $h$ in \eqref{BootstrapFull}, 
for all $t\in[T_1,T_2]$ we have
\begin{align}\label{prH}
{\| h(t) \|}_{L^2}^2 + {\| |\partial_x|^{1/2}\phi(t) \|}_{L^2}^2 
  \lesssim \mathcal{H}(h,\phi)(t) \lesssim A^2(t).
\end{align}
In particular, by the definition of $\omega$, for all $t_1\leq t_2 \in [T_1,T_2]$,
\begin{align}\label{pr6}
\begin{split}
& {\| P_{\leq 4} h(t_2) \|}_{H^{N_0}}^2 + {\| P_{\leq 4}|\partial_x|^{1/2}\omega(t_2) \|}_{H^{N_0}}^2
  \lesssim {\| h(t_2) \|}_{L^2}^2 + {\| |\partial_x|^{1/2}\phi(t_2) \|}_{L^2}^2 \lesssim A^2(t_1).
\end{split}
\end{align}
Using \eqref{pr1}, \eqref{pr5}, \eqref{pr6} we obtain
\begin{align}\label{pr10}
\begin{split}
{\| h(t_2) \|}_{H^{N_0}}^2 + {\| |\partial_x|^{1/2}\omega(t_2) \|}_{H^{N_0}}^2
  & \lesssim A^2(t_1) + \e(t_2) A^2(t_2)
  + A^2 \int_{t_1}^{t_2} \e^3(s) \, ds,
\end{split}
\end{align}
which gives the desired bounds \eqref{MainGrowth}.


\medskip
\section{Long-term regularity: proof of Theorem \ref{LongTermDet}}\label{Deterministic}

In this section we prove Theorem \ref{LongTermDet} on the long-term regularity of the system \eqref{gWW} with small initial-data. We provide the argument in the (harder) case of finite $R$, but the argument applies with almost no changes in the Euclidean case $R=\infty$ as well (the Euclidean Strichartz estimates in Lemma \ref{LemStrichartz} are global $T=\infty$, and the point is that all the estimates hold uniformly in $R$).

\subsection{Integration by parts and Strichartz estimates}\label{Strc}
We use sometimes integration by parts to estimate oscillatory integrals with non-degenerate phases, according to the following:

\begin{lemma}\label{tech5} Assume that $0<\eps\leq 1/\eps\leq K$, $N\geq 1$ is an integer, and $f,g\in C^N_0(\mathbb{R}^d)$, $d\geq 1$. Then
\begin{equation}\label{ln1}
\Big|\int_{\mathbb{R}^d}e^{iKf}g\,dx\Big|\lesssim_N (K\eps)^{-N}\big[\sum_{|\alpha|\leq N}\eps^{|\alpha|}\|D^\alpha_xg\|_{L^1}\big],
\end{equation}
provided that $f$ is real-valued, 
\begin{equation}\label{ln2}
|\nabla_x f|\geq \mathbf{1}_{{\mathrm{supp}}\,g},\quad\text{ and }\quad\|D_x^\alpha f \cdot\mathbf{1}_{{\mathrm{supp}}\,g}\|_{L^\infty}\lesssim_N\eps^{1-|\alpha|},\,2\leq |\alpha|\leq N+1.
\end{equation}
\end{lemma}

Let $\Lambda(\xi):=|\xi|^{1/2}$ denote the dispersion relation associated to gravity water waves. The proof of Theorem \ref{LongTermDet} is based on the following Strichartz estimates on the torus $\mathbb{T}_R$:

\begin{lemma}\label{LemStrichartz}
If $R\geq 1$, $k\in\mathbb{Z}$ and $T\leq R2^{k^-/2}$ then, for any $f\in L^2(\mathbb{T}_R)$,
\begin{equation}\label{Strc2}
\|e^{-it\Lambda} P_kf\|_{L^4_tL^\infty_x(\mathbb{T}_R\times [-T,T])}\lesssim 2^{3k/8}\|f\|_{L^2(\mathbb{T}_R)}.
\end{equation}
\end{lemma}

\begin{proof}
We adapt the standard proof of Strichartz estimates to our large box setting. Using Sobolev embedding
\begin{equation*}
\|e^{-it\Lambda} P_kf\|_{L^4_tL^\infty_x(\mathbb{T}_R\times [-T,T])}\lesssim T^{1/4}2^{k/2}\|f\|_{L^2(\mathbb{T}_R)},
\end{equation*}
so the desired bounds \eqref{Strc2} follow if $2^k\lesssim R^{-1}$. So we may assume that $R2^k\geq 2^{10}$. Let $K_k$ denote the kernel of the operator $e^{-it\Lambda}P_k$, i.e.
\begin{equation}\label{Strc3}
\begin{split}
&e^{-it\Lambda}P_kf(x)=\int_{\mathbb{T}_R}f(y)K_k(x-y,t),\\
&K_k(z,t):=\frac{1}{2\pi R}\sum_{\xi\in\mathbb{Z}/R}\varphi_k(\xi)e^{-it\Lambda(\xi)}e^{i\xi z}.
\end{split}
\end{equation}
Then
\begin{equation*}
\|e^{-it\Lambda} P_kf\|_{L^4_tL^\infty_x(\mathbb{T}_R\times [-T,T])}\lesssim\sup_{\|g\|_{L^{4/3}_tL^1_x(\mathbb{T}_R\times [-T,T])}\leq 1}\Big|\int_{\mathbb{T}_R^2\times [-T,T]}f(y)K_k(x-y,t)g(x,t)\,dtdxdy\Big|.
\end{equation*}
Therefore, for \eqref{Strc2} it suffices to prove that
\begin{equation}\label{Strc4}
\|Ig\|_{L^2_y}\lesssim 2^{3k/8}\|g\|_{L^{4/3}_tL^1_x},\qquad \text{ where }\,\,Ig(y):=\int_{\mathbb{T}_R\times [-T,T]}K_k(x-y,t)g(x,t)\,dtdx.
\end{equation}

The formula for the kernel $K_k$ in \eqref{Strc3} shows that
\begin{equation}\label{Strc5}
\begin{split}
&\int_{\mathbb{T}_R}K_k(x-y,t)\overline{K_k(z-y,s)}\,dy=K'_k(x-z,t-s),\\
&K'_k(w,t):=\frac{1}{2\pi R}\sum_{\xi\in\mathbb{Z}/R}\varphi^2_k(\xi)e^{-it\Lambda(\xi)}e^{i\xi w},
\end{split}
\end{equation}
for any $x,z\in\mathbb{T}_R$ and $t,s\in\mathbb{R}$. Then
\begin{equation*}
\int_{\mathbb{T}_R}|Ig(y)|^2\,dy=\int_{\mathbb{T}^2_R\times [-T,T]^2}g(x,t)\overline{g(z,s)}K'_k(x-z,t-s)\,dtdsdxdz.
\end{equation*}
By fractional integration, for \eqref{Strc4} it suffices to prove that
\begin{equation}\label{Strc7}
|K'_k(w,t)|\lesssim 2^{3k/4}|t|^{-1/2}\qquad\text{ for any }t\in[-2T,2T]\text{ and }w\in\mathbb{T}_R.
\end{equation}

To prove the dispersive bounds \eqref{Strc7} we need to use assumption $T\leq R2^{k^-/2}$. In view of the definition \eqref{Strc5} it suffices to show that
\begin{equation}\label{Strc9}
\frac{1}{2\pi R}\Big|\sum_{n\in\mathbb{Z}}\varphi^2_0(n/(2^kR))e^{-itR^{-1/2}\sqrt{|n|}}e^{in w}\Big|\lesssim 2^{3k/4}|t|^{-1/2}
\end{equation}
for any $t\in[-2T,2T]$ and $w\in[0,2\pi]$. We use the Poisson summation formula for the function $f(x):=(2\pi R)^{-1}\varphi^2_0(x/(2^kR))e^{-itR^{-1/2}\sqrt{|x|}}e^{ix w}$, thus
\begin{equation}\label{Strc9.5}
\frac{1}{2\pi R}\sum_{n\in\mathbb{Z}}\varphi^2_0(n/(2^kR))e^{-itR^{-1/2}\sqrt{|n|}}e^{in w}=\sum_{n\in\mathbb{Z}}\widehat{f}(2\pi n),
\end{equation}
where $\widehat{f}$ denotes the Fourier transform of $f$ on $\mathbb{R}$. We write now
\begin{equation}\label{Strc10}
\begin{split}
\widehat{f}(\xi)&=\frac{1}{2\pi R}\int_{\mathbb{R}}\varphi^2_0(x/(2^kR))e^{-itR^{-1/2}\sqrt{|x|}}e^{-ix (\xi-w)}\,dx\\
&=\frac{2^k}{2\pi}\int_{\mathbb{R}}\varphi^2_0(x)e^{-i(t2^{k/2})\sqrt{|x|}}e^{-i(2^kR) (\xi-w)x}\,dx.
\end{split}
\end{equation}
If $|n|\geq 10$ then we recall that $|t2^{k/2}|\leq 2R2^k$ and $|w|\leq 2\pi$, and use integration by parts (Lemma \ref{tech5} with $K\approx r^kR$, $\eps\approx 1$) to see that $|\widehat{f}(2\pi n)|\lesssim 2^k(2^kR)^{-2}n^{-2}$. On the other hand, if $|n|\leq 10$ then we use general stationary phase estimates to show that $|\widehat{f}(2\pi n)|\lesssim 2^k|t2^{k/2}|^{-1/2}$. The desired bounds \eqref{Strc9} follow using also \eqref{Strc9.5}.
\end{proof}

\subsection{Proof of Theorem \ref{LongTermDet}}\label{LongTermDetpr}
Given a solution of \eqref{gWW} with data as in \eqref{mai1}, 
we use the Strichartz estimate \ref{Strc2} and the energy estimate \eqref{MainGrowth}
to prove the following bootstrap argument, which implies Theorem \ref{LongTermDet}: let $N_1:=N_0-3\geq 4$,
$$u := h+i|\partial_x|^{1/2}\omega, \qquad \omega = \phi - T_Bh,$$ 
and assume that $u\in C([0,T]:H^{N_0})$, $T\leq T_{R,\varep} := c_0^2 \e^{-3} \min \big( \e^{-3}, R^{3/4} \big)$, is a solution satisfying
\begin{align}\label{mai3'}
\begin{split}
\sup_{t\in[0,T]}\big[{\big\|u(t) \big\|}_{H^{N_0}}+\big\|\,|\partial_x|^{1/2}\phi(t) \big\|_{H^{N_0-1/2}}\big] &\leq 2C_0\e,
  \qquad ,
\\
\sup_{t_1\leq t_2\in[0,T],\,|t_2-t_1|\leq T_S}\|\delta(.)\|_{L^4_{[t_1,t_2]}}&\leq 2C^2_0 \e,\\
\delta(s):=\big\|u(s) \big\|_{\dot{W}^{N_1,-1/4}_x}&+\big\|\,|\partial_x|^{1/2}\phi(s)\big\|_{\dot{W}^{N_1-1/2,-1/4}_x},
\end{split}
\end{align} 
where $C_0\gg 1$ and $c_0=C_0^{-20}\ll 1$ are suitable constants, and $T_S := c_0\min(\e^{-4},R)$. Then
\begin{align}
\label{mai3conc1}
\sup_{t\in[0,T]}\big[{\big\|u(t) \big\|}_{H^{N_0}}+\big\|\,|\partial_x|^{1/2}\phi(t) \big\|_{H^{N_0-1/2}}\big]&\leq C_0\e,
\\
\label{mai3conc2}
\sup_{t_1\leq t_2\in[0,T],\,|t_2-t_1|\leq T_S} \|\delta(.)\|_{L^4_{[t_1,t_2]}}&\leq C^2_0 \e.
\end{align}

\subsubsection*{Proof of \eqref{mai3conc1}}
With $\overline{C}$ denoting sufficiently large constants (depending only on the implicit constants in \eqref{MainGrowth} and \eqref{Strc}, it follows from \eqref{MainGrowth} that for any $t\in[0,T]$
\begin{align*}
{\| u (t) \|}_{H^{N_0}}^2 & \leq \overline{C}\e^2 + \overline{C}(C_0\e)^2 \int_0^t [\delta(s)]^3 \, ds
\\ 
& \leq \overline{C}\e^2 + \overline{C}(C_0\e)^2 \cdot T^{1/4} 
  \cdot {\|\delta(.)\|}_{L^4_{[0,T]}}^3
\\ 
& \leq \overline{C}\big[\e^2 + (C_0\e)^2 \cdot T_{R,\e}^{1/4} 
  \cdot \big( (T_{R,\e}/T_S) \cdot (C_0^2\e)^4 \big)^{3/4} \big]
\\
& \leq \overline{C}\big[ \e^2 + C_0^8\cdot \e^5 \cdot T_{R,\e} \cdot T_S^{-3/4} \big].
\end{align*}
For the last inequality we have split the time integral over $[0,T]$
into at most $T_{R,\e}/T_S$ many intervals of size $T_S$ and used the assumption on the Strichartz norm \eqref{mai3'}.
The desired bounds on $u$ in \eqref{mai3conc1} follow by choosing $C_0$ large enough and then $c_0=C_0^{-20}$ small enough.

The desired bounds on $|\partial_x|^{1/2}\phi$ follow if $\varepsilon$ is sufficiently small, since $|\partial_x|^{1/2}\phi=\Im(u)+|\partial_x|^{1/2}T_Bh$ and $\|B(s)\|_{\dot{W}^{N_0-7/4,-1/4}_x}\leq\overline{C}C_0\varepsilon$ for any $s\in[0,T]$ (due to \eqref{DefBV2} and Sobolev embedding).

\subsubsection*{Proof of \eqref{mai3conc2}}
We first need to introduce a new variable satisfying a cubic equation.
Let 
\begin{align}\label{defv}
v:= h+i|\partial_x|^{1/2}\phi + \frac{1}{2}|\partial_x|\Big( \frac{\partial_x}{|\partial_x|}h \cdot \frac{\partial_x}{|\partial_x|} h\Big)
  + i\frac{\partial_x}{|\partial_x|^{1/2}} \Big( \frac{\partial_x}{|\partial_x|} h\cdot |\partial_x|\phi \Big),
\end{align}
compare with \cite[Lemma 5.1]{IoPu2}. It follows from the assumptions \eqref{mai3'} and Sobolev embedding that
\begin{equation}\label{pac1}
\|h(s)\|_{\dot{W}^{N_0-3/4.-1/4}}+\|\,|\partial_x|^{1/2}\phi(s)\|_{\dot{W}^{N_0-5/4.-1/4}}+\|B(s)\|_{\dot{W}^{N_0-7/4.-1/4}}\leq\overline{C}C_0\varepsilon,
\end{equation}
for any $s\in[0,T]$. The definition \eqref{defv} and the product estimates in Lemma \ref{algeProp} show that
\begin{align}\label{Struv0}
\begin{split}
{\| (u-v)(t)\|}_{{\dot{W}^{N_1,-1/4}_x}} \leq  \overline{C}C_0\e \delta(t)
\end{split}
\end{align}
for all $t\in[0,T]$, Therefore, if $\e$ is sufficiently small and $I\subseteq [0,T]$ is an interval with $|I| \leq T_S$,
\begin{align}\label{Struv}
\begin{split}
{\| u\big\|}_{L^4_t(I){\dot{W}^{N_1,-1/4}_x}}
  \leq {\| v\big\|}_{L^4_t(I){\dot{W}^{N_1,-1/4}_x}}+C^2_0\varepsilon/4.
\end{split}
\end{align}

On the other hand, the point of the quadratic correction in definition \eqref{defv} is that the variable $v$ solves a cubic equation
\begin{align}\label{cubicv}
(\partial_t + i |\partial_x|^{1/2})v = \mathcal{N}_3,
\end{align}
where $ \mathcal{N}_3$ is a cubic and higher order nonlinearity in $(h,\phi)$ that
satisfies the bounds
\begin{align}\label{cubicN0}
{\big\|\mathcal{N}_3(t)\big\|}_{H^{N_0-2}} \leq\overline{C}
  \big({\| h(t)\|}_{H^{N_0}}+\|\,|\partial_x|^{1/2}\phi(t)\|_{H^{N_0-1/2}}\big)\delta(t)^2\leq\overline{C}(C_0\e)\delta(t)^2
\end{align}
for any $t\in[0,T]$. Indeed, this follows easily from \eqref{expand0} and \eqref{expand2},  or from the analysis in \cite[Section 5]{IoPu2}. 
In particular, for any intervals $I\subseteq [0,T]$ with $|I|\leq T_S$ we have
\begin{align}\label{cubicN}
{\big\|\mathcal{N}_3 \big\|}_{L^2_{t}(I)H^{N_0-2}_x} \leq\overline{C}C_0^5\e^3.
\end{align}

We apply the Strichartz estimate \eqref{Strc2} to the Duhamel formula associated to \eqref{cubicv}:
for any $I=[t_1,t_2] \subseteq [0,T]$, with $|I| \leq T_S$, we have
\begin{align*}
{\| v \|}_{L^4_t(I)\dot{W}^{N_1,-1/4}_x}
  & \leq \overline{C}{\| v(t_1) \|}_{H^{N_1+1/2}}
  + \overline{C}{\Big\| \int_{t_1}^{t} e^{i(t-s)|\partial_x|^{1/2}} \mathcal{N}_3(s)\,ds \Big\|}_{L^4_t(I)\dot{W}^{N_1,-1/4}_x}
  \\
  \nonumber
  & \leq\overline{C}{\| u(t_1) \|}_{H^{N_0}}
  + \overline{C}{\big\| \mathcal{N}_3 \big\|}_{L^{1}_t(I)H^{N_1+1/2}_x}
  \\
  \nonumber
  & \leq \overline{C}C_0\e + \overline{C}|I|^{1/2} {\big\| \mathcal{N}_3 \big\|}_{L^{2}_t(I)H^{N_0-2}_x}  \\
  & \leq \overline{C}C_0\e + \overline{C}|T_S|^{1/2} \cdot C_0^5 \e^3,
\end{align*}
having used \eqref{cubicN} for the last inequality.
Since $T_S \leq c_0\e^{-4}$ and $c_0=C_0^{-20}$ is small enough,
and using also \eqref{Struv}, we obtain
\begin{equation*}
{\| u\big\|}_{L^4_t(I){\dot{W}^{N_1,-1/4}_x}}+{\| v\big\|}_{L^4_t(I){\dot{W}^{N_1,-1/4}_x}}
  \leq C^2_0\varepsilon/2,
  \end{equation*}
for any interval $I\subseteq [0,T]$ with $|I| \leq T_S$. The desired bounds \eqref{mai3conc2} follow using also \eqref{pac1}, since $|\partial_x|^{1/2}\phi=\Im(u)+|\partial_x|^{1/2}T_Bh$.

\medskip
\section{Proofs of Lemmas \ref{DNmainpro} and \ref{BVexpands}}\label{TechnicalProofs}

\subsection{Proof of Lemma \ref{DNmainpro}} We adapt the approach in \cite[Appendix B]{IoPu4}.

{\bf{Step 1.}} We use complex analysis to define our main objects. Given any function $f:\T_R\to\mathbb{C}$ let $f^\#$ denote its $2\pi R$-periodic extension to $\R$. In particular, let $h^\#$ denote the periodic extension of the height function $h$. Let
\begin{equation}\label{na2}
\gamma(x):=x+ih^\#(x)\qquad\text{ and }\qquad\Omega^\#:=\{x+iy\in\mathbb{C}:y<h^\#(x)\}.
\end{equation}

For any $f\in H^2(\mathbb{T}_R)$ we define the {\it{perturbed Hilbert transform}}
\begin{equation}\label{na3}
\begin{split}
&(\mathcal{H}_\gamma f)(\alpha):=\frac{1}{\pi i}\mathrm{p.v.} \int_{\mathbb{R}}
\frac{f^\#(\beta)\gamma'(\beta)}{\gamma(\alpha)-\gamma(\beta)}\,d\beta:=\lim_{\eps\to 0}\frac{1}{\pi i}\int_{|\beta-\alpha|\in[\eps,\eps^{-1}]}
\frac{f^\#(\beta)\gamma'(\beta)}{\gamma(\alpha)-\gamma(\beta)}\,d\beta.
\end{split}
\end{equation}
Clearly $\mathcal{H}_\gamma f$ is $2\pi R$-periodic. For any $z\in\Omega^\#$ and $f\in H^2(\mathbb{T}_R)$ we define
\begin{equation}\label{na4}
F_f(z):=\frac{1}{2\pi i}\int_{\mathbb{R}}
\frac{f^\#(\beta)\gamma'(\beta)}{z-\gamma(\beta)}\,d\beta:=\lim_{\lambda\to\infty}\frac{1}{2\pi i}\int_{|\beta-\Re z|\leq\lambda}
\frac{f^\#(\beta)\gamma'(\beta)}{z-\gamma(\beta)}\,d\beta.
\end{equation}
Clearly, $F_f$ is an analytic function in $\Omega^\#$. For $\eps>0$ let
\begin{equation}\label{na5}
(\mathcal{T}_\gamma^\eps f)(\alpha):=F_f(\gamma(\alpha)-i\eps)=\frac{1}{2\pi i}\int_{\mathbb{R}}
\frac{f^\#(\beta)\gamma'(\beta)}{\gamma(\alpha)-i\eps-\gamma(\beta)}\,d\beta.
\end{equation}
It is easy to see that
\begin{equation}\label{na7}
\lim_{\eps\to 0}\mathcal{T}_\gamma^\eps f=\frac{1}{2}(I+\mathcal{H}_\gamma)f
\end{equation}
uniformly on $\mathbb{R}$, for any $f\in H^2(\mathbb{T}_R)$.

{\bf{Step 2.}} The perturbed Hilbert transform can be used to derive explicit formulas for $G(h)\phi$. To justify the calculations we need a technical lemma:

\begin{lemma}\label{Wboundedness}
Assume that $\varep,A$ are as in Proposition \ref{DNmainpro} and $f\in H^3(\T_R)$ satisfies
\begin{equation}\label{na23.1}
\||\partial_x|^{1/2}f\|_{\dot{W}^{1,0}}\leq 1.
\end{equation}
Let
\begin{equation}\label{na20}
(H_0 f)(\alpha):=\frac{-1}{\pi }\mathrm{p.v.} \int_{\mathbb{R}}\frac{f^\#(\beta)}{\alpha-\beta}\,d\beta,
\end{equation}
denote the standard Hilbert transform, and consider the real operators $T_1$ and $T_2$ defined by
\begin{equation}\label{na21}
\begin{split}
&(T_1 f)(\alpha):=\frac{1}{\pi}\int_{\mathbb{R}}
\frac{h^\#(\alpha)-h^\#(\beta)-\partial_\be h^\#(\beta)(\alpha-\beta)}{|\gamma(\alpha)-\gamma(\beta)|^2}f^\#(\beta)\,d\beta,\\
&(T_2 f)(\alpha):=\frac{1}{\pi}\int_{\mathbb{R}}
\frac{h^\#(\alpha)-h^\#(\beta)-\partial_\be h^\#(\beta)(\alpha-\beta)}{|\gamma(\alpha)-\gamma(\beta)|^2}\frac{h^\#(\alpha)-h^\#(\beta)}{\alpha-\beta}f^\#(\beta)\,d\beta.
\end{split}
\end{equation}
For $n\in\mathbb{Z}_+$, we define the real operators
\begin{equation}\label{na8}
 (R_nf)(\alpha):=\frac{1}{\pi}\int_{\mathbb{R}}
\frac{h^\#(\alpha)-h^\#(\beta)-\partial_\be h^\#(\beta)(\alpha-\beta)}{(\alpha-\beta)^2}\Big(\frac{h^\#(\alpha)-h^\#(\beta)}{\alpha-\beta}\Big)^nf^\#(\beta)\,d\beta.
\end{equation}

Then
\begin{equation}\label{na22}
\mathcal{H}_\gamma=iH_0-T_1+iT_2,
\end{equation}
and
\begin{equation}\label{na23}
T_1=\sum_{n\geq 0}(-1)^{n}R_{2n},\qquad T_2=\sum_{n\geq 0}(-1)^{n}R_{2n+1}.
\end{equation}
Moreover, for $a\in\{0,1,2\}$
\begin{equation}\label{na23.7}
\sum_{n\geq a}\|R_nf\|_{\dot{H}^{N_0,1/10}}\lesssim A\varep^{a},\qquad \sum_{n\geq a}\|R_nf\|_{\dot{W}^{N_1,1/10}}\lesssim \varep^{a+1}.
\end{equation}
\end{lemma}

This is similar to \cite[Lemma B.3]{IoPu4}, and we sketch the proof in subsection \ref{zxc100} below. 

{\bf{Step 3.}} We can return now to the proof of Proposition \ref{DNmainpro}. For $\alpha\in[-2,2]$ let $\mathcal{E}^\alpha$ denote the space defined by the norm
\begin{equation*}
\|f\|_{\mathcal{E}^\al}:=A^{-1}\|f\|_{\dot{H}^{N_0+\alpha,1/2}}+\varep^{-1}\|f\|_{\dot{W}^{N_1+\alpha,1/2}}.
\end{equation*}
The Cauchy integral formula applied to the analytic function $F_f$ defined in \eqref{na4} and the limit in \eqref{na7} show that
\begin{equation}\label{na31}
\mathcal{H}_\gamma^2=I\text{ on }\mathcal{E}^{-2}.
\end{equation}

We are now ready to define the conjugate function $\psi$. Indeed, recall that $\|T_1\|_{\mathcal{E}^{0}\to\mathcal{E}^{0}}\leq 1/2$ (due to \eqref{na23}--\eqref{na23.7}), and define $\psi\in\mathcal{E}^{0}$ such that
\begin{equation}\label{na35}
(I+T_1)\psi=(H_0+T_2)\phi.
\end{equation}
Clearly, $\psi$ is real-valued. Let $P:=(I-\mathcal{H}_\gamma)(-i\phi+\psi)$ and notice that $(I+\mathcal{H}_\gamma)P=0$, as a consequence
of \eqref{na31}. Moreover, $\Re P=0$, as a consequence of \eqref{na35}. Therefore $P=0$, using again the fact that $T_1$ is a contraction on $\mathcal{E}^{0}$. To summarize, this shows that if $\phi\in \mathcal{E}^{0}$ is a real-valued function then there is a unique real-valued function $\psi\in \mathcal{E}^{0}$ such that
\begin{equation}\label{na32}
(I-\mathcal{H}_\gamma)(\phi+i\psi)=0.
\end{equation}
Moreover, the function $F:\Omega^\#\to\mathbb{C}$,
\begin{equation}\label{na33}
F(z):=\frac{1}{2\pi i}\int_{\mathbb{R}}
\frac{(\phi+i\psi)^\#(\beta)\gamma'(\beta)}{z-\gamma(\beta)}\,d\beta
\end{equation}
in a bounded analytic function in $\Omega^\#$, satisfying $F(z+2\pi R)=F(z)$, which extends to a $C^1$ function in $\overline{\Omega^\#}$
with the property that $F(x+ih^\#(x))=(\phi+i\psi)^\#(x)$ for any $x\in\mathbb{R}$. 

We decompose now $F=G+iH$ and define
\begin{equation*}
v_1:=\partial_xG=\partial_yH,\qquad v_2:=\partial_yG=-\partial_xH.
\end{equation*}
Notice that $G(x+ih^\#(x))=\phi^\#(x)$ and $H(x+ih^\#(x))=\psi^\#(x)$. Thus
\begin{equation*}
-\partial_x\psi^\#=\partial_yG(x+ih^\#(x))-\partial_xh^\#(x)\partial_xG(x+ih^\#(x))=(G(h)\phi)^\#(x).
\end{equation*}
Let $\widetilde{V}(x):=v_1(x+ih^\#(x))$ and $\widetilde{B}(x):=v_2(x+ih^\#(x))$. Taking derivatives (and dropping the superscripts $\#$ since the functions involved are all $2\pi R$-periodic on $\mathbb{R}$), we have $\partial_x\phi=\widetilde{V}+h_x\widetilde{B}$ and $-\partial_x\psi=\widetilde{B}-h_x\widetilde{V}$, where $h_x=\partial_xh$. Therefore
\begin{equation}\label{na37}
\widetilde{V}=\frac{\partial_x\phi+h_x\partial_x\psi}{1+h_x^2}=V,\qquad \widetilde{B}=\frac{h_x\partial_x\phi-\partial_x\psi}{1+h_x^2}=B,
\end{equation}
where $V,B$ are defined in \eqref{DefBV}. In addition, the definition \eqref{na3} shows that
\begin{equation*}
\partial_x\mathcal{H}_\gamma f=\gamma_x\mathcal{H}_\gamma(\partial_xf/\gamma_x),
\end{equation*}
where $\gamma_x=\partial_x\gamma=1+ih_x$. Therefore the identity \eqref{na32} gives
\begin{equation*}
 (I-\mathcal{H}_\gamma)\Big(\frac{\partial_x\phi+i\partial_x\psi}{\gamma_x}\Big)=0.
\end{equation*}
Notice also that $\partial_x\phi+i\partial_x\psi=(1+ih_x)(V-iB)$, as a consequence of \eqref{na37}. Therefore
\begin{equation*}
 (I-\mathcal{H}_\gamma)\big(V-iB\big)=0.
\end{equation*}

To summarize, we proved the identities
\begin{equation}\label{na41}
G(h)\phi=-\partial_x\psi,\qquad V=\frac{\partial_x\phi+h_x\partial_x\psi}{1+h_x^2},\qquad B=\frac{h_x\partial_x\phi-\partial_x\psi}{1+h_x^2},
\end{equation}
and
\begin{equation}\label{na42}
\partial_x\phi=V+h_xB,\qquad -\partial_x\psi=B-h_xV,\qquad (I-\mathcal{H}_\gamma)\big(V-iB\big)=0.
\end{equation}

Using now the definition \eqref{na35}, for any $g\in\{\partial_x\phi,\partial_x\psi\}$,
\begin{equation}\label{na42.5}
A^{-1}\|g\|_{\dot{H}^{N_0-1,-1/2}}+\varep^{-1}\|g\|_{\dot{W}^{N_1-1,-1/2}}\lesssim 1.
\end{equation}
Using also Lemma \ref{algeProp}, for any $f\in\{\partial_x\phi,\partial_x\psi,V,B\}$ we have
\begin{equation}\label{na43}
A^{-1}\|f\|_{H^{N_0-1}}+\varep^{-1}\|f\|_{\widetilde{W}^{N_1-1}}\lesssim 1.
\end{equation}

{\bf{Step 4.}} It remains to prove the paralinearization identities in part (iii), in particular the quartic bounds \eqref{DNmain4} for the remainders $G_{\geq 4}$ defined according to \eqref{DNmainformula}. The main issue is to avoid loss of derivative, so we will prove first that
\begin{equation}\label{ctu5.5}
\big\|P_{\geq 0}G_{\geq 4}\big\|_{H^{N_0}}\lesssim A\varep^3\quad\text{ and }\quad\big\|P_{\geq 0}G_{\geq 4}\big\|_{\wt{W}^{N_1}}\lesssim \varep^4.
\end{equation}

Taking real and imaginary parts of the identity $(I-\mathcal{H}_\gamma)\big(V-iB\big)=0$ we have
\begin{equation}\label{na74}
V-H_0B=-T_1V+T_2B,\qquad -B-H_0V=T_2V+T_1B.
\end{equation}
Using also the formulas $-\partial_x\psi=B-h_xV$, see \eqref{na42}, and $|\partial_x|\phi=-H_0(\partial_x\phi)$, we have
\begin{equation}\label{na75}
\begin{split}
|\partial_x|\phi-G(h)\phi&=-H_0(\partial_x\phi)-B+h_xV\\
&=T_2V+T_1B+h_xV-H_0(h_xB).
\end{split}
\end{equation}

Let
\begin{equation}\label{na76}
D:=T_2V+T_1B=-H_0V-B.
\end{equation}
With $G_2$ defined as in \eqref{DNmain2}, let 
\begin{equation}\label{na77}
\begin{split}
G_{\geq 3}&:=G(h)\phi-|\partial_x|\phi+|\partial_x|T_Bh+\partial_xT_Vh-G_2=I+II,\\
I:&=-h_xV-H_0(h_xH_0V)-|\partial_x|T_{H_0V}h+\partial_xT_Vh-G_2+R_0(H_0V),\\
II:&=-D-R_0(H_0V)-H_0(h_xD)-|\partial_x|T_Dh.
\end{split}
\end{equation}

The multiplier of the Hilbert transform $H_0$ is $\eta\to i\,\sgn(\eta)$. Using the definitions and \eqref{na16}, we write
\begin{equation*}
\widehat{I}(\xi)=\frac{1}{2\pi R}\sum_{\eta\in\Z/R}\widehat{h}(\eta)\widehat{V}(\xi-\eta)p(\xi,\eta)-\widehat{G_2}(\xi)
\end{equation*}
where
\begin{equation*}
\begin{split}
p(\xi,\eta):&=-i\eta+i\eta\sgn(\xi)\sgn(\xi-\eta)-i|\xi|\sgn(\xi-\eta)\chi(\xi-\eta,\eta)+i\xi\chi(\xi-\eta,\eta)\\
&-i(\xi-\eta)[1-\sgn(\xi)\sgn(\xi-\eta)]\\
&=-i\xi[1-\chi(\xi-\eta,\eta)][1-\sgn(\xi)\sgn(\xi-\eta)].
\end{split}
\end{equation*}
Using also the formula $V=\partial_x\phi-h_xB$ (see \eqref{na42}), we have
\begin{equation}\label{ctu1}
\widehat{I}(\xi)=\frac{1}{2\pi R}\sum_{\eta\in\Z/R}
  \widehat{h}(\eta)\mathcal{F}(V-\partial_x\phi)(\xi-\eta)p(\xi,\eta)
  =\frac{-1}{2\pi R}\sum_{\eta\in\Z/R}\widehat{h}(\eta)\mathcal{F}(h_xB)(\xi-\eta)p(\xi,\eta).
\end{equation}

Using the definitions we calculate now
\begin{equation}\label{ctu2}
\begin{split}
II&=-T_2V-T_1B-R_0(H_0V)-[H_0(h_xT_1B)+|\partial_x|T_{T_1B}h]-[H_0(h_xT_2V)+|\partial_x|T_{T_2V}h]\\
&=-R_1V-R_0B-R_0(H_0V)-[H_0(h_xR_0B)+|\partial_x|T_{R_0B}h]+G_{\geq 4}^1,
\end{split}
\end{equation}
where
\begin{equation}\label{ctu3}
\begin{split}
G_{\geq 4}^1&:=-(T_2-R_1)V-(T_1-R_0)B-[H_0(h_x(T_1-R_0)B)+|\partial_x|T_{(T_1-R_0)B}h]\\
&-[H_0(h_xT_2V)+|\partial_x|T_{T_2V}h].
\end{split}
\end{equation}
Notice that, for any $F\in H^2$,
\begin{equation}\label{ctu4}
\mathcal{F}\big\{H_0(h_xF)+|\partial_x|T_Fh\big\}(\xi)=\frac{1}{2\pi R}\sum_{\eta\in\Z/R}\widehat{h}(\eta)\widehat{F}(\xi-\eta)\,\sgn(\xi)[\xi\chi(\xi-\eta,\eta)-\eta].
\end{equation}
Assuming now that $F\in\{(T_1-R_0)B,\,T_2V\}$, we notice that $\|F\|_{\dot{H}^{N_0,1/10}}\lesssim A\varep^2$ and $\|F\|_{\dot{W}^{N_1,1/10}}\lesssim\varep^3$, as a consequence of \eqref{na23.7}. Using the formula \eqref{ctu4} it is easy to see that
\begin{equation*}
\begin{split}
&\big\|P_{\geq 0}\big\{H_0(h_xF)+|\partial_x|T_Fh\big\}\big\|_{H^{N_0}}\lesssim A\varep^3,\\
&\big\|P_{\geq 0}\big\{H_0(h_xF)+|\partial_x|T_Fh\big\}\big\|_{\widetilde{W}^{N_1}}\lesssim \varep^4,
\end{split}
\end{equation*}
for $G\in\{(T_1-R_0)B,\,T_2V\}$. Using again \eqref{na23.7} 
together with the bounds \eqref{na43} we see that the terms $P_{\geq 0}(T_2-R_1)V$ and $P_{\geq 0}(T_1-R_0)V$ 
satisfy similar quartic $H^{N_0}$ bounds. 
Therefore $P_{\geq 0}G_{\geq 4}^1$ is an acceptable quartic error,
\begin{equation}\label{ctu5}
\big\|P_{\geq 0}G_{\geq 4}^1\big\|_{H^{N_0}}
  \lesssim A\varep^3\quad\text{ and }\quad\big\|P_{\geq 0}G_{\geq 4}^1\big\|_{\wt{W}^{N_1}}\lesssim \varep^4.
\end{equation}

We examine now the main terms in the right-hand side of \eqref{ctu2}. Using \eqref{na74} we write
\begin{equation*}
\begin{split}
&-R_1V-R_0B-R_0(H_0V)=-R_1(H_0B)+R_0(R_0B)+G_{\geq 4}^2,\\
&G_{\geq 4}^2:=R_1(T_1V-T_2B)+R_0(T_2V)+R_0(T_1-R_0)B.
\end{split}
\end{equation*}
Using \eqref{na23.7} and \eqref{na43} as before, we see that $P_{\geq 0}G_{\geq 4}^2$ is an acceptable quartic error, satisfying bounds similar to \eqref{ctu5}. Moreover, using \eqref{na19}, \eqref{na16}, and \eqref{ctu4}, we write the cubic terms in $II$ in the form
\begin{equation*}
\begin{split}
-\mathcal{F}[R_1(H_0B)](\xi)&=\frac{1}{(2\pi R)^2}\sum_{\eta,\rho\in\Z/R}\widehat{h}(\eta)\widehat{h}(\rho)\widehat{B}(\xi-\eta-\rho)p_1(\xi,\eta,\rho),\\
\mathcal{F}[R_0(R_0B)](\xi)&=\frac{1}{(2\pi R)^2}\sum_{\eta,\rho\in\Z/R}\widehat{h}(\eta)\widehat{h}(\rho)\widehat{B}(\xi-\eta-\rho)p_2(\xi,\eta,\rho),\\
-\mathcal{F}[H_0(h_xR_0B)+|\partial_x|T_{R_0B}h](\xi)&=\frac{1}{(2\pi R)^2}\sum_{\eta,\rho\in\Z/R}\widehat{h}(\eta)\widehat{h}(\rho)\widehat{B}(\xi-\eta-\rho)p_3(\xi,\eta,\rho),
\end{split}
\end{equation*}
where
\begin{equation*}
\begin{split}
p_1(\xi,\eta,\rho)&:=(1/2)|\xi-\eta-\rho|[|\xi|+|\xi-\eta-\rho|-|\xi-\eta|-|\xi-\rho|],\\
p_2(\xi,\eta,\rho)&:=[\sgn(\xi)-\sgn(\xi-\eta)](\xi-\eta)[\sgn(\xi-\eta)-\sgn(\xi-\eta-\rho)](\xi-\eta-\rho),\\
p_3(\xi,\eta,\rho)&:=-\sgn(\xi)[\xi\chi(\xi-\eta,\eta)-\eta][\sgn(\xi-\eta)-\sgn(\xi-\eta-\rho)](\xi-\eta-\rho).
\end{split}
\end{equation*}
Recalling also \eqref{ctu1}, we have
\begin{equation*}
\begin{split}
\widehat{I}(\xi)&=\frac{1}{(2\pi R)^2}\sum_{\eta,\rho\in\Z/R}\widehat{h}(\eta)\widehat{h}(\rho)\widehat{B}(\xi-\eta-\rho)p_4(\xi,\eta,\rho),\\
p_4(\xi,\eta,\rho)&:=-\xi\rho[1-\chi(\xi-\eta,\eta)][1-\sgn(\xi)\sgn(\xi-\eta)].
\end{split}
\end{equation*}

Using also \eqref{na77} and symmetrizing the multipliers in $\eta$ and $\rho$, we get
\begin{equation*}
\mathcal{F}\{G_{\geq 3}\}(\xi)=\frac{1}{(2\pi R)^2}\sum_{\eta,\rho\in\Z/R}\widehat{h}(\eta)\widehat{h}(\rho)\widehat{B}(\xi-\eta-\rho)p^\ast(\xi,\eta,\rho)+\mathcal{F}\{G_{\geq 4}^1\}(\xi)+\mathcal{F}\{G^2_{\geq 4}\}(\xi),
\end{equation*}
where
\begin{equation*}
\begin{split}
p^\ast&(\xi,\eta,\rho):=\frac{1}{2}|\xi|[1-\chi(\xi-\eta,\eta)][|\xi-\eta|-|\xi-\eta-\rho|-\rho\,\sgn(\xi)]\\
&+\frac{1}{2}|\xi|[1-\chi(\xi-\rho,\rho)][|\xi-\rho|-|\xi-\eta-\rho|-\eta\,\sgn(\xi)]+\frac{1}{2}\big[|\xi||\xi-\eta-\rho|-\xi(\xi-\eta-\rho)\big].
\end{split}
\end{equation*}
We notice now that the symbol $p^\ast$ is symmetric in $\eta$ and $\rho$ and smoothing, in the sense that $p^\ast(\xi,\eta,\rho)=0$ if $|\eta|+|\rho|\leq 2^{-40}|\xi|$ or if $|\eta|+|\xi-\eta-\rho|\leq 2^{-40}|\xi|$ or if $|\rho|+|\xi-\eta-\rho|\leq 2^{-40}|\xi|$. Moreover, for any $k, k_1, k_2, k_3\in\Z$ we have
\begin{equation*}
\big\|\mathcal{F}^{-1}\big[p^\ast(\xi,\eta,\rho)\varphi_k(\xi)\varphi_{\leq k_1}(\eta)\varphi_{\leq k_2}(\rho)\varphi_{\leq k_3}(\xi-\eta-\rho)\big]\big\|_{L^1(\T_R^3)}\lesssim 2^k2^{\max(k_1,k_2,k_3)},
\end{equation*}
These bounds follow easily from the explicit formula of $p^\ast$ above. Therefore, using \eqref{na42} and \eqref{na35} we may replace $B$ with $-\partial_x\psi$ and then with $|\partial_x|\phi$ at the expenses of acceptable quartic  errors. The desired bounds \eqref{ctu5.5} follow.

{\bf{Step 5.}} Finally, we show that the low frequencies of the quartic remainders satisfy
\begin{equation}\label{ctu5.6}
\|P_{\leq 0}G_{\geq 4}\|_{\dot{H}^{N_0,-3/4}}\lesssim A\varep^3,\qquad \|P_{\leq 0}G_{\geq 4}\|_{\dot{W}^{N_1,-3/4}}\lesssim \varep^4.
\end{equation}
The proof is easier than before, but we need to write the remainders $G_{\geq 4}$ differently, as perfect derivatives. We start from the definition \eqref{na35}, which we rewrite in the form
\begin{equation*}
(I+R_0)\psi=(H_0+R_1)\phi+(T_2-R_1)\phi-(T_1-R_0)\psi.
\end{equation*}
We apply the operator $-\partial_x(I-R_0+R_0^2)$ to both sides, so
\begin{equation}\label{ctu5.7}
\begin{split}
-\partial_x\psi&=-\partial_x(I-R_0+R_0^2)H_0\phi-\partial_xR_1\phi+G_{\geq 4}^4,\\
G_{\geq 4}^4&:=\partial_x(R_0-R_0^2)R_1\phi-\partial_x(I-R_0+R_0^2)(T_2-R_1)\phi\\
&+\partial_x(I-R_0+R_0^2)(T_1-R_0)\psi+\partial_xR_0^3\psi.
\end{split}
\end{equation}
Recalling now that $G(h)\phi=-\partial_x\psi$ we have
\begin{equation}\label{ctu5.8}
\begin{split}
G(h)\phi&-\big\{|\partial_x|\omega- \partial_x(T_V h) + G_2(\phi,h) + G_3(\phi,h,h)\big\}\\
&=\big\{\partial_xR_0H_0\phi+|\partial_x|T_Bh+\partial_xT_Vh-G_2(\phi,h)\big\}\\
&-\big\{\partial_xR_0^2H_0\phi+\partial_xR_1\phi+G_3(\phi,h,h)\big\}+G_{\geq 4}^4.
\end{split}
\end{equation}

It follows from the formulas in \eqref{ctu5.7} and Lemma \ref{Wboundedness} that
\begin{equation}\label{ctu5.9}
\|G_{\geq 4}^4\|_{\dot{H}^{N_0-1,-3/4}}\lesssim A\varep^3,\qquad \|G_{\geq 4}^4\|_{\dot{W}^{N_1-1,-3/4}}\lesssim\varep^4,
\end{equation}
and
\begin{equation}\label{ctu5.10}
\begin{split}
&G(h)\phi=|\partial_x|\phi+\partial_xR_0H_0\phi+[G(h)\phi]_{\geq 3},\\
&\|[G(h)\phi]_{\geq 3}\|_{\dot{H}^{N_0-1,-3/4}}\lesssim A\varep^2,\qquad \|[G(h)\phi]_{\geq 3}\|_{\dot{W}^{N_1-1,-3/4}}\lesssim\varep^3.
\end{split}
\end{equation}
Using the identities \eqref{na41} and Lemma \ref{algeProp} (i) it follows that
\begin{equation}\label{ctu5.11}
\begin{split}
&V=\partial_x\phi-\partial_xh|\partial_x|\phi+V_{\geq 3},\\
&B=|\partial_x|\phi+\partial_xR_0H_0\phi+\partial_xh\partial_x\phi+B_{\geq 3},
\end{split}
\end{equation}
where the cubic remainders $V_{\geq 3}$ and $B_{\geq 3}$ satisfy
\begin{equation}\label{ctu5.12}
\|V_{\geq 3}\|_{H^{N_0-1}}+\|B_{\geq 3}\|_{H^{N_0-1}}\lesssim A\varep^2,\qquad \|V_{\geq 3}\|_{\wt{W}^{N_1-1}}+\|B_{\geq 3}\|_{\wt{W}^{N_1-1}}\lesssim\varep^3.
\end{equation}

We examine now the formula \eqref{ctu5.8}. The quadratic term in the right-hand side is
\begin{equation*}
\partial_xR_0H_0\phi+|\partial_x|T_{|\partial_x|\phi}h+\partial_xT_{\partial_x\phi}h-G_2(\phi,h),
\end{equation*}
which is identically equal to $0$ due to the formulas \eqref{na16} and \eqref{DNmain2}. The cubic term is 
\begin{equation*}
|\partial_x|T_{\partial_xR_0H_0\phi+\partial_xh\partial_x\phi}h+\partial_xT_{-\partial_xh|\partial_x|\phi}h-\big\{\partial_xR_0^2H_0\phi+\partial_xR_1\phi+G_3(\phi,h,h)\big\},
\end{equation*}
due to the expansion \eqref{ctu5.11}. An explicit calculation using the formulas \eqref{DNmain3}, \eqref{na13.3}, \eqref{na15}, and \eqref{na19} shows that this expression vanishes identically as well. The remaining terms $|\partial_x|T_{B_\geq 3}h$ and $\partial_xT_{V_\geq 3}h$ are acceptable quartic remainders, satisfying bounds similar to \eqref{ctu5.9} as a consequence of \eqref{ctu5.12}. The desired bounds \eqref{ctu5.6} follow using \eqref{ctu5.8} and \eqref{ctu5.9}.

\subsection{Proof of Lemma \ref{Wboundedness}}\label{zxc100} The identities \eqref{na22}-\eqref{na23} follow directly from definitions. The main point is to prove the bounds \eqref{na23.7}. 

{\bf{Step 1.}} We rewrite first the operators $R_n$ in the form
\begin{equation*}
 (R_nf)(\alpha)=\frac{1}{\pi}\int_{\mathbb{R}}
\frac{h^\#(\alpha)-h^\#(\al+\rho)+\rho{h'}^\#(\al+\rho)}{\rho^2}\Big(\frac{h^\#(\alpha+\rho)-h^\#(\alpha)}{\rho}\Big)^nf^\#(\al+\rho)\,d\rho.
\end{equation*}
Notice that
\begin{equation*}
\frac{h^\#(\alpha+\rho)-h^\#(\alpha)}{\rho}=\frac{1}{2\pi R}\sum_{\eta\in\Z/R}\widehat{h}(\eta)e^{i\eta\alpha}\frac{e^{i\eta\rho}-1}{\rho}.
\end{equation*}
Therefore, we can take the Fourier transform on $\T_R$ and rewrite
\begin{equation}\label{na12}
\begin{split}
\mathcal{F}\big[R_nf\big](\xi)=\frac{1}{\pi(2\pi R)^{n+1}}&\sum_{\eta_1,\ldots,\eta_{n+1}\in\Z/R}
\widetilde{M}_{n+1}(\xi;\eta_1,\ldots,\eta_{n+1})\\
&\times\widehat{f}(\xi-\eta_1-\ldots-\eta_{n+1})\widehat{h}(\eta_1)\cdot\ldots\cdot\widehat{h}(\eta_{n+1}),
\end{split}
\end{equation}
where
\begin{equation}\label{na13}
\begin{split}
\widetilde{M}_{n+1}(\xi;\eta_1,\ldots,\eta_{n+1}):=\int_{\mathbb{R}}\frac{e^{-i\xi\rho}}{\rho}\frac{e^{i\eta_{n+1}\rho}-1-i\eta_{n+1}\rho}{\rho}\prod_{l=1}^n\frac{e^{i\eta_l\rho}-1}{\rho}\,d\rho.
\end{split}
\end{equation}

As in \cite[Subsection B.2]{IoPu4} we can symmetrize in the variables $\eta_1,\ldots,\eta_{n+1}$ to get more fa\-vo\-rable formulas. Indeed, letting
\begin{equation}\label{na13.2}
M_{n+1}(\xi;\underline{\eta}):=\frac{-i(\xi-\eta_1-\ldots-\eta_{n+1})}{n+1}\int_{\mathbb{R}}e^{-i\xi\rho}\prod_{l=1}^{n+1}\frac{e^{i\eta_l\rho}-1}{\rho}\,d\rho,
\end{equation}
where $\underline{\eta}:=(\eta_1,\ldots,\eta_{n+1})$, it follows by symmetrization from \eqref{na12} that
\begin{equation}\label{na13.3}
\begin{split}
\widehat{R_nf}(\xi)=\frac{1}{\pi(2\pi R)^{n+1}}\sum_{\eta_1,\ldots,\eta_{n+1}\in \Z/R}M_{n+1}(\xi;\underline{\eta})&\widehat{f}(\xi-\eta_1-\ldots-\eta_{n+1})\widehat{h}(\eta_1)\cdot\ldots\cdot\widehat{h}(\eta_{n+1}).
\end{split}
\end{equation}

{\bf{Step 2.}} We consider first the operators $R_0$ and $R_1$, for which we can derive explicit formulas. Recall the general (formal) identities
\begin{equation}\label{na14}
\int_{\mathbb{R}}e^{-i\xi x}\,dx=2\pi\delta_0(\xi),\qquad \int_{\mathbb{R}}e^{-i\xi x}\frac{1}{x}\,dx=-i\pi\,\mathrm{sgn}(\xi),\qquad\int_{\mathbb{R}}\frac{e^{-i\xi x}-1}{x^2}\,dx=-\pi\,|\xi|,
\end{equation}
for any $\xi\in\mathbb{R}$. Using these formulas, the symbol $M_1(\xi;\eta_1)$ can be calculated easily,
\begin{equation}\label{na15}
M_{1}(\xi;\eta_{1})=\pi(\xi-\eta_1)\big[\,\mathrm{sgn}(\xi)-\,\mathrm{sgn}(\xi-\eta_1)\big].
\end{equation}
Therefore
\begin{equation}\label{na16}
\widehat{R_0f}(\xi)=\frac{1}{2\pi R}\sum_{\eta_1\in\Z/R}\big[\,\mathrm{sgn}(\xi)-\,\mathrm{sgn}(\xi-\eta_1)\big](\xi-\eta_1)\widehat{f}(\xi-\eta_1)\widehat{h}(\eta_1).
\end{equation}

The main observation is that the multiplier of the operator $R_0$ vanishes when $|\eta_1|\leq|\xi|$. Recalling that $\||\partial_x|^{1/2}f\|_{\dot{W}^{1,0}}\lesssim 1$ we can use Lemma \ref{touse} to estimate, for $p\in\{2,\infty\}$,
\begin{equation*}
\begin{split}
\|P_kR_0f\|_{L^p}&\lesssim \sum_{k_2+8\geq \max(k,k_1)}2^{k_1}\|P'_{k_1}f\|_{L^\infty}\|P'_{k_2}h\|_{L^p}\lesssim \sum_{k_2+8\geq k}\|P'_{k_2}h\|_{L^p},
\end{split}
\end{equation*}
for any $k\in\overline{\mathbb{Z}}$. Thus
\begin{equation}\label{na16.4}
\|R_0f\|_{\dot{H}^{N_0,1/10}}\lesssim A,\qquad \|R_0f\|_{\dot{W}^{N_1,1/10}}\lesssim \varep.
\end{equation}

We can analyze similarly the operator $R_1$. Using \eqref{na14}, its multiplier is
\begin{equation}\label{na19}
M_{2}(\xi;\eta_{1},\eta_2)= \frac{i\pi}{2} (\xi-\eta_1-\eta_{2}) \big[|\xi-\eta_1-\eta_2|+|\xi|-|\xi-\eta_1|-|\xi-\eta_2|\big].
\end{equation}
As before, we observe that $M_{2}(\xi;\eta_{1},\eta_2)=0$ if $|\eta_1|+|\eta_2|\leq|\xi|$. Moreover, it is easy to see that
\begin{equation*}
\Big\|\mathcal{F}^{-1}\big[M_{2}(\xi;\eta_{1},\eta_2)\varphi_{k_1}(\eta_1)\varphi_{k_2}(\eta_2)\varphi_{k_3}(\xi-\eta_1-\eta_2)\varphi_k(\xi)\big]\Big\|_{L^1(\mathbb{R}^3)}\lesssim 2^{k_3}2^{\min(k_1,k_2)}.
\end{equation*}
Therefore, using Lemma \ref{touse}, the assumption $\||\partial_x|^{1/2}f\|_{\dot{W}^{1,0}}\lesssim 1$, and the bounds \eqref{Asshu} on $h$,
\begin{equation*}
\begin{split}
\|P_kR_1f\|_{L^p}&\lesssim \sum_{k_2+10\geq\max(k,k_1,k_3)}2^{k_1+k_3}\|P'_{k_1}h\|_{L^\infty}\|P'_{k_2}h\|_{L^p}\|P'_{k_3}f\|_{L^\infty}\lesssim \e\sum_{k_2+10\geq k}\|P'_{k_2}h\|_{L^p},
\end{split}
\end{equation*}
for $p\in\{2,\infty\}$ and any $k\in\overline{\mathbb{Z}}$. Therefore
\begin{equation}\label{na19.8}
\|R_1f\|_{\dot{H}^{N_0,1/10}}\lesssim A\varep,\qquad \|R_1f\|_{\dot{W}^{N_1,1/10}}\lesssim \varep^2.
\end{equation}

{\bf{Step 3.}} To analyze the operators $R_n$, $n\geq 2$, we prove in fact the slightly stronger bounds \eqref{na152} below, which clearly suffice to  complete the proof of the desired estimates \eqref{na23.7}.

\begin{lemma}\label{na149}
For any $n\geq 1$ and $\underline{k}=(k_1,\ldots,k_{n+1})\in\overline{\mathbb{Z}}^{n+1}$ we define the operators $F_{n;\underline{k}}$ by
\begin{equation}\label{na150}
\begin{split}
\mathcal{F}\big(F_{n;\underline{k}}(f_{\underline{k}})\big)(\xi):=\frac{1}{(2\pi R)^{n+1}}&\sum_{\eta_1,\ldots,\eta_{n+1}\in\Z/R}M'_{n+1}(\xi;\underline{\eta})\varphi_{k_1}(\eta_1)\ldots \varphi_{k_{n+1}}(\eta_{n+1})\\
&\times\widehat{f_{\underline{k}}}(\xi-\eta_1-\ldots-\eta_{n+1})\widehat{h}(\eta_1)\cdot\ldots\cdot\widehat{h}(\eta_{n+1}),
\end{split}
\end{equation}
where 
\begin{equation}\label{na150.1}
M'_{n+1}(\xi;\underline{\eta}):=\int_{\mathbb{R}}e^{-i\xi\rho}\prod_{l=1}^{n+1}\frac{e^{i\eta_l\rho}-1}{\rho}\,d\rho.
\end{equation}
Assume that $h$ satisfies the bounds in \eqref{Asshu} and the functions $f_{\underline{k}}$ satisfy the uniform bounds
\begin{equation}\label{na151}
\|f_{\underline{k}}\|_{\widetilde{W}^{1/2}}\leq 1.
\end{equation}
Then there is a constant $C_0\geq 1$ such that if $n\geq 2$, $p\in\{2,\infty\}$, and $l\in\overline{\Z}$ then
\begin{equation}\label{na152}
\sum_{\underline{k}\in\overline{\mathbb{Z}}^{n+1}}\big\|P_l[F_{n;\underline{k}}(f_{\underline{k}})]\big\|_{L^p}\leq (C_0\e)^{n}\big(\sum_{l'\geq l-6n}\|P_{l'}h\|_{L^p}\big).
\end{equation}
\end{lemma}

\begin{proof}
Notice that
\begin{equation}\label{na201}
M'_{n+1}(\xi,\eta_1,\ldots,\eta_{n+1})=0\qquad\text{ if }\qquad \eta_1\cdot\ldots\cdot\eta_{n+1}=0.
\end{equation}
Taking partial derivatives in $\underline\eta$ it follows that
\begin{equation}\label{na202}
M'_{n+1}(\xi,\eta_1,\ldots,\eta_{n+1})=0\qquad\text{ if }\qquad |\eta_1|+\ldots+|\eta_{n+1}|\leq|\xi|.
\end{equation}

Using the Fourier inversion formula we rewrite, in the physical space,
\begin{equation}\label{na205}
F_{n,\underline{k}}(f_{\underline{k}})(x)=\int_{\mathbb{R}}\int_{\T_R^{n+1}}f_{\underline{k}}(x-\rho)\prod_{l=1}^{n+1}P'_{k_l}h(x-\rho-y_l)\cdot \prod_{l=1}^{n+1}L_{k_l}(y_l,\rho)\,d\underline{y}d\rho ,
\end{equation}
where, with $K_k$ defined as in \eqref{Pkconv1},
\begin{equation}\label{na206}
L_k(y,\rho):=\frac{1}{2\pi R}\sum_{\eta\in\Z/R}\varphi_k(\eta)e^{iy\eta}\frac{e^{i\eta\rho}-1}{\rho}=\frac{K_k(y+\rho)-K_k(y)}{\rho}.
\end{equation}
Using the identities \eqref{Pkconv2} it is easy to see that
\begin{equation}\label{na207}
\int_{\T_R}\big|L_k(y,\rho)\big|\,dy\lesssim \min\big(|\rho|^{-1},2^k\big).
\end{equation}
Therefore, assuming that $\underline{k}=(k_1,\ldots,k_{n+1})\in\Z^{n+1}$ and letting $\widetilde{k}_1\leq\ldots\leq\widetilde{k}_{n+1}$ denote the nondecreasing rearrangement of the integers $k_1,\ldots,k_{n+1}$, we have
\begin{equation}\label{na208}
\int_{\mathbb{R}}\int_{\T_R^{n+1}}\Big|\prod_{l=1}^{n+1}L_{k_l}(y_l,\rho)\Big|\, d\underline{y} d\rho\leq C^{n+1} 2^{\widetilde{k}_1+\ldots+\widetilde{k}_n}(1+|\widetilde{k}_{n+1}-\widetilde{k}_n|),
\end{equation}
for some constant $C\geq 1$ (which is allowed to change from now on from line to line). 

Therefore, using \eqref{na205}, for any $\underline{k}\in\Z^{n+1}$ and $p\in\{2,\infty\}$,
\begin{equation}\label{na209}
\big\|F_{n,\underline{k}}(f_{\underline{k}})\big\|_{L^p}\leq C^{n+1}(1+|\widetilde{k}_{n+1}-\widetilde{k}_n|)\cdot\|f_{\underline{k}}\|_{L^\infty}\|P'_{\widetilde{k}_{n+1}}h\|_{L^p}\prod_{l=1}^n\big(2^{\widetilde{k}_l}\|P'_{\widetilde{k}_{l}}h\|_{L^\infty}\big).
\end{equation}
We notice that the bounds \eqref{na209} suffice to prove \eqref{na152} if $l\leq 0$, due to the identity \eqref{na202}. On the other hand, if $l\geq 0$ then these bounds suffice to control the part of the sum in \eqref{na152} corresponding to $P_{>\widetilde{k}_{n+1}-6n}f_{\underline{k}}$, and also the part of the sum corresponding to multi-indices $\underline{k}$ for which $\widetilde{k}_n\geq\widetilde{k}_{n+1}-6n$. After these reductions for \eqref{na152} it remains to prove that
\begin{equation}\label{na213}
\sum_{k_2,\ldots,k_{n+1}\leq k_1-6n}\Big\|F_{n,\underline{k}}(P_{\leq k_1-6n}f_{\underline{k}})\Big\|_{L^p}\leq (C_1\varep)^{n} \|P'_{k_1}h\|_{L^p},
\end{equation}
for $p\in\{2,\infty\}$, $k_1\geq -6$, and some constant $C_1\geq 1$.

The bounds \eqref{na209} cannot be used directly to prove \eqref{na213}, due to derivative loss. We use instead induction over $n$. The case $n=1$ follows from our earlier discussion of the operator $R_1$. To prove the induction step we decompose
\begin{equation*}
L_{k_{n+1}}(y_{n+1},\rho)=\widetilde{L}_{k_{n+1}}(y_{n+1},\rho)+K'_{k_{n+1}}(y_{n+1}),
\end{equation*}
where $K'_k(y)=(\partial_y K_k)(y)$ and 
\begin{equation*}
\widetilde{L}_k(y,\rho):=\frac{K_k(y+\rho)-K_k(y)-\rho K'_k(y)}{\rho}.
\end{equation*}
Then we decompose
\begin{equation*}
F_{n,\underline{k}}(P_{\leq k_1-6n}f_{\underline{k}})=F^1_{n,\underline{k}}(P_{\leq k_1-6n}f_{\underline{k}})
+F^2_{n,\underline{k}}(P_{\leq k_1-6n}f_{\underline{k}}),
\end{equation*}
where
\begin{equation}\label{na216}
F^1_{n,\underline{k}}(g)(x):=\int_{\mathbb{R}}\int_{\mathbb{T}_R^{n+1}}g(x-\rho)\prod_{l=1}^{n+1}P'_{k_l}h(x-\rho-y_l)\cdot \widetilde{L}_{k_{n+1}}(y_{n+1},\rho)\prod_{l=1}^{n}L_{k_l}(y_l,\rho)\, d\underline{y}d\rho,
\end{equation}
and
\begin{equation}\label{na217}
F^2_{n,\underline{k}}(g)(x):=\int_{\mathbb{R}}\int_{\mathbb{T}_R^{n+1}}g(x-\rho)\prod_{l=1}^{n+1}P'_{k_l}h(x-\rho-y_l)\cdot K'_{k_{n+1}}(y_{n+1})\prod_{l=1}^{n}L_{k_l}(y_l,\rho)\, d\underline{y}d\rho.
\end{equation}

Using \eqref{Pkconv2} it is easy to see that the kernels $\widetilde{L}_k$ satisfy the better estimates
\begin{equation*}
\int_{\mathbb{T}_R}\big|\widetilde{L}_k(y,\rho)\big|\,dy\lesssim 2^k|\rho|\min\big(|\rho|^{-1},2^k\big),
\end{equation*}
compare with \eqref{na207}. Using also \eqref{na207} it follows that if $k_2,\ldots,k_{n+1}\in (-\infty,k_1-6n]$ then
\begin{equation*}
\int_{\mathbb{R}}\int_{\mathbb{T}_R^{n+1}}\Big|\widetilde{L}_{k_{n+1}}(y_{n+1},\rho)\prod_{l=1}^{n}L_{k_l}(y_l,\rho)\Big|\, d\underline{y}d\rho\leq C^{n+1}(1+|\widetilde{k}_n-\widetilde{k}_{n-1}|)2^{\widetilde{k}_1+\ldots+\widetilde{k}_n},
\end{equation*}
where $\widetilde{k}_1\leq\ldots\leq\widetilde{k}_{n}$ denote the nondecreasing rearrangement of the integers $k_2,\ldots,k_{n+1}$. This is slightly stronger than the inequality \eqref{na208}. Therefore
\begin{equation*}
\Big\|F^1_{n,\underline{k}}(P_{\leq k_1-6n}f_{\underline{k}})\Big\|_{L^p}\leq C^{n+1}(1+|\widetilde{k}_n-\widetilde{k}_{n-1}|)\cdot\|f_{\underline{k}}\|_{L^\infty}\|P'_{k_1}h\|_{L^p}\prod_{l=2}^{n+1}\big(2^{k_l}\|P'_{k_{l}}h\|_{L^\infty}\big).
\end{equation*}
Therefore, for $p\in\{2,\infty\}$ and $k_1\geq -6$,
\begin{equation}\label{na218}
\sum_{k_2,\ldots,k_{n+1}\leq k_1-6n}\Big\|F^1_{n,\underline{k}}(P_{\leq k_1-6n}f_{\underline{k}})\Big\|_{L^p}\leq (C\varep)^{n} \|P'_{k_1}h\|_{L^p}.
\end{equation}

We estimate now the contributions of the terms $F^2_{n,\underline{k}}(P_{\leq k_1-6n}f_{\underline{k}})$. We integrate the variable $y_{n+1}$ in the defining formula \eqref{na217}. Let
\begin{equation*}
g^{k_{n+1}}_{(k_1,\ldots,k_{n})}(z):=P_{\leq k_1-6n}f_{(k_1,\ldots,k_{n+1})}(z)\int_{\mathbb{T}_R}P'_{k_{n+1}}h(z-y_{n+1})K'_{k_{n+1}}(y_{n+1})\,dy_{n+1}.
\end{equation*}
Using Lemma \ref{algeProp} (i) and the assumptions \eqref{na151} and \eqref{Asshu}, it is easy to see that
\begin{equation*}
\|g^{k_{n+1}}_{(k_1,\ldots,k_{n})}\|_{\widetilde{W}^{1/2}}\lesssim\varep 2^{-|k_{n+1}|/2},
\end{equation*}
for any $k_1,\ldots,k_{n+1}$ as in \eqref{na213}. Since 
\begin{equation*}
F^2_{n,\underline{k}}(P_{\leq k_1-6n}f_{\underline{k}})=F_{n-1,\underline{k}}(P_{\leq k_1-6n+6}g_{\underline{k}}),
\end{equation*}
the induction hypothesis shows that
\begin{equation*}
\sum_{k_2,\ldots,k_{n+1}\leq k_1-6n}\Big\|F^2_{n,\underline{k}}(P_{\leq k_1-6n}f_{\underline{k}})\Big\|_{L^p}\lesssim \sum_{k_{n+1}\in\Z}(C_1\varep)^{n-1}\varep 2^{-|k_{n+1}|/2} \|P'_{k_1}h\|_{L^p}.
\end{equation*}
The desired bounds \eqref{na213} follow, using also \eqref{na218} provided that $C_1$ is sufficiently large.
\end{proof}

\subsection{Proof of Lemma \ref{BVexpands}} Using \eqref{na16} we have
\begin{equation}\label{exn1}
\partial_xR_0H_0\phi=-\partial_x(h\partial_x\phi)-|\partial_x|(h|\partial_x|\phi).
\end{equation}
The identities \eqref{expand1} and the bounds \eqref{expand2} for the cubic remainders $B_{\geq 3}$ and $V_{\geq 3}$ follow from \eqref{ctu5.11}--\eqref{ctu5.12}. The identities \eqref{expand0} and the bounds \eqref{expand2} for the remainders $\dot{h}_{\geq 3}$ and $\dot{\phi}_{\geq 3}$ follow from the evolution equations \eqref{gWW}, the identities and inequalities \eqref{ctu5.10}, and Lemma \ref{algeProp} (i).

The remaining conclusions of Lemma \ref{BVexpands} follow easily from the following:

\begin{lemma}\label{dotpsi}
Assume $(h,\phi)\in C([T_1,T_2]:H^{N_0}\times\dot{H}^{N_0,1/2})$ is a solution of the system \eqref{gWW} satisfying the bounds \eqref{maineqas}. We define $\psi\in\mathcal{E}^0$ as in \eqref{na35}. Then
\begin{equation}\label{exn2}
\begin{split}
&\partial_t\psi=-H_0h+\dot{\psi}_2+\dot{\psi}_{\geq 3},\\
&\dot{\psi}_2:=\big[|\partial_x|\phi\partial_x\phi-h\partial_xh\big]-H_0\Big[\frac{1}{2}(|\partial_x|\phi)^2+\frac{1}{2}(\partial_x\phi)^2-h|\partial_x|h\Big],\\
&\|\dot{\psi}_{\geq 3}\|_{\dot{H}^{N_0-1,1/10}}\lesssim A[\varep(t)]^2,\qquad \|\dot{\psi}_{\geq 3}\|_{\dot{W}^{N_1-1,1/10}}\lesssim [\varep(t)]^3.
\end{split}
\end{equation}
Moreover
\begin{equation}\label{exn3}
\begin{split}
&\partial^2_t\psi=-\partial_x\phi+\ddot{\psi}_2+\ddot{\psi}_{\geq 3},\\
&\ddot{\psi}_2:=-\partial_xh|\partial_x|\phi-|\partial_x|h\partial_x\phi+2H_0\big[\partial_xh\partial_x\phi+|\partial_x|h|\partial_x|\phi\big],\\
&\|\ddot{\psi}_{\geq 3}\|_{\dot{H}^{N_0-2,1/10}}\lesssim A[\varep(t)]^2,\qquad \|\ddot{\psi}_{\geq 3}\|_{\dot{W}^{N_1-2,1/10}}\lesssim [\varep(t)]^3.
\end{split}
\end{equation}
\end{lemma}

Assuming this lemma, we can easily complete the proof of Lemma \ref{BVexpands}. Indeed, the identities \eqref{na41} show that
\begin{equation*}
\partial_tV=\frac{\partial^2_{xt}\phi+\partial_xh\partial^2_{xt}\psi+\partial^2_{xt}h\partial_x\psi}{1+(\partial_xh)^2}-\frac{2\partial_xh\partial^2_{xt}h(\partial_x\phi+\partial_xh\partial_x\psi)}{[1+(\partial_xh)^2]^2}
\end{equation*}
and
\begin{equation*}
\partial_tB=\frac{-\partial^2_{xt}\psi+\partial_xh\partial^2_{xt}\phi+\partial^2_{xt}h\partial_x\phi}{1+(\partial_xh)^2}-\frac{2\partial_xh\partial^2_{xt}h(-\partial_x\psi+\partial_xh\partial_x\phi)}{[1+(\partial_xh)^2]^2}.
\end{equation*}
The desired identities \eqref{expand3}--\eqref{expand4} follow using the expansions \eqref{expand0}, \eqref{ctu5.10}, and \eqref{exn2}, and then using Lemma \ref{algeProp} (i) to estimate the cubic remainders $\dot{B}_{\geq 3}$ and $\dot{V}_{\geq 3}$ as claimed in \eqref{expand6.5}. 

We take one more time derivative to calculate
\begin{equation*}
\begin{split}
\partial^2_tB&=\frac{-\partial^3_{xtt}\psi+2\partial^2_{xt}h\partial^2_{xt}\phi+\partial_{x}h\partial^3_{xtt}\phi+\partial^3_{xtt}h\partial_x\phi}{1+(\partial_xh)^2}-\frac{2\partial_xh\partial^2_{xt}h(-\partial^2_{xt}\psi+\partial_xh\partial^2_{xt}\phi+\partial^2_{xt}h\partial_x\phi)}{[1+(\partial_xh)^2]^2}\\
&-\frac{2(\partial^2_{xt}h)^2(-\partial_x\psi+\partial_xh\partial_x\phi)+2\partial_xh\partial^2_{xt}h(-\partial^2_{xt}\psi+\partial^2_{xt}h\partial_x\phi+\partial_{x}h\partial^2_{xt}\phi)}{[1+(\partial_xh)^2]^2}\\
&+\frac{8(\partial_xh\partial^2_{xt}h)^2(-\partial_x\psi+\partial_xh\partial_x\phi)}{[1+(\partial_xh)^2]^3}.
\end{split}
\end{equation*}
We observe that only the expression in the first fraction contributes linear and quadratic terms. All the terms in the last 3 fractions are cubic and higher order, so they can be incorporated as part of the remainder $\ddot{B}_{\geq 3}$ and estimated using Lemma \ref{algeProp} (i). Moreover, notice that
\begin{equation}\label{exn4}
\begin{split}
\|\partial_t^2h+|\partial_x|h\|_{H^{N_0-2}}\lesssim A\varepsilon(t),\qquad \|\partial_t^2h+|\partial_x|h\|_{\wt{W}^{N_1-2}}\lesssim [\varepsilon(t)]^2,\\
\|\partial_t^2\phi+|\partial_x|\phi\|_{H^{N_0-2}}\lesssim A\varepsilon(t),\qquad \|\partial_t^2\phi+|\partial_x|\phi\|_{\wt{W}^{N_1-2}}\lesssim [\varepsilon(t)]^2,
\end{split}
\end{equation}
using again the main system \eqref{gWW} and the identities \eqref{expand0} and \eqref{exn2}. We can then use these bounds and the identities \eqref{expand0} and \eqref{exn3} to extract the linear and the quadratic terms from the first fraction in the formula of $\partial_t^2B$ above, and control the cubic remainder as claimed in \eqref{expand5}--\eqref{expand6.5}. This completes the proof of Lemma \ref{BVexpands}.

\subsubsection{Proof of Lemma \ref{dotpsi}} We start from the defining formula $(I+T_1)\psi=(H_0+T_2)\phi$ and recall the main definitions \eqref{na23} and \eqref{na13.2}--\eqref{na13.3}. 

{\bf{Step 1.}} For $n\geq 0$ we define the operators $\dot{R}_n$, $\ddot{R}_{n,1}$, and $\ddot{R}_{n,2}$ by the formulas
\begin{equation}\label{exn5}
\begin{split}
\widehat{\dot{R}_nf}(\xi)=\frac{1}{\pi(2\pi R)^{n+1}}\sum_{\eta_1,\ldots,\eta_{n+1}\in \Z/R}&M_{n+1}(\xi;\underline{\eta})\widehat{f}(\xi-\eta_1-\ldots-\eta_{n+1})\\
\times&\widehat{h}(\eta_1)\cdot\ldots\cdot\widehat{h}(\eta_{n})\cdot\widehat{(\partial_th)}(\eta_{n+1}),
\end{split}
\end{equation}
\begin{equation}\label{exn6}
\begin{split}
\widehat{\ddot{R}_{n,1}f}(\xi)=\frac{1}{\pi(2\pi R)^{n+1}}\sum_{\eta_1,\ldots,\eta_{n+1}\in \Z/R}&M_{n+1}(\xi;\underline{\eta})\widehat{f}(\xi-\eta_1-\ldots-\eta_{n+1})\\
\times&\widehat{h}(\eta_1)\cdot\ldots\cdot\widehat{h}(\eta_{n})\cdot\widehat{(\partial_t^2h)}(\eta_{n+1}),
\end{split}
\end{equation}
\begin{equation}\label{exn7}
\begin{split}
\widehat{\ddot{R}_{n,2}f}(\xi)=\frac{1}{\pi(2\pi R)^{n+1}}\sum_{\eta_1,\ldots,\eta_{n+1}\in \Z/R}&M_{n+1}(\xi;\underline{\eta})\widehat{f}(\xi-\eta_1-\ldots-\eta_{n+1})\\
\times&\widehat{h}(\eta_1)\cdot\ldots\cdot\widehat{h}(\eta_{n-1})\cdot\widehat{(\partial_th)}(\eta_{n})\cdot\widehat{(\partial_th)}(\eta_{n+1}).
\end{split}
\end{equation}

We claim that if $f\in H^3(\T_R)$ satisfies $\||\partial_x|^{1/2}f\|_{\dot{W}^{1,0}}\leq 1$ then
\begin{equation}\label{exn8}
\sum_{n\geq a}2^n\|\dot{R}_nf\|_{\dot{H}^{N_0-1,1/10}}\lesssim A\varep^{a},\qquad \sum_{n\geq a}2^n\|\dot{R}_nf\|_{\dot{W}^{N_1-1,1/10}}\lesssim \varep^{a+1}
\end{equation}
and
\begin{equation}\label{exn9}
\sum_{n\geq a}2^n\|\ddot{R}_{n,b}f\|_{\dot{H}^{N_0-2,1/10}}\lesssim A\varep^{a},\qquad \sum_{n\geq a}2^n\|\ddot{R}_{n,b}f\|_{\dot{W}^{N_1-2,1/10}}\lesssim \varep^{a+1}
\end{equation}
for $a\in\{0,1\}$ and $b\in\{1,2\}$. These are similar to the bounds \eqref{na23.7} in Lemma \ref{Wboundedness} once we notice that the functions $\partial_th$ and $\partial_t^2h$ satisfy estimates of the form
\begin{equation*}
\|\partial_t^lh(t)\|_{H^{N_0-l}}\lesssim A,\qquad \|\partial_t^lh(t)\|_{\wt{W}^{N_1-l}}\lesssim\varep, \qquad l\in\{1,2\},
\end{equation*}
for any $t\in[T_1,T_2]$ (see \eqref{expand0}, \eqref{expand2}, and \eqref{exn4}).

{\bf{Step 2.}} We prove now the expansions in Lemma \ref{dotpsi}. It follows from \eqref{na35} that
\begin{equation}\label{exn10}
(I+T_1)\dot{\psi}=(H_0+T_2)\dot{\phi}+\dot{T}_2\phi-\dot{T}_1\psi,
\end{equation}
where $\dot\psi:=\partial_t\psi$, $\dot{\phi}:=\partial_t\phi$, and
\begin{equation}\label{exn11}
\dot{T}_1:=\sum_{n\geq 0}(-1)^{n}(2n+1)\dot{R}_{2n},\qquad \dot{T}_2:=\sum_{n\geq 0}(-1)^{n}(2n+2)\dot{R}_{2n+1}.
\end{equation}
It follows from \eqref{na23.7}, \eqref{exn8}, and \eqref{exn10} that for any $t\in[T_1,T_2]$
\begin{equation}\label{exn12}
\|\dot{\psi}(t)\|_{\dot{H}^{N_0-1,1/10}}\lesssim A,\qquad \|\dot{\psi}(t)\|_{\dot{W}^{N_1-1,1/10}}\lesssim\varep(t).
\end{equation}
We apply the operator $I-R_0$ to \eqref{exn10} to see that
\begin{equation*}
\dot{\psi}-R_0^2\dot{\psi}=(I-R_0)\big\{(H_0+T_2)\dot{\phi}+\dot{T}_2\phi-\dot{T}_1\psi-(T_1-R_0)\dot\psi\big\},
\end{equation*}
which can be rewritten in the form
\begin{equation*}
\dot{\psi}=(I-R_0)H_0\dot{\phi}-\dot{R}_0\psi+(I-R_0)\big\{T_2\dot{\phi}+\dot{T}_2\phi-(\dot{T}_1-\dot{R}_0)\psi-(T_1-R_0)\dot\psi\big\}+R_0\dot{R}_0\psi+R_0^2\dot{\psi}.
\end{equation*}
We notice that most of the terms in the formula above are cubic and higher order acceptable remainders, due to \eqref{exn8}, and only the first two terms in the right-hand side give linear and quadratic contributions. Using also \eqref{expand0} these contributions are
\begin{equation*}
H_0(-h+\dot{\phi}_2)+R_0H_0h-\dot{R}_0\psi.
\end{equation*}
The identities \eqref{exn2} follow since $R_0f=-H_0(h\partial_xf)-h|\partial_x|f$ and $\dot{h}-|\partial_x|\phi$ and $\psi-H_0\phi$ are suitable quadratic terms (see \eqref{expand0}). 

The proof of the expansion \eqref{exn3} is similar. We start from from \eqref{na35} and take 2 time derivatives, so
\begin{equation}\label{exn14}
(I+T_1)\ddot{\psi}=(H_0+T_2)\ddot{\phi}+2\dot{T}_2\dot\phi+\ddot{T}_2\phi-2\dot{T}_1\dot\psi-\ddot{T}_1\psi,
\end{equation}
where $\ddot\psi:=\partial^2_t\psi$, $\ddot{\phi}:=\partial^2_t\phi$, and $\ddot{T}_1, \ddot{T}_2$ are defined naturally as in \eqref{exn11}. Using \eqref{na23.7} and \eqref{exn8}--\eqref{exn9} it follows that for any $t\in[T_1,T_2]$
\begin{equation}\label{exn15}
\|\ddot{\psi}(t)\|_{\dot{H}^{N_0-2,1/10}}\lesssim A,\qquad \|\ddot{\psi}(t)\|_{\dot{W}^{N_1-2,1/10}}\lesssim\varep(t).
\end{equation}
As before, we apply the operator $I-R_0$ to \eqref{exn14} to see that $\ddot{\psi}$ is equal to
\begin{equation*}
(I-R_0)H_0\ddot{\phi}-2\dot{R}_0\dot\psi-\ddot{R}_0\psi
\end{equation*}
up to acceptable cubic remainders. 
Then we notice that $\psi-H_0\phi$, $\dot{\psi}+H_0h$, $\dot{h}-|\partial_x|\phi$,
and $\ddot{h}+|\partial_x|h$ satisfy suitable quadratic bounds, while
$\ddot{\phi}+|\partial_x|\phi-\partial_x(h\partial_x\phi)-|\partial_x|(h|\partial_x|\phi)
  -\partial_xh\partial_x\phi+|\partial_x|h|\partial_x|\phi$ is an acceptable cubic remainder. 
The desired expansion \eqref{exn3} follows after some simplifications of the quadratic terms, 
which completes the proof of the lemma.

\medskip
\section{Vanishing at the Benjamin-Feir resonances}\label{uyt22sec}

In this section we prove Proposition \ref{BFprop}. Assume that $(y_1,y_2,y_3,y_4)\in\mathbb{H}^3_\infty$ satisfies
\begin{equation}\label{verf3}
\begin{split}
&y_1\in(-\infty,0],\qquad y_2,y_3,y_4\in[0,\infty),\qquad|y_1|\geq y_3\geq y_4\geq y_2,\\
&y_3\geq|y_1|/8,\qquad y_2\geq y_4/8.
\end{split}
\end{equation}
We would like to calculate explicitly the symmetrized symbols $SK_1^a(y_1,y_2,y_3,y_4)$, $a\in\{1,\ldots,9\}$, and $SK_2^b(y_1,y_2,y_3,y_4)$, $b\in\{1,\ldots,4\}$. 

Since $y_1+y_2+y_3+y_4=0$, it follows from \eqref{verf3} that many of the $\chi$ coefficients vanish, i.e.
\begin{equation}\label{verf4}
0=\chi(y_1,y_j)=\chi(y_3,y_j)=\chi(y_4,y_2)=\chi(y_2,y_4)
\end{equation}
for any $j\in\{1,2,3,4\}$. Similarly, for any $j\neq k\in\{1,2,3,4\}$,
\begin{equation}\label{verf5}
\begin{split}
&0=\chi(y_1,y_j+y_k)=\chi(y_3,y_j+y_k)=\chi(y_2,y_1+y_3)=\chi(y_4,y_1+y_3),\\
&0=\chi(y_j+y_k,y_2)=\chi(y_j+y_k,y_4),\\
&0=\chi(y_2+y_3,y_1)=\chi(y_3+y_4,y_1)=\chi(y_1+y_2,y_3)=\chi(y_1+y_4,y_3).
\end{split}
\end{equation}
For simplicity of notation let $\chi_{i,j}:=\chi(y_i,y_j)$, 
$\chi_{i,j+k}:=\chi(y_i,y_j+y_k)$, $\chi_{i+j,k}:=\chi(y_i+y_j,y_k)$, $i,j,k\in\{1,2,3,4\}$.
Using \eqref{chieq} it is easy to see that the nontrivial cutoffs $\chi$ satisfy
\begin{equation}\label{verf5.5}
\begin{split}
&\chi_{2,1+4}=\chi_{2,3},\qquad\chi_{2,3+4}=\chi_{2,1},\qquad\chi_{4,1+2}=\chi_{4,3},\qquad\chi_{4,3+2}=\chi_{4,1},\\
&\chi_{2+4,1}=\chi_{2+4,3}=:\chi_{2+4}.
\end{split}
\end{equation}

We can also calculate explicitly the symbols $n^{rr}_2$, $n^{ii}_2$, and $n^{ri}_2$ 
defined in \eqref{Nh}--\eqref{Nii2}: for $j\in\{2,4\}$ we have
\begin{equation}\label{verf6}
\begin{split}
&n_2^{rr}(y_1,y_3)=n_2^{rr}(y_3,y_1)=n_2^{rr}(y_2,y_4)=n_2^{rr}(y_4,y_2)=0,\\
&n_2^{rr}(y_1,y_j)=n_2^{rr}(y_j,y_1)=(1/4)|y_1+y_j|^{1/2}|y_j|\chi_{j,1},\\
&n_2^{rr}(y_3,y_j)=n_2^{rr}(y_j,y_3)=(1/4)|y_3+y_j|^{1/2}|y_j|\chi_{j,3},
\end{split}
\end{equation}
\begin{equation}\label{verf7}
\begin{split}
&n_2^{ii}(y_1,y_3)=n_2^{ii}(y_3,y_1)=0,\\
&n_2^{ii}(y_2,y_4)=n_2^{ii}(y_4,y_2)=|y_2+y_4|^{1/2}|y_2|^{1/2}|y_4|^{1/2},\\
&n_2^{ii}(y_1,y_j)=n_2^{ii}(y_j,y_1)=-|y_1+y_j|^{1/2}|y_1|^{1/2}|y_j|^{1/2}\chi_{j,1}/2,\\
&n_2^{ii}(y_3,y_j)=n_2^{ii}(y_j,y_3)=|y_3+y_j|^{1/2}|y_3|^{1/2}|y_j|^{1/2}(1-\chi_{j,3}/2),
\end{split}
\end{equation}
\begin{equation}\label{verf8}
\begin{split}
&n_2^{ri}(y_1,y_3)=-2|y_1+y_3||y_3|^{1/2},\qquad n_2^{ri}(y_3,y_1)=0,\\
&n_2^{ri}(y_2,y_4)=n_2^{ri}(y_4,y_2)=0,\\
&n_2^{ri}(y_1,y_j)=|y_1+y_j||y_j|^{1/2}(-2+\chi_{j,1})-|y_j|^{3/2}\chi_{j,1}/2,\qquad n_2^{ri}(y_j,y_1)=-|y_1|^{1/2}|y_j|\chi_{j,1}/2,\\
&n_2^{ri}(y_3,y_j)=|y_j|^{1/2}(|y_3|+|y_j|/2)\chi_{j,3},\qquad n_2^{ri}(y_j,y_3)=-|y_3|^{1/2}|y_j|\chi_{j,3}/2.
\end{split}
\end{equation}

For simplicity of notation let $m_j:=m_{N_0}(y_j)$, $m_{j+k}:=m_{N_0}(y_j+y_k)$, $j,k\in\{1,2,3,4\}$. Recall that the symbols $A'$ and $B'''$ are symmetric in all three variables. Using \eqref{sus30}, \eqref{slov26}, and the identities \eqref{verf4}--\eqref{verf5.5} we calculate explicitly 
\begin{equation}\label{A1ex}
4A'(y_1,y_2,y_3+y_4)=\frac{|y_1|^{1/2}\big\{m_1^2(-2+\chi_{2,1})+m_{3+4}^2\chi_{2,1}\big\}}{3|y_2|^{1/2}|y_3+y_4|^{1/2}},
\end{equation}
\begin{equation}\label{A2ex}
4A'(y_1,y_4,y_3+y_2)=\frac{|y_1|^{1/2}\big\{m_1^2(-2+\chi_{4,1})+m_{3+2}^2\chi_{4,1}\big\}}{3|y_4|^{1/2}|y_3+y_2|^{1/2}},
\end{equation}
\begin{equation}\label{A3ex}
4A'(y_3,y_2,y_1+y_4)=\frac{|y_1+y_4|^{1/2}\big\{m_3^2\chi_{2,3}+m_{1+4}^2(-2+\chi_{2,3})\big\}}{3|y_3|^{1/2}|y_2|^{1/2}},
\end{equation}
\begin{equation}\label{A4ex}
4A'(y_3,y_4,y_1+y_2)=\frac{|y_1+y_2|^{1/2}\big\{m_3^2\chi_{4,3}+m_{1+2}^2(-2+\chi_{4,3})\big\}}{3|y_3|^{1/2}|y_4|^{1/2}},
\end{equation}
\begin{equation}\label{A5ex}
4A'(y_2,y_4,y_1+y_3)=\frac{-2|y_1+y_3|^{1/2}m_{1+3}^2}{3|y_2|^{1/2}|y_4|^{1/2}},
\end{equation}
\begin{equation}\label{A6ex}
4A'(y_1,y_3,y_2+y_4)=\frac{|y_1|^{1/2}\big\{m_1^2(-2+\chi_{2+4})+m_3^2\chi_{2+4}\big\}}{3|y_3|^{1/2}|y_2+y_4|^{1/2}},
\end{equation}
\begin{equation}\label{BB1ex}
4B'''(y_1,y_2,y_3+y_4)=4|y_2|m_2^2+4|y_3+y_4|m_{3+4}^2-2\chi_{2,1}\big[|y_1|m_1^2+|y_1|m_{3+4}^2-2|y_2|m_1m_{3+4}\big],
\end{equation}
\begin{equation}\label{BB2ex}
4B'''(y_1,y_4,y_3+y_2)=4|y_4|m_4^2+4|y_3+y_2|m_{3+2}^2-2\chi_{4,1}\big[|y_1|m_1^2+|y_1|m_{3+2}^2-2|y_4|m_1m_{3+2}\big],
\end{equation}
\begin{equation}\label{BB3ex}
4B'''(y_3,y_2,y_1+y_4)=4|y_3|m_3^2+4|y_2|m_{2}^2-2\chi_{2,3}\big[|y_1+y_4|m_{1+4}^2+|y_1+y_4|m_{3}^2-2|y_2|m_3m_{1+4}\big],
\end{equation}
\begin{equation}\label{BB4ex}
4B'''(y_3,y_4,y_1+y_2)=4|y_3|m_3^2+4|y_4|m_{4}^2-2\chi_{4,3}\big[|y_1+y_2|m_{1+2}^2+|y_1+y_2|m_{3}^2-2|y_4|m_3m_{1+2}\big],
\end{equation}
\begin{equation}\label{BB5ex}
4B'''(y_2,y_4,y_1+y_3)=4|y_2|m_2^2+4|y_4|m_4^2,
\end{equation}
\begin{equation}\label{BB6ex}
4B'''(y_1,y_3,y_2+y_4)=4|y_3|m_3^2+4|y_2+y_4|m_{2+4}^2-2\chi_{2+4}\big[|y_1|m_1^2+|y_1|m_{3}^2-2|y_2+y_4|m_1m_{3}\big].
\end{equation}

The symbols $B''$ are symmetric in the first 2 variables. Using \eqref{slov24}--\eqref{slov25} we calculate
\begin{equation}\label{B1ex}
\begin{split}
4B''(y_1,y_2,y_3+y_4)&=-2|y_1|^{1/2}|y_2|^{1/2}\big[2m_2^2+\chi_{2,1}(m_1^2-m_{3+4}^2)\big],\\
4B''(y_3+y_4,y_1,y_2)&=-2|y_1|^{1/2}|y_3+y_4|^{1/2}\big[2m_{3+4}^2-\chi_{2,1}(m_1^2+m_{3+4}^2)\big],\\
4B''(y_3+y_4,y_2,y_1)&=4|y_2|^{1/2}|y_3+y_4|^{1/2}\big[m_{3+4}^2+m_2^2-\chi_{2,1}m_{3+4}^2\big],
\end{split}
\end{equation}
\begin{equation}\label{B2ex}
\begin{split}
4B''(y_1,y_4,y_3+y_2)&=-2|y_1|^{1/2}|y_4|^{1/2}\big[2m_4^2+\chi_{4,1}(m_1^2-m_{3+2}^2)\big],\\
4B''(y_3+y_2,y_1,y_4)&=-2|y_1|^{1/2}|y_3+y_2|^{1/2}\big[2m_{3+2}^2-\chi_{4,1}(m_1^2+m_{3+2}^2)\big],\\
4B''(y_3+y_2,y_4,y_1)&=4|y_4|^{1/2}|y_3+y_2|^{1/2}\big[m_{3+2}^2+m_4^2-\chi_{4,1}m_{3+2}^2\big],
\end{split}
\end{equation}
\begin{equation}\label{B3ex}
\begin{split}
4B''(y_3,y_2,y_1+y_4)&=4|y_3|^{1/2}|y_2|^{1/2}\big[m_3^2+m_2^2-\chi_{2,3}m_3^2\big],\\
4B''(y_1+y_4,y_3,y_2)&=-2|y_1+y_4|^{1/2}|y_3|^{1/2}\big[2m_3^2-\chi_{2,3}(m_{1+4}^2+m_3^2)\big],\\
4B''(y_1+y_4,y_2,y_3)&=-2|y_1+y_4|^{1/2}|y_2|^{1/2}\big[2m_2^2+\chi_{2,3}(m_{1+4}^2-m_3^2)\big],
\end{split}
\end{equation}
\begin{equation}\label{B4ex}
\begin{split}
4B''(y_3,y_4,y_1+y_2)&=4|y_3|^{1/2}|y_4|^{1/2}\big[m_3^2+m_4^2-\chi_{4,3}m_3^2\big],\\
4B''(y_1+y_2,y_3,y_4)&=-2|y_1+y_2|^{1/2}|y_3|^{1/2}[2m_3^2-\chi_{4,3}(m_{1+2}^2+m_3^2)],\\
4B''(y_1+y_2,y_4,y_3)&=-2|y_1+y_2|^{1/2}|y_4|^{1/2}[2m_4^2+\chi_{4,3}(m_{1+2}^2-m_3^2)],
\end{split}
\end{equation}
\begin{equation}\label{B5ex}
\begin{split}
4B''(y_2,y_4,y_1+y_3)&=4|y_2|^{1/2}|y_4|^{1/2}(m_2^2+m_4^2),\\
4B''(y_1+y_3,y_2,y_4)&=-4|y_1+y_3|^{1/2}|y_2|^{1/2}m_2^2,\\
4B''(y_1+y_3,y_4,y_2)&=-4|y_1+y_3|^{1/2}|y_4|^{1/2}m_4^2,
\end{split}
\end{equation}
\begin{equation}\label{B6ex}
\begin{split}
4B''(y_1,y_3,y_2+y_4)&=-2|y_1|^{1/2}|y_3|^{1/2}\big[2m_3^2-\chi_{2+4}(m_1^2+m_3^2)\big],\\
4B''(y_2+y_4,y_1,y_3)&=-2|y_1|^{1/2}|y_2+y_4|^{1/2}\big[2m_{2+4}^2+\chi_{2+4}(m_1^2-m_3^2)\big],\\
4B''(y_2+y_4,y_3,y_1)&=4|y_3|^{1/2}|y_2+y_4|^{1/2}\big[m_3^2+m_{2+4}^2-\chi_{2+4}m_3^2\big].
\end{split}
\end{equation}

\smallskip
\subsection{The symbols $SK_i^j$}\label{uyt22seccalc}
We calculate now the symmetrized symbols $SK$,
where $K$ is any of the symbols from Lemma \ref{quartic3},
and evaluate them at frequencies $(y_1,y_2,y_3,y_4)$ satisfying \eqref{verf3}. Recall the main formulas
\begin{equation}\label{mainSKform}
\begin{split}
&SK(y_1,y_2,y_3,y_4)=3\widetilde{K}(y_1,y_2,y_3,y_4)+3\widetilde{K}(y_1,y_2,y_4,y_3)-3\widetilde{K}(y_3,y_4,y_1,y_2)-3\widetilde{K}(y_3,y_4,y_2,y_1),\\
&3\widetilde{K}(\xi,\eta,\rho,\si)=K(\xi,\eta,\rho,\sigma) + K(\eta,\rho,\xi,\sigma) + K(\rho,\xi,\eta,\sigma),
\end{split}
\end{equation}
see \eqref{Ksymm} and \eqref{verf1}. For simplicity of notation let $x_j:=|y_j|^{1/2}$, $j\in\{1,2,3,4\}$. 

\smallskip
{\bf{Contributions of $K_1^1$}}. 
Using \eqref{verf4}--\eqref{verf5.5} we calculate the components
\begin{equation*}
\begin{split}
K_1^1(y_1,y_2,y_3,y_4)=(1/2)m_1^2\chi_{2,3}\chi_{4,1}\cdot |y_1||y_2||y_4|^{1/2}-m_1^2\chi_{2+4}\cdot |y_1||y_2||y_4|^{1/2},
\end{split}
\end{equation*}
\begin{equation*}
\begin{split}
K_1^1(y_2,y_1,y_3,y_4)=2m_2^2|y_3||y_2||y_4|^{1/2}-2m_2^2\chi_{4,3}|y_3+y_4||y_2||y_4|^{1/2},
\end{split}
\end{equation*}
\begin{equation*}
\begin{split}
K_1^1(y_3,y_2,y_1,y_4)=m_3^2|y_4|^{1/2}|y_2||y_3|\big\{2-\chi_{2+4}-\chi_{2,1}+\chi_{2,1}\chi_{4,3}/2\big\}.
\end{split}
\end{equation*}
\begin{equation*}
\begin{split}
K_1^1(y_1,y_2,y_4,y_3)=0,
\end{split}
\end{equation*}
\begin{equation*}
\begin{split}
K_1^1(y_2,y_1,y_4,y_3)=m_2^2|y_2||y_4||y_3|^{1/2}(2-\chi_{4,1}),
\end{split}
\end{equation*}
\begin{equation*}
\begin{split}
K_1^1(y_4,y_2,y_1,y_3)=m_4^2|y_4||y_2||y_3|^{1/2}(2-\chi_{2,1}).
\end{split}
\end{equation*}
\begin{equation*}
\begin{split}
K_1^1(y_1,y_3,y_4,y_2)=(1/2)m_1^2\chi_{4,3}\chi_{2,1}\cdot |y_1||y_4||y_2|^{1/2}-m_1^2\chi_{2+4}\cdot |y_1||y_4||y_2|^{1/2},
\end{split}
\end{equation*}
\begin{equation*}
\begin{split}
K_1^1(y_3,y_1,y_4,y_2)=m_3^2|y_2|^{1/2}|y_4||y_3|\big\{2-\chi_{2+4}-\chi_{4,1}+\chi_{2,3}\chi_{4,1}/2\big\},
\end{split}
\end{equation*}
\begin{equation*}
\begin{split}
K_1^1(y_4,y_1,y_3,y_2)=2m_4^2|y_3||y_4||y_2|^{1/2}-2m_4^2\chi_{2,3}|y_3+y_2||y_4||y_2|^{1/2},
\end{split}
\end{equation*}
\begin{equation*}
\begin{split}
K_1^1(y_2,y_3,y_4,y_1)=K_1^1(y_3,y_2,y_4,y_1)=K_1^1(y_4,y_2,y_3,y_1)=0.
\end{split}
\end{equation*}

In terms of the variables $x_j=|y_j|^{1/2}$, $j\in\{1,2,3,4\}$, the formulas above show that
\begin{equation}\label{Sk11}
4SK_1^1(y_1,y_2,y_3,y_4)=\sum_{l\in\{1,\ldots,4\}}4H_{1,l}(y_1,y_2,y_3,y_4),
\end{equation}
where
\begin{equation}\label{H11xy}
\begin{split}
4H_{1,1}(y_1,y_2,&y_3,y_4)=m_1^2\chi_{2+4}\cdot 4x_1^2x_2x_4(x_4-x_2)\\
&+m_1^2\chi_{2,3}\chi_{4,1}\cdot 2x_1^2x_2^2x_4-m_1^2\chi_{4,3}\chi_{2,1}\cdot 2x_1^2x_4^2x_2,
\end{split}
\end{equation}
\begin{equation}\label{H12xy}
\begin{split}
4H_{1,2}(y_1,y_2,&y_3,y_4)=m_2^2\cdot 8x_2^2x_3x_4(x_3+x_4)\\
&-m_2^2\chi_{4,3}\cdot 8(x_3^2+x_4^2)x_2^2x_4-m_2^2\chi_{4,1}\cdot 4x_3x_2^2x_4^2,
\end{split}
\end{equation}
\begin{equation}\label{H13xy}
\begin{split}
4&H_{1,3}(y_1,y_2,y_3,y_4)=m_3^2\cdot 8x_3^2x_2x_4(x_2-x_4)-m_3^2\chi_{2+4}\cdot 4x_3^2x_2x_4(x_2-x_4)\\
&-m_3^2\chi_{2,1}\cdot 4x_3^2x_2^2x_4+m_3^2\chi_{2,1}\chi_{4,3}\cdot 2x_3^2x_2^2x_4+m_3^2\chi_{4,1}\cdot 4x_3^2x_4^2x_2-m_3^2\chi_{4,1}\chi_{2,3}\cdot 2x_3^2x_4^2x_2,
\end{split}
\end{equation}
\begin{equation}\label{H14xy}
\begin{split}
4H_{1,4}(y_1,y_2,&y_3,y_4)=m_4^2\cdot 8x_4^2x_3x_2(x_2-x_3)\\
&+m_4^2\chi_{2,3}\cdot 8(x_3^2+x_2^2)x_4^2x_2-m_4^2\chi_{2,1}\cdot 4x_3x_2^2x_4^2.
\end{split}
\end{equation}

{\bf{Contributions of $K_1^2$}}. The nontrivial components are
\begin{equation*}
\begin{split}
K_1^2(y_2,y_1,y_3,y_4)=&-m_1m_{3+4}|y_2||y_3+y_4||y_4|^{1/2}\chi_{2,1}\chi_{4,3}\\
&+m_3m_{1+4}|y_2||y_1+y_4||y_4|^{1/2}\chi_{2,3}(2-\chi_{4,1}),
\end{split}
\end{equation*}
\begin{equation*}
\begin{split}
K_1^2(y_4,y_1,y_3,y_2)=&-m_1m_{3+2}|y_4||y_3+y_2||y_2|^{1/2}\chi_{4,1}\chi_{2,3}\\
&+m_3m_{1+2}|y_4||y_1+y_2||y_2|^{1/2}\chi_{4,3}(2-\chi_{2,1}).
\end{split}
\end{equation*}
Therefore, in terms of the variables $x_j=|y_j|^{1/2}$,
\begin{equation}\label{Sk12}
\begin{split}
4&SK_1^2(y_1,y_2,y_3,y_4)=4m_1m_{3+2}x_4^2x_2(x_3^2+x_2^2)\chi_{2,3}\chi_{4,1}-4m_1m_{3+4}x_2^2x_4(x^2_3+x^2_4)\chi_{4,3}\chi_{2,1}\\
&+4m_3m_{1+2}x_4^2x_2(x^2_3+x^2_4)\chi_{4,3}(\chi_{2,1}-2)-4m_3m_{1+4}x_2^2x_4(x^2_3+x^2_2)\chi_{2,3}(\chi_{4,1}-2).
\end{split}
\end{equation}

{\bf{Contributions of $K_1^3$}}. The nontrivial components are
\begin{equation*}
\begin{split}
K_1^{3}(y_1,y_2,y_3,y_4)=-(1/2)m_{3+2}^2|y_2||y_1+y_4||y_4|^{1/2}\chi_{2,3}(2-\chi_{4,1}),
\end{split}
\end{equation*}
\begin{equation*}
\begin{split}
K_1^{3}(y_3,y_2,y_1,y_4)=(1/2)m_{1+2}^2|y_2||y_3+y_4||y_4|^{1/2}\chi_{2,1}\chi_{4,3}.
\end{split}
\end{equation*}
\begin{equation*}
\begin{split}
K_1^{3}(y_1,y_3,y_4,y_2)=-(1/2)m_{3+4}^2|y_4||y_1+y_2||y_2|^{1/2}\chi_{4,3}(2-\chi_{2,1}),
\end{split}
\end{equation*}
\begin{equation*}
\begin{split}
K_1^{3}(y_3,y_1,y_4,y_2)=(1/2)m_{1+4}^2|y_4||y_3+y_2||y_2|^{1/2}\chi_{2,3}\chi_{4,1},
\end{split}
\end{equation*}
Since $m_{1+2}=m_{3+4}$ and $m_{1+4}=m_{3+2}$ it follows that
\begin{equation}\label{Sk13}
\begin{split}
4SK_1^3(y_1,y_2,y_3,y_4)=&-2m_{1+4}^2x_2x_4(x_3^2+x_2^2)\chi_{2,3}\big\{x_2(2-\chi_{4,1})+x_4\chi_{4,1}\big\}\\
&+2m_{1+2}^2x_2x_4(x_3^2+x_4^2)\chi_{4,3}\big\{x_4(2-\chi_{2,1})+x_2\chi_{2,1}\big\}.
\end{split}
\end{equation}

{\bf{Contributions of $K_1^4$}}. The nontrivial components are
\begin{equation*}
K_1^4(y_2,y_1,y_3,y_4)=2m_1m_3|y_2||y_4|^{3/2}\chi_{2+4},
\end{equation*}
\begin{equation*}
K_1^4(y_4,y_1,y_3,y_2)=2m_1m_3|y_4||y_2|^{3/2}\chi_{2+4},
\end{equation*}
where we used the definition \eqref{ite4}. Therefore
\begin{equation}\label{Sk14}
4SK_1^4(y_1,y_2,y_3,y_4)=m_1m_3\chi_{2+4}\cdot 8x_2^2x_4^2(x_4-x_2).
\end{equation}

{\bf{Contributions of $K_1^5$}}. The nontrivial components are
\begin{equation*}
K_1^5(y_1,y_2,y_3,y_4)=-(1/2)m_1m_{3+2}|y_2||y_4|^{3/2}\chi_{2,3}\chi_{4,1},
\end{equation*}
\begin{equation*}
K_1^5(y_3,y_2,y_1,y_4)=-(1/2)m_3m_{1+2}|y_2||y_4|^{3/2}\chi_{2,1}\chi_{4,3}.
\end{equation*}
\begin{equation*}
K_1^5(y_1,y_3,y_4,y_2)=-(1/2)m_1m_{3+4}|y_4||y_2|^{3/2}\chi_{4,3}\chi_{2,1},
\end{equation*}
\begin{equation*}
K_1^5(y_3,y_1,y_4,y_2)=-(1/2)m_3m_{1+4}|y_4||y_2|^{3/2}\chi_{4,1}\chi_{2,3},
\end{equation*}
Therefore
\begin{equation}\label{Sk15}
\begin{split}
4SK_1^5(y_1,y_2,y_3,y_4)=&-2m_1m_{3+2}\chi_{4,1}\chi_{2,3}\cdot x_2^2x_4^3+2m_1m_{3+4}\chi_{2,1}\chi_{4,3}\cdot x_4^2x_2^3\\
&-2m_3m_{1+2}\chi_{2,1}\chi_{4,3}\cdot x_2^2x_4^3+2m_3m_{1+4}\chi_{4,1}\chi_{2,3}\cdot x_4^2x_2^3.
\end{split}
\end{equation}

{\bf{Contributions of $K_1^6$}}. The nontrivial components are
\begin{equation*}
K_1^6(y_2,y_1,y_4,y_3)=-(1/2)m_3(m_{1+4}-m_3)|y_3|^{1/2}|y_2||y_4|\chi_{4,1}\chi_{2,3},
\end{equation*}
\begin{equation*}
K_1^6(y_4,y_2,y_1,y_3)=-(1/2)m_3(m_{1+2}-m_3)|y_3|^{1/2}|y_2||y_4|\chi_{2,1}\chi_{4,3}.
\end{equation*}
\begin{equation*}
K_1^6(y_2,y_3,y_4,y_1)=-(1/2)m_1(m_{3+4}-m_1)|y_1|^{1/2}|y_2||y_4|\chi_{4,3}\chi_{2,1},
\end{equation*}
\begin{equation*}
K_1^6(y_4,y_2,y_3,y_1)=-(1/2)m_1(m_{3+2}-m_1)|y_1|^{1/2}|y_2||y_4|\chi_{2,3}\chi_{4,1},
\end{equation*}
Therefore, in terms of the variables $x_j=|y_j|^{1/2}$,
\begin{equation}\label{Sk16}
\begin{split}
4SK_1^6(y_1,&y_2,y_3,y_4)=2x_2^2x_4^2x_1m_1\big\{m_{3+4}\chi_{2,1}\chi_{4,3}+m_{3+2}\chi_{2,3}\chi_{4,1}\big\}\\
&-2x_2^2x_4^2x_3m_3\big\{m_{1+4}\chi_{2,3}\chi_{4,1}+m_{1+2}\chi_{2,1}\chi_{4,3}\big\}\\
&-2x_2^2x_4^2x_1m_1^2[\chi_{2,1}\chi_{4,3}+\chi_{2,3}\chi_{4,1}]+2x_2^2x_4^2x_3m_3^2[\chi_{2,1}\chi_{4,3}+\chi_{2,3}\chi_{4,1}].
\end{split}
\end{equation}

{\bf{Contributions of $K_1^7$}}. The nontrivial components are
\begin{equation*}
K_1^7(y_1,y_2,y_4,y_3)=2m_3(m_1-m_3)\chi_{2+4}|y_3|^{1/2}|y_2||y_4|,
\end{equation*}
\begin{equation*}
K_1^7(y_3,y_2,y_4,y_1)=2m_1(m_3-m_1)\chi_{2+4}|y_1|^{1/2}|y_2||y_4|,
\end{equation*}
where we used the definition \eqref{ite3}. Therefore
\begin{equation}\label{Sk17}
4SK_1^7(y_1,y_2,y_3,y_4)=-8m_1m_3x^2_2x^2_4(x_1-x_3)\chi_{2+4}+8\big(x_1m_1^2-x_3m_3^2\big)x_2^2x_4^2\chi_{2+4}.
\end{equation}

{\bf{Contributions of $K_1^{8}$}}. Using the formulas \eqref{ku113} and \eqref{verf6}--\eqref{verf8} we calculate 
\begin{equation*}
\begin{split}
K_1^{8}(y_1,y_2,y_3,y_4)=&-3A'(y_3,y_2,y_1+y_4)\\
&\times|y_3|^{1/2}|y_3+y_2|^{1/2}|y_2|^{1/2}|y_4|^{1/2}\big\{|y_3+y_2|(4-2\chi_{4,1})+|y_4|\chi_{4,1}\big\},
\end{split}
\end{equation*}
\begin{equation*}
\begin{split}
K_1^{8}(y_2,y_1,y_3,y_4)=0,
\end{split}
\end{equation*}
\begin{equation*}
\begin{split}
K_1^{8}(y_3,y_2,&y_1,y_4)=+3A'(y_3,y_4,y_1+y_2)\cdot |y_3+y_4|^{1/2}|y_3|^{1/2}|y_4||y_2|\chi_{2,1}\\
&+3A'(y_1,y_2,y_3+y_4)\cdot |y_1|^{1/2}|y_2|^{1/2}|y_3+y_4|^{1/2}|y_4|^{1/2}(2|y_3|+|y_4|)\chi_{4,3},
\end{split}
\end{equation*}
\begin{equation*}
\begin{split}
K_1^{8}(y_1,y_2,y_4,y_3)=-12A'(y_2,y_4,y_1+y_3)\cdot |y_2|^{1/2}|y_4|^{1/2}|y_3|^{1/2}|y_2+y_4|^{3/2},
\end{split}
\end{equation*}
\begin{equation*}
\begin{split}
K_1^{8}(y_2,y_1,y_4,y_3)=&+3A'(y_3,y_2,y_1+y_4)\cdot |y_3+y_2|^{1/2}|y_2|^{1/2}|y_3||y_4|\chi_{4,1}\\
&-3A'(y_1,y_4,y_3+y_2)\cdot |y_1|^{1/2}|y_4|^{1/2}|y_3+y_2|^{1/2}|y_3|^{1/2}|y_2|\chi_{2,3},
\end{split}
\end{equation*}
\begin{equation*}
\begin{split}
K_1^{8}(y_4,y_2,y_1,y_3)=&+3A'(y_3,y_4,y_1+y_2)\cdot |y_3+y_4|^{1/2}|y_4|^{1/2}|y_3||y_2|\chi_{2,1}\\
&-3A'(y_1,y_2,y_3+y_4)\cdot |y_1|^{1/2}|y_2|^{1/2}|y_3+y_4|^{1/2}|y_3|^{1/2}|y_4|\chi_{4,3}.
\end{split}
\end{equation*}
\begin{equation*}
\begin{split}
K_1^{8}(y_1,y_3,y_4,y_2)=&-3A'(y_3,y_4,y_1+y_2)\\
&\times|y_3|^{1/2}|y_3+y_4|^{1/2}|y_4|^{1/2}|y_2|^{1/2}\big\{|y_3+y_4|(4-2\chi_{2,1})+|y_2|\chi_{2,1}\big\},
\end{split}
\end{equation*}
\begin{equation*}
\begin{split}
K_1^{8}(y_3,y_1,&y_4,y_2)=+3A'(y_3,y_2,y_1+y_4)\cdot |y_3+y_2|^{1/2}|y_3|^{1/2}|y_2||y_4|\chi_{4,1}\\
&+3A'(y_1,y_4,y_3+y_2)\cdot |y_1|^{1/2}|y_4|^{1/2}|y_3+y_2|^{1/2}|y_2|^{1/2}(2|y_3|+|y_2|)\chi_{2,3},
\end{split}
\end{equation*}
\begin{equation*}
\begin{split}
K_1^{8}(y_4,y_1,y_3,y_2)=0,
\end{split}
\end{equation*}
\begin{equation*}
\begin{split}
K_1^{8}(y_2,y_3,y_4,y_1)=&+3A'(y_1,y_2,y_3+y_4)\cdot |y_3+y_4|^{3/2}|y_2|^{1/2}|y_4|\chi_{4,3}\\
&-3A'(y_3,y_4,y_1+y_2)\cdot |y_1|^{1/2}|y_3|^{1/2}|y_2||y_4|^{1/2}|y_3+y_4|^{1/2}\chi_{2,1},
\end{split}
\end{equation*}
\begin{equation*}
\begin{split}
K_1^{8}(y_3,y_2,y_4,y_1)=0,
\end{split}
\end{equation*}
\begin{equation*}
\begin{split}
K_1^{8}(y_4,y_2,y_3,y_1)=&+3A'(y_1,y_4,y_3+y_2)\cdot |y_3+y_2|^{3/2}|y_4|^{1/2}|y_2|\chi_{2,3}\\
&-3A'(y_3,y_2,y_1+y_4)\cdot |y_1|^{1/2}|y_3|^{1/2}|y_4||y_2|^{1/2}|y_3+y_2|^{1/2}\chi_{4,1},
\end{split}
\end{equation*}

Therefore, using \eqref{verf1},
\begin{equation}\label{Sk18}
4SK_1^{8}(y_1,y_2,y_3,y_4)=\sum_{l\in\{1,\ldots,5\}}4I_{1,l}(y_1,y_2,y_3,y_4)
\end{equation}
where
\begin{equation}\label{I11x}
\begin{split}
I_{1,1}(y_1,y_2,&y_3,y_4):=3A'(y_1,y_2,y_3+y_4)\cdot |y_2|^{1/2}|y_4|^{1/2}|y_3+y_4|^{1/2}\chi_{4,3}\\
&\times\big\{-|y_4|^{1/2}|y_3+y_4|+|y_1|^{1/2}(2|y_3|+|y_4|)-|y_1|^{1/2}|y_3|^{1/2}|y_4|^{1/2}\big\},
\end{split}
\end{equation}
\begin{equation}\label{I12x}
\begin{split}
I_{1,2}(y_1,y_2,&y_3,y_4):=-3A'(y_1,y_4,y_3+y_2)\cdot |y_2|^{1/2}|y_4|^{1/2}|y_3+y_2|^{1/2}\chi_{2,3}\\
&\times\big\{|y_2|^{1/2}|y_3+y_2|+|y_1|^{1/2}(2|y_3|+|y_2|)+|y_1|^{1/2}|y_3|^{1/2}|y_2|^{1/2}\big\},
\end{split}
\end{equation}
\begin{equation}\label{I13x}
\begin{split}
&I_{1,3}(y_1,y_2,y_3,y_4):=-3A'(y_3,y_2,y_1+y_4)\cdot |y_3|^{1/2}|y_3+y_2|^{1/2}|y_2|^{1/2}|y_4|^{1/2}\\
&\times\big\{4|y_3+y_2|+\chi_{4,1}\big(|y_4|-2|y_3+y_2|-|y_3|^{1/2}|y_4|^{1/2}+|y_2|^{1/2}|y_4|^{1/2}-|y_1|^{1/2}|y_4|^{1/2}\big)\big\},
\end{split}
\end{equation}
\begin{equation}\label{I14x}
\begin{split}
&I_{1,4}(y_1,y_2,y_3,y_4):=3A'(y_3,y_4,y_1+y_2)\cdot |y_3|^{1/2}|y_3+y_4|^{1/2}|y_4|^{1/2}|y_2|^{1/2}\\
&\times\big\{4|y_3+y_4|+\chi_{2,1}\big(|y_2|-2|y_3+y_4|+|y_3|^{1/2}|y_2|^{1/2}+|y_2|^{1/2}|y_4|^{1/2}+|y_1|^{1/2}|y_2|^{1/2}\big)\big\},
\end{split}
\end{equation}
\begin{equation}\label{I15x}
\begin{split}
I_{1,5}(y_1,y_2,y_3,y_4):=-3A'(y_2,y_4,y_1+y_3)\cdot 4|y_2|^{1/2}|y_4|^{1/2}|y_3|^{1/2}|y_2+y_4|^{3/2}.
\end{split}
\end{equation}

Using the formulas \eqref{A1ex}--\eqref{A6ex} we can rewrite, in terms of the variables $x_j=|y_j|^{1/2}$
\begin{equation}\label{I11xy}
\begin{split}
4I_{1,1}(y_1,y_2,&y_3,y_4)=-m_1^2\chi_{4,3}\cdot 2x_1x_4\big[-x_4(x_3^2+x_4^2)+x_1(2x_3^2+x_4^2-x_3x_4)\big]\\
&+m_1^2\chi_{4,3}\chi_{2,1}\cdot x_1x_4\big[-x_4(x_3^2+x_4^2)+x_1(2x_3^2+x_4^2-x_3x_4)\big]\\
&+m_{3+4}^2\chi_{4,3}\chi_{2,1}\cdot x_1x_4\big[-x_4(x_3^2+x_4^2)+x_1(2x_3^2+x_4^2-x_3x_4)\big],
\end{split}
\end{equation}
\begin{equation}\label{I12xy}
\begin{split}
4I_{1,2}(y_1,y_2,&y_3,y_4)=m_1^2\chi_{2,3}\cdot 2x_1x_2\big[x_2(x_3^2+x_2^2)+x_1(2x_3^2+x_2^2+x_3x_2)\big]\\
&-m_1^2\chi_{4,1}\chi_{2,3}\cdot x_1x_2\big[x_2(x_3^2+x_2^2)+x_1(2x_3^2+x_2^2+x_3x_2)\big]\\
&-m_{3+2}^2\chi_{4,1}\chi_{2,3}\cdot x_1x_2\big[x_2(x_3^2+x_2^2)+x_1(2x_3^2+x_2^2+x_3x_2)\big],
\end{split}
\end{equation}
\begin{equation}\label{I13xy}
\begin{split}
4I_{1,3}(y_1,y_2,&y_3,y_4)=m_{1+4}^2\cdot 8x_4(x_3^2+x_2^2)^2\\
&-m_3^2\chi_{2,3}\cdot 4x_4(x_3^2+x_2^2)^2-m_{1+4}^2\chi_{2,3}\cdot 4x_4(x_3^2+x_2^2)^2\\
&-m_{1+4}^2\chi_{4,1}\cdot 2x_4(x_3^2+x_2^2)(x_1x_4+x_3x_4+2x_3^2+2x_2^2-x_4^2-x_2x_4)\\
&+m_3^2\chi_{2,3}\chi_{4,1}\cdot x_4(x_3^2+x_2^2)(x_1x_4+x_3x_4+2x_3^2+2x_2^2-x_4^2-x_2x_4)\\
&+m_{1+4}^2\chi_{2,3}\chi_{4,1}\cdot x_4(x_3^2+x_2^2)(x_1x_4+x_3x_4+2x_3^2+2x_2^2-x_4^2-x_2x_4),
\end{split}
\end{equation}
\begin{equation}\label{I14xy}
\begin{split}
4I_{1,4}(y_1,y_2,&y_3,y_4)=-m_{1+2}^2\cdot 8x_2(x_3^2+x_4^2)^2\\
&+m_3^2\chi_{4,3}\cdot 4x_2(x_3^2+x_4^2)^2+m_{1+2}^2\chi_{4,3}\cdot 4x_2(x_3^2+x_4^2)^2\\
&-m_{1+2}^2\chi_{2,1}\cdot 2x_2(x_3^2+x_4^2)(x_1x_2+x_3x_2-2x_3^2-2x_4^2+x_2^2+x_2x_4)\\
&+m_3^2\chi_{2,1}\chi_{4,3}\cdot x_2(x_3^2+x_4^2)(x_1x_2+x_3x_2-2x_3^2-2x_4^2+x_2^2+x_2x_4)\\
&+m_{1+2}^2\chi_{2,1}\chi_{4,3}\cdot x_2(x_3^2+x_4^2)(x_1x_2+x_3x_2-2x_3^2-2x_4^2+x_2^2+x_2x_4),
\end{split}
\end{equation}
\begin{equation}\label{I15xy}
\begin{split}
4I_{1,5}(y_1,y_2,&y_3,y_4)=m_{1+3}^2\cdot 8x_3(x_2^2+x_4^2)^2.
\end{split}
\end{equation}

{\bf{Contributions of $K_1^{9}$}}. Using the formulas \eqref{ku114} and \eqref{verf6}--\eqref{verf8} we calculate 
\begin{equation*}
\begin{split}
K_1^{9}(y_1,y_2,y_3,y_4)=&-(1/2)B''(y_3+y_2,y_4,y_1)\cdot |y_2||y_3+y_2|^{1/2}\chi_{2,3}\\
&-(1/2)B'''(y_3,y_2,y_1+y_4)\cdot |y_4|^{1/2}\big\{|y_3+y_2|(4-2\chi_{4,1})+|y_4|\chi_{4,1}\big\},
\end{split}
\end{equation*}
\begin{equation*}
\begin{split}
K_1^{9}(y_2,y_1,y_3,y_4)=0,
\end{split}
\end{equation*}
\begin{equation*}
\begin{split}
K_1^{9}(y_3,y_2,y_1,y_4)=&-(1/2)B''(y_1+y_2,y_4,y_3)\cdot |y_2||y_3+y_4|^{1/2}\chi_{2,1}\\
&+(1/2)B'''(y_1,y_2,y_3+y_4)\cdot (2|y_3|+|y_4|)|y_4|^{1/2}\chi_{4,3},
\end{split}
\end{equation*}
\begin{equation*}
\begin{split}
K_1^{9}(y_1,y_2,y_4,y_3)=&-2B'''(y_2,y_4,y_1+y_3)\cdot|y_2+y_4||y_3|^{1/2},
\end{split}
\end{equation*}
\begin{equation*}
\begin{split}
K_1^{9}(y_2,y_1,y_4,y_3)=&-(1/2)B''(y_1+y_4,y_3,y_2)\cdot|y_4||y_3+y_2|^{1/2}\chi_{4,1}\\
&-(1/2)B'''(y_1,y_4,y_3+y_2)\cdot|y_2||y_3|^{1/2}\chi_{2,3},
\end{split}
\end{equation*}
\begin{equation*}
\begin{split}
K_1^{9}(y_4,y_2,y_1,y_3)=&-(1/2)B''(y_1+y_2,y_3,y_4)\cdot|y_2||y_3+y_4|^{1/2}\chi_{2,1}\\
&-(1/2)B'''(y_1,y_2,y_3+y_4)\cdot|y_4||y_3|^{1/2}\chi_{4,3},
\end{split}
\end{equation*}
\begin{equation*}
\begin{split}
K_1^{9}(y_1,y_3,y_4,y_2)=&-(1/2)B''(y_3+y_4,y_2,y_1)\cdot|y_4||y_3+y_4|^{1/2}\chi_{4,3}\\
&-(1/2)B'''(y_3,y_4,y_1+y_2)\cdot |y_2|^{1/2}\big\{|y_3+y_4|(4-2\chi_{2,1})+|y_2|\chi_{2,1}\big\},
\end{split}
\end{equation*}
\begin{equation*}
\begin{split}
K_1^{9}(y_3,y_1,y_4,y_2)=&-(1/2)B''(y_1+y_4,y_2,y_3)\cdot |y_4||y_3+y_2|^{1/2}\chi_{4,1}\\
&+(1/2)B'''(y_1,y_4,y_3+y_2)\cdot (2|y_3|+|y_2|)|y_2|^{1/2}\chi_{2,3},
\end{split}
\end{equation*}
\begin{equation*}
\begin{split}
K_1^{9}(y_4,y_1,y_3,y_2)=0,
\end{split}
\end{equation*}
\begin{equation*}
\begin{split}
K_1^{9}(y_2,y_3,y_4,y_1)=&-(1/2)B''(y_3+y_4,y_1,y_2)\cdot |y_4||y_3+y_4|^{1/2}\chi_{4,3}\\
&-(1/2)B'''(y_3,y_4,y_1+y_2)\cdot|y_2||y_1|^{1/2}\chi_{2,1},
\end{split}
\end{equation*}
\begin{equation*}
\begin{split}
K_1^{9}(y_3,y_2,y_4,y_1)=0,
\end{split}
\end{equation*}
\begin{equation*}
\begin{split}
K_1^{9}(y_4,y_2,y_3,y_1)=&-(1/2)B''(y_3+y_2,y_1,y_4)\cdot |y_2||y_3+y_2|^{1/2}\chi_{2,3}\\
&-(1/2)B'''(y_3,y_2,y_1+y_4)\cdot|y_4||y_1|^{1/2}\chi_{4,1}.
\end{split}
\end{equation*}
Therefore, using \eqref{verf1},
\begin{equation}\label{Sk110}
4SK_1^{9}(y_1,y_2,y_3,y_4)=\sum_{l\in\{1,\ldots,5\}}4J_{1,l}(y_1,y_2,y_3,y_4)
\end{equation}
where
\begin{equation}\label{Sk11x}
\begin{split}
4&J_{1,1}(y_1,y_2,y_3,y_4):=2B'''(y_1,y_2,y_3+y_4)\cdot |y_4|^{1/2}(2|y_3|+|y_4|-|y_4|^{1/2}|y_3|^{1/2})\chi_{4,3}\\
&+2\big[B''(y_3+y_4,y_2,y_1)+B''(y_3+y_4,y_1,y_2)\big]\cdot |y_4||y_3+y_4|^{1/2}\chi_{4,3},
\end{split}
\end{equation}\begin{equation}\label{Sk12x}
\begin{split}
4&J_{1,2}(y_1,y_2,y_3,y_4):=-2B'''(y_1,y_4,y_3+y_2)\cdot |y_2|^{1/2}(2|y_3|+|y_2|+|y_2|^{1/2}|y_3|^{1/2})\chi_{2,3}\\
&-2\big[B''(y_3+y_2,y_4,y_1)-B''(y_3+y_2,y_1,y_4)\big]\cdot |y_2||y_3+y_2|^{1/2}\chi_{2,3},
\end{split}
\end{equation}
\begin{equation}\label{Sk13x}
\begin{split}
4&J_{1,3}(y_1,y_2,y_3,y_4)\\
&:=-2B'''(y_3,y_2,y_1+y_4)\cdot |y_4|^{1/2}\big\{|y_3+y_2|(4-2\chi_{4,1})+(|y_4|-|y_4|^{1/2}|y_1|^{1/2})\chi_{4,1}\big\}\\
&+2\big[B''(y_1+y_4,y_2,y_3)-B''(y_1+y_4,y_3,y_2)\big]\cdot |y_4||y_3+y_2|^{1/2}\chi_{4,1},
\end{split}
\end{equation}
\begin{equation}\label{Sk14x}
\begin{split}
4&J_{1,4}(y_1,y_2,y_3,y_4)\\
&:=2B'''(y_3,y_4,y_1+y_2)\cdot |y_2|^{1/2}\big\{|y_3+y_4|(4-2\chi_{2,1})+(|y_2|+|y_2|^{1/2}|y_1|^{1/2})\chi_{2,1}\big\}\\
&-2\big[B''(y_1+y_2,y_4,y_3)+B''(y_1+y_2,y_3,y_4)\big]\cdot |y_2||y_3+y_4|^{1/2}\chi_{2,1},
\end{split}
\end{equation}
\begin{equation}\label{Sk15x}
\begin{split}
4J_{1,5}(y_1,y_2,y_3,y_4):=-8B'''(y_2,y_4,y_1+y_3)\cdot|y_2+y_4||y_3|^{1/2}.
\end{split}
\end{equation}

With $x_j:=|y_j|^{1/2}$ as before, we use \eqref{BB1ex} and \eqref{B1ex} to rewrite the formula \eqref{Sk11x} as 
\begin{equation}\label{Sk11xy}
\begin{split}
4J_{1,1}&(y_1,y_2,y_3,y_4)=m_2^2\chi_{4,3}\cdot 2x_2x_4\big[x_4(x_3^2+x_4^2)+x_2(2x_3^2+x_4^2-x_3x_4)\big]\\
&+m_{3+4}^2\chi_{4,3}\cdot 2x_4(x_3^2+x_4^2)\big[(x_2-x_1)x_4+2x_3^2+x_4^2-x_3x_4\big]\\
&+m_1m_{3+4}\chi_{2,1}\chi_{4,3}\cdot 2x_2^2x_4(2x_3^2+x_4^2-x_3x_4)\\
&+m_{1}^2\chi_{2,1}\chi_{4,3}\cdot x_1x_4\big[x_4(x_3^2+x_4^2)-x_1(2x_3^2+x_4^2-x_3x_4)\big]\\
&+m_{3+4}^2\chi_{2,1}\chi_{4,3}\cdot x_4\big[x_4(x_1-2x_2)(x_3^2+x_4^2)-x_1^2(2x_3^2+x_4^2-x_3x_4)\big].
\end{split}
\end{equation}
Using \eqref{BB2ex} and \eqref{B2ex} we can rewrite the formula \eqref{Sk12x} as 
\begin{equation}\label{Sk12xy}
\begin{split}
4J_{1,2}&(y_1,y_2,y_3,y_4)=-m_4^2\chi_{2,3}\cdot 2x_2x_4\big[x_2(x_3^2+x_2^2)+x_4(2x_3^2+x_2^2+x_3x_2)\big]\\
&-m_{3+2}^2\chi_{2,3}\cdot 2x_2(x_3^2+x_2^2)\big[(x_4+x_1)x_2+2x_3^2+x_2^2+x_3x_2\big]\\
&-m_1m_{3+2}\chi_{4,1}\chi_{2,3}\cdot 2x_2x_4^2(2x_3^2+x_2^2+x_3x_2)\\
&+m_{1}^2\chi_{4,1}\chi_{2,3}\cdot x_1x_2\big[x_2(x_3^2+x_2^2)+x_1(2x_3^2+x_2^2+x_3x_2)\big]\\
&+m_{3+2}^2\chi_{4,1}\chi_{2,3}\cdot x_2\big[x_2(x_1+2x_4)(x_3^2+x_2^2)+x_1^2(2x_3^2+x_2^2+x_3x_2)\big].
\end{split}
\end{equation}
Using \eqref{BB3ex} and \eqref{B3ex} we can rewrite the formula \eqref{Sk13x} as 
\begin{equation}\label{Sk13xy}
\begin{split}
&4J_{1,3}(y_1,y_2,y_3,y_4)=-m_2^2\cdot 8x_4x_2^2(x_3^2+x_2^2)-m_{3}^2\cdot 8x_4x_3^2(x_3^2+x_2^2)\\
&+m_3^2\chi_{2,3}\cdot 4x_4(x_3^2+x_2^2)^2+m_{1+4}^2\chi_{2,3}\cdot 4x_4(x_3^2+x_2^2)^2-m_3m_{1+4}\chi_{2,3}\cdot 8x_2^2x_4(x_3^2+x_2^2)\\
&+m_2^2\chi_{4,1}\cdot 2x_2x_4\big[x_2(x_1x_4-x_4^2+2x_3^2+2x_2^2)-x_4(x_3^2+x_2^2)\big]\\
&+m_3^2\chi_{4,1}\cdot 2x_4x_3\big[x_3(x_1x_4-x_4^2+2x_3^2+2x_2^2)+x_4(x_3^2+x_2^2)\big]\\
&-m_3^2\chi_{2,3}\chi_{4,1}\cdot x_4(x_3^2+x_2^2)\big[x_1x_4-x_4^2+2x_3^2+2x_2^2+x_3x_4-x_2x_4\big]\\
&-m_{1+4}^2\chi_{2,3}\chi_{4,1}\cdot x_4(x_3^2+x_2^2)\big[x_1x_4-x_4^2+2x_3^2+2x_2^2+x_3x_4+x_2x_4\big]\\
&+m_3m_{1+4}\chi_{2,3}\chi_{4,1}\cdot 2x_2^2x_4(x_1x_4-x_4^2+2x_3^2+2x_2^2).
\end{split}
\end{equation}
Using \eqref{BB4ex} and \eqref{B4ex} we can rewrite the formula \eqref{Sk14x} as 
\begin{equation}\label{Sk14xy}
\begin{split}
&4J_{1,4}(y_1,y_2,y_3,y_4)=m_4^2\cdot 8x_2x_4^2(x_3^2+x_4^2)+m_{3}^2\cdot 8x_2x_3^2(x_3^2+x_4^2)\\
&-m_3^2\chi_{4,3}\cdot 4x_2(x_3^2+x_4^2)^2-m_{1+2}^2\chi_{4,3}\cdot 4x_2(x_3^2+x_4^2)^2+m_3m_{1+2}\chi_{4,3}\cdot 8x_2x_4^2(x_3^2+x_4^2)\\
&+m_4^2\chi_{2,1}\cdot 2x_2x_4\big[x_4(x_1x_2+x_2^2-2x_3^2-2x_4^2)+x_2(x_3^2+x_4^2)\big]\\
&+m_3^2\chi_{2,1}\cdot 2x_2x_3\big[x_3(x_1x_2+x_2^2-2x_3^2-2x_4^2)+x_2(x_3^2+x_4^2)\big]\\
&-m_3^2\chi_{2,1}\chi_{4,3}\cdot x_2(x_3^2+x_4^2)\big[x_1x_2+x_2^2-2x_3^2-2x_4^2+x_2x_3+x_2x_4\big]\\
&-m_{1+2}^2\chi_{2,1}\chi_{4,3}\cdot x_2(x_3^2+x_4^2)\big[x_1x_2+x_2^2-2x_3^2-2x_4^2+x_2x_3-x_2x_4\big]\\
&+m_3m_{1+2}\chi_{2,1}\chi_{4,3}\cdot 2x_2x_4^2(x_1x_2+x_2^2-2x_3^2-2x_4^2).
\end{split}
\end{equation}
Finally, using \eqref{BB5ex} we can rewrite the formula \eqref{Sk15x} as 
\begin{equation}\label{Sk15xy}
4J_{1,5}(y_1,y_2,y_3,y_4):=-m_2^2\cdot 8x_3x_2^2(x_2^2+x_4^2)-m_4^2\cdot 8x_3x_4^2(x_2^2+x_4^2).
\end{equation}

{\bf{Contributions of $K_2^1$}}. Using \eqref{verf4}--\eqref{verf5.5} we calculate the components
\begin{equation*}
\begin{split}
K_2^1(y_1,y_2,y_3,y_4)=-m_1^2|y_1|^{1/2}|y_3|^{1/2}|y_2|^{1/2}|y_4|(2-\chi_{2+4}),
\end{split}
\end{equation*}
\begin{equation*}
\begin{split}
K_2^1(y_2,y_1,y_3,y_4)=0,
\end{split}
\end{equation*}
\begin{equation*}
\begin{split}
K_2^1(y_3,y_2,y_1,y_4)=m_3^2|y_1|^{1/2}|y_3|^{1/2}|y_4||y_2|^{1/2}\chi_{2+4},
\end{split}
\end{equation*}
\begin{equation*}
\begin{split}
&K_2^1(y_1,y_2,y_4,y_3)=-2m_1^2|y_1|^{1/2}|y_2|^{1/2}|y_4|^{1/2}|y_3|-m_1^2\chi_{2+4}|y_1|^{1/2}|y_2|^{1/2}|y_4|^{1/2}|y_2+y_4|\\
&+m_1^2|y_1|^{1/2}|y_2|^{1/2}|y_4|^{1/2}|y_3+y_4|\chi_{4,3}(2-\chi_{2,1})+m_1^2|y_1|^{1/2}|y_4|^{1/2}|y_2|^{1/2}|y_3+y_2|\chi_{2,3}(2-\chi_{4,1}),
\end{split}
\end{equation*}
\begin{equation*}
\begin{split}
K_2^1(y_2,y_1,y_4,y_3)=0,
\end{split}
\end{equation*}
\begin{equation*}
\begin{split}
K_2^1(y_4,y_2,y_1,y_3)=0.
\end{split}
\end{equation*}
\begin{equation*}
\begin{split}
K_2^1(y_1,y_3,y_4,y_2)=-m_1^2|y_1|^{1/2}|y_3|^{1/2}|y_4|^{1/2}|y_2|(2-\chi_{2+4}),
\end{split}
\end{equation*}
\begin{equation*}
\begin{split}
K_2^1(y_3,y_1,y_4,y_2)=m_3^2|y_1|^{1/2}|y_3|^{1/2}|y_2||y_4|^{1/2}\chi_{2+4},
\end{split}
\end{equation*}
\begin{equation*}
\begin{split}
K_2^1(y_4,y_1,y_3,y_2)=0,
\end{split}
\end{equation*}
\begin{equation*}
\begin{split}
K_2^1(y_2,y_3,y_4,y_1)=-2m_2^2|y_3|^{1/2}|y_2|^{3/2}|y_4|^{1/2},
\end{split}
\end{equation*}
\begin{equation*}
\begin{split}
K_2^1(&y_3,y_2,y_4,y_1)=-2m_3^2|y_3|^{3/2}|y_2|^{1/2}|y_4|^{1/2}-m_3^2\chi_{2+4}|y_3|^{1/2}|y_2|^{1/2}|y_4|^{1/2}|y_2+y_4|\\
&+2m_{3}^2|y_3|^{1/2}|y_2|^{1/2}|y_4|^{1/2}|y_3+y_2|\chi_{2,3}+2m_{3}^2|y_3|^{1/2}|y_4|^{1/2}|y_2|^{1/2}|y_3+y_4|\chi_{4,3}\\
&-m_3^2|y_3|^{1/2}|y_2|^{1/2}|y_4|^{1/2}|y_1+y_4|\chi_{4,1}\chi_{2,3}-m_3^2|y_3|^{1/2}|y_4|^{1/2}|y_2|^{1/2}|y_1+y_2|\chi_{2,1}\chi_{4,3}
\end{split}
\end{equation*}
\begin{equation*}
\begin{split}
K_2^1(y_4,y_2,y_3,y_1)=-2m_4^2|y_3|^{1/2}|y_4|^{3/2}|y_2|^{1/2},
\end{split}
\end{equation*}

Therefore, using \eqref{verf1},
\begin{equation}\label{Sk21}
4SK_2^1(y_1,y_2,y_3,y_4)=\sum_{l\in\{1,\ldots,4\}}4H_{2,l}(y_1,y_2,y_3,y_4),
\end{equation}
where
\begin{equation}\label{H21xy}
\begin{split}
4H_{2,1}(y_1,y_2,y_3,&y_4)=-m_1^2\cdot 8(x_1x_3x_2x_4^2-x_1x_3x_4x_2^2+x_1x_2x_4x_3^2)\\
&+m_1^2\chi_{2+4}\cdot 4[x_1x_3x_2x_4^2-x_1x_3x_4x_2^2-x_1x_2x_4(x_2^2+x_4^2)]\\
&+m_1^2\chi_{2,3}\cdot 8x_1x_2x_4(x_3^2+x_2^2)-m_1^2\chi_{2,3}\chi_{4,1}\cdot 4x_1x_2x_4(x_3^2+x_2^2)\\
&+m_1^2\chi_{4,3}\cdot 8x_1x_4x_2(x_3^2+x_4^2)-m_1^2\chi_{4,3}\chi_{2,1}\cdot 4x_1x_4x_2(x_3^2+x_4^2),
\end{split}
\end{equation}
\begin{equation}\label{H22xy}
\begin{split}
4H_{2,2}(y_1,y_2,y_3,&y_4)=m_2^2\cdot 8x_3x_4x_2^3,
\end{split}
\end{equation}
\begin{equation}\label{H23xy}
\begin{split}
4H_{2,3}(y_1,y_2,y_3,&y_4)=m_3^2\cdot 8x_3^3x_2x_4\\
&+m_3^2\chi_{2+4}\cdot 4\big[x_3x_2x_4(x_2^2+x_4^2)-x_1x_3x_4x_2^2+x_1x_3x_2x_4^2\big]\\
&-m_{3}^2\chi_{2,3}\cdot 8x_3x_2x_4(x_3^2+x_2^2)+m_3^2\chi_{4,1}\chi_{2,3}\cdot 4x_3x_2x_4(x_3^2+x_2^2)\\
&-m_{3}^2\chi_{4,3}\cdot 8x_3x_4x_2(x_3^2+x_4^2)+m_3^2\chi_{2,1}\chi_{4,3}\cdot 4x_3x_4x_2(x_3^2+x_4^2),
\end{split}
\end{equation}
\begin{equation}\label{H24xy}
\begin{split}
4H_{2,4}(y_1,y_2,y_3,&y_4)=m_4^2\cdot 8x_3x_2x_4^3.
\end{split}
\end{equation}

{\bf{Contributions of $K_2^2$}}.
From the formula \eqref{ku24}, and using the definition \eqref{ite3},
it is easy to see that there are no nontrivial components. Thus
\begin{equation}\label{Sk22}
4SK_2^2(y_1,y_2,y_3,y_4)=0.
\end{equation}

{\bf{Contributions of $K_2^{3}$}}. Using the formulas \eqref{ku25} and \eqref{verf6}--\eqref{verf8} we calculate 
\begin{equation*}
\begin{split}
K_2^{3}(y_1,y_2,y_3,y_4)=+6A'(y_1,y_4,y_3+y_2)\cdot |y_2|^{1/2}|y_3|^{1/2}|y_4|^{1/2}|y_3+y_2|^{3/2}(2-\chi_{2,3}),
\end{split}
\end{equation*}\begin{equation*}
\begin{split}
K_2^{3}(y_2,y_1,y_3,y_4)=0,
\end{split}
\end{equation*}
\begin{equation*}
\begin{split}
K_2^{3}(y_3,y_2,y_1,y_4)=&-6A'(y_3,y_4,y_1+y_2)\cdot |y_3+y_4|^{1/2}|y_4|^{1/2}|y_2|^{1/2}|y_1|^{1/2}|y_3|\chi_{2,1}\\
&-3A'(y_1,y_2,y_3+y_4)\cdot |y_3+y_4|^{1/2}|y_3|^{1/2}|y_4||y_2|\chi_{4,3},
\end{split}
\end{equation*}
\begin{equation*}
\begin{split}
K_2^{3}(y_1,y_2,y_4,y_3)=+12A'(y_1,y_3,y_2+y_4)\cdot |y_2|^{1/2}|y_3|^{1/2}||y_4|^{1/2}|y_2+y_4|^{3/2},
\end{split}
\end{equation*}
\begin{equation*}
\begin{split}
K_2^{3}(y_2,y_1,y_4,y_3)=&-6A'(y_3,y_2,y_1+y_4)\cdot |y_3+y_2|^{1/2}|y_3|^{1/2}|y_4|^{1/2}|y_1|^{1/2}|y_2|\chi_{4,1}\\
&+3A'(y_1,y_4,y_3+y_2)\cdot |y_3+y_2|^{1/2}|y_2|^{1/2}|y_4|(2|y_3|+|y_2|)\chi_{2,3},
\end{split}
\end{equation*}
\begin{equation*}
\begin{split}
K_2^{3}(y_4,y_2,y_1,y_3)=&-6A'(y_3,y_4,y_1+y_2)\cdot |y_3+y_4|^{1/2}|y_3|^{1/2}|y_2|^{1/2}|y_1|^{1/2}|y_4|\chi_{2,1}\\
&+3A'(y_1,y_2,y_3+y_4)\cdot |y_3+y_4|^{1/2}|y_4|^{1/2}|y_2|(2|y_3|+|y_4|)\chi_{4,3},
\end{split}
\end{equation*}
\begin{equation*}
\begin{split}
K_2^{3}(y_1,y_3,y_4,y_2)=+6A'(y_1,y_2,y_3+y_4)\cdot |y_2|^{1/2}|y_3|^{1/2}|y_4|^{1/2}|y_3+y_4|^{3/2}(2-\chi_{4,3}),
\end{split}
\end{equation*}
\begin{equation*}
\begin{split}
K_2^{3}(y_3,y_1,y_4,y_2)=&-6A'(y_3,y_2,y_1+y_4)\cdot |y_3+y_2|^{1/2}|y_2|^{1/2}|y_4|^{1/2}|y_1|^{1/2}|y_3|\chi_{4,1}\\
&-3A'(y_1,y_4,y_3+y_2)\cdot |y_3+y_2|^{1/2}|y_3|^{1/2}|y_2||y_4|\chi_{2,3},
\end{split}
\end{equation*}
\begin{equation*}
\begin{split}
K_2^{3}(y_4,y_1,y_3,y_2)=0,
\end{split}
\end{equation*}
\begin{equation*}
\begin{split}
&K_2^{3}(y_2,y_3,y_4,y_1)=K_2^{3}(y_3,y_2,y_4,y_1)=K_2^{3}(y_4,y_2,y_3,y_1)=0.
\end{split}
\end{equation*}

Therefore, using \eqref{verf1},
\begin{equation}\label{Sk23}
4SK_2^{3}(y_1,y_2,y_3,y_4)=\sum_{l\in\{1,2,3,4,6\}}4I_{2,l}(y_1,y_2,y_3,y_4)
\end{equation}
where
\begin{equation}\label{I21x}
\begin{split}
I_{2,1}(y_1,y_2,&y_3,y_4):=3A'(y_1,y_2,y_3+y_4)\cdot |y_3+y_4|^{1/2}|y_2|^{1/2}|y_4|^{1/2}\\
&\times\big\{-2|y_3+y_4||y_3|^{1/2}(2-\chi_{4,3})+\chi_{4,3}|y_2|^{1/2}(2|y_3|+|y_4|-|y_3|^{1/2}|y_4|^{1/2})\big\},
\end{split}
\end{equation}
\begin{equation}\label{I22x}
\begin{split}
I_{2,2}(y_1,y_2,&y_3,y_4):=3A'(y_1,y_4,y_3+y_2)\cdot |y_3+y_2|^{1/2}|y_2|^{1/2}|y_4|^{1/2}\\
&\times\big\{2|y_3+y_2||y_3|^{1/2}(2-\chi_{2,3})+\chi_{2,3}|y_4|^{1/2}(2|y_3|+|y_2|+|y_3|^{1/2}|y_2|^{1/2})\big\},
\end{split}
\end{equation}
\begin{equation}\label{I23x}
\begin{split}
I_{2,3}(y_1,y_2,y_3,y_4)&:=-6A'(y_3,y_2,y_1+y_4)\cdot |y_3+y_2|^{1/2}|y_3|^{1/2}|y_1|^{1/2}\\
&\qquad\qquad\times |y_4|^{1/2}|y_2|^{1/2}(|y_2|^{1/2}-|y_3|^{1/2})\chi_{4,1},
\end{split}
\end{equation}
\begin{equation}\label{I24x}
\begin{split}
I_{2,4}(y_1,y_2,y_3,y_4)&:=-6A'(y_3,y_4,y_1+y_2)\cdot |y_3+y_4|^{1/2}|y_3|^{1/2}|y_1|^{1/2}\\
&\qquad\qquad\times |y_2|^{1/2}|y_4|^{1/2}(|y_4|^{1/2}+|y_3|^{1/2})\chi_{2,1},
\end{split}
\end{equation}
\begin{equation}\label{I26x}
\begin{split}
I_{2,6}(y_1,y_2,y_3,y_4):=12A'(y_1,y_3,y_2+y_4)\cdot |y_2|^{1/2}|y_3|^{1/2}||y_4|^{1/2}|y_2+y_4|^{3/2}.
\end{split}
\end{equation}

Using the formulas \eqref{A1ex}--\eqref{A6ex} and letting $x_j=|y_j|^{1/2}$ as before, we can rewrite
\begin{equation}\label{I21xy}
\begin{split}
4I_{2,1}(y_1,y_2,&y_3,y_4)=m_1^2\cdot 8x_1x_3x_4(x_3^2+x_4^2)\\
&-m_1^2\chi_{2,1}\cdot 4x_1x_3x_4(x_3^2+x_4^2)-m_{3+4}^2\chi_{2,1}\cdot 4x_1x_3x_4(x_3^2+x_4^2)\\
&-m_1^2\chi_{4,3}\cdot 2x_1x_4[2x_3(x_3^2+x_4^2)+x_2(2x_3^2+x_4^2-x_3x_4)]\\
&+m_1^2\chi_{4,3}\chi_{2,1}\cdot x_1x_4[2x_3(x_3^2+x_4^2)+x_2(2x_3^2+x_4^2-x_3x_4)]\\
&+m_{3+4}^2\chi_{4,3}\chi_{2,1}\cdot x_1x_4[2x_3(x_3^2+x_4^2)+x_2(2x_3^2+x_4^2-x_3x_4)],
\end{split}
\end{equation}
\begin{equation}\label{I22xy}
\begin{split}
4I_{2,2}(y_1,y_2,&y_3,y_4)=-m_1^2\cdot 8x_1x_3x_2(x_3^2+x_2^2)\\
&+m_1^2\chi_{4,1}\cdot 4x_1x_3x_2(x_3^2+x_2^2)+m_{3+2}^2\chi_{4,1}\cdot 4x_1x_3x_2(x_3^2+x_2^2)\\
&-m_1^2\chi_{2,3}\cdot 2x_1x_2[-2x_3(x_3^2+x_2^2)+x_4(2x_3^2+x_2^2+x_3x_2)]\\
&+m_1^2\chi_{2,3}\chi_{4,1}\cdot x_1x_2[-2x_3(x_3^2+x_2^2)+x_4(2x_3^2+x_2^2+x_3x_2)]\\
&+m_{3+2}^2\chi_{2,3}\chi_{4,1}\cdot x_1x_2[-2x_3(x_3^2+x_2^2)+x_4(2x_3^2+x_2^2+x_3x_2)],
\end{split}
\end{equation}
\begin{equation}\label{I23xy}
\begin{split}
4I_{2,3}(y_1,y_2,&y_3,y_4)=-m_{1+4}^2\chi_{4,1}\cdot 4x_1x_4(x_3-x_2)(x_3^2+x_2^2)\\
&+m_3^2\chi_{4,1}\chi_{2,3}\cdot 2x_1x_4(x_3-x_2)(x_3^2+x_2^2)\\
&+m_{1+4}^2\chi_{4,1}\chi_{2,3}\cdot 2x_1x_4(x_3-x_2)(x_3^2+x_2^2),
\end{split}
\end{equation}
\begin{equation}\label{I24xy}
\begin{split}
4I_{2,4}(y_1,y_2,&y_3,y_4)=m_{1+2}^2\chi_{2,1}\cdot 4x_1x_2(x_3+x_4)(x_3^2+x_4^2)\\
&-m_3^2\chi_{2,1}\chi_{4,3}\cdot 2x_1x_2(x_3+x_4)(x_3^2+x_4^2)\\
&-m_{1+2}^2\chi_{2,1}\chi_{4,3}\cdot 2x_1x_2(x_3+x_4)(x_3^2+x_4^2),
\end{split}
\end{equation}
\begin{equation}\label{I26xy}
\begin{split}
4I_{2,6}(y_1,y_2,&y_3,y_4)=-m_1^2\cdot 8x_1x_2x_4(x_2^2+x_4^2)\\
&+m_1^2\chi_{2+4}\cdot 4x_1x_2x_4(x_2^2+x_4^2)+m_3^2\chi_{2+4}\cdot 4x_1x_2x_4(x_2^2+x_4^2).
\end{split}
\end{equation}

{\bf{Contributions of $K_2^{4}$}}. Using the formulas \eqref{ku26} and \eqref{verf6}--\eqref{verf8} we calculate 
\begin{equation*}
\begin{split}
K_2^{4}(y_1,y_2,y_3,y_4)=&-B''(y_3+y_2,y_1,y_4)\cdot |y_2|^{1/2}|y_3|^{1/2}|y_3+y_2|^{1/2}(2-\chi_{2,3})\\
&+(1/2)B''(y_3,y_2,y_1+y_4)\cdot |y_4||y_1|^{1/2}\chi_{4,1},
\end{split}
\end{equation*}
\begin{equation*}
\begin{split}
K_2^{4}(y_2,y_1,y_3,y_4)=0,
\end{split}
\end{equation*}
\begin{equation*}
\begin{split}
K_2^{4}(y_3,y_2,y_1,y_4)=&+B''(y_1+y_2,y_3,y_4)\cdot|y_2|^{1/2}|y_1|^{1/2}|y_3+y_4|^{1/2}\chi_{2,1}\\
&+(1/2)B''(y_1,y_2,y_3+y_4)\cdot|y_4||y_3|^{1/2}\chi_{4,3},
\end{split}
\end{equation*}
\begin{equation*}
\begin{split}
K_2^{4}(y_1,y_2,y_4,y_3)=&-2B''(y_2+y_4,y_1,y_3)\cdot|y_2|^{1/2}|y_4|^{1/2}|y_2+y_4|^{1/2},
\end{split}
\end{equation*}
\begin{equation*}
\begin{split}
K_2^{4}(y_2,y_1,y_4,y_3)=&+B''(y_1+y_4,y_2,y_3)\cdot|y_4|^{1/2}|y_1|^{1/2}|y_3+y_2|^{1/2}\chi_{4,1}\\
&-(1/2)B''(y_1,y_4,y_3+y_2)\cdot(2|y_3|+|y_2|)|y_2|^{1/2}\chi_{2,3},
\end{split}
\end{equation*}
\begin{equation*}
\begin{split}
K_2^{4}(y_4,y_2,y_1,y_3)=&+B''(y_1+y_2,y_4,y_3)\cdot|y_2|^{1/2}|y_1|^{1/2}|y_3+y_4|^{1/2}\chi_{2,1}\\
&-(1/2)B''(y_1,y_2,y_3+y_4)\cdot(2|y_3|+|y_4|)|y_4|^{1/2}\chi_{4,3},
\end{split}
\end{equation*}
\begin{equation*}
\begin{split}
K_2^{4}(y_1,y_3,y_4,y_2)=&-B''(y_3+y_4,y_1,y_2)\cdot|y_4|^{1/2}|y_3|^{1/2}|y_3+y_4|^{1/2}(2-\chi_{4,3})\\
&+(1/2)B''(y_3,y_4,y_1+y_2)\cdot|y_2||y_1|^{1/2}\chi_{2,1},
\end{split}
\end{equation*}
\begin{equation*}
\begin{split}
K_2^{4}(y_3,y_1,y_4,y_2)=&+B''(y_1+y_4,y_3,y_2)\cdot|y_4|^{1/2}|y_1|^{1/2}|y_3+y_2|^{1/2}\chi_{4,1}\\
&+(1/2)B''(y_1,y_4,y_3+y_2)\cdot|y_2||y_3|^{1/2}\chi_{2,3},
\end{split}
\end{equation*}
\begin{equation*}
\begin{split}
K_2^{4}(y_4,y_1,y_3,y_2)=0,
\end{split}
\end{equation*}
\begin{equation*}
\begin{split}
K_2^{4}(y_2,y_3,y_4,y_1)=&-B''(y_3+y_4,y_2,y_1)\cdot|y_3|^{1/2}|y_4|^{1/2}|y_3+y_4|^{1/2}(2-\chi_{4,3})\\
&+(1/2)B''(y_3,y_4,y_1+y_2)\cdot |y_2|^{1/2}\big\{|y_3+y_4|(4-2\chi_{2,1})+|y_2|\chi_{2,1}\big\},
\end{split}
\end{equation*}
\begin{equation*}
\begin{split}
K_2^{4}(y_3,y_2,y_4,y_1)=&-2B''(y_2+y_4,y_3,y_1)\cdot|y_2|^{1/2}|y_4|^{1/2}|y_2+y_4|^{1/2}\\
&+2B''(y_2,y_4,y_1+y_3)\cdot|y_2+y_4||y_3|^{1/2},
\end{split}
\end{equation*}
\begin{equation*}
\begin{split}
K_2^{4}(y_4,y_2,y_3,y_1)=&-B''(y_3+y_2,y_4,y_1)\cdot|y_3|^{1/2}|y_2|^{1/2}|y_3+y_2|^{1/2}(2-\chi_{2,3})\\
&+(1/2)B''(y_3,y_2,y_1+y_4)\cdot |y_4|^{1/2}\big\{|y_3+y_2|(4-2\chi_{4,1})+|y_4|\chi_{4,1}\big\}.
\end{split}
\end{equation*}

Therefore
\begin{equation}\label{Sk25}
4SK_2^{4}(y_1,y_2,y_3,y_4)=\sum_{l\in\{1,\ldots,6\}}4J_{2,l}(y_1,y_2,y_3,y_4)
\end{equation}
where
\begin{equation}\label{Sk21x}
\begin{split}
4&J_{2,1}(y_1,y_2,y_3,y_4):=-2B''(y_1,y_2,y_3+y_4)\cdot|y_4|^{1/2}\big\{2|y_3|+|y_4|-|y_3|^{1/2}|y_4|^{1/2}\big\}\chi_{4,3}\\
&+4[B''(y_3+y_4,y_1,y_2)+B''(y_3+y_4,y_2,y_1)]\cdot|y_4|^{1/2}|y_3|^{1/2}|y_3+y_4|^{1/2}(2-\chi_{4,3}),
\end{split}
\end{equation}
\begin{equation}\label{Sk22x}
\begin{split}
4&J_{2,2}(y_1,y_2,y_3,y_4):=-2B''(y_1,y_4,y_3+y_2)\cdot|y_2|^{1/2}\big\{2|y_3|+|y_2|+|y_3|^{1/2}|y_2|^{1/2}\big\}\chi_{2,3}\\
&-4[B''(y_3+y_2,y_1,y_4)-B''(y_3+y_2,y_4,y_1)]\cdot |y_2|^{1/2}|y_3|^{1/2}|y_3+y_2|^{1/2}(2-\chi_{2,3}),
\end{split}
\end{equation}
\begin{equation}\label{Sk23x}
\begin{split}
4&J_{2,3}(y_1,y_2,y_3,y_4)\\
&:=-2B''(y_3,y_2,y_1+y_4)\cdot |y_4|^{1/2}\big\{|y_3+y_2|(4-2\chi_{4,1})+(|y_4|-|y_4|^{1/2}|y_1|^{1/2})\chi_{4,1}\big\}\\
&-4[B''(y_1+y_4,y_3,y_2)-B''(y_1+y_4,y_2,y_3)]\cdot|y_4|^{1/2}|y_1|^{1/2}|y_3+y_2|^{1/2}\chi_{4,1},
\end{split}
\end{equation}
\begin{equation}\label{Sk24x}
\begin{split}
4&J_{2,4}(y_1,y_2,y_3,y_4)\\
&:=-2B''(y_3,y_4,y_1+y_2)\cdot |y_2|^{1/2}\big\{|y_3+y_4|(4-2\chi_{2,1})+(|y_2|+|y_2|^{1/2}|y_1|^{1/2})\chi_{2,1}\big\}\\
&+4[B''(y_1+y_2,y_3,y_4)+B''(y_1+y_2,y_4,y_3)]\cdot|y_2|^{1/2}|y_1|^{1/2}|y_3+y_4|^{1/2}\chi_{2,1},
\end{split}
\end{equation}
\begin{equation}\label{Sk25x}
\begin{split}
4J_{2,5}(y_1,y_2,y_3,y_4):=-8B''(y_2,y_4,y_1+y_3)\cdot|y_2+y_4||y_3|^{1/2},&
\end{split}
\end{equation}
\begin{equation}\label{Sk26x}
\begin{split}
4J_{2,6}(y_1,&y_2,y_3,y_4):=8[B''(y_2+y_4,y_3,y_1)-B''(y_2+y_4,y_1,y_3)]\cdot|y_2|^{1/2}|y_4|^{1/2}|y_2+y_4|^{1/2}.
\end{split}
\end{equation}

Using the identities \eqref{B1ex}--\eqref{B6ex} we rewrite these formulas in terms of the variables $x_j$ as
\begin{equation}\label{Sk21xy}
\begin{split}
4J_{2,1}(y_1,y_2,&y_3,y_4)=m_2^2\cdot 8x_3x_4x_2(x_3^2+x_4^2)+m_{3+4}^2\cdot 8x_3x_4(x_2-x_1)(x_3^2+x_4^2)\\
&+m_1^2\chi_{2,1}\cdot 4x_3x_4x_1(x_3^2+x_4^2)+m_{3+4}^2\chi_{2,1}\cdot 4x_3x_4(x_1-2x_2)(x_3^2+x_4^2)\\
&+m_2^2\chi_{4,3}\cdot 2x_2x_4\big[x_1(2x_3^2+x_4^2-x_3x_4)-2x_3(x_3^2+x_4^2)\big]\\
&+m_{3+4}^2\chi_{4,3}\cdot 4x_3x_4(x_1-x_2)(x_3^2+x_4^2)\\
&+m_{1}^2\chi_{2,1}\chi_{4,3}\cdot x_4x_1\big[x_2(2x_3^2+x_4^2-x_3x_4)-2x_3(x_3^2+x_4^2)\big]\\
&+m_{3+4}^2\chi_{2,1}\chi_{4,3}\cdot x_4\big[2x_3(2x_2-x_1)(x_3^2+x_4^2)-x_1x_2(2x_3^2+x_4^2-x_3x_4)\big],
\end{split}
\end{equation}
\begin{equation}\label{Sk22xy}
\begin{split}
4J_{2,2}(&y_1,y_2,y_3,y_4)=m_4^2\cdot 8x_3x_2x_4(x_3^2+x_2^2)+m_{3+2}^2\cdot 8x_3x_2(x_4+x_1)(x_3^2+x_2^2)\\
&-m_1^2\chi_{4,1}\cdot 4x_3x_2x_1(x_3^2+x_2^2)-m_{3+2}^2\chi_{4,1}\cdot 4x_3x_2(x_1+2x_4)(x_3^2+x_2^2)\\
&+m_4^2\chi_{2,3}\cdot 2x_2x_4\big[x_1(2x_3^2+x_2^2+x_3x_2)-2x_3(x_3^2+x_2^2)\big]\\
&-m_{3+2}^2\chi_{2,3}\cdot 4x_3x_2(x_1+x_4)(x_3^2+x_2^2)\\
&+m_{1}^2\chi_{2,3}\chi_{4,1}\cdot x_2x_1\big[x_4(2x_3^2+x_2^2+x_3x_2)+2x_3(x_3^2+x_2^2)\big]\\
&+m_{3+2}^2\chi_{2,3}\chi_{4,1}\cdot x_2\big[2x_3(2x_4+x_1)(x_3^2+x_2^2)-x_1x_4(2x_3^2+x_2^2+x_3x_2)\big],
\end{split}
\end{equation}
\begin{equation}\label{Sk23xy}
\begin{split}
4J_{2,3}(&y_1,y_2,y_3,y_4)=-m_2^2\cdot 8x_2x_3x_4(x_3^2+x_2^2)-m_{3}^2\cdot 8x_2x_3x_4(x_3^2+x_2^2)\\
&+m_3^2\chi_{2,3}\cdot 8x_2x_3x_4(x_3^2+x_2^2)\\
&+m_2^2\chi_{4,1}\cdot 2x_2x_4\big[x_3(x_1x_4-x_4^2+2x_3^2+2x_2^2)-2x_1(x_3^2+x_2^2)\big]\\
&+m_3^2\chi_{4,1}\cdot 2x_4x_3\big[x_2(x_1x_4-x_4^2+2x_3^2+2x_4^2)+2x_1(x_3^2+x_2^2)\big]\\
&-m_3^2\chi_{2,3}\chi_{4,1}\cdot 2x_4\big[x_3x_2(x_1x_4-x_4^2+2x_3^2+2x_2^2)+x_1(x_3-x_2)(x_3^2+x_2^2)\big]\\
&-m_{1+4}^2\chi_{2,3}\chi_{4,1}\cdot 2x_1x_4(x_3+x_2)(x_3^2+x_2^2),
\end{split}
\end{equation}
\begin{equation}\label{Sk24xy}
\begin{split}
4J_{2,4}(&y_1,y_2,y_3,y_4)=-m_4^2\cdot 8x_2x_3x_4(x_3^2+x_4^2)-m_{3}^2\cdot 8x_2x_3x_4(x_3^2+x_4^2)\\
&+m_3^2\chi_{4,3}\cdot 8x_2x_3x_4(x_3^2+x_4^2)\\
&-m_4^2\chi_{2,1}\cdot 2x_2x_4\big[x_3(x_1x_2+x_2^2-2x_3^2-2x_4^2)+2x_1(x_3^2+x_4^2)\big]\\
&-m_3^2\chi_{2,1}\cdot 2x_2x_3\big[x_4(x_1x_2+x_2^2-2x_3^2-2x_4^2)+2x_1(x_3^2+x_4^2)\big]\\
&+m_3^2\chi_{2,1}\chi_{4,3}\cdot 2x_2\big[x_3x_4(x_1x_2+x_2^2-2x_3^2-2x_4^2)+x_1(x_3+x_4)(x_3^2+x_4^2)\big]\\
&+m_{1+2}^2\chi_{2,1}\chi_{4,3}\cdot 2x_1x_2(x_3-x_4)(x_3^2+x_4^2),
\end{split}
\end{equation}
\begin{equation}\label{Sk25xy}
\begin{split}
4J_{2,5}(y_1,y_2,y_3,y_4)=-m_2^2\cdot 8x_2x_3x_4(x_2^2+x_4^2)-m_4^2\cdot 8x_2x_3x_4(x_2^2+x_4^2),
\end{split}
\end{equation}
\begin{equation}\label{Sk26xy}
\begin{split}
4J_{2,6}(y_1,&y_2,y_3,y_4)=m_3^2\cdot 8x_2x_3x_4(x_2^2+x_4^2)+m_{2+4}^2\cdot 8x_2x_4(x_1+x_3)(x_2^2+x_4^2)\\
&-m_3^2\chi_{2+4}\cdot 4x_2x_4(x_2^2+x_4^2)(x_1+2x_3)+m_1^2\chi_{2+4}\cdot 4x_1x_2x_4(x_2^2+x_4^2).
\end{split}
\end{equation}

\smallskip
\subsection{Proof of Proposition \ref{BFprop}}\label{uyt22secsymm}
We restrict now to the case of resonant frequencies and prove the identity \eqref{BFprop2}. In terms of the variables $x_j$ the two main conditions $y_1+y_2+y_3+y_4=0$ and $|y_1|^{1/2}+|y_2|^{1/2}-|y_3|^{1/2}-|y_4|^{1/2}=0$ show that
\begin{equation}\label{sumg3}
x_1=x_3+x_4-x_2\,\,\text{ and }\,\,x_2x_3+x_2x_4-x_3x_4=0.
\end{equation}

We now re-arrange the formulas in \S\ref{uyt22seccalc} above to calculate
\begin{equation}\label{sumg1}
\begin{split}
\sum_{a\in\{1,\ldots,9\}}&4SK_1^a(y_1,y_2,y_3,y_4)-\sum_{a\in\{1,\ldots,4\}}4SK_2^a(y_1,y_2,y_3,y_4)
\\
&=X_1m_1^2+X_2m_2^2+X_3m_3^2+X_4m_4^2+X_5m_{1+2}^2+X_6m_{1+4}^2+X_7m_{1+3}^2
\\
&\quad\,+Y_1m_1m_{3+2}+Y_2m_1m_{3+4}+Y_3m_3m_{1+2}+Y_4m_3m_{1+4}+Y_5m_1m_3.
\end{split}
\end{equation}
Notice how we are organizing the terms according to the factors $m_1,m_2$ and so on.
We will then organize each of these expressions using the cutoffs $\chi$, see for example \eqref{sumgx11}.

\smallskip
{\it The coefficient $X_1$}. 
This is obtained by examining the identities \eqref{H11xy}, \eqref{Sk16}, 
\eqref{Sk17}, \eqref{I11xy}, \eqref{I12xy}, \eqref{Sk11xy}, \eqref{Sk12xy}, \eqref{H21xy},
\eqref{I21xy}, \eqref{I22xy}, \eqref{I26xy}, \eqref{Sk21xy}, \eqref{Sk22xy}, \eqref{Sk26xy}:
\begin{equation}\label{sumgx11}
\begin{split}
X_1&=X_1^1+\chi_{2+4}\cdot X_1^2+\chi_{2,1}\cdot X_1^3+\chi_{4,3}\cdot X_1^4+\chi_{2,1}\chi_{4,3}\cdot X_1^5\\
&+\chi_{2,3}\cdot X_1^6+\chi_{4,1}\cdot X_1^7+\chi_{2,3}\chi_{4,1}\cdot X_1^8,
\end{split}
\end{equation}
where
\begin{equation*}
\begin{split}
X_1^1&=8(x_1x_3x_2x_4^2-x_1x_3x_4x_2^2+x_1x_2x_4x_3^2)-8x_1x_3x_4(x_3^2+x_4^2)\\
&+8x_1x_3x_2(x_3^2+x_2^2)+8x_1x_2x_4(x_2^2+x_4^2),
\end{split}
\end{equation*}
\begin{equation*}
\begin{split}
X_1^2&=4x_1^2x_2x_4(x_4-x_2)+8x_1x_2^2x_4^2-4[x_1x_3x_2x_4^2-x_1x_3x_4x_2^2-x_1x_2x_4(x_2^2+x_4^2)]\\
&-4x_1x_2x_4(x_2^2+x_4^2)-4x_1x_2x_4(x_2^2+x_4^2),
\end{split}
\end{equation*}
\begin{equation*}
\begin{split}
X_1^3&=4x_1x_3x_4(x_3^2+x_4^2)-4x_3x_4x_1(x_3^2+x_4^2),
\end{split}
\end{equation*}
\begin{equation*}
\begin{split}
X_1^4&=-2x_1x_4\big[-x_4(x_3^2+x_4^2)+x_1(2x_3^2+x_4^2-x_3x_4)\big]-8x_1x_2x_4(x_3^2+x_4^2)\\
&+2x_1x_4[2x_3(x_3^2+x_4^2)+x_2(2x_3^2+x_4^2-x_3x_4)],
\end{split}
\end{equation*}
\begin{equation*}
\begin{split}
X_1^5&=-2x_1^2x_4^2x_2-2x_2^2x_4^2x_1+x_1x_4\big[-x_4(x_3^2+x_4^2)+x_1(2x_3^2+x_4^2-x_3x_4)\big]\\
&+x_1x_4\big[x_4(x_3^2+x_4^2)-x_1(2x_3^2+x_4^2-x_3x_4)\big]+4x_1x_2x_4(x_3^2+x_4^2)\\
&-x_1x_4[2x_3(x_3^2+x_4^2)+x_2(2x_3^2+x_4^2-x_3x_4)]\\
&-x_4x_1\big[x_2(2x_3^2+x_4^2-x_3x_4)-2x_3(x_3^2+x_4^2)\big],
\end{split}
\end{equation*}
\begin{equation*}
\begin{split}
X_1^6&=2x_1x_2\big[x_2(x_3^2+x_2^2)+x_1(2x_3^2+x_2^2+x_3x_2)\big]-8x_1x_2x_4(x_3^2+x_2^2)\\
&+2x_1x_2[-2x_3(x_3^2+x_2^2)+x_4(2x_3^2+x_2^2+x_3x_2)],
\end{split}
\end{equation*}
\begin{equation*}
\begin{split}
X_1^7&=-4x_1x_3x_2(x_3^2+x_2^2)+4x_3x_2x_1(x_3^2+x_2^2),
\end{split}
\end{equation*}
\begin{equation*}
\begin{split}
X_1^8&=2x_1^2x_2^2x_4-2x_2^2x_4^2x_1-x_1x_2\big[x_2(x_3^2+x_2^2)+x_1(2x_3^2+x_2^2+x_3x_2)\big]\\
&+x_1x_2\big[x_2(x_3^2+x_2^2)+x_1(2x_3^2+x_2^2+x_3x_2)\big]+4x_1x_2x_4(x_3^2+x_2^2)\\
&-x_1x_2[-2x_3(x_3^2+x_2^2)+x_4(2x_3^2+x_2^2+x_3x_2)]\\
&-x_2x_1\big[x_4(2x_3^2+x_2^2+x_3x_2)+2x_3(x_3^2+x_2^2)\big].
\end{split}
\end{equation*}
Clearly, $X_1^3=X_1^7=0$. Also, using the identities \eqref{sumg3} it is easy to see that $X_1^1=X_1^2=X_1^4=X_1^5=X_1^6=X_1^8=0$; for example
\begin{equation*}
\begin{split}
X_1^1&=8x_1(x_2^2+x_3^2+x_4^2)(x_2x_4+x_2x_3-x_3x_4)=0.
\end{split}
\end{equation*}
Therefore $X_1=0$ as desired.

\smallskip
{\it The coefficient $X_2$}. 
The coefficient $X_2$ in \eqref{sumg1} is obtained by examining the identities \eqref{H12xy}, \eqref{Sk11xy}, 
\eqref{Sk13xy}, \eqref{Sk15xy}, \eqref{H22xy}, \eqref{Sk21xy}, \eqref{Sk23xy}, \eqref{Sk25xy},
\begin{equation}\label{sumgx21}
\begin{split}
X_2&=X_2^1+\chi_{4,1}\cdot X_2^2+\chi_{4,3}\cdot X_2^3,
\end{split}
\end{equation}
where
\begin{equation*}
\begin{split}
X_2^1&=8x_2^2x_3x_4(x_3+x_4)-8x_4x_2^2(x_3^2+x_2^2)-8x_3x_2^2(x_2^2+x_4^2)\\
&-8x_3x_4x_2^3-8x_3x_4x_2(x_3^2+x_4^2)+8x_2x_3x_4(x_3^2+x_2^2)+8x_2x_3x_4(x_2^2+x_4^2),
\end{split}
\end{equation*}
\begin{equation*}
\begin{split}
X_2^2&=-4x_3x_2^2x_4^2+2x_2x_4\big[x_2(x_1x_4-x_4^2+2x_3^2+2x_2^2)-x_4(x_3^2+x_2^2)\big]\\
&-2x_2x_4\big[x_3(x_1x_4-x_4^2+2x_3^2+2x_2^2)-2x_1(x_3^2+x_2^2)\big],
\end{split}
\end{equation*}
\begin{equation*}
\begin{split}
X_2^3&=-8(x_3^2+x_4^2)x_2^2x_4+2x_2x_4\big[x_4(x_3^2+x_4^2)+x_2(2x_3^2+x_4^2-x_3x_4)\big]\\
&-2x_2x_4\big[x_1(2x_3^2+x_4^2-x_3x_4)-2x_3(x_3^2+x_4^2)\big],
\end{split}
\end{equation*}
Using again the identities \eqref{sumg3} we verify easily that $X_2^1=X_2^2=X_2^3=0$, thus $X_2=0$ as desired.

\smallskip
{\it The coefficient $X_3$}. 
The coefficient $X_3$ is obtained by examining the identities \eqref{H13xy}, \eqref{Sk16}, \eqref{Sk17}, \eqref{I13xy}, \eqref{I14xy}, \eqref{Sk13xy}, \eqref{Sk14xy}, \eqref{H23xy}, \eqref{I23xy}, \eqref{I24xy}, \eqref{I26xy}, \eqref{Sk23xy}, \eqref{Sk24xy}, \eqref{Sk26xy},
\begin{equation}\label{sumgx31}
\begin{split}
X_3&=X_3^1+\chi_{2+4}\cdot X_3^2+\chi_{2,1}\cdot X_3^3+\chi_{4,3}\cdot X_3^4+\chi_{2,1}\chi_{4,3}\cdot X_3^5\\
&+\chi_{2,3}\cdot X_3^6+\chi_{4,1}\cdot X_3^7+\chi_{2,3}\chi_{4,1}\cdot X_3^8
\end{split}
\end{equation}
where
\begin{equation*}
\begin{split}
X_3^1&=8x_3^2x_2x_4(x_2-x_4)-8x_4x_3^2(x_3^2+x_2^2)+8x_2x_3^2(x_3^2+x_4^2)-8x_3^3x_2x_4\\
&+8x_2x_3x_4(x_3^2+x_2^2)+8x_2x_3x_4(x_3^2+x_4^2)-8x_2x_3x_4(x_2^2+x_4^2),
\end{split}
\end{equation*}
\begin{equation*}
\begin{split}
X_3^2&=-4x_3^2x_2x_4(x_2-x_4)-8x_3x_2^2x_4^2-4\big[x_3x_2x_4(x_2^2+x_4^2)-x_1x_3x_4x_2^2+x_1x_3x_2x_4^2\big]\\
&-4x_1x_2x_4(x_2^2+x_4^2)+4x_2x_4(x_2^2+x_4^2)(x_1+2x_3),
\end{split}
\end{equation*}
\begin{equation*}
\begin{split}
X_3^3&=-4x_3^2x_2^2x_4+2x_2x_3\big[x_3(x_1x_2+x_2^2-2x_3^2-2x_4^2)+x_2(x_3^2+x_4^2)\big]\\
&+2x_2x_3\big[x_4(x_1x_2+x_2^2-2x_3^2-2x_4^2)+2x_1(x_3^2+x_4^2)\big],
\end{split}
\end{equation*}
\begin{equation*}
\begin{split}
X_3^4&=4x_2(x_3^2+x_4^2)^2-4x_2(x_3^2+x_4^2)^2+8x_3x_4x_2(x_3^2+x_4^2)-8x_2x_3x_4(x_3^2+x_4^2),
\end{split}
\end{equation*}
\begin{equation*}
\begin{split}
X_3^5&=2x_3^2x_2^2x_4+2x_2^2x_4^2x_3+x_2(x_3^2+x_4^2)(x_1x_2+x_3x_2-2x_3^2-2x_4^2+x_2^2+x_2x_4)\\
&-x_2(x_3^2+x_4^2)\big[x_1x_2+x_2^2-2x_3^2-2x_4^2+x_2x_3+x_2x_4\big]\\
&-4x_3x_4x_2(x_3^2+x_4^2)+2x_1x_2(x_3+x_4)(x_3^2+x_4^2)\\
&-2x_2\big[x_3x_4(x_1x_2+x_2^2-2x_3^2-2x_4^2)+x_1(x_3+x_4)(x_3^2+x_4^2)\big],
\end{split}
\end{equation*}
\begin{equation*}
\begin{split}
X_3^6&=-4x_4(x_3^2+x_2^2)^2+4x_4(x_3^2+x_2^2)^2+8x_3x_2x_4(x_3^2+x_2^2)-8x_2x_3x_4(x_3^2+x_2^2),
\end{split}
\end{equation*}
\begin{equation*}
\begin{split}
X_3^7&=4x_3^2x_4^2x_2+2x_4x_3\big[x_3(x_1x_4-x_4^2+2x_3^2+2x_2^2)+x_4(x_3^2+x_2^2)\big]\\
&- 2x_4x_3\big[x_2(x_1x_4-x_4^2+2x_3^2+2x_4^2)+2x_1(x_3^2+x_2^2)\big],
\end{split}
\end{equation*}
\begin{equation*}
\begin{split}
X_3^8&=-2x_3^2x_4^2x_2+2x_2^2x_4^2x_3+x_4(x_3^2+x_2^2)(x_1x_4+x_3x_4+2x_3^2+2x_2^2-x_4^2-x_2x_4)\\
&-x_4(x_3^2+x_2^2)\big[x_1x_4-x_4^2+2x_3^2+2x_2^2+x_3x_4-x_2x_4\big]\\
&-4x_3x_2x_4(x_3^2+x_2^2)-2x_1x_4(x_3-x_2)(x_3^2+x_2^2)\\
&+2x_4\big[x_3x_2(x_1x_4-x_4^2+2x_3^2+2x_2^2)+x_1(x_3-x_2)(x_3^2+x_2^2)\big].
\end{split}
\end{equation*}
Using again the identities \eqref{sumg3} we verify that $X_3^j=0$, $j\in\{1,\ldots,8\}$, thus $X_3=0$ as desired.

\smallskip
{\it The coefficient $X_4$}. 
The coefficient $X_4$ is obtained by examining the identities \eqref{H14xy}, 
\eqref{Sk12xy}, \eqref{Sk14xy}, \eqref{Sk15xy}, \eqref{H24xy}, \eqref{Sk22xy}, \eqref{Sk24xy}, \eqref{Sk25xy},
\begin{equation}\label{sumgx41}
\begin{split}
X_4&=X_4^1+\chi_{2,1}\cdot X_4^2+\chi_{2,3}\cdot X_4^3,
\end{split}
\end{equation}
where
\begin{equation*}
\begin{split}
X_4^1&=8x_4^2x_3x_2(x_2-x_3)+8x_2x_4^2(x_3^2+x_4^2)-8x_3x_4^2(x_2^2+x_4^2)-8x_3x_2x_4^3\\
&-8x_3x_2x_4(x_3^2+x_2^2)+8x_2x_3x_4(x_3^2+x_4^2)+8x_2x_3x_4(x_2^2+x_4^2),
\end{split}
\end{equation*}
\begin{equation*}
\begin{split}
X_4^2&=-4x_3x_2^2x_4^2+2x_2x_4\big[x_4(x_1x_2+x_2^2-2x_3^2-2x_4^2)+x_2(x_3^2+x_4^2)\big]\\
&+2x_2x_4\big[x_3(x_1x_2+x_2^2-2x_3^2-2x_4^2)+2x_1(x_3^2+x_4^2)\big],
\end{split}
\end{equation*}
\begin{equation*}
\begin{split}
X_4^3&=8(x_3^2+x_2^2)x_4^2x_2-2x_2x_4\big[x_2(x_3^2+x_2^2)+x_4(2x_3^2+x_2^2+x_3x_2)\big]\\
&-2x_2x_4\big[x_1(2x_3^2+x_2^2+x_3x_2)-2x_3(x_3^2+x_2^2)\big].
\end{split}
\end{equation*}
Using again the identities \eqref{sumg3} we verify that $X_4^1=X_4^2=X_4^3=0$, thus $X_4=0$ as desired.

\smallskip
{\it The coefficient $X_5$}. 
Since $m_{1+2}=m_{3+4}$, the coefficient $X_5$ is obtained by examining the identities \eqref{Sk13}, \eqref{I11xy}, \eqref{I14xy}, \eqref{Sk11xy}, \eqref{Sk14xy}, \eqref{I21xy}, \eqref{I24xy}, \eqref{Sk21xy}, \eqref{Sk24xy},
\begin{equation}\label{sumgx51}
\begin{split}
X_5&=X_5^1+\chi_{2,1}\cdot X_5^2+\chi_{4,3}\cdot X_5^3+\chi_{2,1}\chi_{4,3}\cdot X_5^4,
\end{split}
\end{equation}
where
\begin{equation*}
\begin{split}
X_5^1&=-8x_2(x_3^2+x_4^2)^2-8x_3x_4(x_2-x_1)(x_3^2+x_4^2),
\end{split}
\end{equation*}
\begin{equation*}
\begin{split}
X_5^2&=-2x_2(x_3^2+x_4^2)(x_1x_2+x_3x_2-2x_3^2-2x_4^2+x_2^2+x_2x_4)+4x_1x_3x_4(x_3^2+x_4^2)\\
&-4x_1x_2(x_3+x_4)(x_3^2+x_4^2)-4x_3x_4(x_1-2x_2)(x_3^2+x_4^2),
\end{split}
\end{equation*}
\begin{equation*}
\begin{split}
X_5^3&=4x_2x_4^2(x_3^2+x_4^2)+4x_2(x_3^2+x_4^2)^2+2x_4(x_3^2+x_4^2)\big[(x_2-x_1)x_4+2x_3^2+x_4^2-x_3x_4\big]\\
&-4x_2(x_3^2+x_4^2)^2-4x_3x_4(x_1-x_2)(x_3^2+x_4^2),
\end{split}
\end{equation*}
\begin{equation*}
\begin{split}
X_5^4&=2x_2x_4(x_3^2+x_4^2)(x_2-x_4)+x_1x_4\big[-x_4(x_3^2+x_4^2)+x_1(2x_3^2+x_4^2-x_3x_4)\big]\\
&+x_2(x_3^2+x_4^2)(x_1x_2+x_3x_2-2x_3^2-2x_4^2+x_2^2+x_2x_4)\\
&-x_1x_4[2x_3(x_3^2+x_4^2)+x_2(2x_3^2+x_4^2-x_3x_4)]+2x_1x_2(x_3+x_4)(x_3^2+x_4^2)\\
&+x_4\big[x_4(x_1-2x_2)(x_3^2+x_4^2)-x_1^2(2x_3^2+x_4^2-x_3x_4)\big]\\
&-x_2(x_3^2+x_4^2)\big[x_1x_2+x_2^2-2x_3^2-2x_4^2+x_2x_3-x_2x_4\big]\\
&-x_4\big[2x_3(2x_2-x_1)(x_3^2+x_4^2)-x_1x_2(2x_3^2+x_4^2-x_3x_4)\big]-2x_1x_2(x_3-x_4)(x_3^2+x_4^2).
\end{split}
\end{equation*}
Using the identities \eqref{sumg3} we verify that $X_5^1=X_5^2=X_5^3=X_5^4=0$, thus $X_5=0$ as desired.

\smallskip
{\it The coefficient $X_6$}. 
Similarly, since $m_{1+4}=m_{3+2}$, the coefficient $X_6$ is obtained by examining the identities \eqref{Sk13}, \eqref{I12xy}, \eqref{I13xy}, \eqref{Sk12xy}, \eqref{Sk13xy}, \eqref{I22xy}, \eqref{I23xy}, \eqref{Sk22xy}, \eqref{Sk23xy},
\begin{equation}\label{sumgx61}
\begin{split}
X_6&=X_6^1+\chi_{4,1}\cdot X_6^2+\chi_{2,3}\cdot X_6^3+\chi_{4,1}\chi_{2,3}\cdot X_6^4,
\end{split}
\end{equation}
where
\begin{equation*}
\begin{split}
X_6^1&=8x_4(x_3^2+x_2^2)^2-8x_3x_2(x_4+x_1)(x_3^2+x_2^2),
\end{split}
\end{equation*}
\begin{equation*}
\begin{split}
X_6^2&=-2x_4(x_3^2+x_2^2)(x_1x_4+x_3x_4+2x_3^2+2x_2^2-x_4^2-x_2x_4)-4x_1x_3x_2(x_3^2+x_2^2)\\
&+4x_1x_4(x_3-x_2)(x_3^2+x_2^2)+4x_3x_2(x_1+2x_4)(x_3^2+x_2^2),
\end{split}
\end{equation*}
\begin{equation*}
\begin{split}
X_6^3&=-4x_2^2x_4(x_3^2+x_2^2)-4x_4(x_3^2+x_2^2)^2-2x_2(x_3^2+x_2^2)\big[(x_4+x_1)x_2+2x_3^2+x_2^2+x_3x_2\big]\\
&+4x_4(x_3^2+x_2^2)^2+4x_3x_2(x_1+x_4)(x_3^2+x_2^2),
\end{split}
\end{equation*}
\begin{equation*}
\begin{split}
X_6^4&=-2x_2x_4(x_3^2+x_2^2)(x_4-x_2)-x_1x_2\big[x_2(x_3^2+x_2^2)+x_1(2x_3^2+x_2^2+x_3x_2)\big]\\
&+x_4(x_3^2+x_2^2)(x_1x_4+x_3x_4+2x_3^2+2x_2^2-x_4^2-x_2x_4)\\
&-x_1x_2[-2x_3(x_3^2+x_2^2)+x_4(2x_3^2+x_2^2+x_3x_2)]-2x_1x_4(x_3-x_2)(x_3^2+x_2^2)\\
&+x_2\big[x_2(x_1+2x_4)(x_3^2+x_2^2)+x_1^2(2x_3^2+x_2^2+x_3x_2)\big]\\
&-x_4(x_3^2+x_2^2)\big[x_1x_4-x_4^2+2x_3^2+2x_2^2+x_3x_4+x_2x_4\big]\\
&-x_2\big[2x_3(2x_4+x_1)(x_3^2+x_2^2)-x_1x_4(2x_3^2+x_2^2+x_3x_2)\big]+2x_1x_4(x_3+x_2)(x_3^2+x_2^2).
\end{split}
\end{equation*}
Using again the identities \eqref{sumg3} we verify that $X_6^1=X_6^2=X_6^3=X_6^4=0$, thus $X_6=0$ as desired.

\smallskip
{\it The coefficient $X_7$}. 
The coefficient $X_7$ is obtained by examining the identities \eqref{I15xy}, \eqref{Sk26xy}, 
\begin{equation*}
\begin{split}
X_7&=8x_3(x_2^2+x_4^2)^2-8x_2x_4(x_1+x_3)(x_2^2+x_4^2)=0,
\end{split}
\end{equation*}
where we used \eqref{sumg3} in the last line.

\smallskip
{\it The coefficients $Y_j$, $j=1,\dots 5$}. 
 The coefficient $Y_1$ is obtained by examining the identities \eqref{Sk12}, \eqref{Sk15}, \eqref{Sk16}, \eqref{Sk12xy},
\begin{equation*}
\begin{split}
Y_1&=4x_4^2(x_3^2+x_2^2)x_2\cdot\chi_{2,3}\chi_{4,1}-2x_2^2x_4^3\cdot\chi_{2,3}\chi_{4,1}+2x_2^2x_4^2x_1\cdot\chi_{2,3}\chi_{4,1}\\
&\qquad\qquad-\chi_{4,1}\chi_{2,3}\cdot 2x_2x_4^2(2x_3^2+x_2^2+x_3x_2)=0.
\end{split}
\end{equation*}

The coefficient $Y_2$ is obtained by examining the identities \eqref{Sk12}, \eqref{Sk15}, \eqref{Sk16}, \eqref{Sk11xy},
\begin{equation*}
\begin{split}
Y_2&=-4x_2^2(x_3^2+x_4^2)x_4\cdot\chi_{4,3}\chi_{2,1}+2x_4^2x_2^3\cdot\chi_{4,3}\chi_{2,1}+2x_2^2x_4^2x_1\cdot\chi_{4,3}\chi_{2,1}\\
&\qquad\qquad+\chi_{2,1}\chi_{4,3}\cdot 2x_2^2x_4(2x_3^2+x_4^2-x_3x_4)=0.
\end{split}
\end{equation*}

The coefficient $Y_3$ is obtained by examining the identities \eqref{Sk12}, \eqref{Sk15}, \eqref{Sk16}, \eqref{Sk14xy},
\begin{equation*}
\begin{split}
Y_3&=4x_4^2(x_3^2+x_4^2)x_2\cdot\chi_{4,3}(\chi_{2,1}-2)-2x_2^2x_4^3\cdot\chi_{2,1}\chi_{4,3}-2x_2^2x_4^2x_3\cdot\chi_{2,1}\chi_{4,3}\\
&\qquad\qquad+\chi_{4,3}\cdot 8x_2x_4^2(x_3^2+x_4^2)+\chi_{2,1}\chi_{4,3}\cdot 2x_2x_4^2(x_1x_2+x_2^2-2x_3^2-2x_4^2)=0.
\end{split}
\end{equation*}

The coefficient $Y_4$ is obtained by examining the identities \eqref{Sk12}, \eqref{Sk15}, \eqref{Sk16}, \eqref{Sk14xy},
\begin{equation*}
\begin{split}
Y_4&=-4x_2^2(x_3^2+x_2^2)x_4\cdot\chi_{2,3}(\chi_{4,1}-2)+2x_4^2x_2^3\cdot\chi_{4,1}\chi_{2,3}-2x_2^2x_4^2x_3\cdot\chi_{2,3}\chi_{4,1}\\
&\qquad\qquad-\chi_{2,3}\cdot 8x_2^2x_4(x_3^2+x_2^2)+\chi_{2,3}\chi_{4,1}\cdot 2x_2^2x_4(x_1x_4-x_4^2+2x_3^2+2x_2^2)=0.
\end{split}
\end{equation*}

The coefficient $Y_5$ is obtained by examining the identities \eqref{Sk17}, \eqref{Sk14},
\begin{equation*}
\begin{split}
Y_5&=-8x_2^2x_4^2(x_1-x_3)\chi_{2+4}+8\chi_{2+4}(x_2^2x_4^3-x_4^2x_2^3)=0.
\end{split}
\end{equation*}

Therefore we verified that $X_1=\ldots=X_7=Y_1=\ldots=Y_5=0$, and the conclusion of Proposition \ref{BFprop} follows from \eqref{sumg1}.

\end{document}